\documentclass[11pt]{amsart}   	
\usepackage[top=1in,bottom=1in,left=1in,right=1in]{geometry}                		

\usepackage{xcolor}

\usepackage[unicode=true,colorlinks=true,linkcolor=blue,urlcolor=blue, citecolor=blue]{hyperref}
\usepackage{multirow}
\usepackage{array}

\usepackage{cellspace}
\usepackage{todonotes}

\usepackage{enumerate}
\usepackage{amssymb,amsfonts,amsthm,amsmath, mathrsfs}
\usepackage{thmtools}

\usepackage{graphicx}
\usepackage{xypic}
\usepackage{hhline}

\usepackage{colordvi}
\usepackage{caption}

\usepackage{amscd}
\usepackage{xcolor}

\usepackage{tikz}
\usetikzlibrary{calc}
\usetikzlibrary{fit}
\usetikzlibrary{shapes}
\tikzset{cross/.style={cross out, draw=black, minimum size=2*(#1-\pgflinewidth), inner sep=0pt, outer sep=0pt}, cross/.default={1pt}}
\usepackage{rotating}

\numberwithin{equation}{section}
\declaretheorem[style=plain,parent=section]{theorem}
\declaretheorem[style=plain,sibling=theorem]{corollary}
\declaretheorem[style=plain,sibling=theorem]{lemma}

\declaretheorem[style=plain,sibling=theorem]{proposition}
\declaretheorem[style=plain,sibling=theorem]{conjecture}

\declaretheorem[style=definition,sibling=theorem]{definition}

\declaretheorem[style=definition, qed=\hfill $\diamond$, sibling=definition]{example}
\declaretheorem[style=remark,sibling=theorem]{remark}

\newcommand{\Trop}{\text{Trop}}
\newcommand{\trop}{\text{trop}}
\newcommand{\charF}{\operatorname{char}}
\newcommand{\mult}{\operatorname{mult}}
\newcommand{\Spec}{\operatorname{Spec}}

\newcommand{\RR}{\mathbb{R}}
\newcommand{\CC}{\mathbb{C}}
\newcommand{\ZZ}{\mathbb{Z}}
\newcommand{\QQ}{\mathbb{Q}}

\newcommand{\cY}{\mathcal{Y}}

\newcommand{\pr}{\mathbb{P}}
\newcommand{\PS}{\CC\{\!\{t\}\!\}}
\newcommand{\PSR}{\RR\{\!\{t\}\!\}}
\newcommand{\val}{\operatorname{val}}
\newcommand{\init}{\operatorname{in}}

\newcommand{\TP}{\mathbb{T}}
\newcommand{\TPr}{\TP\pr}

\newcommand{\K}{\ensuremath{\mathbb{K}}}
\newcommand{\LL}{\ensuremath{\mathbb{L}}}
\newcommand{\KR}{\ensuremath{\K_{\RR}}}
\newcommand{\resK}{\ensuremath{\widetilde{\K}}}
\newcommand{\resKR}{\ensuremath{\widetilde{\KR}}}
\newcommand{\resL}{\ensuremath{\widetilde{\LL}}}
\newcommand{\Star}{\ensuremath{\operatorname{Star}}}

\newcommand{\Jac}{\ensuremath{\operatorname{Jac}}}
\providecommand{\Sn}[1]{\ensuremath{\mathfrak{S}_{#1}}}
\newcommand{\Dn}[1]{\ensuremath{D_{#1}}}
\newcommand{\Sk}{\ensuremath{\operatorname{Sk}}}

\newcommand {\dunion}{\,\mbox {\raisebox{0.25ex}{$\cdot$} \kern-1.83ex $\cup$}
  \,}
\newcommand {\divisor}{\ensuremath{\operatorname{div}}}
\newcommand {\dist}{\ensuremath{\operatorname{dist}}}
\newcommand {\sign}{\ensuremath{\operatorname}{sign}}
\newcommand{\sextic}{\ensuremath{f}} 
\newcommand{\du}{\ensuremath{d}} 
\newcommand{\da}{\ensuremath{\alpha}} 
\newcommand{\ka}{\ensuremath{\beta}} 

\newcommand \sage{\texttt{Sage}}

\makeatletter
\@namedef{subjclassname@2020}{%
 \textup{2020} Mathematics Subject Classification}
\makeatother

\title[Tropical methods for building real tritangents to space sextics]{Tropical methods for building real space sextics with totally real tritangent planes}

\author[M.A.Cueto, Y.~Len, H.~Markwig and Y.~Ren]{Mar\'ia Ang\'elica Cueto \and Yoav Len \and Hannah Markwig${}^{\S}$ \and Yue Ren}
\date{\today} 
\thanks{${\S}$ \emph{Corresponding author}}


\begin{document}

\begin{abstract}
  This paper proposes the use of combinatorial techniques from tropical geometry to build the 120 tritangent planes to a general smooth algebraic space sextic. Although the tropical count is infinite, tropical tritangents come in 15 equivalence classes, each containing the tropicalization of exactly eight classical tritangents. 

  Under mild genericity conditions on the tropical side, we show that liftings of tropical tritangents are defined over quadratic extensions of the ground field over which the input sextic curve is defined.   When the input curve is real, we prove that every complex liftable member of a given tropical tritangent class  either completely lifts to the reals or none of its liftings are defined over the reals. As our main application we use these methods to build examples of real space sextics with $64$ and $120$ totally real tritangents, respectively. The paper concludes with a discussion of our results in the arithmetic setting.  
\end{abstract}

\maketitle

\section{Introduction}\label{sec:introduction}

In recent years, tropical geometry has emerged as a tool for addressing classical questions on linear spaces that are multitangent to smooth plane curves, in the complex and real settings.  This includes Pl\"ucker's classical result~\cite{Plucker} on the 28 bitangent lines to smooth plane quartics, and the real count of such objects due to Zeuthen~\cite{Zeuthen}, which  depends on the topology of the real loci and its dividing type. Even though the analogous tropical count is generally infinite, tropical methods can be exploited to compute the classical multitangent objects~\cite{CM20,gei.pan:24,LM17}, determine their rationality~\cite{MPS23} and provide combinatorial proofs of these counts~\cite{BLMPR,gei:25,gei.pan:23,MPS24}.

The present paper's focus is on a classical question: the tritangent planes to a smooth space sextic curve. When the field is algebraically closed of characteristic not 2 or 3 and the curve is generic, work initiated by Clebsch~\cite{Clebsch} and further developed by Coble~\cite{Coble} confirms that the number of such tritangent planes is 120. This is done by identifying such planes with the number of odd theta characteristics of the input curve~\cite[Theorem 2.2]{HL17}.

As in the quartic case, when the input curve is defined over the reals,  the number of real tritangent planes, i.e. those defined over the reals, to a real smooth sextic curve in $\pr^3_{\RR}$ depends on its topology and how it sits in the Riemann surface of the corresponding complex curve. The precise count can be determined by appealing to the theory of real theta characteristics~\cite{Krasnov}. Interestingly, when the realness of the tangency points is also considered, work of Kummer~\cite{kum:19} provides explicit lower and upper bounds on the number of such totally-real tritangent planes, for each fixed real count.

The search for concrete examples realizing these bounds and all intermediate values has been the focus of several recent works involving numerical methods~\cite{HKSS,KRSMS}. Their construction is a delicate and subtle matter, even for those with the maximal number of totally-real tritangent planes. For instance,~\cite[Theorem 3.2]{HL17} shows that an earlier example constructed by Emch~\cite[\S 49]{Emch} has 108 totally real tritangent planes rather than the  claimed 120 count. 

The present paper contributes to the subject by showing how tropical methods can significantly simplify this task. In particular, we provide two explicit examples with 64 and 120 totally real-tritangents, realizing the maximum number predicted from the topology.

\smallskip

In order to tropicalize the problem at hand, an embedding of the geometric objects must be chosen. We let $\K$ be a non-Archimedean, non-trivially valued field, whose residue field has characteristic different from 2 or 3. A smooth space sextic curve $C$ over $\K$ is a non-hyperelliptic canonical curve of genus four, so it can be realized as the complete intersection in $\pr^3_{\K}$ of two smooth surfaces of degrees two and three, respectively. We let $Q$ be the quadric surface, and assume it is smooth.

After base change to a finite  extension of $\K$ if necessary, we can use the standard classification of homogeneous symmetric quadratic forms in four variables, to fix $Q$ to be the standard Segre quadric $V(x_0x_1-x_2x_3)$. Thus, pulling back via the Segre embedding, we may view $C$ as a smooth curve in $\pr^1\times \pr^1$, defined by a bidegree $(3,3)$ polynomial $f$. A plane in $\pr^3$ becomes a $(1,1)$-curve in $\pr^1\times \pr^1$. The tropicalizations of our classical objects lie in the corresponding product  $\TPr^1\times \TPr^1$ of tropical projective lines.  

Throughout we will assume both that $V(f)$ and its tropicalization $\Gamma = \Trop\, V(f)$ are smooth (i.e., locally isomorphic to the tropicalization of a linear space). Up to $\Dn{4}$-symmetry, there are $5\,941$ combinatorial types of such curves, in correspondence with the number of unimodular triangulations of the square of side length three. By contrast, there are only three possible combinatorial types of tropical $(1,1)$-curves $\Lambda$, corresponding to subdivisions of the standard square. They are depicted in~\autoref{fig:Planesvs11Curves}.

The skeleton of $\Gamma$ is a trivalent genus four metric graph in $\RR^2$. All but one of the sixteen planar graphs with these restrictions appear as skeletons of $(3,3)$-smooth tropical curves, and, furthermore, their edge lengths are linearly restricted~\cite[Figure 17, Table 3]{PlaneModuli2015}. We will assume $\Gamma$ satisfies some mild genericity constraints involving edge lengths on its skeleton (see~\autoref{rm:genericityOfGamma}) and that  $\sextic$ is a generic element on the fiber of $\Gamma$ under the tropicalization map. In particular, the curve $V(\sextic)$ will have no hyperflexes.

In analogy with the classical setting, effective tropical theta characteristics on tropical skeleta and the theory of divisor classes on tropical curves via chip-firing~\cite{bak.jen:16} can be used to define the notion of tritangencies in the tropical setting. This enables a combinatorial study of such objects. Our first main result is a combinatorial classification of local tangencies between such curves:

\begin{theorem}\label{thm:main1} Up to $\Dn{4}$-symmetry, there are 38 possible local tangencies between a $(1,1)$-tropical curve $\Lambda$ and a smooth $(3,3)$-tropical curve $\Gamma$ in $\TPr^1\times \TPr^1$.  They are depicted in~\autoref{fig:classificationLocalTangencies}. 
  \end{theorem}

Once these local tangency types are determined, it is desirable to know whether the local combinatorics can provide information on their realizability via classical tritangent $(1,1)$-curves to $V(\sextic)$. More precisely, we wish to determine what is the number of lifts a tropical tritangent $\Lambda$ to $\Gamma$ can have. Our second main result provides a precise answer to this question:

\begin{theorem}\label{thm:main2} Assume that the $(3,3)$-polynomial $\sextic$ with coefficients in $\K$ is generic with respect to $\Gamma$. Under mild genericity assumptions on $\Gamma$, each tropical tritangent $\Lambda$ to $\Gamma$ has either 0, 1, 2, 4 or 8 lifts to a classical tritangent to $V(\sextic)$ in $\pr^1_{\overline{\K}}\times \pr^1_{\overline{\K}}$. Precise values of the lifting multiplicities of all local tangency types can be found in \autoref{tab:LiftingMultiplicities}.
\end{theorem}

Our method of proof includes solving a local version of the classical system defining a tangency between two curves  to obtain the initial data of the tangencies and the equation  defining the tritangent $(1,1)$-curve, tropical modifications and re-embeddings induced by it (also known as tropical refinements). Although all these techniques have already been used in similar settings~\cite{CM20,LM17}, there are several technical difficulties involved in combining these tools, given the high number of local tangencies to consider and their combinatorial complexity. In particular, the ``tree-shape'' tangency, corresponding to the case when $\Gamma \cap \Lambda$ is a trivalent tree with four leaves, was proven to be extremely challenging both from the combinatorial and the computational perspective. The number of pages spanning~\autoref{sec:tang-mult-six} attests to this fact.

After determining individual lifting multiplicities, arithmetic considerations become natural. More precisely, we wish to determine the field of definition of the tritangent curve and the three tangency points. Inspired by  joint work of the third author with Payne and Shaw~\cite{MPS23} we show that when $\sextic$ and $\Gamma$ are generic in the sense of \autoref{thm:main2}, these fields are compositions of degree two extensions of $\K$. In agreement with~\cite{MPS23}, we call such extensions quadratic.

\begin{theorem}\label{thm:main3}
  Let $\Lambda$ be a tropical tritangent to $\Gamma$ with positive lifting multiplicity. Then, all lifts $V(\ell)$ of $\Lambda$ are defined over the same minimal quadratic extension $\LL$ of $\K$. In addition, in almost all cases, whenever $V(\ell)$ is defined over $\LL$, all three tangency points between $V(\ell)$ and $V(\sextic)$ are also defined over $\LL$, that is, any lift of $\Lambda$ that is $\LL$-rational, is automatically totally $\LL$-rational. For the exceptional cases, the tangency points are defined over a degree two extension of $\LL$.
\end{theorem}

\begin{table}[t]\begin{center}
              \begin{tabular}{|c|c|}
                \hline Multiplicity & Tangency types \\
                \hline
                0 &  (1),  (2b), (3ab), (3cb), (3bb), (3bb1), (3bb2), (4b), (7); (1'),  
                (3a')  \\
                1 & (2a), (4a)$^*$,   (5b), (6a)$^*$,  (6b); (2a'), (4a')$^*$, (4b'), (6a'),  (6b')\\
                 2 & (3a), (3c),  (3aa), (3ac), (3cc), (3d), (3h), (4a)$^*$,  (5a), (6a)$^*$; (3c'), (4a')$^*$  \\
                4 & (3f) \\
                8 & (8) \\
                \hline
              \end{tabular}
            \end{center}
  \caption{Possible values for the lifting multiplicities of all 38 local tangency types, following the notation of~\autoref{fig:classificationLocalTangencies}.Trivalent types and $4$-valent ones are separated by ';'. Repetitions highlight that more than one value can occur (the value depends on the remaining tangency points). For types (2), (4), (5) and (6), we use the labels `a' and `b' to distinguish between tangencies of multiplicity two and four, respectively. The types marked with $^*$ have more than one possible  lifting multiplicity: their precise value depends on the complementing tangencies.\label{tab:LiftingMultiplicities}}
          \end{table}

The construction of tritangent $(1,1)$-curves to $\Gamma$ via effective theta characteristics on its skeleton $\Sk(\Gamma)$ confirms the existence of infinitely many tritangents to $\Gamma$ in most cases. Even though we witness a superabundance of tropical objects, the collection of these infinitely many tropical tritangent curves to $\Gamma$ can be grouped into 15 equivalence classes, one for each non-zero element of the group $H^1(\Sk(\Gamma),\ZZ/2\ZZ)$~\cite[Theorem 5.2]{HL17}. These \emph{tritangent classes} correspond to local perturbations obtained by moving its vertices that preserve the tritangency condition. Their structure and properties are studied in a companion paper by the first, third and fourth authors~\cite{CLMR25Partition}.

Work of the second author and Jensen confirms the  numerology regarding the expected number of liftings of each tritangent classes. Indeed, by~\cite[Theorem 4.5]{JL18}, each tritangent class has precisely 8 lifts (counted with multiplicity) over $\overline{\K}$. The fact that the number of real tritangents to a space sextic in $\pr^3_{\RR}$ and of complex lifts of each  tritangent class is a multiple of eight, makes it natural to expect the same for the number of real lifts of any tritangent class.

Via Tarski's principle~\cite{JL89}, the real counts can be done by working over a real closed valued field $\KR$, such as $\PSR$. Note that our original approach to lifting in $\pr^1\times \pr^1$ cannot be used directly because it depends on the signature of the symmetric quadratic form defining the surface $Q$ containing the input space sextic curve $C$. By working with avoidance loci to curves in  $\pr^3_{\KR}$ in the spirit of~\cite{MPS24}, we confirm our expected numerology for real counts:

\begin{theorem}\label{thm:main4}
  Let $C$ be a smooth space sextic curve  defined over a real closed valued field $\KR$. Assume that its tropicalization $\Trop\, C\subseteq \TPr^3$ is smooth. Then, 
each tritangent class to $\Trop\, C$ has either 0 or 8 lifts to a tritangent plane to $C$. 
\end{theorem}

Our proof technique shows that when $Q \simeq \pr^1\times \pr^1$ and $C$ has smooth tropicalization in $\pr^1\times \pr^1$, the tangency points of any lift of a tropical tritangent to $C$ that is defined over $\KR$ are also  defined over $\KR$.
  It follows from this fact that tropical methods can be used to construct explicit examples in $\pr^1_{\KR}\times \pr^1_{\KR}$ with the maximum number of totally real tritangents given a fixed number of real ones. In~\autoref{sec:examples}, we provide two concrete examples, one with 120 and one with 64 totally real tritangents. As far as we know, the latter is only the second example available in the literature for this count after~\cite{HKSS}.

The rest of the paper is organized as follows.~\autoref{sec:preliminaries} recalls the construction of tropical tritangents to tropical $(3,3)$-curves and their connection to tropical theta characteristics. In addition, it reviews the methods devised in~\cite{LM17} for lifting tropical tangencies at the local level under mild genericity assumptions on the input classical and tropical curves.  \autoref{sec:comb-local-tang} provides a combinatorial classification of local tangencies on $\Gamma$, including a proof of~\autoref{thm:main1}.

Once the classification is established, local lifting multiplicities for each tangency type must be determined. This is carried out in~
{Sections}~\ref{sec:trivalentLifts} through~\ref{sec:tang-mult-six}. \autoref{sec:proofsMain2-3} contains the proofs of~
{Theorems}~\ref{thm:main2} and~\ref{thm:main3}, which are based on the local computations summarized in~\autoref{tab:LiftingMultiplicities}.
\autoref{sec:avoidance-loci-real} discusses avoidance loci of space sextic curves over real closed valued fields and its use to prove~\autoref{thm:main4}. This result is then applied in~\autoref{sec:examples}  to build new examples of curves with the maximal number of totally real tritangents given their topology.

\autoref{sec:arithm-constr-lift} discusses extensions of our lifting results in the arithmetic context and provides evidence supporting our expectation that the lifting behavior of tritangent classes over other fields matches the  real case.  The paper concludes with two appendices, one fixing a minor mistake in the proof of a result from~\cite{CM20} that is needed in the present paper, and a second one listing data that is heavily used in the re-embedding task carried out in~\autoref{ssec:build-an-algebr}.

\section*{Acknowledgments}

This project started during a week-long Research in Pairs stay at The Oberwolfach Institute for Mathematics (Germany). The authors thank the institute and its staff for their hospitality and  excellent working conditions. MAC acknowledges the Universit\`a degli Studi di Parma (Italy) for hosting her during the final stages of this work.

MAC was supported by NSF Standard Grants DMS-1700194 and DMS-1954163 (USA), as well as INdAM Research Project E53C23001740001 (Italy), YL was supported by EPSRC New Investigator Award (grant number EP/X002004/1), HM was supported by the DFG project MA 4797/9-1, and YR was supported both by UKRI Future Leaders Fellowship MR/Y003888/1 and the EPSRC grant EP/Y028872/1.

Computations were made in~\texttt{Sage}~\cite{sagemath} and \texttt{Singular}~\cite{DGPS}, using the package \texttt{tropical.lib}~\cite{JMM07a}. We have included the scripts used for these calculations on the latest \texttt{arXiv} submission of this paper. They can be obtained by downloading the source files.

\medskip

\section{Preliminaries: tropical tritangent curves and lifting methods}\label{sec:preliminaries}

Throughout this paper, we let $\K$ be a complete, non-Archimedean valued field with valuation ring $(R,\mathfrak{M})$. We write $\resK=R/\mathfrak{M}$ for its residue field. We assume that the characteristic of $\resK$ is not 2 or 3. Furthermore, we suppose the valuation map is non-trivial and admits a splitting, which we denote by $w\mapsto t^{w}$, following~\cite[Chapter 2]{MSBook}. In particular, we set $t^{0}=1$. The three main examples to keep in mind are the fields of  complex and real Puiseux series, denoted by $\PS$ and $\PSR$, respectively,  and the $p$-adics $\QQ_p$. 

Initial forms (which we now define) will play a crucial role in our work.

\begin{definition}\label{def:initialForms}
  Given a non-zero element $a$ of $\K$ with fixed valuation $\val(a)=\alpha$, we define its \emph{initial form} $\bar{a}$ as the class of  $t^{-\alpha} a\in R$ in the residue field $\resK$.
\end{definition}

As was mentioned in~\autoref{sec:introduction}, any smooth space sextic $C$ is the complete intersection of a unique quadric surface $Q$ and a smooth cubic one~\cite[page 118]{ACGH1}. Throughout, we assume the surface $Q$ is smooth. In particular, $Q$ is determined by a non-zero symmetric non-singular bilinear form in $\K^4$. Since $\charF \K \neq 2$, this form can be diagonalized over $\K$ through Sylverster's Law of Inertia~\cite[\S 10.2, Theorem 3]{hoff.kun:04}. This allows us to assume that, after base change,  $Q$ is the Segre quadric $Q = V(x_0x_1-x_2x_3) \subseteq \pr^3$, which is  isomorphic to  $\pr^1\times \pr^1$ under the Segre embedding. 

Pulling back $C$ via this embedding yields a curve in $\pr^1\times \pr^1$ of bidegrees $(3,3)$ defined by some
\begin{equation}\label{eq:sextic}
\sextic(x,y) = \sum_{0\leq i,j\leq 3} a_{ij} x^iy^j.
\end{equation}
Assuming $a_{00}a_{30}a_{03}a_{33}\neq 0$,  its Newton polytope will be the standard square of side length three. We  let $\Gamma$ be the associate tropical $(3,3)$-curve in $\TPr^1\times \TPr^1$.

In turn, for each plane $\pi$ in $\pr^3$ meeting the dense torus, the pull-back of $\pi\cap Q$ produces a $(1,1)$-curve in $\pr^1\times \pr^1$.
 By B\'ezout's Theorem, these two curves intersect at six points (counted with multiplicity) defined over $\overline{\K}$. The tritangency condition arises when these six points collide in pairs, yielding intersection points of even multiplicity. In the absence of hyperflexes, we will have exactly three distinct intersection points, of multiplicity two each. 

Throughout, a tritangent $(1,1)$-curve to $V(\sextic)$ will be determined by the fixed equation
\begin{equation}\label{eq:tritangent}
\ell = y + m + n\,x + \du\,x\,y \qquad \text{ with } m,n,\du \in \overline{\K}^*.
\end{equation}
Its tropicalization $\Lambda$ will be one of the three graphs seen in~\autoref{fig:Planesvs11Curves}. 

 \begin{definition} Given the polynomials $\sextic$ and $\ell$ defined above, we say $(\ell, p, p',p'')$ is a \emph{tritangent tuple} to $V(\sextic)$ defined over $\overline{\K}$ if $V(\ell)$ is tritangent to $V(\sextic)$ with tangency points $p$, $p'$ and $p''$.
 \end{definition}

\begin{figure}[t]
  \includegraphics[scale=0.35]{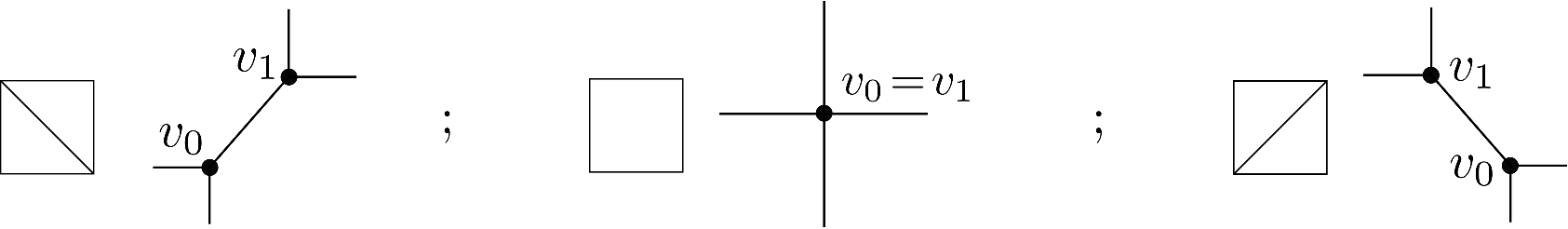}
  \caption{Tropical 
    $(1,1)$-curves   with the corresponding Newton subdivisions of $\ell$.\label{fig:Planesvs11Curves}}
  \end{figure}

Tropical curves will be defined following the \emph{max} convention. In particular, the tropicalization $\Gamma$ of $V(\sextic)$ in $\TPr^1\times \TPr^1$ will be determined by the Newton subdivision of $\sextic$.  The smooth condition will correspond to  having a unimodular triangulation of the Newton polytope of $\sextic$. \autoref{fig:example64} shows an example of such curves.
In turn, the tropicalization of a curve embedded in $(\K^*)^n$ by an ideal $I$ in the Laurent polynomial ring $\K[x_1^{\pm}, \ldots, x_n^{\pm}]$ will be obtained as the (Euclidean) closure in $\RR^n$ of the image of $V(I)\subseteq (\overline{\K}^*)^n$ under the coordinatewise negative valuation. Our ideals $I$ will be determined by tropical modifications of $\RR^2$ induced by the  tritangent curves $\Lambda$ to the input curve $\Gamma$. For an overview of these and other notions in tropical geometry, we refer the reader to~\cite{MSBook, Mi06}.

\begin{remark}\label{rm:actionD4} The  dihedral group $\Dn{4}$ of order eight records automorphisms of the unit square preserving its boundary. By construction, this group acts on $\pr^1\times \pr^1$ by monomial Cremona transformations. This action extends to $\ZZ^2$ and the space of tropical plane curves of bidegree $(d,d)$ in $\TPr^1\times \TPr^1$ for any $d\in \ZZ_{>0}$. In particular, it acts on the space of smooth  $(3,3)$-curves and their tropical tritangent $(1,1)$-curves. \autoref{tab:D4Action} describes the action of two generators of   $\Dn{4}$ on the classical and tropical sides, including its effect on the coefficients $m, n, \du$ featured in~\eqref{eq:tritangent}.
\end{remark}

The tropical version of B\'ezout's Theorem (see e.g.,~\cite[Corollary 1.3.4]{MSBook}) ensures that the curves $\Gamma$ and $\Lambda$ meet at six points (counted with multiplicity) when considering their stable intersections. If their intersection is transverse (or proper), these six points are precisely $\Gamma \cap \Lambda$. In the non-proper case, the number of intersection points, counted with multiplicity, is obtained by the fan displacement rule along a generic direction $v$, i.e.,
\begin{equation*}
 \Gamma \cap_{st} \Lambda = \lim_{\varepsilon \to 0} \Gamma \cap (\Lambda + \varepsilon v).
\end{equation*}

The classical tritangency condition has a natural tropical analog: each of the connected components of the set-intersection $\Gamma \cap \Lambda$ must contain  an even number of points from the stable intersection $\Gamma  \cap_{st} \Lambda$, counted with multiplicity. The next purely combinatorial definition provides more specifics:

\begin{definition}\label{def:tropicalTritangents} Let $\Gamma$ and $\Lambda$  be two tropical curves in $\TPr^1\times \TPr^1$ of bidegrees $(3,3)$ and $(1,1)$, respectively. Assume that $\Gamma$ is smooth. Then, we say that $\Lambda$ is \emph{tritangent} to $\Gamma$ if any of the following three conditions hold:
  \begin{enumerate}[(i)]
  \item $\Gamma \cap \Lambda$ has three components, each with total stable intersection multiplicity 2, or
    \item $\Gamma \cap \Lambda$ has two components, of total stable intersection multiplicities 2 and 4, respectively, or
\item $\Lambda \cap \Gamma$ is connected, and its stable intersection multiplicity is 6.
  \end{enumerate}
\end{definition}

Notice that by construction, the stable intersection of $\Gamma$ and $\Lambda$ contains no point in the relative interior of a leg of $\Gamma$. Thus, we may view $\Gamma \cap_{st} \Lambda$ in the metric graph $\Gamma'$ of genus four obtained from $\Gamma$ by removing the relative interiors of  its 12 legs. We view $\Gamma'$ as a skeleton for the algebraic curve $C$ after base change to $\Spec \overline{\K}$. In particular, $\Gamma \cap_{st} \Lambda$ is an effective divisor on $\Gamma'$ of degree six.

\begin{table}[t]
  \begin{tabular}{|c|c|c|c|c|}
\hline    Gen. & Projective & Lattice & Tropical & coefficients of $\ell$\\
    \hline
    $\tau_0$ & $x\longleftrightarrow y$ & $(i,j) \mapsto (j,i)$ & $X\longleftrightarrow Y$ & $(m, n, \du) \mapsto (m/n, 1/n, \du/n)$\\
    \hline {$\tau_1$} & $(x,y)\mapsto (y,1/x)$ & $(i,j)\mapsto (j,3-i)$ & $(X,Y)\mapsto (Y,-X)$ & {$(m, n, \du) \mapsto (n/m, \du/m, 1/m)$}
      \\
     \hline
      \end{tabular}
    \caption{Generators of $\Dn{4}$ as projective, polytopal and tropical isomorphisms, and their action on the coefficients of the polynomial $\ell$ from~\eqref{eq:tritangent}.\label{tab:D4Action}}
\end{table}

Once a tritangent curve $\Lambda$ to $\Gamma$ is determined, the theory of  divisors on tropical curves (see, e.g.,\cite{bak.jen:16}) can be used to determine the tangency points between these two curves. More precisely, starting from the divisor $D:=\Gamma \cap_{st} \Lambda$ of degree six on $\Gamma$, chip-firing can move  the points on $D$ towards each other on each component of $\Gamma' \cap \Lambda$ to produce a linearly equivalent divisor of the form $2P +2P'+2 P''$ for some points $P, P', P'' \in \Gamma'$ (not necessarily distinct). These tropical divisors correspond to non-zero \emph{effective tropical theta characteristics} on $\Gamma'$. They are in one-to-one correspondence with the 15 non-zero elements on the homology group $H^1(\Gamma', \ZZ/2\ZZ)$ and can be recovered from it via Zharkov's algorithm~\cite{Zha10}.

The connection between tropical tritangent curves to $\Gamma$ and non-zero effective theta characteristics on $\Gamma'$ has a natural consequence: each tropical linear system produces a family of tritangent curves to $\Gamma$ by continuous deformations of $(1,1)$-tropical curves that preserve the tritangency condition. This defines a natural equivalence relation among tritangent curves to $\Gamma$, analogous to the one for bitangents to tropical smooth plane quartics described in~\cite[Definition 3.8]{BLMPR}.

\begin{definition}\label{def:tropicalTritangentPlaneClasses} An equivalence class of tropical tritangents to $\Gamma$ is called a \emph{tritangent class} of $\Gamma$.
  \end{definition}

In particular, any  tritangent class of $\Gamma$ with infinitely many members yields an infinite number of tropical tritangents to $\Gamma$. Nonetheless, the  classical finite count of 120 tritangents to the smooth general $(3,3)$-curve $C\subseteq \pr^1\times \pr^1$ can be recovered from the 15 tritangent classes of $\Gamma=\Trop\, C$ through the specialization map from classical to tropical theta characteristics (see~\cite[Theorem 1.1]{JL18}). Here is the precise statement, which can be found in~\cite[Theorem 5.2]{HL17}:

\begin{theorem}\label{thm:15x8=120} If $\Gamma=\Trop\,V(\sextic)$ is smooth, each of the 15 tritangent classes of $\Gamma$ contains the tropicalization of precisely  8  tritangent curves $V(\ell)$ to $V(\sextic)$ defined over $\overline{\K}$.
  \end{theorem}

One of the goals of the present paper is to devise a method to detect which tropical tritangent curves $\Lambda$ to $\Gamma$ arise as tropicalization of classical tritangent curves to $V(\sextic)$, and, most importantly, determine the parameters $m$, $n$ and $\du$ seen in~\eqref{eq:tritangent} from the tropical tangency data $(\Lambda, P,P',P'')$. The following definition arises naturally:

\begin{definition}\label{def:lift} We say $\Lambda$ \emph{lifts} over $\overline{\K}$ if there exists a tritangent tuple $(\ell,p,p,p'')$ defined over $\overline{\K}$ whose tropicalization is $(\Lambda, P,P',P'')$, i.e.,
  \[\Trop\, V(\ell) = \Lambda, \quad \Trop\, p = P, \quad \Trop\, p' = P' \quad\text{and} \quad \Trop\, p'' = P''.
  \]
  The \emph{lifting multiplicity} of $\Lambda$ equals the number of such tritangent tuples.
\end{definition}

Tritangent lifts to a tuple $(\Lambda,P,P',P'')$ can be obtained from the local lifting methods of tropical tangencies devised in~\cite{LM17}, which we now briefly review. Fix a $(1,1)$-polynomial $\ell$ from~\eqref{eq:tritangent} and a point $p\in V(\ell)\cap V(\sextic) \subseteq \pr^1_{\overline{\K}}\times \pr^1_{\overline{\K}}$ contained in the dense torus. The point $p=(x,y)$ is a tangency point between these two curves precisely if it satisfies the system $\sextic=\ell = W=0$, where $W:=\det(J(\sextic, \ell;x,y))$ is the Wro\'nskian between $\sextic$ and $\ell$.

From this it follows that $\bar{m}, \bar{n}, \bar{\du}$ and $\bar{p} :=(\bar{x}, \bar{y})\in (\resK^*)^2$ must satisfy the  \emph{local equations} at $P=\Trop \,p$, corresponding to the vanishing of the initial forms (with respect to $P$) of $\sextic, \ell$ and $W$, that is, $\sextic_P = \ell_P=W_P=0$. Here,
\[
W_P :=\det(J(\sextic_P,\ell_P;\bar{x}, \bar{y})) \; \text{ and } \;\sextic_P :=\sum_{(i,j)\in P^{\vee}} \overline{a_{ij}}\,\bar{x}^i\,\bar{y}^j,
\]
where $P^{\vee}$ is the cell of the Newton subdivision of $\sextic$ dual to $P$. The expression for $\ell_P$ is obtained as in $\sextic_P$, but it will involve the initials of all terms whose exponent are contained in  the cell of the Newton subdivision of $\ell$ dual to $P$. For example, if $\Lambda$ is trivalent with a slope one edge and $P$ is its top vertex, then $\ell_P= \bar{y} +\bar{\du}\bar{x}\bar{y} + \bar{n}\bar{x}$.

The local equations at  $P$ recover the initial form  of both entries of $p$ and either one of the parameters of $\ell$ or ratios thereof, depending on the location of $P$ within $\Lambda$. The exact data will be determined from the combinatorics of the local tangencies, which is discussed in~\autoref{sec:comb-local-tang}. 

Whenever the local equations admit infinitely many solutions over the algebraic closure of $\resK$, we make use of \emph{tropical modifications} of $\RR^2$ along $\Lambda$ and suitable re-embeddings of $V(\sextic)$ to ensure, in the new coordinates,  the corresponding local systems have finitely many solutions. These situations arise precisely whenever the component of $\Gamma \cap \Lambda'$ containing $P$ is infinite. If the number of non-zero solutions (potentially after modifications) in the algebraic closure of  $\resK$ is finite, we call this quantity  the \emph{local lifting multiplicity} of $(\Lambda, P)$. We will denote it by $\mult(\Lambda, P)$.

\smallskip
Combining the local equations at all tropical tangency points  yields a system of nine equations in nine unknowns, namely the three initial forms  of the three parameters in $\ell$ and the six coordinates of $p$, $p'$ and $p''$. In order to exploit this local approach to compute the 120 classical tritangents to $V(\sextic)$ two things must occur. First, these square  systems should have only finitely many solutions and, second, the data $(m,n,\du,p,p',p'')$ should be uniquely determined by the  initial values.

This last uniqueness property can be derived  from a multivariate analog of Hensel's Lemma, which we now state.  Here, $\bar{h}$ denotes the initial form of any $h\in R[x_1,\ldots, x_n]$ obtained from the initial form of its coefficients. 
For details, see~\cite[Theorem 2.2]{LM17} and references therein.

\begin{lemma}[\textbf{Multivariate Hensel's Lemma}]\label{lm:multivariateHensel} Consider a tuple of polynomials $\mathbf{f}:= (f_1,\ldots, f_n)$ with $f_i\in R[x_1,\ldots, x_n]\smallsetminus \mathfrak{M}R[x_1,\ldots, x_n]$ and let $J_{\mathbf{f}}$ be its Jacobian matrix. Consider a  point $a=(a_1,\ldots, a_n)\in (\resK^*)^n$ satisfying $J_{\overline{\mathbf{f}}}(a)\neq 0 \in \resK$ and $\overline{f_i}(a)=0$ for all $i=1,\ldots,n$. Then, there exists a unique $b = (b_1,\ldots, b_n)\in R^n$ with $f_i(b) = 0$, $\val(b_i)=0$ and $\overline{b_i} = a_i$ for all $i=1,\ldots, n$.
\end{lemma}

The finiteness of the solutions to the $9\times 9$ local systems (after tropical modifications have been considered) is not guaranteed from the combinatorics of $\Gamma$. Instead, we must impose some mild genericity constraints analogous to those outlined in~\cite[Remark 3.3]{LM17}. Here is the precise statement:

\pagebreak

\begin{definition}\label{def:fgenericRelToGamma} Given $\Gamma$, we say that $\sextic$ is \emph{generic relative to $\Gamma$} if the following conditions hold:
\begin{enumerate}[(i)]
  \item the polynomial  $\sextic$  defines a smooth curve in $\TPr^1\times \TPr^1$ with $\Trop V(\sextic) = \Gamma$,
\item the curve $V(\sextic)$ has no hyperflexes,
\item if two tropical tangencies occur on the relative interior of the same leg or edge of $\Lambda$, the corresponding local systems are inconsistent,
  \item if $\Gamma$ is trivalent, and  $\Lambda$ has a vertex whose two adjacent legs contain no tangency points in their closure, then the $9\times 9$ system of local equations at the tropical tangency  points between $\Gamma$ and $\Lambda$ has no solutions in the $9$-dimensional torus over the algebraic closure of $\resK$.
  \end{enumerate}
\end{definition}

\begin{remark}\label{rm:genericityOfGamma} The lifting result presented in~\autoref{thm:15x8=120} has no requirements on $\Gamma$ other than its smoothness. However, as was shown in~\cite{LM17}, the modification approach to obtain a finite number of solutions to the local equations at a given $P$ does rely on  $\Gamma$ being mildly generic. More precisely,  the tropical tangencies arising from components of $\Gamma \cap \Lambda$ of infinite size should not be vertices of $\Lambda$. This condition ensures that, after modifying $\RR^2$ along $\Lambda$, the tangency points of the re-embedded tropical curve  lie in the relative interior of a top-dimensional cell of the modified tropical plane.

  As it will be clear from the discussion in the next section, this condition holds if we impose two particular  (generic) edge length requirements on $\Gamma$ up to $\Dn{4}$-symmetry:
  \begin{enumerate}[(i)]
  \item if a component of $\Gamma \cap \Lambda$ is a tripod with three edges with directions $(-1,0)$, $(0,-1)$ and  $(1,1)$, then the shortest length of these three edges is unique,
    \item if a component of $\Gamma \cap \Lambda$ is a balanced trivalent quartet tree with an internal slope 1 edge of fixed length (called $L_3$), and  leaf-edges of directions $(\pm 1, 0)$ and $(0, \pm 1)$ (with edge lengths $L_1, L_2, L_4, L_5$), then the shortest of the lengths $L_1,L_2,L_4,L_5$ must be unique; furthermore, if this minimum value is $L_1$ and $L_2$ is the length of its  adjacent leaf-edge, then $\min\{L_3 , L_5 - L_1 , L_4 - L_1\}$ must also be attained exactly once. 
  \end{enumerate}
  The first condition corresponds to the ``star-intersections'' discussed in~\cite[Proposition 3.12]{LM17}, whereas the second one can be seen in~\autoref{fig:chipFiringTreeShapeIntersections}. 
\end{remark}

\section{Combinatorics of local tropical tangencies}\label{sec:comb-local-tang}

In this section, we provide a complete classification of local tangencies between planar $(1,1)$- and $(3,3)$-tropical curves, labeled $\Lambda$ and $\Gamma$, respectively.  We assume that $\Gamma$ is always smooth, but $\Lambda$ may be 4-valent. We exploit the $\Dn{4}$-symmetry of $\TPr^1\times \TPr^1$ whenever possible.  In particular, if $\Lambda$ is trivalent, we suppose its unique edge has slope one. Our methodology follows closely that of~\cite{CM20, LM17}.

\begin{figure}[tb]
\includegraphics[scale=0.33]{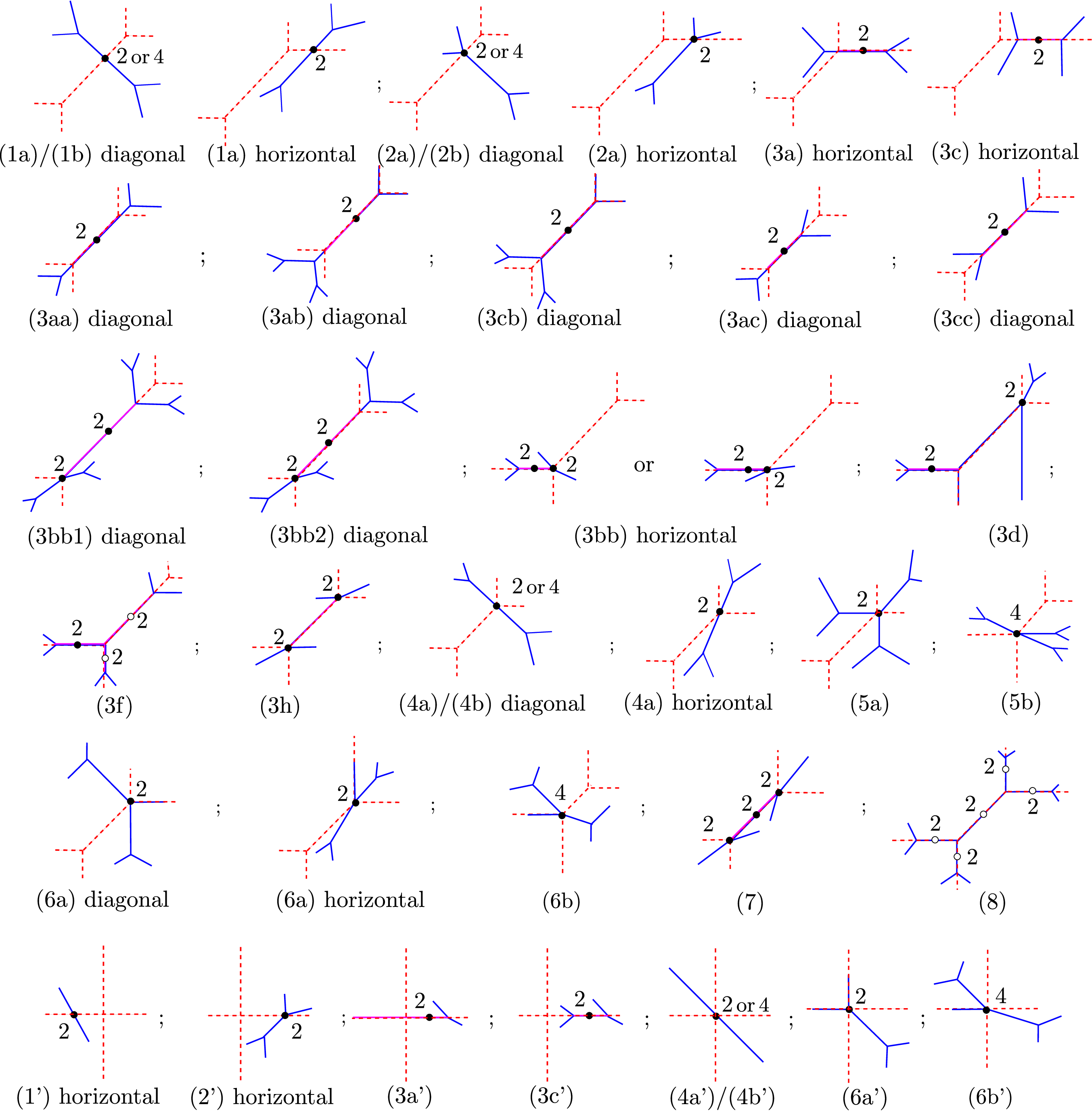}
  \caption{Representatives of all possible local tangency types between a $(1,1)$-tropical curve $\Lambda$ (in dashed lines) and a smooth $(3,3)$-tropical curve $\Gamma$, under the action of $\Dn{4}$. The last row corresponds to the cases when $\Gamma$ is $4$-valent. For types (1), (2), (4) (5) and (6), we use the labels `a' and `b' to distinguish between tangencies of multiplicity two and four, respectively. The black dots show the location of the tropical tangency point. The numbers adjacent to each dot indicate possible tangency multiplicities. Unfilled dots in (3f) and (8) denote potential tangencies, and their total multiplicities are four and six, respectively.\label{fig:classificationLocalTangencies}}
  \end{figure}

\begin{figure}[tb]
  \begin{center}  \includegraphics[scale=0.33]{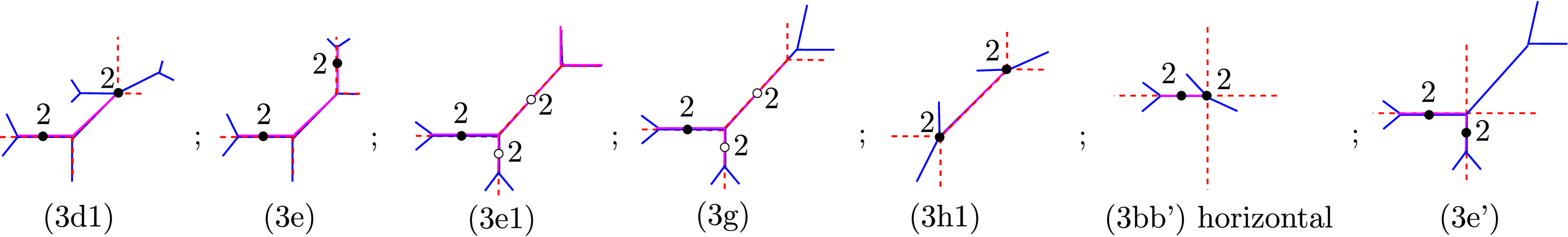}
  \end{center}
  \caption{Representatives of potential local tangency types between $\Lambda$ and $\Gamma$ that cannot occur for bidegree reasons.\label{fig:absentTypes}}
\end{figure}

\begin{theorem}\label{thm:classificationRealizableLocalTangencies} Up to symmetry, there are 38 possible local tangencies between a $(1,1)$-tropical curve $\Lambda$ and a smooth $(3,3)$-tropical curve $\Gamma$ in $\TPr^1\times \TPr^1$. Five of these arise from $4$-valent $(1,1)$-curves, while the remaining ones come from trivalent ones. They are depicted in~\autoref{fig:classificationLocalTangencies}. \autoref{tab:LiftingMultiplicities} lists the local lifting multiplicities of each such tangency. 
  \end{theorem}

\begin{proof}   The classification follows by simple inspection after analyzing the possible combinations of connected components and multiplicities. In addition to the 38 tangencies listed in the figure we will have four more, featured in \autoref{fig:absentTypes}. 
  {Lemmas}~\ref{lm:no3gNo3h1} and~\ref{lm:3bbOr3eCross} below rule out the latter options.

  Throughout, we let $\cY$ be a fixed connected component  of $\Gamma \cap \Lambda$, and fix its multiplicity as the sum of the multiplicity of all points in $\cY\cap (\Gamma\cap_{\text{st}} \Lambda)$.  By construction, this number equals 2, 4, or 6 for each component. After determining the combinatorics involving each tangency type, a standard chip-firing argument on $\cY$ determines the precise location of the tangency points within $\cY$.

We distinguish two cases, depending on whether $\Lambda$ is trivalent or not. First, we assume $\Lambda$ is trivalent with a slope one edge. The classification (up to $\Dn{4}$-symmetry) depends on the number of vertices of $\Lambda$ contained in $\cY$. If $\cY$ has at most a single vertex of $\Lambda$, any such local tangency can be viewed locally as one  between a (max- or min-) tropical line and a smooth tropical curve in $\TPr^2$. These have been determined in~\cite{CM20, LM17}. They correspond to types (1), (2), (4), (5), (6), and a subset of type (3) tangencies (i.e., some in the first two rows of~\autoref{fig:classificationLocalTangencies}, (3bb1) diagonal, (3bb) horizontal and (3f)), depending on whether the tangency lies on a leg  of  $\Lambda$ or on its unique edge. Notice that type (3f) corresponds to the star-shaped configuration from~\cite[Figure 15]{LM17}.

New tangency types  arise when $\cY$ contains both vertices of $\Lambda$. We distinguish two cases, depending on whether $\cY$ agrees with the edge of $\Lambda$, or if it strictly contains it. 
First, we assume $\cY$ agrees with the slope one edge of $\Lambda$. If its  multiplicity equals  two, we get the remaining type (3) tangencies seen on the second row of the figure. In turn, when the value is  four, then   $\cY\cap (\Gamma\cap_{\text{st}}\Lambda)$ consists of the two vertices of this edge. Two possibilities arise depending on the tuple $(n_0,n_1)$ recording the stable intersection multiplicities at the vertices  of $\Lambda$.

By symmetry, we can assume the lower vertex of $\Lambda$ carries the larger multiplicity, so there are two options, namely $(n_0,n_1)=(2,2)$ or $(3,1)$.
The first one yields  types (3h) and (3h1), whereas the second one gives a (3bb2) diagonal tangency. \autoref{lm:no3gNo3h1} below rules out the type (3h1) tangency. 

Finally, if the intersection  multiplicity of the edge $\cY$ is six, the bidegree of both $\Lambda$ and $\Gamma$ forces the two endpoints to have multiplicity three. The tangency points are located at the vertices and the midpoint of this edge, producing the type (7) seen in the picture.

On the contrary, if $\cY$ strictly contains the edge of $\Lambda$, the possible tangency types depend on the number of vertices and ends of $\Gamma$ contained in $\cY$. We label these quantities as $d$ and $q$, respectively.
By construction, we have $2\leq d\leq 6$, whereas the bidegree of $\Gamma$ forces $0\leq q\leq 2$.  
Notice that $d\neq 5$ since this will force $q=1$ and an intersection multiplicity five on $\cY$, which cannot occur.

The various combinations $(d,q)$  produce different tangency types. For example, if $d=2$, the bidegree of $\Lambda$ forces $q=2$ and, furthermore, both legs of $\Gamma$ in $\cY$ must be adjacent to the same vertex of $\Lambda$, which by symmetry can be assumed is the top one. Such situation produces a type (3ab) diagonal tangency. A similar analysis restrict the possible combinations of $(d,q)$ to $(3,0)$, $(3,1)$, $(4,2)$ and $(6,0)$, corresponding to types (3g), (3d) or (3d1), (3e) or (3e1), and (8), respectively.  \autoref{lm:no3gNo3h1} below discards all of these types except for (3d) and (8).

To conclude, we determine the types when $\Lambda$ is 4-valent. We let $v$ be its unique vertex. If $\cY$ is a point, its type will be (1'), (4a') or (4b') depending on its location within $\Lambda$.  The bidegree of $\Lambda$ restricts the multiplicity in both cases to be either 2 (in the first two cases) or 4.

On the contrary, if $\cY$ arises from a non-proper intersection, its type depends on the number of legs of $\Lambda$ containing points of $\cY$ in their relative interiors. Our options are limited by the smoothness of $\Gamma$ and the balancing condition.

First, assume $\cY$ lies in a single leg of $\Lambda$, say the negative horizontal one. If $\cY$ is unbounded, the multiplicity value of $\cY$ implies that  both that  $v\in \cY$ and that $v$ is a vertex of $\Gamma$. The smoothness of $\Gamma$ and the evenness of the multiplicity of $\cY$ yields a type (6b') tangency. On the contrary, if $\cY$ is bounded, we obtained either a tangency of type (3bb') or (3c') depending on whether or not $v$ is part of $\cY$. \autoref{lm:3bbOr3eCross} below rules out the first option.

Second, if $\cY$ lies in two legs of $\Lambda$, we may assume these are the horizontal negative and positive vertical ones. By construction, it follows that $v$ is a vertex of $\Gamma$ and $v\in \cY$. The evenness of the multiplicity of $\cY$ forces this set to be either bounded or unbounded in both the horizontal and vertical directions. This produces the type (3e') and (6a') tangencies, respectively. Once again, \autoref{lm:3bbOr3eCross} discards the first possibility.
\end{proof}

The next two results confirm that none of the orbit representatives for  local tangencies listed in \autoref{fig:absentTypes} can occur between $\Lambda$ and $\Gamma$ given their prescribed bidegreees.

\begin{lemma}\label{lm:no3gNo3h1} No trivalent tropical $(1,1)$-curve $\Lambda$ tritangent to a smooth tropical $(3,3)$-curve $\Gamma$ can carry a type (3d1), (3e), (3e1), (3g) nor a (3h1) local tangency.
\end{lemma}

\begin{proof} We prove the statement for the representatives depicted in~\autoref{fig:absentTypes}. By construction, all the listed types have multiplicity four. We show that $\Lambda$ and $\Gamma$ meet at two extra points of  multiplicity one each, contradicting the tritangency condition. 

 First, we analyze the type (3g) tangency. We let $v_0$ be the lower vertex of $\Lambda$.  The bidegree of $\Gamma$ fixes the vertices of $v_0^{\vee}$ to be $(1,1)$, $(2,1)$ and $(1,2)$. Thus, the  two legs of $\Lambda$ adjacent to $v_1$ intersect $\Gamma$ at an extra point besides $v_1$ in the boundary of the chambers $(2,1)^{\vee}$ and $(1,2)^{\vee}$. These chambers are the regions of $\RR^2\smallsetminus \Gamma$ dual to $(2,1)$ and $(1,2)$, respectively.  For bidegree reasons, the intersection multiplicity at these two points is one.

 Next, we suppose we have a tangency of type (3e). Similarly to the previous cases,  the set of vertices of $v_0^{\vee}$ and $v_1^{\vee}$ consists of $\{(2,1), (3,0)\}$ complemented by either $(2,0)$ or $(3,1)$, respectively. As a consequence, $\Lambda$ intersects $\Gamma$ at two extra  points, each lying on the negative horizontal and positive vertical legs of $\Gamma$.  Similarly, for type (3e1), we will see two extra intersections, on the negative horizontal and vertical legs of $\Lambda$,  in the boundary of the chambers $(1,i)^{\vee}$ and $(j,1)^{\vee}$, respectively, for suitable $0\leq i,j\leq 3$. The argument for type (3d1) is similar: the interior of either the positive horizontal or vertical leg of $\Lambda$ carries a multiplicity one point in the boundary of either $(2+i,0)^{\vee}$ or $(1+i,2)^{\vee}$ for $i=0,1$.

 Finally, assume we have a type (3h1) tangency. There are two cases to consider, depending on the slopes of the edges adjacent to $v_0$ and $v_1$ in $\Gamma$. If $v_0$ and $v_1$ are adjacent to a vertical and horizontal edges, there are precisely four possible locations for the dual cells to $v_0$ and $v_1$ in the Newton subdivision of $\sextic$. For all these choices, we will find an extra multiplicity one intersection point in a vertical leg of $\Lambda$, in the boundary of the chamber of $\RR^2\smallsetminus \Gamma$ dual to either $(1,2)$, $(2,2)$, $(2,1)$ or $(3,1)$.

 On the contrary, if $v_0$ is adjacent to a vertical edge but $v_1$ is adjacent to a slope $2/3$ edge of $\Gamma$, there are two possible locations for $v_0^{\vee}$ and $v_1^{\vee}$ in the Newton subdivision of $\sextic$. In both cases, there is a multiplicity one intersection point along a horizontal leg of $\Lambda$. Such point will belong to the chamber of $\RR^2\smallsetminus \Gamma$ dual to either $(1,1)$ or $(2,0)$.
\end{proof}

\begin{lemma}\label{lm:3bbOr3eCross}
  No  4-valent tropical  (1,1)-curve $\Lambda$ tritangent to a smooth tropical $(3,3)$-curve $\Gamma$  can have a local tangency of type (3bb') or (3e').
\end{lemma}
\begin{proof}  We let $v$ be the vertex of $\Lambda$ and establish the statement with similar arguments as those used in the proof of~\autoref{lm:no3gNo3h1}.  The multiplicity four condition for  type (3bb') combined with the $\Dn{4}$-symmetry forces the star of $\Gamma$ at $v$ to contain the ray $(2,-1)$. This determines the cell dual to $v$ to be the triangle with  vertices $\{(0,0), (0, 1), (1,2)\}$ shifted by either $(1,0)$, $(1,1)$, $(2,0)$ or $(2,1)$. Analyzing the possible remaining intersections between $\Lambda$ and $\Gamma$  gives two more intersection points, each of multiplicity one, along two legs of $\Lambda$. Thus, $\Lambda$ is not tritangent to $\Gamma$.
  
  Similarly, in the presence of a (3e') tangency on $\Lambda$, we note that the cell $v^{\vee}$ in the Newton subdivision of $\sextic$ is the triangle with vertices $\{(0,0), (1,0), (0,1)\}$ shifted by $(1,1)$, $(1,2)$, $(2,1)$ or $(2,2)$. As before, the bidegree of $\Gamma$ implies that $\Gamma\cap \Lambda$ has two intersection one points on two different legs of $\Lambda$ outside the component defining the (3e') tangency. The result follows.
\end{proof}

To determine the lifting behavior of each tangency type through its local data, we must first collect the slopes of all rays in the stars of $\Gamma$ at vertices featured on each type.  The next remark  and its subsequent lemmas (regarding   types (3) and (7)) accomplish this task.

\begin{remark}\label{rm:StarsOfSomeTangencies}  The star of $\Gamma$ at a tangency point $P$ of types (1), (2), (4), (5) or (6) can be determined, up to $\Dn{4}$-symmetry, using the same methods as those from~\cite[Remark 2.6]{CM20}. In what follows, we choose the orbit representatives listed in~\autoref{fig:classificationLocalTangencies}. \autoref{tab:edgeDirections1And2} shows this data when $P$ has types (1), (1') or (2). For type (4a), the point $P$ is in the interior of an edge of direction $(1,2)$, if the tangency is horizontal, or of directions $(1,-1)$ or $(1,-3)$  in the diagonal case, depending on the multiplicity at $P$. In turn, for type (4'), the directions are either $(1,-1)$ or $(3,-1)$, depending on the value of the multiplicity at $P$. For type (5a), $\Star_{\Gamma}(P)$ agrees with the star of a min-tropical curve in $\TPr^2$ at its unique vertex. In turn, for type (5b) or (4b), the point $P$ is adjacent to or in the relative interior of an edge with direction $(-3,1)$. For types (6b) and (6b'), $P$ is a vertex of $\Gamma$, adjacent to two edges with directions $(3,-1)$ and $(-2,1)$. The star for a type (6a') is obtained from the balancing condition. Finally, for type (6a), the star of $\Gamma$ at $P$ contains either a negative vertical leg (in the diagonal case) or a positive diagonal leg (in the horizontal case). The remaining ray of $\Star_{\Gamma}(P)$ is determined, again, by balancing.
\end{remark}

\begin{table}[tb]\begin{center}
              \begin{tabular}{|c|c|c|}
\hline                Edge/Leg &  Direction of $e$ & Direction of $e^{\vee}$ \\
\hline
Horizontal & $(-3,2), (-1,2), (1,2), (3, 2)$ & $(2,3), (2,1), (-2,1), (-2,3)$ \\
                Vertical & $(2,3), (2,1), (2,-1), (2,-3)$ & $(-3,2), (-1,2), (1,2), (3,2)$  \\
                Diagonal & $(1,3), (3,1), (1,-1), \underline{(-3,1)}, \underline{(-1, 3)}$ & $(-3,1), (-1,3), (1,1), \underline{(1,3)}, \underline{(3,1)}$\\
                \hline
  \end{tabular}
            \end{center}
            \caption{Data for local tangencies of types (1), (1') and (2) for each leg or edge of $\Lambda$, up to symmetry. Underline directions correspond to multiplicity four, while the rest have multiplicity two. Here, $e$ is the edge of $\Gamma$ containing the tangency and $e^{\vee}$ denotes the corresponding dual edge in the Newton subdivision of $\sextic$.\label{tab:edgeDirections1And2}}
\end{table}

\begin{lemma}\label{lm:starsTopVertexMostType3d3h} Assume the vertex $v_1$ of  $\Lambda$  is part of a tangency of type (3h) 
  between $\Lambda$ and $\Gamma$. Then, up to symmetry, the star of $\Gamma$ at $v_1$ has rays of direction $(-1,-1)$, $(2,1)$ and $(-1,0)$.  
\end{lemma}

\begin{proof} The proof follows by direct computation. Up to applying the permutation $\tau_0$ from~\autoref{tab:D4Action}, we may assume one of the rays of $\Star_{\Gamma}(v_1)$, called $\rho$, lies in the positive orthant, while the remaining one lies on the sector determined by the inequality $x<\min\{y,0\}$. We write $\rho = (a,b)$.

  The multiplicity two condition at $v_1$ forces the determinant of the matrix  $\begin{pmatrix} a & 0 \\ b & 1 \end{pmatrix}$ to be $\pm 2$. In turn, the smoothness of $\Gamma$ at $v_1$ yields $a-b = \pm 1$. Combining both facts with the positivity of both $a$ and $b$ and our assumptions on $\Star_{\Gamma}(v_1)$ forces $\rho = (2,1)$. The balancing condition determines the remaining ray of the star.
\end{proof}

\begin{lemma}\label{lm:starsTangency3bb} Assume the vertex $v_0$ of $\Lambda$ is part of a tangency of type (3bb) between $\Lambda$ and $\Gamma$. If the tangency is diagonal, the star of $\Gamma$ at $v_0$ has rays with directions $(1,1)$, $(2,1)$ and $(-3,-2)$ (up to $\langle\tau_0\rangle$-symmetry). For horizontal tangencies, the rays of $\Star_{\Gamma}(v_0)$ have directions either $(-1,0)$, $(2, -1)$ and $(-1,1)$, or $(-1,0)$, $(3,1)$ and $(-2,-1)$.
  \end{lemma}

\begin{proof} We first note that the presence of two tangency points in the connected component of $\Gamma\cap \Lambda$ containing $v_0$ forces the intersection multiplicity between $\Gamma$ and $\Lambda$ at this vertex to be three.

  We start by discussing the diagonal case, corresponding to types (3bb1) and (3bb2) in~\autoref{fig:classificationLocalTangencies}. Up to applying the permutation $\tau_0$, we may assume one of the rays of $\Star_{\Gamma}(v_0)$ (called $\rho$) lies on the negative orthant, while the other one belongs to the sector $x>\max\{0,y\}$. Combining the multiplicity information at $v_0$ with the  smoothness of $\Gamma$ forces the direction of $\rho$ to be $(-3,-2)$. The balancing condition confirms the remaining ray of $\Star_{\Gamma}(P)$ has direction $(2,1)$.

  In the horizontal case, we must consider two cases for the location of the unknowns rays of $\Star_{\Gamma}(P)$. First, we assume they lie in the sectors given by the inequalities $x>\max\{0,y\}$ and $y>\max\{0,x\}$. The multiplicity at $v_0$ and the smoothness of $\Gamma$ forces the directions of these two rays to be $(2,-1)$ and $(-1,1)$, respectively.
  On the contrary, if the rays lie in the negative orthant and in the sector determined by $x>\max\{0,y\}$, the same arguments show that the directions of these rays are  $(-2,-1)$ and $(3,1)$, respectively.   
\end{proof}

\begin{lemma}\label{lm:starsTopVertexMostType7a} Assume $\Lambda$ has a type (7) tangency with  $\Gamma$. Then, up to symmetry, the rays of the star of $\Gamma$  at $v_0$ are $(1,1)$, $(-3,-2)$ and $(2,1)$, whereas those for $v_1$ are $(-1,-1)$, $(2,3)$ and $(-1,-2)$.
\end{lemma}

\begin{proof} By construction, the stable intersection between $\Lambda$ and $\Gamma$ equals $3\,v_0 + 3\,v_1$. The spanning rays of the stars of $\Gamma$ at both $v_0$ and $v_1$  are determines from the multiplicity 3 constraint, the smoothness of $\Gamma$, the bidegree of $\Gamma$ and the balancing condition, as we did in the proof of~\autoref{lm:starsTangency3bb}. 
\end{proof}

In the remainder of this section, we give more precisions regarding tangencies of type (8). We let  $P, P'$ and $P''$ be the corresponding tropical tangencies between $\Gamma$ and $\Lambda$. As usual, we assume $\Lambda$ has a slope one edge. The intersection $\Gamma\cap \Lambda$ is a quartet-tree in $\RR^2$ with edges of lengths $L_1, \ldots, L_5$. The distribution of lengths in the skeleton of  $\Gamma$ is shown in the top-left picture from~\autoref{fig:chipFiringTreeShapeIntersections}.

By definition, the divisor $2P+2P'+2P''$ encoding the tangency points is linearly equivalent to the divisor $D$ on $\operatorname{Sk}(\Gamma)$ corresponding to the  stable intersection between $\Gamma$ and $\Lambda$. Up to translation, we may assume $v_0=(0,0)$. With this convention, $D$ consists of the following six points in $\RR^2$, each counted with multiplicity one:
\begin{equation}\label{eq:p1Top6}
  \begin{array}{lll}
    p_1\!:= (-L_2,0), &       p_2 \!:= (0,-L_1), &       p_3 \!:= (0,0),\\
    p_4\!:= (L_3, L_3+L_4), &       p_5\!:= (L_3+L_5, L_3), &       p_6\!:= (L_3,L_3).
  \end{array}
   \end{equation}

By exploiting symmetry and invoking the tropical genericity of $\Gamma$, we may assume $L_1<L_2, L_4, L_5$. Our next objective is to determine the possible locations of $P$, $P'$ and $P''$ in the tree $\Gamma\cap \Lambda$. The answer  depends on the relative order between suitable linear combinations of $L_1,\ldots, L_5$. The next result gives precise formulas for each case, in agreement with~\autoref{rm:genericityOfGamma}.

\begin{proposition}\label{pr:TreeShapeIntersections} Assume $\Lambda$ has a type (8) tangency with $\Gamma$, and let $L_1, \ldots, L_5$ be the five edge lengths of the tree-shaped intersection $\Gamma\cap \Lambda$. If the values of $L_1, \ldots, L_5$   are generic subject to the condition  $L_1<L_2,L_4,L_5$, there are precisely three options for the location of the tangency points between $\Lambda$ and $\Gamma$,  depending on which  term achieves the unique minimum value $\min\{L_3, L_4-L_1, L_5-L_1\}$. The corresponding three divisors are  depicted at the bottom of~\autoref{fig:chipFiringTreeShapeIntersections}.
\end{proposition}

\begin{figure}
  \includegraphics[scale=0.32]{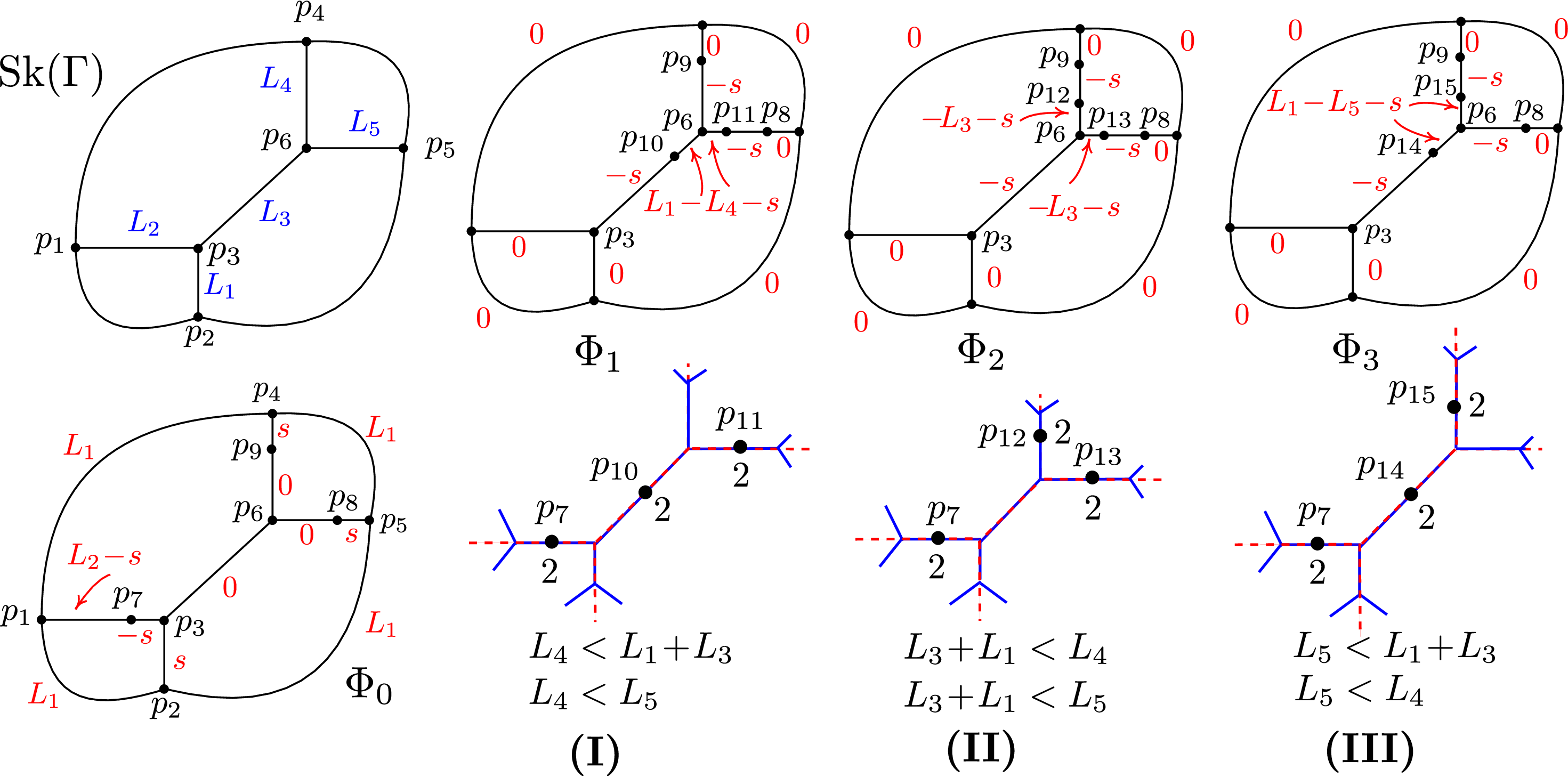}
  \caption{From left to right and top to bottom: skeleton of $\Gamma$ with a type (8) tangency, and four tropical regular functions on it leading to the three possibilities for the distribution of the tangency points under the generic assumption that $L_1<L_2, L_4, L_5$ and that $\min\{L_3, L_5-L_1, L_4-L_1\}$ is attained exactly once. Planar coordinates for the 15 points $p_1, \ldots, p_{15}$ can be found in~\eqref{eq:p1Top6} and~\eqref{eq:p7Top15}.\label{fig:chipFiringTreeShapeIntersections}}
  \end{figure}

\begin{proof} The result follows by a direct chip-firing computation on the skeleton of $\Gamma$. Coordinates for its six vertices appear in~\eqref{eq:p1Top6}. Next, we consider the following nine points on the skeleton of $\Gamma$, seen in~\autoref{fig:chipFiringTreeShapeIntersections}:
  \begin{equation}\label{eq:p7Top15}
    \begin{array}{lll}
        p_7\!:= (\frac{L_1-L_2}{2},0), & p_8\!:= (L_3+L_5-L_1, L_3), &       p_9\!:= (L_3, L_3+L_4-L_1),\\
 p_{10}\!:= (\frac{L_3+L_4-L_1}{2}, \frac{L_3+L_4-L_1}{2}), &           p_{11}\!:= (L_3 + \frac{L_5-L_4}{2}, L_3), &           p_{12}\!:= (L_3, \frac{L_3+L_4-L_1}{2}) ,\\
  p_{13}\!:= (\frac{L_3+L_5-L_1}{2}, L_3), &       p_{14}\!:= (\frac{L_3+L_5-L_1}{2},\frac{L_3+L_5-L_1}{2}) , &       p_{15}\!:= (L_3, L_3 + \frac{L_4-L_5}{2}). 
    \end{array}
  \end{equation}

\noindent  We claim that  $D := p_1 +\ldots + p_6$ is linearly equivalent to one of the following three divisors:
  \[D_1:=2p_7 + 2p_{10}+2p_{11}, \; D_2:=2p_7+ 2p_{12} + 2p_{13} \; \text{ or }\; D_3:=2p_7+2p_{14} + 2p_{15}.
  \]
We prove this  in two steps. First, we consider the tropical function $\Phi_0$ from~\autoref{fig:chipFiringTreeShapeIntersections}. We have 
  \[\divisor(\Phi_0) = -p_1 +2 p_7 -p_2 - p_4 - p_5 + p_8 +p_{9}.\]
  This function produces a new tropical divisor on $\Sk(\Gamma)$ linearly equivalent to $D$, namely,
  \[D_0 := D +\divisor(\Phi_0) = 2 p_7 + p_3 + p_6+p_8+p_9.\]

  After the divisor $D_0$ is produced, we can fire from $p_3$, $p_8$ and $p_9$ towards the point $p_6$. Note that $\dist(p_3,p_6) = L_3$, $\dist(p_6,p_8)=L_5-L_1$ and $\dist(p_6,p_9)=L_4-L_1$. Our genericity assumptions on the edge lengths of the graph ensures these three expressions are distinct. In particular, $\min\{L_3, L_5-L_1, L_4-L_1\}$ is  uniquely achieved. The minimum value determines which  point reaches $p_6$ first. In all cases, we obtain two chips on the point $p_6$ and one chip on each of the edges adjacent to $p_6$ not realizing the minimum value. Moving these last two chips towards $p_6$ and both chips on $p_6$ towards them  produces one of the three divisors  $D_1$, $D_2$ or $D_3$.

  The tropical functions $\Phi_1$, $\Phi_2$ and $\Phi_3$ seen in~\autoref{fig:chipFiringTreeShapeIntersections} encode the combination of the two moves described above starting from the divisor $D_0$ for each choice of shortest distance. They can be used to confirm that $D_0$, and thus, $D$ is linearly equivalent two one of the three divisors seen at the bottom of the figure. More precisely, we have
  \[\begin{cases}
  D_1 = D_0 + \divisor(\Phi_1) &  \text{ if }L_4-L_1 < L_3, L_5 -L_1,\\
  D_2 = D_0 + \divisor(\Phi_2) &  \text{ if }L_3<L_4-L_1, L_5 -L_1,\\
  D_3 = D_0 + \divisor(\Phi_2) &  \text{ if }L_5-L_1 < L_3, L_4 -L_1.
  \end{cases}\qedhere
  \]
  \end{proof}

\section{Local lifting of multiplicity two tangencies: the trivalent case}\label{sec:trivalentLifts}

In this section, we discuss the local lifting multiplicities of trivalent $(1,1)$-tropical curves $\Lambda$ that are tangent to a fixed smooth $(3,3)$-tropical curves labeled $\Gamma$ at a point of tropical multiplicity two. We are particularly interested in determining the degree of the field extension of $\K$ where such lifts are defined. We defer the discussion of higher multiplicity tangencies to~\autoref{sec:appendix1}.

The action of $\Dn{4}$ on tropical curves in $\TPr^1\times \TPr^1$ allows us to exploit symmetry to restrict the number of cases to be analyzed. In particular, throughout this section we will always assume the unique edge of $\Lambda$ has slope one. Furthermore, we will often reduce to the case when tangencies on legs of $\Lambda$ occur along the positive horizontal leg, and tangencies at vertices occur at $v_1$. Such conditions follow the conventions of the configurations depicted in~\autoref{fig:classificationLocalTangencies}. 

Local tangencies of multiplicity two between $\Lambda$ and $\Gamma$ can be treated using the techniques developed by Len and Markwig in~\cite{LM17} for lifting tropical tangent lines to tropical curves in $\TPr^2$. In particular,~\cite[Proposition 3.5]{LM17} ensures that type (1) tangencies between $\Lambda$ and $\Gamma$ do not lift to classical tangencies between $V(\ell)$ and $V(\sextic)$. Type (3) tangencies will be discussed in~\autoref{sec:type-3-tangencies}.


Next, we analyze type (2) tangencies. By~\cite[Proposition 3.6]{LM17}, we know that the local lifting multiplicity of a type (2) tangency between a tropical line and $\Gamma$ equals one. Our next result confirms that the same is true between $\Lambda$ and $\Gamma$ since there is a  linear constraint between the initial forms of the relevant coefficients of $\ell$.

\begin{lemma}\label{lm:type2Horiz} Let $(\ell, p)$ be a local lift of a tangency point $P$ of type  (2) between $\Lambda$ and $\Gamma$. Then, $\bar{p}$ is unique and the  coefficients of $\ell_P$ satisfy a linear relation. Furthermore, $p$ and the relevant coefficients from $\ell$ are uniquely determined by their initial forms featured in $\ell_P$. 
  {Tables}~\ref{tab:initialFormsType2Horiz} and~\ref{tab:initialFormsType2Diag} give precise values for $(\ell_P, \bar{p})$ for convenient $\Dn{4}$-orbit representatives. 
\end{lemma}

\begin{proof} Exploiting the $\Dn{4}$-symmetry, we assume the unique edge of $\Lambda$ has slope one. It is enough to treat two cases, namely  the tangency occurs either at the edge  or at the positive horizontal leg  of $\Lambda$. In the first situation, $\ell_P = \bar{y}+ \bar{n}\, \bar{x}$, whereas in the second one we have $\ell_P = \bar{x}\,\bar{\du}(\bar{y} + \bar{n}/\bar{\du})$.   We use $\ell_P$  to eliminate the variable $\bar{y}$ from the local equations $\sextic_P=\ell_P=W_P=0$.

  Manipulating the ideal $I_P=\langle\sextic_P, \ell_P,W_p\rangle \cap\resK[\bar{\du}^{\pm}, \bar{n}^{\pm}][\bar{x}]$ yields one linear and one constant constraint for each of the tangency types seen in~\autoref{fig:localNPType2}. For the eight possible horizontal tangencies (seen in the top row of the figure), we get two possible constants, namely
    \begin{equation}\label{eq:linearunType2Horiz}
 4\bar{a}\,\bar{c}\,\bar{n} + \bar{b}^2\, \bar{\du}\quad \text{(for types (2A), (2C), (2E), (2G))}\;\text{ and }\;4\bar{a}\,\bar{c}\,\bar{\du} + \bar{b}^2\, \bar{n} \quad \text{(for the other types)}.
  \end{equation}
    In turn, $I_P$ contains one of  five  possible linear polynomials depending on the star of $\Gamma$ at $P$, namely, $\bar{b}\,\bar{n}\,\bar{x} - 2\,\bar{a}\,\bar{\du}$ (for types (2A) and (2F)), $-\bar{b}\,\bar{\du}\,\bar{x} + 2\,\bar{a}\,\bar{n}$ (for types (2B) and (2E)), $\bar{b}\,\bar{x} + 2\,\bar{a}$ (for types (2C) and (2D)),    $\bar{b}\,\bar{\du}^2\bar{x} + 2\,\bar{a}\,\bar{n}^2$ (for type (2G)) and $2\bar{c}\,\bar{n}\,\bar{x} - \bar{b}\,\bar{\du}$ (for type (2H)). We conclude from here that  $(\bar{n}/\bar{\du}, \bar{p})$ is uniquely determined. The precise values are listed in~\autoref{tab:initialFormsType2Horiz}.

For diagonal tangencies, we obtain four possible constant polynomials (two of which correspond to two different types) and six possible constant polynomial (one for each type). The resulting solutions $(\bar{n}, \bar{p})$ are unique. We list their values in~\autoref{tab:initialFormsType2Diag}.



It remains to determine the lifting multiplicity of $(\Lambda,P)$. The last column in both 
{Tables}~\ref{tab:initialFormsType2Horiz} and~\ref{tab:initialFormsType2Diag} shows the initial form of the $3\times 3$-Jacobian of the local system defining the tangency $P$ with respect to the variables $x, y$ and $n/\du$ or $n$, respectively. Since these expressions are Laurent monomials in the coefficients of $\sextic_P$, we conclude from \autoref{lm:multivariateHensel} that the values of $x,y$ and $n/\du$, respectively $n$, are uniquely determined by their initial forms. This concludes our proof.
\end{proof}

\begin{figure}
  \includegraphics[scale=0.33]{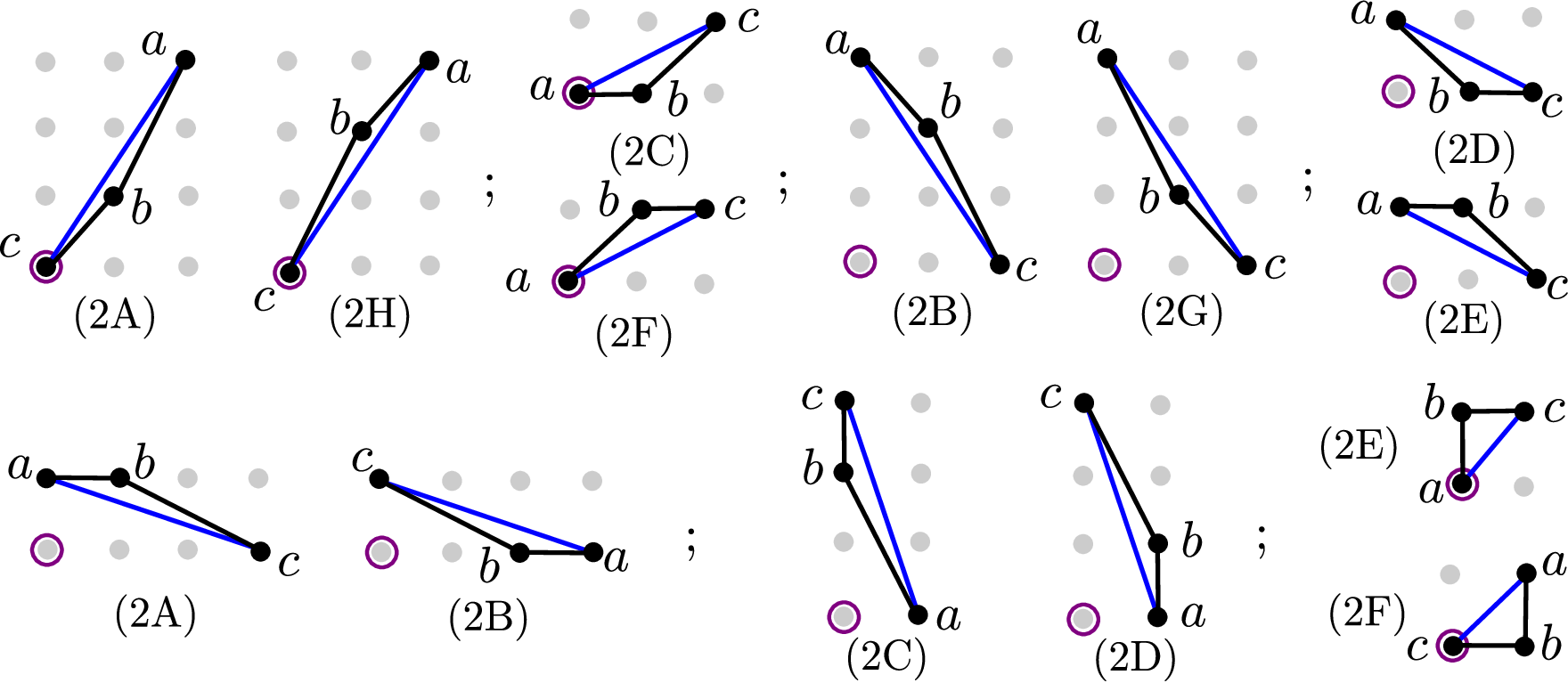}
\caption{From top to bottom: relevant cells in the Newton subdivision of $\sextic$ and notation for the corresponding coefficients for all type (2) tangencies of multiplicity two that occur on $\Gamma$, up to $\Dn{4}$-symmetry. We group them in pairs if they are connected by an edge dual to the one carrying the tangency. The first row corresponds to horizontal tangencies, whereas the second row marks those occurring along the slope one edge of $\Lambda$. The marked lattice point on each case determines the polytope. \label{fig:localNPType2}}
\end{figure}

        \begin{table}[tb]\begin{center}
  \begin{tabular}{|c||c|c||c|}
    \hline tangency & $\bar{n}/\bar{\du}$ & $\bar{p}$ & $\det(\Jac(\sextic_P, \ell_P,W_P; \bar{n}/\bar{\du}, \bar{x}, \bar{y})(\bar{n}/\bar{\du};\bar{p}))$ \\
\hline
(2A) &  $-4\,\bar{a}\,\bar{c}/\bar{b}^2$  & $(-8\,\bar{a}^2\bar{c}/b^3, \;\bar{b}^2/(4\,\bar{a}\,\bar{c}))$ & $\bar{a}\,\bar{b}^2/(2\,\bar{c})$ \\
(2H) & $-\bar{b}^2/(4\,\bar{a}\,\bar{c})$ &  $(-\bar{b}^3/(8\,\bar{a}\,\bar{c}^2), \;4\,\bar{a}\,\bar{c}/\bar{b}^2)$ & $-8\,\bar{a}^3\bar{c}/\bar{b}^2$
 \\
\hline
(2C) & $-4\,\bar{a}\,\bar{c}/\bar{b}^2$ & $(-2\,\bar{a}/\bar{b},\; \bar{b}^2/(4\,\bar{a}\,\bar{c}))$ &  $\bar{a}\,\bar{b}^4/(8\,\bar{c}^3)$\\
(2F) & $-\bar{b}^2/(4\,\bar{a}\,\bar{c})$ & $(-\bar{b}/(2\,\bar{c}), \;4\,\bar{a}\,\bar{c}/\bar{b}^2)$ & $-32\,\bar{a}^5\bar{c}/\bar{b}^4$\\
\hline
(2E) & $-4\,\bar{a}\,\bar{c}/\bar{b}^2$ & $(-\bar{b}/(2\,\bar{c}),\; \bar{b}^2/(4\,\bar{a}\,\bar{c}))$ & $\bar{b}^8/(128\,\bar{a}\,\bar{c}^5)$\\
(2D) & $-\bar{b}^2/(4\,\bar{a}\,\bar{c})$ & $(-2\,\bar{a}/\bar{b},\; 4\,\bar{a}\,\bar{c}/\bar{b}^2)$ & $-512\,\bar{a}^7\bar{c}^3/\bar{b}^8$\\
\hline
(2G) & $-4\,\bar{a}\,\bar{c}/\bar{b}^2$ & $(-\bar{b}^3/(8\,\bar{a}\,\bar{c}^2),\; \bar{b}^2/(4\,\bar{a}\,\bar{c}))$ & $\bar{b}^{20}/(524288\,\bar{a}^7\bar{c}^{11})$ \\
(2B) &  $-\bar{b}^2/(4\,\bar{a}\,\bar{c})$ & $(-\bar{b}^3/(8\,\bar{a}^2\bar{c}), \; 4\,\bar{a}\,\bar{c}/\bar{b}^2)$ & $-2097152\,\bar{a}^{13}\bar{c}^9/\bar{b}^{20}$\\
\hline
  \end{tabular}
          \end{center}
    \caption{Values for $(\bar{n}/\bar{\du}, \bar{p})$ corresponding to a lift $(\ell,p)$ of a type (2) tangency point $P$ between $\Lambda$ and $\Gamma$ occurring along the top horizontal leg of $\Lambda$ when the later has a slope one edge. The coefficients $a,b,c$ associated to points in  $P^{\vee}$ are indicated in~\autoref{fig:localNPType2}.\label{tab:initialFormsType2Horiz}}
        \end{table}

        \begin{table}[tb]
  \begin{tabular}{|c||c|c||c|}
    \hline tangency & $\bar{n}$ & $\bar{p}$ & $\det(\Jac(\sextic_P, \ell_P,W_P;  \bar{n}, \bar{x}, \bar{y})(\bar{n};\bar{p}))$\\
    \hline
    (2A) & $-4\,\bar{a}\,\bar{c}/\bar{b}^2$ & $(-2\bar{a}/\bar{b},\; -8\,\bar{a}^2\bar{c}/ \bar{b}^3)$ & $-8\,\bar{a}^3\,\bar{c}/\bar{b}^2$ \\
    (2B) & $-\bar{b}^2/(4\,\bar{a}\,\bar{c})$ &  $(-\bar{b}/(2\bar{a}),\;-\bar{b}^3/(8\,\bar{a}^2\bar{c}))$ & $\bar{b}^2\,\bar{c}/(2\,\bar{a})$\\ 
    \hline 
    (2C) & $-4\,\bar{a}\,\bar{c}/\bar{b}^2$ 
&      $(-\,\bar{b}^3/(8\,\bar{a}\,\bar{c}^2),\; -\bar{b}/(2\bar{c}))$
       & $ -\bar{a}\,\bar{b}^2/(2\,\bar{c})$\\
      (2D) & $-\bar{b}^2/(4\,\bar{a}\,\bar{c})$
      & $( -8\,\bar{a}^2\bar{c}/\bar{b}^3,\;  -2\,\bar{a}/\bar{b})$
 & $-8\,\bar{a}^3\,\bar{c}/\bar{b}^2$\\
    \hline
    (2E) & $-4\bar{a}\,\bar{c}/\bar{b}^2$ & $(-\bar{b}/(2\,\bar{c}),\;-2\,\bar{a}/\bar{b})$ & $-2\,\bar{a}\,\bar{c}$\\
    (2F) &  $-\bar{b}^2/(4\bar{a}\,\bar{c})$ & $(-2\,\bar{c}/\bar{b},\;-\bar{b}/(2\,\bar{a}))$ & $-2\,\bar{a}\,\bar{c}$\\
    \hline
  \end{tabular}
    \caption{Values for $(\bar{n}, \bar{p})$ and initial form of the determinant of the Jacobian at $p$, corresponding to a lift $(\ell,p)$ of a type (2) tangency point $P$ between $\Lambda$ and $\Gamma$ occurring along the slope one edge of $\Lambda$. The coefficients $a,b,c$ associated to points in  $P^{\vee}$ are indicated in~\autoref{fig:localNPType2}. Rows grouped in pairs are related by 
      $\tau_1^{-1}$.  \label{tab:initialFormsType2Diag}}
        \end{table}


        Our next task is to analyze tangencies of multiplicity two  occurring at the top vertex $v_1$ of $\Lambda$. We label these cases as (4a), (5a) and (6a). The first and last are treated jointly, since their behavior is similar. Type (5a) is discussed at the end of this section.  We analyze two separate scenarios, depending on the number of tangencies lying on the interior of a leg adjacent to $v_1$. 
Note that our genericity assumption on $\sextic$ relative to $\Gamma$ prevents any liftable tritangent from having tangencies on  the relative interior of both legs. Indeed, such situation will lead to a local system with 9 equations in 8 variables, which will have no solution when $\sextic$ is generic.

        \begin{remark}\label{rem:linearConstraintType2Horizontal}
  The proof of~\autoref{lm:type2Horiz} confirms that when the tangency point occurs along the positive horizontal leg of a $(1,1)$-curve with a slope one edge, then there are two possibilities for the linear equation relating $\bar{\du}$ and $\bar{n}$, as seen in~\eqref{eq:linearunType2Horiz}. These equations will play a role in determining lifting multiplicities of tangency types (4a) and (6a).
\end{remark}


\begin{proposition}\label{pr:4a6aNoAdjacentLeg} Let $\sextic$ be  generic relative to the curve $\Gamma$. Assume that $\Lambda$ is trivalent and is tangente to  $\Gamma$ at a point $P$ of type (4a) or (6a). Furthermore, suppose that $\Lambda$ contains no other tangency point  along a leg adjacent to  $P$. Then, the pair $(\Lambda, P)$ has lifting multiplicity one if the tangency $P$ is horizontal or vertical, and two in  the diagonal case.
\end{proposition}

\begin{proof} Without loss of generality, we assume $\Lambda$ has an edge of slope one and that $P$ is its top vertex. We prove the statement when the tangency $P$ is either horizontal or diagonal. The vertical case will follow by applying the map $\tau_0$ from~\autoref{tab:D4Action}. In turn, the result when the tangency occurs at the lower vertex of $\Lambda$ is obtained by applying the map $\tau_{1}^2$.

        \begin{figure}
  \includegraphics[scale=0.45]{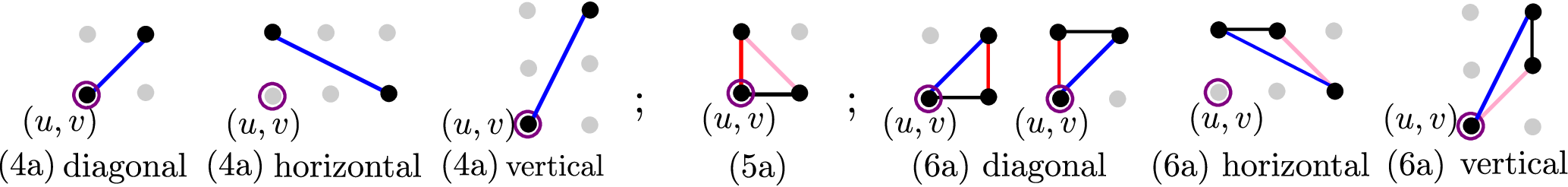}
  \caption{Relevant cells in the Newton subdivision of $\sextic$ and their position in $\ZZ^2$, for representatives of $\Dn{4}$-symmetric classes of tangency types (4a), (5a) and (6a). We assume $\Lambda$ has a slope one edge and the tangency occurs at its top vertex. We record the direction of the edge or leg of $\Lambda$ responsible for each tangency and one distinguished vertex $(u,v)$ in the Newton subdivision. \label{fig:LocalNP}}
\end{figure}

 We write $p=(x,y)$.  As it is customary, we assume $P=(0,0)$ and $\sextic \in R[x,y]\smallsetminus \mathfrak{M} R[x,y]$. In particular, we have that $\du,n,x,y\in R\smallsetminus \mathfrak{M}$ for any lift $(\ell,p)$ of $(\Lambda,P)$.  The local equations arising from the tangency $P$ provide information on the initial forms of the coefficients $\du$ and $n$. The condition that no tangency between $\Lambda$ and $\Gamma$ lies on a leg of $\Lambda$ adjacent to the vertex $P$ ensures that the value of $\bar{n}$ is determined by the remaining tangency points. We show that $\bar{\du},\bar{p}$ are uniquely determined by $\bar{n}$. In turn, the uniqueness of $\du$ and $p$ will be a consequence of~\autoref{lm:multivariateHensel}.

We start by proving the horizontal case.  Following the notation of~\autoref{fig:LocalNP}, the initial forms $(\bar{\du}, \bar{p})$ are determined by the local system  $\sextic_P=\ell_P=W_P=0$, where
 \begin{equation*}\label{eq:4a6aH}
\sextic_P=\bar{x}^u\bar{y}^{v}(\overline{a_{u,v+1}}\,\bar{y} + \varepsilon\,\overline{a_{u+1,v+1}} \,\bar{x}\,\bar{y} + \overline{a_{u+2,v}}\, \bar{x}^2),\; \ell_P =  \bar{y} + \bar{n}\,\bar{x} + \bar{\du}\,\bar{x}\,\bar{y} \;\; \text{and}\; \;  W_P=\det(\Jac(\sextic_P,\ell_P; \bar{x},\bar{y})).
 \end{equation*}
 We set $\varepsilon = 0$ if the tangency type of $P$ is (4a), and $\varepsilon =1$ if it is (6a).

 After neglecting the monomial factor $\bar{x}^{u}\bar{y}^v$ from $\sextic_P$ when computing the Wronskian, we have
 \[W_P = -2\,\overline{a_{u+2,v}}\,\bar{\du}\, \bar{x}^2 + \overline{a_{u,v+1}}\,\bar{\du}\, \bar{y} + \overline{a_{u,v+1}}\,\bar{n} - 2 \,\overline{a_{u+2,v}}\,\bar{x} + \varepsilon \,\overline{a_{u+1,v+1}} (\bar{n}\,\bar{x} -\bar{y}).
 \]
 A~\sage~computation produces the following two polynomials after eliminating the variable $\bar{y}$ from the ideal $\langle \sextic_P, \ell_P, W_P\rangle \subseteq R[\bar{n}^{\pm}, \bar{\du}^{\pm},\bar{x}^{\pm}, \bar{y}^{\pm}]$:
\begin{equation*}\label{eq:4a6aNoy}
   (\varepsilon\,\overline{a_{u+1,v+1}}\,\bar{n} - \overline{a_{u+2,v}})\,\bar{x} + 2\,\overline{a_{u,v+1}}\, \bar{n} \quad \text{ and } \quad
     4\,\overline{a_{u,v+1}}\,\overline{a_{u+2,v}}\,\bar{n}\,\bar{\du} + (\varepsilon\,\overline{a_{u+1,v+1}}\,\bar{n}   -\overline{a_{u+2,v}})^2.
\end{equation*}
The right-most expression confirms that $\bar{\du}$ is a Laurent monomial in $\bar{n}$, whereas $\bar{x}$ is a rational function in $\bar{n}$. If $P$ has tangency type (4a), $\bar{x}$ becomes a Laurent monomial in $\bar{n}$. In turn, for type (6a) tangencies, we must ensure the expression $(\overline{a_{u+1,v+1}}\,\bar{n} - \overline{a_{u+2,v}})$ is not zero. We do so by invoking the genericity of $\sextic$. Indeed, since the legs of $\Lambda$ adjacent to $P$ have no other tangency points, then the coefficient $a_{u+1, v+1}$ plays no role in the determination of $\bar{n}$. Therefore, the genericity of $\sextic$ relative to $\Gamma$ ensures that $\bar{n} \neq \overline{a_{u+2,v}}/\overline{a_{u+1,v+1}}$, as we wanted.

Replacing the expressions for $\bar{x}$ and $\bar{\du}$ back in $\ell_P$ determines $\bar{y}$ and $\bar{p}$,  uniquely. We have
\[
\bar{\du} = -\frac{(\varepsilon\,\overline{a_{u+1,v+1}}\,\bar{n}   -\overline{a_{u+2,v}})^2}{4\, \overline{a_{u,v+1}}\, \overline{a_{u+2,v}}\,\bar{n}}
 \quad \text{and} \quad
\bar{p} = \left(-\frac{2\overline{a_{u,v+1}}\, \bar{n}}{\varepsilon\,\overline{a_{u+1,v+1}}\,\bar{n}   -\overline{a_{u+2,v}}}, \, \frac{4\, \overline{a_{u,v+1}}\, \overline{a_{u+2,v}}\,\bar{n}^2}{(\varepsilon\,\overline{a_{u+1,v+1}}\,\bar{n})^2   -\overline{a_{u+2,v}}^2}\right).
\]
A calculation in~\sage~confirms that the $P$-initial form of the Jacobian of $(\sextic, \ell, W)$ with respect to $(\du,x,y)$ does not vanish at $(\bar{\du}, \bar{x}, \bar{y})$, since it equals $8\,\overline{a_{u,v+1}}^3\,\overline{a_{u+2,v}}\,\bar{n}^3/(\varepsilon\,\overline{a_{u+1,v+1}}\,\bar{n}   -\overline{a_{u+2,v}})^2$.

\smallskip

Finally, we discuss the diagonal case. Keeping our earlier conventions for the variable $\varepsilon$, and exploiting symmetry, we see that  $(\bar{\du}, \bar{p})$ is determined by the  system  $\sextic_P=\ell_P=W_P=0$, where
 \begin{equation*}\label{eq:4a6aD}
   \sextic_P=\bar{x}^u\bar{y}^{v}(\overline{a_{u,v}}\,\bar{y} + \overline{a_{u+1,v+1}} \,\bar{x}\,\bar{y} + \varepsilon\,\overline{a_{u+1,v}}\, \bar{x}),   \; \ell_P =  \bar{y} + \bar{n}\,\bar{x} + \bar{\du}\,\bar{x}\bar{y} \;\; \text{and} \;\;  W_P=\det(\Jac(\sextic_P,\ell_P; \bar{x},\bar{y})).
 \end{equation*}
  Up to a Laurent monomial in $\bar{x}, \bar{y}$, we have
 \[
 W_P = (\overline{a_{u+1,v+1}}\,\bar{n} - \varepsilon\,\overline{a_{u+1,v}}\,\bar{\du})\,\bar{x} - \overline{a_{u+1,v+1}}\,\bar{y} - \varepsilon\,\overline{a_{u+1,v}}.
 \]
 This gives a unique expression for $\bar{y}$ as a rational function in $\bar{n}, \bar{\du}$ and $\bar{x}$.  Substituting this value back in $\sextic_P$ and $\ell_P$, and manipulating the resulting polynomials in $\bar{x}$ yields two constraints: 
 \[(-\overline{a_{u,v}}\,\bar{\du} -  \varepsilon\,\overline{a_{u+1,v+1}})\,\bar{x} - 2\,\overline{a_{u,v}} = 
 \overline{a_{u,v}}^2{\bar{\du}\,}^2  - 2\,\varepsilon\,\overline{a_{u+1,v+1}}\,\overline{a_{u,v}}\, \bar{\du} + 4\,\overline{a_{u,v}}\,\overline{a_{u+1,v}}\,\bar{n}  +  (\varepsilon\,\overline{a_{u+1,v+1}})^2=0.
 \]
 
 The second equation has two solutions in $\bar{\du}$. Each of them gives a unique solution in $\bar{x}$ for the first equation, and, thus, a unique $\bar{p}$. 
 The precise formulas are:
 \[\bar{\du}\! = \frac{\varepsilon\,\overline{a_{u+1,v}} \pm 2\sqrt{-\overline{a_{u,v}}\, \overline{a_{u+1,v+1}}\,\bar{n}}}{\overline{a_{u,v}}} \;\; \text{and}\;\; \bar{p}\!=\left(\frac{-\overline{a_{u,v}}}{\varepsilon\,\overline{a_{u+1,v}} \pm \sqrt{-\overline{a_{u,v}}\, \overline{a_{u+1,v+1}}\,\bar{n}}} , \mp \frac{\sqrt{-\overline{a_{u,v}}\, \overline{a_{u+1,v+1}}\,\bar{n}}}{\overline{a_{u+1,v+1}}} \right)\!.
 \]
 As before, $\overline{a_{u+1,v+1}}$ plays no role in determining $\bar{n}$, hence the expression for $\overline{p}$ is well-defined per our genericity assumption on $\sextic$.
 
A \sage~computation confirms that the Jacobian of  $(\sextic_P, \ell_P, W_P)$ in $(\bar{\du}, \bar{x},\bar{y})$ has determinant
 \[2\sqrt{-\overline{a_{u,v}}\, \overline{a_{u+1,v+1}}\,\bar{n}}\, (\varepsilon\,\overline{a_{u+1,v}}\mp \sqrt{-\overline{a_{u,v}}\, \overline{a_{u+1,v+1}}\,\bar{n}})^2 / (\varepsilon\,\overline{a_{u+1,v}}\pm \sqrt{-\overline{a_{u,v}}\, \overline{a_{u+1,v+1}}\,\bar{n}})^2.
 \]
 This expression is invertible, so each  solution $(\bar{\du}, \bar{p})$ lifts to a unique $(\du, p)$  by~\autoref{lm:multivariateHensel}. 
\end{proof}

The proof of the previous result gives precise arithmetic information regarding the defining field for the lifting of such local tangencies.

\begin{corollary}\label{cor:4a6aField} Let  $P$ be a tangency of type (4a) or (6a) on the top vertex of $\Lambda$. Assume that  $\Lambda$ is trivalent, of slope one, and that it  has no other tangency points on any  leg adjacent to $P$. Then:
  \begin{enumerate}[(i)]
  \item If $P$ is either a horizontal or vertical tangency, the lifting of $(\Lambda, P)$ is defined over $\K$.
  \item In the diagonal case, the two liftings of $(\Lambda, P)$ are defined over a quadratic extension of $\K(n)$. Furthermore, they are defined over $\K$ if, and only if,  $\sqrt{-\overline{a_{u,v}}\, \overline{a_{u+1,v+1}}\,\bar{n}}\in \resK$.
  \end{enumerate}
\end{corollary}

\begin{proposition}\label{pr:4a6aWithAdjacentLeg} Let $\sextic$ be 
  generic relative to  $\Gamma$. Assume that $\Lambda$ is trivalent and that it has a tangency point of type (4a) or (6a) along one of its vertices, which we label $P'$, and a second tangency, denoted by $P$, on a leg  of $\Lambda$ adjacent to $P'$.  If the remaining leg adjacent to $P'$ has no further tangencies with $\Gamma$, then  the lifting multiplicity of the triple $(\Lambda, P,P')$ agrees with that of $(\Lambda,P)$. Thus, the local lifting multiplicity of $(\Lambda, P')$ equals one. Furthermore, a local lifting of the triple is defined over $\K$ if, and only if, the one for $(\Lambda,P)$ is.\end{proposition}

\begin{proof} Exploiting the $\Dn{4}$-symmetry, we assume that $\Lambda$ has an edge of slope one, $P'$ is the top vertex of $\Lambda$ and $P$ lies on the positive horizontal leg of $\Lambda$. Note that by construction,  $P$ and $P'$ belong to different connected components of $\Gamma \cap \Lambda$.
  Under these conditions, the lifting multiplicity of $(\Lambda, P)$ is determined by~\cite[Table 1]{LM17}:  $P$ will lift only if it has type (2) or (3). In what follows we confirm that the local lifting multiplicity of $(\Lambda,P)$ determines that of the triple $(\Lambda, P, P')$.  

  Before treating each case, we discuss the combinatorics of $P^{\vee}$ and $(P')^{\vee}$. By our hypothesis, the remaining tangency point between $\Lambda$ and $\Gamma$ cannot lie on the vertical leg of $\Lambda$ adjacent to $P$. 
   Thus,  $(P')^{\vee}$ must contain a point with second coordinate 3. This information and the bidegree of $\Gamma$ combined restrict the tangency type of $P$ and the combinatorics of $(P')^{\vee}$. For type (4a), the tangency should be either vertical or diagonal, and $(P')^{\vee}$ must contain the vertex $(0,3)$. In turn, if $P'$ is of type (6a), the tangency can only be diagonal, and the marked vertex of $(P')^{\vee}$ in the notation of~\autoref{fig:LocalNP}    is $(0,2)$.

  First, assume $P$ has type (2). The combinatorial constraints on $P^{\vee}$ and $(P')^{\vee}$ discussed restricts the tangency type of $P$. If $P'$ is a diagonal tangency, then   $P$ can have type (2E), (2D), (2B) or (2G), whereas for a vertical one, the type of $P$ can only be  (2D),  (2E), (2C) or (2F).  Finally, if $P$ has type (3), then its type is (3c).

  In order to prove the statement   we proceed in two steps. First, we show that the local systems $\sextic_P = \ell_P = W_P = 0$ and $\sextic_{P'} = \ell_{P'} = W_{P'}=0$ have the expected number of joint solutions. More precisely, any solution $(\bar{n}/\bar{\du}, \bar{p})$ for the second system yields a unique solution $(\bar{n}, \bar{\du}, \bar{p'})$ of the first one. This is done in~\autoref{lm:VertexAndType23} below.

  Secondly, we show that the corresponding $6\times 6$ Jacobian matrix of the local equations describing $P$ and $P'$ is non-singular by computing the expected initial form of its determinant. This is the content of~\autoref{lm:uniquenesVType23}.
\end{proof}
\begin{remark}\label{rm:4a6aWithTwoTangenciesOnLegs} The genericity of $\sextic$ relative to $\Gamma$ confirms that in the presence of a type (4a) or (6a) tangency at a vertex $P$ of $\Lambda$ with tangencies on both legs adjacent to $P$, the curve $\Lambda$ will not lift to a classical tritangent to $V(\sextic)$. Indeed, the joint system of local equations at all three tangencies will have 9 equations in at most 8 parameters. The genericity of $\sextic$ will yield an empty solution in the torus $\overline{\K}^9$.
  \end{remark}
  
The next two lemmas discuss how the lifting information arising from a tangency $P$ of type (2) or (3) along a leg adjacent to the vertex $P'$ of $\Lambda$ corresponding to a type (4a) or (6a) tangency determines the local information at $P'$.  \autoref{tab:VertexAndHorizontal} lists the values of the unique tuples $(\bar{n},\bar{\du}, \bar{p'})$ obtained from the local data at $P$ for type (2). The ones for type (3) can be computed similarly.

        \begin{table}[tb]
  \begin{tabular}{|c|c||c|c|}
    \hline type $P'$  & $P^{\vee}$  & $\bar{n}$ & $\bar{\du}$ \\
    \hline
    (4a)  vertical & (2D), (2F) & $\overline{b}^2/(\bar{c}\,\bar{\da})$ & $-\bar{b}^4/(4\,\bar{a}\,\bar{c}^2\bar{\da})$  \\
    (4a)  vertical & (2C), (2E) & $16\,\bar{a}^2\bar{c}/(\bar{b}^2\bar{\da})$ & $-64\,\bar{a}^3\bar{c}^2/(\bar{b}^4\bar{\da})$  \\
    \hline
    (4a) diagonal & (2E), (2G) & $-\bar{b}^4/(4\,\bar{a}\,\bar{c}^2\bar{\da})$ & $\overline{b}^2/(\bar{c}\bar{\da})$ \\
    (4a) diagonal & (2B), (2D) & $-64\,\bar{a}^3\bar{c}^2/(\bar{b}^4\bar{\da})$ & $16\,\bar{a}^2\bar{c}/(\bar{b}^2\bar{\da})$  \\
    \hline
    (6a) diagonal & (2E), (2G) & $-4\bar{a}\,\bar{b}^2\bar{\da}/ (4\,\bar{a}\,\bar{c}\,\bar{\da} + \bar{b}^2\,\bar{\ka})^2$ & $16\bar{a}^2\bar{b}^2\bar{c}\,\bar{\da}/ (4\,\bar{a}\,\bar{c}\,\bar{\da} + \overline{b}^2\,\bar{\ka})^2$  \\
    (6a) diagonal & (2B), (2D) & $-64\bar{a}^3\bar{c}^2\bar{\da}/ (4\,\bar{a}\,\bar{c}\,\bar{\ka} + \bar{b}^2\,\bar{\da})^2$ & $16\bar{a}^2\bar{b}^2\bar{c}\,\bar{\da}/ (4\,\bar{a}\,\bar{c}\,\bar{\ka} + \bar{b}^2\,\bar{\da})^2$ \\
    \hline
  \end{tabular}

  \begin{tabular}{|c|c||c|}
    \hline type $P'$  &  $P^{\vee}$  & $\overline{p'}$ \\
    \hline
(4a)  vertical & (2D), (2F) & $(-4\,\bar{a}\,\bar{c}^2\bar{\da}/\bar{b}^4,\; 2\,\bar{a}\,\bar{c}/\overline{b}^2)$\\
    (4a)  vertical & (2C), (2E) & $(-\bar{b}^4\bar{\da}/(64\,\bar{a}^3\bar{c}^2), \; \bar{b}^2/(8\,\bar{a}\,\bar{c}))$   \\
    \hline
    (4a) diagonal & (2E), (2G) & $(-2\,\bar{c}\,\bar{\da}/\overline{b}^2, \; \overline{b}^2/(2\,\bar{a}\,\bar{c}))$\\
(4a) diagonal & (2B), (2D) & $(-\bar{b}^2\bar{\da}/(8\,\bar{a}^2\bar{c}),\; 8\,\bar{a}\,\bar{c}/\bar{b}^2)$\\
    \hline
(6a) diagonal & (2E), (2G) & $ (-(4\,\bar{a}\,\bar{c}\,\bar{\da} + \bar{b}^2\,\bar{\ka})/(2\,\bar{a}\,\bar{b}^2),\; 2\,\bar{b}^2\bar{\da}/(4\,\bar{a}\,\bar{c}\,\bar{\da} - \overline{b}^2\,\overline{\ka})) $\\
(6a) diagonal & (2B), (2D) &     $ (-(4\,\bar{a}\,\bar{c}\,\bar{\da} + \bar{b}^2\,\bar{\ka})/(8\,\bar{a}^2\bar{c}),\; -8\,\bar{a}\,\bar{c}\,\bar{\da}/(4\,\bar{a}\,\bar{c}\,\bar{\ka} - \bar{b}^2\,\bar{\da})) $\\
\hline
  \end{tabular}
  \caption{Values for $(\bar{\du}, \bar{n}, \overline{p'})$ for  a lift $(\ell,p,p')$ of $(\Lambda, P, P')$,  where $\Lambda$ has a slope one edge, $P'$ is the top vertex of $\Lambda$, and $P$ is a type (2) tangency occurring along the positive horizontal leg of $\Lambda$. The coefficients of $\sextic$ corresponding to  $a$, $b$ and $c$ are indicated in~\autoref{fig:localNPType2}. The parameters $\alpha$ and $\beta$ are defined in the proof of~\autoref{lm:VertexAndType23}.\label{tab:VertexAndHorizontal}}
        \end{table}

\begin{lemma}\label{lm:VertexAndType23} Let $P, P'$ be as in~\autoref{pr:4a6aWithAdjacentLeg}. Assume that $\Lambda$ has a slope one edge,  $P'$ is the top vertex of $\Lambda$ and $P$ lies in the horizontal leg of $\Lambda$ adjacent to $P'$. Then, the tuple $(\bar{n}, \bar{\du},  \bar{p'})$ is uniquely determined by the ratio $\bar{n}/\bar{\du}$.
\end{lemma}

\begin{proof}  As discussed in the proof of~\autoref{pr:4a6aWithAdjacentLeg}, we must determine joint solutions of the systems $\sextic_P = \ell_P = W_P = 0$ and $\sextic_{P'} = \ell_{P'} = W_{P'}=0$. Expression~\eqref{eq:linearunType2Horiz}  confirms that the ratio  $\bar{\du}/\bar{n}$ has a unique value (in $\resK$) if $P$ has type (2), and the precise formula depends on the combinatorics of the cell $P^{\vee}$. In turn, the proof of~\autoref{pr:Prop5.2Corrected} shows that if $P$ has type (3c), we also have a unique value for  $\bar{n}/\bar{\du}$: its agrees with $\overline{a_{2,v}}/\overline{a_{2,v+1}}\in \resK$. Here, the vertices $(2,v+1)$ and $(2,v)$ are the endpoints of the  edge  dual to the horizontal edge of  $\Gamma$ containing $P'$.

As usual, we write  $p'=(x',y')$. The values of $(\bar{n}, \bar{\du})$ for any solution to the local system at $P'$ are subject to one polynomial constraint, dependent on the tangency type of $P'$. They are obtained by standard elimination techniques, as in the  proof of~\autoref{pr:4a6aNoAdjacentLeg}. The precise expressions  are:
    \begin{equation*}\label{eq:nuConstraint4v4D6D}
    \begin{aligned}
      \text{(4a): \quad }&\overline{\da}\,\bar{n}^2 + 4\, \bar{a}\,\bar{\du} = 0 \quad \text{(vertical)} ,\qquad  \overline{\da}\,\bar{\du}^2 + 4\, \bar{a}\,\bar{n} = 0\quad \text{(diagonal)}; \\
            \text{(6a): \quad } & (\overline{\da}\,\bar{\du} - \overline{\ka}\,\bar{n})^2 + 4\,\bar{a}\,\overline{\da}\,\bar{n} = 0 \quad\text{(diagonal)}.
    \end{aligned}
  \end{equation*}
    Here, $(a,\da,\ka) := (a_{1,3}, a_{0,2}, a_{0,3})$ if $P'$ is a diagonal tangency, whereas $(a,\da):= (a_{1,1}, a_{0,3})$ for a vertical one.

Dividing these three expressions by $\bar{\du}^2$, we see that $\bar{\du}$ is uniquely determined by $\bar{n}/\bar{\du}$. Therefore,  $(\bar{\du}, \bar{n})$ is unique and it is defined over $\widetilde{\K}$.

The uniqueness of  $\overline{p'}$ follows from the existence of  a linear polynomial in the variable  $\bar{y'}$ in the ideal $\langle \sextic_{P'}, \ell_{P'}, W_{P'}\rangle \subseteq \resK[\bar{\du}^{\pm}, \bar{n}^{\pm}, \overline{y'}^{\pm}, \overline{x'}^{\pm}]$. More precisely, using our earlier notations, we have:
    \begin{equation*}\label{eq:yConstraint4v4D6D}
    \begin{aligned}
      \text{(4a): \quad }& (2\,\bar{\da}\,\bar{n}^2 + 6 \, \bar{a}\,\bar{\du})\,\bar{y'} - \bar{a}\,\bar{n} \quad \text{(vertical)} ,\qquad  2\,\bar{a}\,\bar{n}\,\bar{y'} - \bar{\da}\,\bar{n}\,\bar{\du} \quad \text{(diagonal)};\\
\text{(6a): \quad }&
(2\,\bar{\ka}\,\bar{n}\,\bar{\du} - 2\,\bar{a}\,\bar{n})\,\bar{y'} + \bar{\ka}\,\bar{n}^2 + \bar{\da}\,\bar{n}\,\bar{\du}   \quad \text{(diagonal)}.
    \end{aligned}
    \end{equation*}
    The vanishing of these equations and $\ell_{P'}$ determines a unique solutions for $\overline{y'}$ and $\overline{x'}$,  in each case. Therefore, $(\bar{n},\bar{\du}, \bar{p'})$ is uniquely determined from the data of $(\Lambda,P)$. 
\end{proof}

\begin{lemma}\label{lm:uniquenesVType23}
Let $P, P'$ be as in~\autoref{pr:4a6aWithAdjacentLeg}. Assume $\sextic$ is generic relative to $\Gamma$. Then, the solutions to the combined local equations defining these tropical tangencies have unique lifts $(\ell,p,p')$. 
\end{lemma}

\begin{proof} We assume $(\Lambda, P, P')$ are as in~\autoref{lm:VertexAndType23}. Our statement is a consequence of~\autoref{lm:multivariateHensel}. But in order to use this result, we must determine which variables and local systems to consider for the $6\times 6$ Jacobian matrix. We treat two cases, depending on the nature of the tangency  $P$. 

  First, we assume $P$ has type (2). In this case, we must determine the initial form of the Jacobian of the six polynomials $(\sextic_P, \ell_P, W_P, \sextic_{P'}, \ell_{P'}, W'_{P'})$ with respect to the variables $(x,y, n, x', y', \du)$. This last step is straightforward to check using~\sage. Indeed, this form is always a Laurent monomial in the initial forms of the relevant coefficient $\bar{a}, \bar{b}, \bar{c}$ of $\sextic_{P'}$ if $P'$ has type (4a). For type (6a), this form is also a Laurent monomial but in the variables $\bar{a}, \bar{b}, \bar{c}$ and $\bar{\du}$.

  On the contrary, if $P$ has type (3c), we must alter the polynomials and variables corresponding to the tangency point $P$, as discussed in the proof of~\autoref{pr:Prop5.2Corrected}. In this situation, they become $(\tilde{\sextic}_{P}, z- m_2, W(x,z))$ and the relevant variables are $x,z, m_2$, where $z = y - a_{u,v}/a_{u,v+1} - m_1$, and $m_1$ is a fixed root in $\K$ of the polynomial seen in~\eqref{eq:univpoly3a}. Thus, the Jacobian certifying the uniqueness of $(\ell,p,p')$ involves the variables $(x,z,m_2, x', y', \du)$. Its block diagonal form ensures that the initial form of the determinant is indeed not vanishing. The first block is invertible since it corresponds to a type (3c) tangency.

  An explicit~\sage~computation confirms that the  remaining block $\Jac(\sextic_{P'},\ell_{P'}, W_{P'}; x,y, \du)$ is invertible since its determinant has non-vanishing initial form. Indeed, if $P'$ is a (4a) tangency, the initial form is $8\,\bar{a}^3/(\bar{\da}\,\bar{n}^2)$ if the tangency is vertical, and $\pm 2\,\bar{a}\,\bar{\da}^{3/2}\,\bar{n}/\sqrt{-\bar{a}\,\bar{n}}$ if it is diagonal. In turn, if $P'$ has type (6a), the initial form becomes
  \[ -{\bar{\du}\,}^7\bar{\da}\, \overline{a_{2,v+1}}/\overline{a_{2,v}} (\bar{\ka} - 2\,\bar{\da})(4\,\bar{\ka}\,\bar{\da} \,\overline{a_{2,v+1}}/\overline{a_{2,v}} - (\bar{\ka}\,\overline{a_{2,v+1}}/\overline{a_{2,v}} -\bar{\da})^2) /(\bar{\ka}\,\overline{a_{2,v+1}}/\overline{a_{2,v}} +\bar{\da})^2.
  \]
  The non-monomial factors are polynomials in the initial forms of relevant coefficients of $\sextic$.  These expressions do not vanish per our genericity assumptions.
        \end{proof}


\begin{proposition}\label{pr:5aWithOrWithoutAdjacentLeg} Let $\sextic$ be generic relative to the curve $\Gamma$. Assume that $\Lambda$ is trivalent and that it has a tangency point $P$  of type (5a). Then, the pair
$(\Lambda, P)$ has lifting multiplicity two. Furthermore, both liftings are defined over the same quadratic field extension of $\K$.
\end{proposition}

\begin{proof}
  By symmetry, we may assume that $\Lambda$ has an edge of slope one, and  $P$ is the top vertex of this edge.   We proceed as in the proof of 
  {Propositions}~\ref{pr:4a6aNoAdjacentLeg} and~\ref{pr:4a6aWithAdjacentLeg}, following the notation from~\autoref{fig:LocalNP} to describe  the cell $P^{\vee}$ in the Newton subdivision of $\sextic$. The local system defining the tangency $P$ is given by $\sextic_P=\ell_P=W_P=0$, where
  \[
  \sextic_P = \bar{x}^u\bar{y}^v\left (\overline{a_{u,v}} + \overline{a_{u+1,v}}\,\bar{x} + \overline{a_{u,v+1}}\,\bar{y}\right ), \; \ell_P =  \bar{y} + \bar{n}\,\bar{x} + \bar{\du}\,\bar{x}\,\bar{y} \quad \text{ and } \quad  W_P=\det(\Jac(\sextic_P,\ell_P; \bar{x},\bar{y})).
  \]
  In this case, the Wronskian has the form $W_P = \overline{a_{u,v+1}}\,\bar{n}+\overline{a_{u,v+1}}\,\bar{\du}\,\bar{y}-\overline{a_{u+1,v}} - \overline{a_{u+1,v}}\,\bar{\du}\,\bar{x}$.
  Its vanishing allows us to express $\bar{y}$ as a Laurent polynomial in $\bar{x}$. Substituting this value in $\sextic_P$ yields \[
  2\,\overline{a_{u+1,v}}\,\bar{\du}\,\bar{x} - \overline{a_{u,v+1}}\,\bar{n} + \overline{a_{u,v}}\,\bar{\du} + \overline{a_{u+1,v}}=0,\]
  so $\bar{x}$ is uniquely determined by $\bar{n}$ and $\bar{\du}$. 
  Replacing the unique values $\bar{x}, \bar{y}$ back in $\ell_P$ and removing Laurent monomial factors yields the following quadratic polynomial in $\bar{n}$ and $\bar{\du}$:
  \begin{equation}\label{eq:constraintsunType5a}
    h:=\overline{a_{u,v}}^2\bar{\du}^2  - 2\,\overline{a_{u,v}}(\overline{a_{u,v+1}} \,\bar{n}+ \overline{a_{u+1,v}})\,\bar{\du}  + (\overline{a_{u,v+1}}\,\bar{n} - \overline{a_{u+1,v}})^2.
  \end{equation}
  
We analyze two cases. First, we assume no leg of $\Lambda$ adjacent to $P$ has a tangency point in its relative interior. In this situation, the possible values of $\bar{n}$ are determined by the remaining tangencies between $\Lambda$ and $\Gamma$. Thus, the lifting multiplicity of the tuple $(\bar{\du}, \bar{p})$ equals two, as we wanted to show. The values of the tuple depend on the parameter $\bar{n}$. More precisely, we have
\begin{equation}\label{eq:5aNoAdjacentLeg}
  \begin{aligned}
    \bar{\du} & = \frac{\overline{a_{u+1,v}}\,(1 \pm \sqrt{\overline{a_{u,v+1}}\,\bar{n}/\overline{a_{u+1,v}}})^2}{\overline{a_{u,v}}}\qquad \text{ and }
\\
    \bar{p} & =\left (
  -\frac{\overline{a_{u,v}}\,(\overline{a_{u+1,v}} \pm \sqrt{\overline{a_{u+1,v}}\,\overline{a_{u,v+1}}\,\bar{n}})}{\overline{a_{u+1,v}}^2\,(1\pm \sqrt{\overline{a_{u,v+1}}\,\bar{n}/\overline{a_{u+1,v}}})^2}
  ,\;
   -\frac{\overline{a_{u,v}}\,(\overline{a_{u,v+1}}\,\bar{n} \pm \sqrt{\overline{a_{u+1,v}}\,\overline{a_{u,v+1}}\,\bar{n}})}{ \overline{a_{u+1,v}}\,\overline{a_{u,v+1}} \,(1\pm \sqrt{\overline{a_{u,v+1}}\,\bar{n}/\overline{a_{u+1,v}}})^2}
   \right ).
  \end{aligned}
\end{equation}
The determinant of the Jacobian $\Jac(q,\ell,W;x,y,u)$ evaluated at $(\du,p)$ has expected initial form $\pm 2\,\overline{a_{u,v}}\,\sqrt{\overline{a_{u+1,v}}\,\overline{a_{u,v+1}}\,\bar{n}}$.
Therefore, by~\autoref{lm:multivariateHensel}, the lifting multiplicity of $(\Lambda, P)$ is two. 

Finally, we assume that a leg of $\Lambda$ adjacent to $P$ contains a tangency point, which we call $P'$. By symmetry, we suppose it is  horizontal. Thus, $P'$ determines the possible values of $\lambda':=n/\du$. Inserting the parameter $\overline{\lambda'}$ in the polynomial $h$ from~\eqref{eq:constraintsunType5a} we obtain a quadratic polynomial in $\bar{\du}$ with two solutions:
\begin{equation}\label{eq:bardu5aAdjLeg}
  \bar{\du} := \overline{a_{u+1,v}}\,\overline{a_{u,v+1}}(1 \pm  \,\sqrt{\overline{a_{u,v}}\,\overline{\lambda'}/\overline{a_{u,v+1}}}\,)^2/(\overline{a_{u,v}}\,\overline{\lambda'} -\overline{a_{u,v+1}})^2.
\end{equation}
The coordinates of the point $\bar{p}$ are Laurent polynomials in $\bar{\du}$ and $\overline{\lambda'}$. The value of $\bar{n}$ equals $\bar{\du}\,\overline{\lambda'}$.

Depending on whether $P'$ has type (1), (2) or (3c) we get either no solution in the first case or exactly one value for $\overline{\lambda'}$ for the last two. In type (2), there are only four options for the cell $P^{\vee}$, namely, (2C), (2F), (2D) and (2E). The corresponding values for $\overline{\lambda'}$ for each of  these options can be found in~\autoref{tab:initialFormsType2Horiz}. For type (3), we have $\overline{\lambda'}= \overline{a_{2,j}}/\overline{a_{2,j+1}}$, where $(2,j)$ is the lower vertex of $(P')^{\vee}$.

To determine the lifting multiplicity of $(\Lambda, P,P')$ we invoke~\autoref{lm:multivariateHensel}, as we did in the proof of~\autoref{lm:uniquenesVType23}. If $P'$ has type (2), the $6\times 6$  Jacobian of the tuple $(\sextic_{P}, \ell_P, W_P, \sextic_{P'}, \ell_{P'}, W_{P'})$ relative to the variables $(x,y, n; x', y', \lambda')$ is block upper-triangular. In turn, if $P'$ has type (3), we must consider the Jacobian of $(\sextic_P, \ell_P, W_P, \sextic_{P'}, z-m_2, W(x',z)_{P'})$ relative to $(x,y,n;x',z, m_2)$. This matrix is also block upper-triangular. 

In both cases, the determinant of the bottom diagonal block has non-vanishing expected initial form, as seen in the proof of~\autoref{lm:type2Horiz} and~\autoref{pr:Prop5.2Corrected}, respectively. The determinant of the top-diagonal block has expected initial form
\[
\frac{\overline{a_{u+1,v}}\,\left (\overline{a_{u,v+1}}\,(\overline{a_{u,v}}\,\overline{\lambda'} - \overline{a_{u,v+1}})^2 + (\overline{a_{u,v}}-\overline{a_{u,v+1}}\,\overline{\lambda'}) ( \overline{a_{u,v+1}} \pm \sqrt{\overline{a_{u,v}}\,\overline{a_{u,v+1}}\,\overline{\lambda'}}\,)^2 \right )}{(\overline{a_{u,v}}\overline{\lambda'} - \overline{a_{u,v+1}})^2}\,.
\]
The genericity of $\sextic$ relative to $\Gamma$ ensures this expression does not vanish. Thus, in the presence of a tangency point $P'$, the lifting multiplicity of $(\Lambda, P)$ is also two.

The statement regarding the field of definition of any local lift of $(\Lambda, P)$ follows by construction, since the same is true for the initial forms of $\du$ and $p$ thanks to~\eqref{eq:5aNoAdjacentLeg} and~\eqref{eq:bardu5aAdjLeg}.
\end{proof}

\begin{remark}\label{rm:(5a)withAdjacentLegQuadratic} In the presence of two tangencies $P$ and $P'$, where $P'$ has type (5a) and $P'$ lies on a leg of $\Lambda$ adjacent to $P$, the proof of~\autoref{pr:5aWithOrWithoutAdjacentLeg} confirms that the local lift $(\ell,p,p')$ of $(\Lambda, P, P')$ is defined over a quadratic extension of $\K$. In addition, the lifting multiplicity of the triple $(\Lambda, P, P')$ equals to the product of the lifting multiplicities corresponding to  $P$ and $P'$.
\end{remark}

\section{Tropical hyperflexes}\label{sec:appendix1}

In this section, we focus our attention on tropical tangency points  of multiplicity four between  $(3,3)$- and $(1,1)$-curves in $\TPr^1\times \TPr^1$.  In analogy with the classical notion, we call these points \emph{tropical hyperflexes}.  We record their local lifts  as  triples $(\ell, p, p')$, where $\Trop\,\ell =\Lambda$, and  $P:=\Trop(p)=\Trop(p')$ is the tropical hyperflex. Our approach follows closely that of~\cite[Appendix 1]{CM20}.

Tropical hyperflexes of multiplicity four at vertices of $\Gamma$  correspond to a tangency of type (2b), (5b), (6b), or (6b'). In turn, those occurring at  interior of edges are of type (4b) or (4b'). We treat both situations  separately. In the remainder of this section, we discuss the trivalent cases, deferring the $4$-valent ones to~\autoref{sec:4ValentLifts}. As usual, we exploit symmetry and pick the representatives of each local tangency as in~\autoref{thm:classificationRealizableLocalTangencies}.

We start our discussion with types (2b) and (4b). We show their local lifting multiplicity is zero.

  \begin{proposition}\label{pr:mult4type2Triv}
Suppose that $V(\sextic)$ has no hyperflexes and  that $\Gamma$ has a tropical hyperflex of type (2b) with the tritangent $\Lambda$. Then,  $\Lambda$  does not lift to a tritangent curve to  $V(\sextic)$ defined over $\overline{\K}$.
  \end{proposition}
  
  \begin{proof} By construction, the tropical hyperflex $P$ is a vertex of $\Gamma$. The bidegree of $\sextic$ restricts the possible configurations of the cell dual to $P$ in the Newton subdivision of $\sextic$ to four cases. These four triangles are in the same orbit under the action of the subgroup $\langle \tau_0, \tau_1^2\rangle$ of $\Dn{4}$ fixing the combinatorial type of  $\Lambda$. Thus, it is enough to treat one case, namely, the one where $P^{\vee}$ has vertices $(u,0)$, $(u+1, 2)$ and $(u+1, 3)$.

    The local system defining the tangencies lifting $P$  is given by the vanishing of the  polynomials
    \[\sextic_P = \bar{x}^u(\overline{a_{u,0}} + \overline{a_{u+1, 2}}\,\bar{x}\,\bar{y}^2 + \overline{a_{u+1, 3}}\,\bar{x}\,\bar{y}^3), \quad \ell_P = \bar{y} +\,\bar{n}\,\bar{x} \quad \text{ and }\quad  W_P= \det(\Jac(\sextic_P,\ell_P; \bar{x},\bar{y})).
    \]
    The vanishing of $\ell_P$ ensures that $\bar{y} = -\bar{n}\,\bar{x}$. Replacing this value in $W_P$ forces a linear constraint on $\bar{x}$, namely
    $4\,\overline{a_{u+1,3}}\,\bar{n}\,\bar{x}-3\,\overline{a_{u+1,2}} = 0$.
    Substituting the resulting Laurent monomial expression for $\bar{x}$ back in the original system gives a unique solution for $(\bar{x}, \bar{y}, \bar{n})$:
    \[
(\bar{x},\bar{y}, \bar{n})= \left(-\frac{64 \, \overline{a_{u+1,3}}^2\, \overline{a_{u,0}}}{9\,\overline{a_{u+1,2}}^2},\, -\frac{3\,\overline{a_{u+1,2}}}{4\,\overline{a_{u+1,3}}} ,\, -\frac{27\, \overline{a_{u+1,2}}^4}{256\, \overline{a_{u+1,3}}^3\,\overline{a_{u,0}}}\right ).
    \]
    Furthermore, the Jacobian of the initial system at this point has non-zero value, so the lifting of this solution over $\overline{\K}$ is unique by~\autoref{lm:multivariateHensel}. Thus, we cannot have two tangency points between $V(\sextic)$ and $V(\ell)$ with tropicalization $P$ in the absence of classical hyperflexes. 
    \end{proof}

  \begin{proposition}\label{pr:mult4type4Triv}
  Assume that $V(\sextic)$ has no hyperflexes and that $\Lambda$ is tritangent to $\Gamma$ and has a tropical hyperflex of type (4b). Then,  $\Lambda$  does not lift  to a tritangent curve to  $V(\sextic)$ defined over $\overline{\K}$.
  \end{proposition}
  \begin{proof} We argue by contradiction and let $(\ell, p,p')$ be a local lift of the tangency $P$.  Our $\Dn{4}$-orbit representative $\Lambda$ has an edge  of  slope one and the tropical hyperflex $P$ is its lower-vertex. Furthermore, $P$ lies in the relative interior of an edge $e$ of $\Gamma$ of slope $-1/3$, and $e^{\vee}$ has endpoints $(u,0)$ and $(u+1,3)$. We let $a$ and $c$ be the corresponding coefficients on $\sextic$.

Up to rescaling of $\sextic$ and translations in $\RR^2$ we assume that $\sextic \in R[x,y]\smallsetminus \mathfrak{M}R[x,y]$ and  set $P=(0,0)$. The local system at $(0,0)$ defining this triple is given by the vanishing of the equations
      \[
  {\sextic_P} = \bar{x}^{u}(\bar{a} + \bar{c} \,\bar{x}\,\bar{y}^3), \quad {\ell_P}\!:= \bar{y} + \overline{m} +\bar{n}\, \bar{x}  \quad \text{ and }\quad {W_P}= \det(\Jac(\sextic_P,\ell_P; \bar{x},\bar{y})) =
  (3\bar{c}\,\bar{n}\,\bar{x} -\bar{y})\, \bar{c}\,\bar{y}^2\,\bar{x}^u.
  \]
  Substituting the value of $\bar{y}$ obtained by the vanishing of $\bar{\ell}$ into ${\sextic_P}$ and ${W_P}$ determines an  ideal $I\subseteq \resK[\overline{m}^{\pm}, \bar{n}^{\pm}][\bar{x}^{\pm}]$ with two generators. Its  vanishing locus consists of two points (counted with multiplicity). However, algebraic manipulations of these two generators using~\sage~produce a non-zero linear polynomial in $I$, namely $-4\,\bar{a}(4\,\bar{n}\,\bar{x} + \overline{m})$. Replacing the corresponding solution in $\bar{x}$ back into $\bar{y}$ and $\sextic_P$ determines a unique solution to the system in terms of the parameter $\overline{m}$, namely:
  \[
  (\bar{x},\bar{y}, \bar{n})= \left(\frac{64\,\overline{a_{u+1,3}}}{27\,\overline{a_{u,0}}\,\overline{m}^3}, -\frac{3\,\overline{m}}{4}, -\frac{27\,\overline{a_{u,0}}\,\overline{m}^4}{256\,\overline{a_{u+1,3}}}
  \right).\]
  However, the Jacobian of the initial system at this point has non-zero initial value, namely, the monomial $9\,\overline{a_{u+1,3}}\,\overline{a_{u,0}}\,\overline{m}^2/4$. Thus, we obtain a unique solution $(x,y,m)$ over $\overline{\K}$ by using~\autoref{lm:multivariateHensel}. This cannot happen if $V(\sextic)$ has no hyperflexes.
  \end{proof}

In the remainder of the section, we discuss types (5b) and (6b). Unlike what we observed for the two types discussed earlier, for the former ones we do obtain initial data for potential lifts $(\ell,p,p')$. 
In the absence of classical hyperflexes, we know that $p\neq p'$. However, this  does not preclude the initial forms $\overline{p}$ and $\overline{p'}$ from matching. The following definition  addresses this situation.

\begin{definition}\label{def:initialHyperflex} 
Fix a triple $(\ell, p, p')$ lifting a tropical hyperflex of multiplicity four on $\Gamma$. If $\overline{p}=\overline{p'}$ we say that the triple has an \emph{initial hyperflex} (of multiplicity four).
\end{definition}

In order to set up the local systems for types (5b) and (6b) we must discuss the partial Newton subdivisions giving rise to them. First, we note that, up to symmetry, these tropical hyperflexes are the vertices of an edge of $\Gamma$ with direction $(-3,1)$. We denote these vertices by $v_r$ and $v_l$, respectively. 
Per our  genericity conditions on $\sextic$ and choice of $\Dn{4}$-orbit representatives for $\Lambda$, any liftable one must have a tangency on one of the two leg adjacent to its top vertex. This constraint fixes the dual cells $v_{l}^\vee$ and $v_r^{\vee}$ in the Newton subdivision of $\sextic$: they are the triangles with vertices $\{(0,0), (0,1), (1,3)\}$  and  $\{(0,0), (1,2), (1,3)\}$. 
The local equations for $\sextic$ at  $v_l$ and $v_r$ become
\begin{equation}\label{eq:localMult4}
  \sextic_l(x,y) = \bar{a} + \bar{b}\, y + \bar{c}\, x\,y^3 \quad \text{ and } \quad \sextic_r(x,y) = \bar{a} + \overline{b'}\, xy^2 + \bar{c}\, xy^3.
  \end{equation}

\noindent We let $\overline{p} = (\bar{x}, \bar{y})$ and  $\overline{p'}=(\overline{r}, \overline{s})$ be the initial forms of $p$ and $p'$.

We let $\Lambda_r$ and $\Lambda_l$ be the  tropical tritangents with lower vertices $v_r$ and $v_l$, respectively. They are depicted in the right of~\autoref{fig:5b6b}.
Our first two lemmas give necessary conditions for $(\Lambda_l,v_l)$ and $(\Lambda_r,v_r)$ to lift to the corresponding triples $(\ell, p, p')$, providing tools to rule out initial hyperflexes.

\begin{figure}[tb]
\includegraphics[scale=0.55]{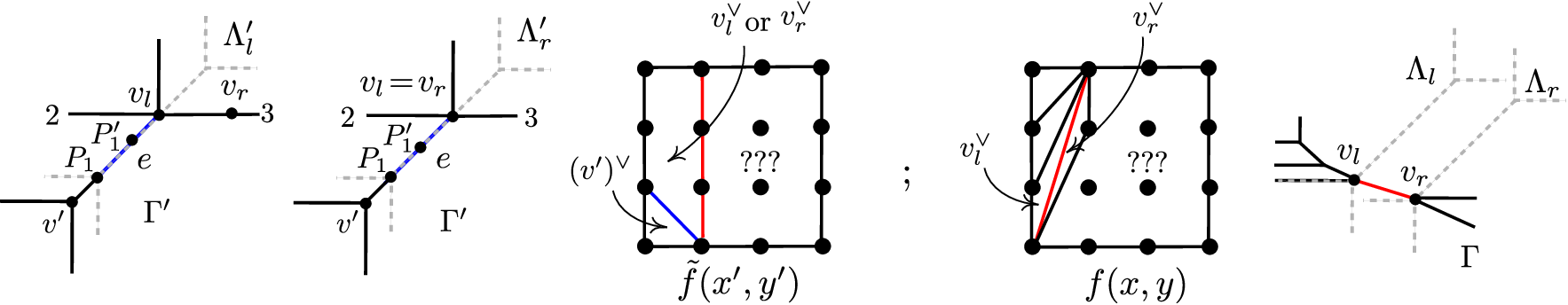}
  \caption{From left to right:  $\Lambda_l'$, $\Lambda_r'$ (viewed relative to $\Gamma'$), Newton subdivisions of both $\tilde{\sextic}$ and $\sextic$, and local information on $\Gamma$, $\Lambda_l$ and  $\Lambda_r$ in the presence of an initial hyperflex. Here, the tropical hyperflex $\Trop\,p=\Trop\,p'$ is either $v_l$ or $v_r$.\label{fig:5b6b}}
  \end{figure}

\begin{lemma}\label{lm:vl_NecConditions} Assume that $\sextic$ is generic relative to $\Gamma$ and let $(\ell,p,p')$ be a tritangent triple with  $\Trop\,V(\ell) = \Lambda_l$. Then,  the initial forms $\overline{m},\bar{n}\in \resK^*$ satisfy
  \begin{equation}\label{eq:conditionsvl}
    (\overline{m},\bar{n})= (-\frac{8\,\bar{a}}{\bar{b}},-\frac{64\,\bar{a}^3\bar{c}}{\bar{b}^4}) \quad\text{ or } \quad(\overline{m}, \bar{n}) =(4\bar{a}/\bar{b}, 16\,\bar{a}^3\,\bar{c}/\bar{b}^4).
  \end{equation}
Furthermore, if the triple has an initial hyperflex, then the second option must occur.
\end{lemma}
\begin{proof} As usual, we assume $\sextic(x,y)$ lies in $R[x,y]\smallsetminus \mathfrak{M}\,R[x,y]$, and $v_l = (0,0)$. In particular, this implies that $a, b, c, m, n \in R\smallsetminus \mathfrak{M}$. The local equations at $v_l$ are given by the vanishing of
  \begin{equation}\label{eq:leftVertexEqn}
    {\sextic_l}\!:=  \bar{a}  + \bar{b}\, \bar{y} + \bar{c}\, \bar{x}\,\bar{y}^3, \quad {\ell_l}\!:= \bar{y} + \overline{m} +\bar{n}\, \bar{x},  \;\text{ and } \;{W_l}\!:= \det(Jac({\sextic_l}, {\ell_l})) = \bar{n}(\bar{b} + 3\,\bar{c}\,\bar{x}\,\bar{y}^2) - \bar{c}\,\bar{y}^3.
  \end{equation}

In what follows, we  find necessary and sufficient conditions in $\overline{m},\bar{n} \in \resK^*$ that guarantee the system ${\sextic_{l}}={\ell_l} = {W_l} = 0$ has two solutions $\bar{x}$ (counted with multiplicity). First, we use ${\ell_l}$ to eliminate $\bar{y}$, by setting $\bar{y}=-\overline{m}-\bar{n}\,\bar{x}$. This leads to an ideal $I$ in $\resK[\overline{m}^{\pm},\bar{n}^{\pm}][\bar{x}^{\pm}]$ generated by the following two polynomials in $\bar{x}$ of degrees four and three, respectively:
  \begin{equation*}
    \begin{aligned}
      {\sextic_l}'& =
-\bar{c}\,\bar{n}^3\,\bar{x}^4 - 3\bar{c}\,\overline{m}\,\bar{n}^2\,\bar{x}^3 - 3\,\bar{c}\,\overline{m}^2\,\bar{n}\,\bar{x}^2 + (-\bar{c}\,\overline{m}^3 - \bar{b}\,\bar{n})\,\bar{x} - \bar{b}\,\overline{m} + \bar{a},\\
  {W_l}' & =
4\,\bar{c}\,\bar{n}^3\,\bar{x}^3 + 9\,\bar{c}\,\overline{m}\,\bar{n}^2\,\bar{x}^2 + 6\,\bar{c}\,\overline{m}^2\,\bar{n}\,\bar{x} + \bar{c}\,\overline{m}^3 + \bar{b}\,\bar{n}.\end{aligned}
  \end{equation*}

  Simple manipulations in~\sage~\cite{sagemath} produce new elements in $I$ of lower degree in $\bar{x}$, namely,
        \begin{equation*}
    \begin{aligned}
      g_l:= &-12\,\bar{b}\,\bar{n}^2\,\bar{x}^2  -16\,\bar{n}\,(\bar{a} - \bar{b}\,\overline{m})\bar{x} +4\,\overline{m}\,(\bar{a} -\bar{b}\,\overline{m}),\\
      h_l:= & 8\,\bar{n}\underbrace{(\bar{b}\,\bar{c}\,\overline{m}^3 + 2\,\bar{a}\,\bar{c}\,\overline{m}^2 - 6\,\bar{b}^2\,\bar{n})}_{=:A_1}\bar{x} + \underbrace{8\,\bar{b}\,\bar{c}\,\overline{m}^4 + 4\,\bar{a}\,\bar{c}\,\overline{m}^3 - 52\,\bar{b}^2\,\overline{m}\,\bar{n} + 64\,\bar{a}\,\bar{b}\,\bar{n}}_{=:A_0}.
    \end{aligned}
  \end{equation*}
        Since $p$ and $p'$ are tangency points and $(\overline{m}, \bar{n})\in (\resK^*)^2$,  the ideal $I$ cannot contain a non-trivial linear polynomial in $\bar{x}$. Therefore, $h_l$ must be the zero polynomial, i.e. $A_0=A_1=0$.

        The Euclidean algorithm used to determine $\gcd(A_0,A_1)$ in the ring $\resK(\bar{n})[\overline{m}]$ yields two polynomials in the ideal $\langle A_0, A_1\rangle $ of $\resK(\bar{n})[\overline{m}]$, namely,
  \[
  (\bar{b}^4\,\bar{n} + 64\,\bar{a}^3\,\bar{c})\,(\bar{b}^4\,\bar{n} - 16\,\bar{a}^3\,\bar{c})\; \; \text{ and } \;\; (\bar{b}^5\,\bar{n} + 24\,\bar{a}^3\,\bar{b}\,\bar{c})\overline{m} + 2\,\bar{a}\,\bar{b}^4\,\bar{n} - 192\,\bar{a}^4\,\bar{c}.\]
  The vanishing of these two polynomials produces the values for the pair $(\overline{m}, \bar{n})$ listed in~\eqref{eq:conditionsvl}.

  A direct computation shows that the  discriminant  of the polynomial $g_{l}$ equals \begin{equation}\label{eq:discrgl}
    \Delta(g_l):=64\,\bar{n}^2(\bar{b}\,\overline{m}-\bar{a})(\bar{b}\,\overline{m}-4\, \bar{a}).
  \end{equation}
  Since $\sextic$ is generic, it follows that  $(\ell, p, p')$ has an initial hyperflex only if $\overline{m} = 4\, \bar{a}/\bar{b}$. 
\end{proof}

\begin{lemma}\label{lm:vr_NecConditions}
Assume that $\sextic$ is generic relative to $\Gamma$ and let $(\ell,p,p')$ be a tritangent triple lifting   $\Lambda_r$. Then,   $\overline{m},\bar{n}\in \resK^*$ satisfy
  \begin{equation}\label{eq:conditionsvr}
    (\overline{m},\bar{n})= (-\overline{b'}/\bar{c}, -\overline{b'}^4/(4\,\bar{a}\,\bar{c}^3)) \text{ or } ((7\pm 4\,\sqrt{2}\,i)\overline{b'}/(9\,\overline{c}),(17 \pm 56\sqrt{2}\,i)\,\overline{b'}^4 /(972\,\bar{a}\,\bar{c}^3)).   
  \end{equation}
  Furthermore, if the triple has an initial hyperflex, then the second option must occur.
\end{lemma}
\begin{proof}
  We follow the same strategy as in the proof of~\autoref{lm:vl_NecConditions}, with $v_r=(0,0)$ and  $\sextic\in R[x,y]\smallsetminus \mathfrak{M}R[x,y]$, so all $a,b',c,m$ and $n$ lie in $R\smallsetminus \mathfrak{M}$. In this case, 
  \begin{equation}\label{eq:rightVertexEqn}
    {\sextic_r}=
    \bar{a} + \overline{b'}\,\bar{x}\,\bar{y}^2 + \bar{c}\,\bar{x}\,\bar{y}^3
, \quad {\ell_r}\!:= \bar{y} + \overline{m} +\bar{n}\, \bar{x}  \;\text{ and }\; {W_r}= 3\,\bar{c}\,\bar{n}\,\bar{x}\,\bar{y}^2 + 2\,\overline{b'}\,\bar{n}\,\bar{x}\,\bar{y} - \bar{c}\,\bar{y}^3 - \overline{b'}\,\bar{y}^2.
  \end{equation}
  We substitute $\bar{y}=-\overline{m}-\bar{n}\bar{x}$ in both ${\sextic_r}$ and ${W_r}$, and obtain the ideal $I_r=\langle {\sextic_r}', {W_r}'\rangle$ in $\resK[\overline{m}^{\pm},\bar{n}^{\pm}][\bar{x}^{\pm}]$. Algebraic manipulations in~\sage~yield two polynomials in $I_r$ of low-degree in $\bar{x}$:
\begin{equation*}
  \begin{aligned}
    g_r:= & 8\,\overline{m}\,\bar{n}(\bar{c}\,\overline{m} - \overline{b'})\,\overline{b'}^2\bar{x}^2
    +(8\,\overline{b'}^2\bar{c}\,\overline{m}^3 - 8\,\overline{b'}^3\overline{m}^2 - 48\,\bar{a}\,\bar{c}^2\overline{m}\,\bar{n} + 16\,\bar{a}\,\overline{b'}\,\bar{c}\,\bar{n})\,\bar{x} \\ & 
  -4\,\bar{a}(3\,\bar{c}^2\overline{m}^2 - 2\,\overline{b'}\,\bar{c}\,\overline{m} + 3\,\overline{b'}^2),
    \\
    h_r:= & (3\,\bar{c}\,\overline{m} - \overline{b'})(2\,\bar{n}\,\underbrace{(\overline{b'}^2\bar{c}^2\overline{m}^4 - 2\,\overline{b'}^3\bar{c}\,\overline{m}^3 + 6\,\bar{a}\,\bar{c}^3\overline{m}^2\bar{n} + \overline{b'}^4\overline{m}^2 - 4\,\bar{a}\,\overline{b'}\,\bar{c}^2\overline{m}\,\bar{n} + 6\,\bar{a}\,\overline{b'}^2\bar{c}\,\bar{n})}_{=:A_1}\,\bar{x} + \\
    & \underbrace{2\,\overline{b'}^2\bar{c}^2\overline{m}^5 - 4\,\overline{b'}^3\bar{c}\,\overline{m}^4 + 3\,\bar{a}\,\bar{c}^3\overline{m}^3\bar{n} + 2\,\overline{b'}^4\overline{m}^3 - 3\,\bar{a}\,\overline{b'}\,\bar{c}^2\overline{m}^2\bar{n} + 17\,\bar{a}\,\overline{b'}^2\bar{c}\,\overline{m}\,\bar{n} - 9\,\bar{a}\,\overline{b'}^3\bar{n}}_{=:A_0}).
  \end{aligned}
\end{equation*}
The tropical hyperflex  forces $h_r$ to be zero.  Thus, either  $(3\,\bar{c}\,\overline{m} - \overline{b'}) = 0$, or $A_0=A_1 = 0$. We treat each case separately.

If $\overline{m} = \overline{b'}/(3\bar{c})$, substituting this value in $g_r$ and ${W_r}$ yields two vanishing polynomials, namely
\begin{equation}\label{eq:opt15b}
  3\,\overline{b'}^2\bar{c}\,\bar{n}\,\bar{x}^2 + \overline{b'}^3\bar{x} + 18\,\bar{a}\,\bar{c}^2 = (18\,\bar{c}^2\bar{n}^2\bar{x}^2 - 6\,\overline{b'}\,\bar{c}\,\bar{n}\,\bar{x} - \overline{b'}^2)(3\,\bar{c}\,\bar{n}\,\bar{x} + \overline{b'}) = 0.
\end{equation}
Long division in $\resK(\bar{n}^{\pm})[\bar{x}]$ of the cubic polynomial by the quadratic one  has remainder
\[(-3\,\bar{c}\,\bar{n}(108\,\bar{a}\,\bar{c}^3\bar{n} + \overline{b'}^4)/\overline{b'}^2)\bar{x} + (108\,\bar{a}\,\bar{c}^3\bar{n} - \overline{b'}^4)/\overline{b'}.\]
The coefficients of this linear polynomial  have no common solution $\bar{n}$ in ${\resK}^*$. Therefore, this option is not allowed in the presence of a tropical hyperflex.

Next, we assume $A_0=A_1=0$. Algebraic manipulations in~\sage\, produce the following two polynomials in the ideal $\langle A_0, A_1\rangle $ of $\resK[\overline{m}^{\pm},\bar{n}^{\pm}]$:
\begin{equation*}
  \begin{aligned}
&\bar{c}(9720\,\bar{a}^2\bar{c}^6\bar{n}^2 + 2090\,\bar{a}\,\overline{b'}^4\bar{c}^3\bar{n} - 199\,\overline{b'}^8)\overline{m} - 3\,\overline{b'}(1944\,\bar{a}^2\bar{c}^6\bar{n}^2 - 14\,\bar{a}\,\overline{b'}^4\bar{c}^3\bar{n} - 87\,\overline{b'}^8), \\
& (3888\,\bar{a}^2\bar{c}^6\bar{n}^2 - 136\,\bar{a}\,\overline{b'}^4\bar{c}^3\bar{n} + 27\,\overline{b'}^8)\,(486\,\bar{a}\,\bar{c}^3\bar{n} + 61\,\overline{b'}^4)^2\,(4\,\bar{a}\,\bar{c}^3\bar{n} + \overline{b'}^4).
  \end{aligned}
\end{equation*}

The vanishing of the first polynomial expresses $\overline{m}$ as a degree zero rational function in $\bar{n}$, whereas the second polynomial has four solutions in $\bar{n}$. Three out of these possible four pairs $(\overline{m}, \bar{n})$ are listed in \eqref{eq:conditionsvr} and solve the system $A_0=A_1=0$. The remaining pair $(\overline{m}, \bar{n}) = (-61\bar{b'}^4/(486\,\bar{a}\bar{c}^3), 42\overline{b'}/(79\bar{c}))$ is not a solution, since the evaluation of $A_0$ at this pair yields the monomial $(55333814855/498483136638)\overline{b'}^7/\overline{c}^3$, which is never zero.

It remains to prove the final claim in the statement. Evaluating the polynomial $g_r$  at the first pair listed in~\eqref{eq:conditionsvr} produces the polynomial
\[(-4\,\overline{b'}^2/(\bar{a}\,\bar{c}^4))(\overline{b'}^6\bar{x}^2 + 8\,\bar{a}\,\overline{b'}^3\bar{c}^2\bar{x} + 8\,\bar{a}^2\bar{c}^4).\]
Its discriminant $\Delta(g_r)$ is the monomial $512\,\overline{b'}^{10}/\overline{c}^4$, which never vanishes. Therefore, this pair does not produce an initial hyperflex. On the contrary, $\Delta(g_r)$ vanishes when evaluating at the second pair listed in the statement (with either choice of signs). This concludes our proof.  
\end{proof}

Our next result is analogous to~\cite[Theorem A.4]{CM20} and we prove it using the same methods. It confirms that no tritangent local lift $(\ell,p,p')$ of $(\Lambda_l,v_l)$ or $(\Lambda_r,v_r)$ has $\bar{p}=\bar{p'}$ in the absence of classical hyperflexes and under the genericity conditions from~\autoref{def:fgenericRelToGamma}.

\begin{theorem}\label{thm:liftingMult4InitHyperflex} Assume that $\sextic$ is generic relative to $\Gamma$ and let 
  $(\ell,p,p')$  be a triple lifting either  $(\Lambda_l,v_l)$ or $(\Lambda_r,v_r)$. Then, we have $\overline{p} \neq \overline{p'}$.
\end{theorem}

\begin{proof} We argue by contradiction and suppose that $(\ell,p,p')$ has an initial hyperflex.  The proofs are the same for both $\Lambda_l$ and $\Lambda_r$, so we present them simultaneously. As in the proof of 
  {Lemmas}~\ref{lm:vl_NecConditions} and~\ref{lm:vr_NecConditions}, we assume $\sextic \in R[x,y]\smallsetminus \mathfrak{M}R[x,y]$ and we let $P=(0,0)$ be the vertex of  $\Lambda$.  Furthermore, these conditions ensure that $a$, $c$ and $b$, respectively $b'$, have valuation zero. In addition, we have $p, p'\in R^2$, and $\val(m)=\val(n) =0$, whereas $\val(\du)>0$.

  The aforementioned lemmas fix the values of the pair $(\overline{m}, \bar{n})$ in the presence of an initial hyperflex. A direct computation reveals the value of such hyperflex, by computing the roots of the polynomials $g_l$ and $g_r$, respectively, specialized at $(\overline{m}, \bar{n})$. We have
  \begin{equation}\label{eq:initialsHyperflex}
    (\overline{m},\bar{n},\overline{p}) =
  \begin{cases}
    (\frac{4\bar{a}}{\bar{b}}, \frac{16\,\bar{a}^3\,\bar{c}}{\bar{b}^4},
(-\frac{\bar{b}^3}{8 \,\bar{a}^2\bar{c}}, -\frac{2\,\bar{a}}{\bar{b}}))
 & \text{ for }\Lambda_l,\\
    (\frac{\overline{b'}(7\pm 4\,\sqrt{2}\,i)}{9\,\overline{c}},\frac{\overline{b'}^4(17 \pm 56\sqrt{2}\,i)}{972\,\bar{a}\,\bar{c}^3},
    ( -\frac{\bar{a}\,\bar{c}^2(44 \mp 2\sqrt{2}\,i)}{9\overline{b'}^3},
  -\frac{\overline{b'}(4 \pm 2\sqrt{2}\,i )}{6\,\bar{c}}))  & \text{ for }\Lambda_r.
  \end{cases}
  \end{equation}
  We set $p=(r,s)$ and $p'=(r',s')$. Since these points are distinct and belong to the curve $V(\ell)$ with $n\neq 0$, we have that $r\neq r'$ and $s\neq s'$. We choose an element $N>0$ in the value group of $R$ with
  \begin{equation}\label{eq:N}\max\{\val(r-r'), \val(s-s')\}\ll N,
  \end{equation}
  and pick  \emph{generic} elements $\varepsilon_1, \varepsilon_2 \in R$ of valuation $N$. Next, we perform tropical modifications of $\RR^2$ along $\max\{X,0\}$ and $\max\{Y,0\}$. Algebraically, we re-embed $V(\sextic)$ and $V(\ell)$ in $\K^4$ using the ideals
    \begin{equation}\label{eq:idealsMult4}
      I := \langle \sextic, x'-(x-r-\varepsilon_1) ,y'-(y-s-\varepsilon_2)\rangle  \text{ and }
      I_\ell := \langle \ell, x'-(x-r-\varepsilon_1) ,y'-(y-s-\varepsilon_2)\rangle.
 \end{equation}
    To visualize the combined modifications, we project $\Trop(V(I))$ and $\Trop(V(I_{\ell}))$ to the $X'Y'$-plane. This is done by tropicalizing the following
    polynomials in $R[x',y']$:
    \begin{equation}\label{eq:modifiedql}
      \tilde{\sextic}(x',y') := \sextic(x'+r+\varepsilon_1,y'+s +\varepsilon_2)\quad \text{ and }\quad \tilde{\ell}(x',y'):= \ell(x'+r+\varepsilon_1,y'+s +\varepsilon_2).
    \end{equation}

      \begin{table}[tb]
  \begin{tabular}{|c||
      c|c|c||c|c|}
    \hline $\Lambda$ & $x'$ &   $x'(y')^3$  & $(y')^3$ & $y'$ & $1$ \\
    \hline \hline
    $\Lambda_l$ & $ cs^3$ & $c$ & $cr$ & $b + 3\,crs^2$  & $crs^3 +bs+a$\\
    \hline
    $\Lambda_r$ & $cs^3$ & $c$ & $cr$ & $3\,crs^2 + 2\,b'rs$ & $crs^3 + b'rs^2 + a$ \\
   \hline 
  \end{tabular}
  \caption{Formulas for the expected initial forms of the coefficients of $\tilde{\sextic}$ for the listed monomials (the $\bar{\hspace{1ex}}$ is omitted to simplify notation). The terms are grouped by their relevance in determining $\Gamma'$.\label{tab:initialFormsTildeSextic}}
  \end{table}

    By construction, the points $p_1:=(-\varepsilon_1, -\varepsilon_2)$ and $p_1':=(r'-r-\varepsilon_1, s'-s-\varepsilon_2)$ are two tangency points between $V(\tilde{\ell})$ and $V(\tilde{\sextic})$. Note that $p_1,p_1'\in (\K^*)^2$ by our choice of $N$. Their tropicalizations become $P_1:=\Trop\,p_1 = (-N, -N)$ and $P_1':=\Trop\,p_1'=(-\val(r'-r), -\val(s'-s))$, respectively. Both points lie in $(\RR_{<0})^2$.
    Since  $p\in V(\ell)$, we have  $\tilde{\ell}= (1+\du(r+\varepsilon_1)) \,y' + m'+n'\,x' +\du\,x'\,y'$ with
    \begin{equation}\label{eq:valuesnewmn}
      m' :=\varepsilon_2 + n\varepsilon_1 + \du((\varepsilon_1+r)\varepsilon_2 + s\,\varepsilon_1), \quad \text{and}\quad n' := n+\du\,(s+\varepsilon_2). 
    \end{equation}

    The given valuation data on  $\varepsilon_1,\varepsilon_2, \du,r,s$ and $n$  imply that  $\val(m')\geq N$, $\val(n')=0$ and $\bar{n}=\overline{n'}$. This information determines the combinatorial type of $\Lambda':=\Trop \,V(\tilde{\ell})$: it is trivalent and it has a slope one edge.  Our genericity condition on  $\varepsilon_1, \varepsilon_2$ is determined by the restrictions
    \begin{equation}\label{eq:genericity5b6b}
      \overline{\varepsilon_2} + \overline{\varepsilon_1}\, \bar{n}\neq 0\; \text{ and }\; 
    \tilde{\sextic}(0,0)\neq 0.      
    \end{equation}
    This last condition is attainable since $\tilde{\sextic}(0,0)\in R[\varepsilon_1, \varepsilon_2]$ is not constant.    The leftmost one  ensures that  $\val(m')=N$ and $\overline{m'} = \overline{\varepsilon_2} + \overline{\varepsilon_1}\, \bar{n}$. In particular, the lower vertex of $\Lambda'$ agrees with $P_1$.

    A direct computation using~\sage~yields the relevant summands in the coefficients of $\tilde{\sextic}$ that are needed to build the Newton subdivision of $\tilde{\sextic}$ and the tropical curve $\Gamma':=\Trop\,V(\tilde{\sextic})$. They are listed in~\autoref{tab:initialFormsTildeSextic}. The monomials featured in the first group have valuation zero. In turn, the values of $\bar{p}$ listed in~\eqref{eq:initialsHyperflex} ensure that the coefficient of $y'$ for both choices of $\Lambda$ also has valuation zero, whereas the constant term lies in $\mathfrak{M}$. 

    \autoref{fig:5b6b} depicts the resulting pieces of the Newton subdivisions of $\tilde{\sextic}$, and the corresponding subgraphs of $\Gamma'$,  $\Lambda_l'$ and $\Lambda_r'$.  By construction, we know that the point $P_1$ lie on the slope-one edge $e$ of $\Gamma'$ adjacent to either $v_r$ or $v_l$. In turn, condition~\eqref{eq:N} ensures that the same is true for $P_1'$.
    Note that this forces $\val(r'-r)=\val(s'-s)$, an identity which was not a priori expected.

    Furthermore, this information also implies that $P_1'$ lies in the relative interior of $e$, but $P_1$ could agree with the lower vertex of $e$, which we denote by $v'$.    Notice that the location of $v'$ is uncertain, due to the lack of precisions on $\val(\tilde{f}(0,0))$.  

    We analyze two cases, depending on the relative position of $P_1$ and $v'$. Both situations would lead to a contradiction, thus proving $\ell$ cannot have an initial hyperflex. As usual, we denote by $\tilde{a}_{i,j}$ the coefficient  associated to the monomial in $x',y'$  with exponent vector $(i,j)$ featured in $\tilde{f}$.

    First, assume  $P_1\neq v'$,  so it lies in the relative interior of $e$.
      Then, the vanishing of the local equations $\tilde{\sextic}_{P_1}$, $\tilde{\ell}_{P_1}$ and their Wronskian at $\overline{p_1}$ together with the identity $\bar{n}=\overline{n'}$ and the definition of $p_1$ force
 $\bar{n} = \overline{\tilde{a}_{10}}/\overline{\tilde{a}_{01}}$ and $\overline{\varepsilon_2} + \overline{\varepsilon_1}\, \bar{n}=0$. This contradicts our genericity assumptions on the parameters $\varepsilon_1, \varepsilon_2$ listed in~\eqref{eq:genericity5b6b}.

    Finally, we treat the case when $P_1= v'$. By~\cite[Corollary 4.2]{Pay09}, the fiber over $v'$ of the tropicalization map $\trop\colon V(\tilde{\sextic})\to \RR^2$ is Zariski dense in $V(\tilde{\sextic})$. Moreover, the coordinatewise initial forms of any point in this fiber lies in the slope one line $V(\init_{v'}(\tilde{\sextic}))$. Therefore, we can find a point $p'':=(r'', s'')$ in this fiber with $\overline{r''}\neq -\overline{\varepsilon_1}$ and $\overline{s''}\neq -\overline{\varepsilon_2}$. Using this point, we modify $\Gamma'$ along $\max\{X', -N\}$, and  $\max\{Y', -N\}$ followed by a linear re-embedding using $x'':=x'-r''-\varepsilon_1'$, and $y'':=y'-s''-\varepsilon_2'$ for generic parameters $\varepsilon_1', \varepsilon_2'$ of valuation greater than $N$. The star of the modified  curve $\Gamma''$ at $v'$ becomes a slope one line (i.e., the edge $e$ is prolonged in $\Gamma''$). However, the modification does not change the location of $P_1$. Therefore, in the new scenario, $P_1$ lies in the relative interior of the edge $e$, and we are reduced to the previous case.
\end{proof}

\begin{corollary}\label{cor:noInitialHyperflex}
Let $P$ be a tropical hyperflex of $\Gamma$ of type (5b) or (6b). Then, its local lifting multiplicity is one. In particular, no local lift $(\ell,p,p')$ of  $(\Lambda_l,v_l)$ or $(\Lambda_r,v_r)$ has an initial hyperflex.  \autoref{tab:liftingMult45b6b} gives explicit formulas for the tuple of initial forms $\overline{m}, \bar{n}, \bar{p}$ and $\overline{p'}$ in both cases. 
\end{corollary}
\begin{proof}
  By \autoref{thm:liftingMult4InitHyperflex} we know that such local lifts have no initial hyperflexes. 
  In such situation, 
  {Lemmas}~\ref{lm:vl_NecConditions} and~\ref{lm:vr_NecConditions} provide unique values for $(\overline{m}, \bar{n})$ in both cases.   In turn, substituting these unique values  into the quadratic polynomials $g_l$ and $g_r$ (defined in the proofs of these lemmas), respectively, determines two solutions for $\bar{x}$, giving the first coordinates of $\bar{p}$ and $\bar{p'}$. The $\bar{y}$ value is obtained from the linear equation $\ell_l$ and $\ell_r$, respectively. This gives the expressions listed in~\autoref{tab:liftingMult45b6b}.

  To finish, we must show that the initial forms $\overline{m}, \bar{n}, \bar{p}, \bar{p'}$ for both $\Lambda_l$ and $\Lambda_r$ uniquely determine $m, n, p, p'$. The statement follows by~\autoref{lm:multivariateHensel}, after confirming that the expected initial forms of the Jacobians in the variables  $(\overline{m},\bar{n},\bar{r},\bar{s},\overline{r'},\overline{s'})$ of the local systems at $v_l$ and $v_r$, respectively, are units in $\resK$. Indeed, a \sage~computation confirms that they are Laurent monomials in the coefficients of~\eqref{eq:localMult4}, namely $-294912\,\sqrt{3}\,\bar{a}^8\,\bar{c}^3/\bar{b}^7$ and $8\,\sqrt{2}\,\bar{a}\,\overline{b'}^7/\bar{c}^4$, respectively. 
\end{proof}

Unlike the case of bitangent lines to quartic curves discussed in~\cite[Proposition A.3]{CM20}, the lifting multiplicity of both  $\Lambda_l$ and $\Lambda_r$ in cannot be solely determined by the local lifting multiplicity at $v_l$ or $v_r$. The remaining tangency point $P''$ between each curve and $\Gamma$ plays a role. The genericity conditions on $V(\sextic)$ stated in~\autoref{def:fgenericRelToGamma} combined with $\Dn{4}$--symmetry restricts the tangency type of $P''$ for a liftable tritangent: it will either be at the top vertex of $\Lambda$, or at its positive horizontal leg. In both cases, the value of $\bar{n}$ features in the formulas for $\bar{\du}$, so it is expected that the lifting multiplicity of $\Lambda$ will depend on the local lifting multiplicity of $P''$. Our next statement confirms that this is indeed the case.

  \begin{table}[tb]
  \begin{tabular}{|c||c|c|c|c|c|c|}
    \hline $\Trop\, \ell$ &  $\overline{m}$ & $\bar{n}$ & $\bar{x}$ & $\bar{y}$ & $\overline{r}$ & $\overline{s}$  \\
    \hline \hline
    
    $\Lambda_l$ & $-\frac{8\,\bar{a}}{\overline{b}}$ & $-\frac{64\,\bar{a}^3\bar{c}}{\overline{b}^4}$ & $-\frac{\overline{b}^3(3+\sqrt{3})}{32\,\bar{a}^2\bar{c}}$ & $\frac{2\,\bar{a}\,(1-\sqrt{3})}{\overline{b}}$ &
    $-\frac{\bar{b}^3(3-\sqrt{3})}{32\,\bar{a}^2\bar{c}}$ &
    $\frac{2\,\bar{a}\,(1+\sqrt{3})}{\overline{b}}$ \\
    \hline
    $\Lambda_r$ & $-\frac{\overline{b'}}{\overline{c}}$ & $-\frac{\overline{b'}^4}{4\,\bar{a}\bar{c}^3}$ & $-\frac{2\,\bar{a}\bar{c}^2(2+\sqrt{2})}{\overline{b'}^3}$ & $-\frac{\sqrt{2}\,\overline{b'}}{2\,\overline{c}}$ &
    $-\frac{2\,\bar{a}\bar{c}^2(2-\sqrt{2})}{\overline{b'}^3}$ & $\frac{\sqrt{2}\,\overline{b'}}{2\,\overline{c}}$\\
    \hline 
  \end{tabular}
  \caption{Formulas for the initial forms of triples $(\ell,p,p')$ lifting $\Lambda_l$ and $\Lambda_r$ in the absence of initial hyperflexes. The value of coefficient $\du$ of $\ell$ is determined by the remaining tangency point between $\Gamma$ and $\Lambda_l$, respectively $\Lambda_r$.\label{tab:liftingMult45b6b}}
\end{table}

  \begin{theorem}\label{thm:liftingFormulasMult4} Assume $V(\sextic)$ has no hyperflexes and is generic relative to $\Gamma$. Then, the lifting multiplicity of  $\Lambda_l$ and $\Lambda_r$ agrees with the local lifting multiplicity of their remaining tangency  point.
\end{theorem}

  \begin{proof} We set $\Lambda$ to be either $\Lambda_r$ or $\Lambda_l$ and let $P$ be the tropical hyperflex between $\Lambda$ and $\Gamma$. By construction, the remaining tangency point between $\Lambda$ and $\Gamma$, which we call $P''$ must lie in the positive horizontal leg of $\Lambda$.

    We let $(\ell,p,p',p'')$ be a tritangent tuple lifting $\Lambda$. By~\autoref{cor:noInitialHyperflex}, we know that the local lifting multiplicity of the pair $(\Lambda, P)$ is one. In turn, the values of $\bar{\du}$ and $\overline{p''}$ are determined by $P''$. Explicit formulas depend on the combinatorial type of $P''$.

    To prove the uniqueness of the whole lift, it is enough to show the  initial forms of the Jacobians for the local systems at $P$ and $P''$ (for suitable choices of variables adapted to $P''$) are Laurent monomials in the coefficients of $\sextic_P$ and $\sextic_{P''}$. The formulas depend on the tangency type of $P''$. The classification  provided in~\autoref{thm:classificationRealizableLocalTangencies} ensures that $P''$ has type (3c), (2a), (4a) or (6a).
    
    First, assume $P''$ is of type (3c). The bidegree of $\Gamma$ fixes $(P'')^{\vee}$ to be the edge with vertices $(2,2)$ and $(2,3)$. After relabeling the parameter $\du$ appearing in $\ell$ at $n/\du$, we obtain    $\ell_{P''}= 
    \bar{x}(\bar{n}/\bar{\du})( \bar{y} + \bar{\du} )$.
    The proof of \autoref{pr:Prop5.2Corrected}, confirms the value of $\bar{\du}$ equals $\overline{a_{2,2}}/\overline{a_{2,3}}$. Furthermore, we can write   $\du = a_{2,2}/a_{2,3} + \du_1 + \du_2$, where $\du_1$ is fixed with $\val(\du_1)>\val(a_{2,2})-\val(a_{2,3})$, $\val(\du_2) >\val(\du_1)$, and
    \[\bar{\du_2} = \pm\,\frac{1}{\overline{a_{2,3}}}\sqrt{2 \, (-1)^{3+j}\,\overline{a_{1,3}} \,\overline{a_{3,j}} \left(\overline{a_{2,2}}/\overline{a_{2,3}}\right)^{3+j-4}}.
    \]
Here,  $(3,j)$ is a vertex of the Newton subdivision of $\sextic$ forming a triangle with the edge $(P'')^{\vee}$. The expected initial form of the $9\times 9$ Jacobian matrix determining the uniqueness of the lift of  $(\overline{m}, \bar{n}, \bar{r}, \bar{s}, \bar{r'}, \bar{s'}, \bar{\du_2}, \bar{r''}, \bar{s''})$ to $(\ell, p,p', p'')$ is the determinant of a block diagonal matrix. The blocks correspond to the local lifting multiplicity computations of $(y + m + n\,x, P, P')$ and $(y+\du, P'')$, respectively. Thus, the lifting multiplicity of $(\ell, P, P', P'')$ agrees with that of $(y+ \du,P'')$, as claimed.

    Next, assume $P''$ has type (2). The location of $v^{\vee}$ restricts the possible combinatorial types of $P''$ to (2B), (2G), (2D) or (2E) as seen in~\autoref{fig:localNPType2}. Furthermore, the vertex of $(P'')^{\vee}$ associated to the coefficient $a$ on each picture is $(2,3)$. The values of $(\bar{\du}/\bar{n}, \overline{p''})$ are listed in~\autoref{tab:initialFormsType2Horiz}. Similarly, if $P''$ has type (4a) or (6a),~\autoref{pr:4a6aNoAdjacentLeg} confirms that $(\bar{\du}/\bar{n}, \overline{p''})$ is unique.

    Finally, if $P''$ has types (2a), (4a) and (6a), a simple computation reveals that  $9\times 9$ Jacobian associated to the local systems at $P$ and $P''$ in the variables $(\overline{m},\bar{n},\overline{p},\overline{p'},\overline{p''},\bar{n}/\bar{\du})$ (for (2a)) or $(\overline{m},\bar{n},\overline{p},\overline{p'},\overline{p''},\bar{\du})$ (for (4a) or (6a)) is block upper-triangular. \autoref{tab:initialFormsType2Horiz} and the proofs of~\autoref{pr:4a6aNoAdjacentLeg} and~\autoref{cor:noInitialHyperflex} confirm that both blocks are invertible. Thus,  \autoref{lm:multivariateHensel} ensures unique lifts $(m,n,p,p',p'',\du)$ from the initial data. 
\end{proof}

  Combining the proof of~\autoref{thm:liftingFormulasMult4} with~\autoref{cor:4a6aField} and the data from~\autoref{tab:liftingMult45b6b}, we obtain arithmetic information about the liftings of $\Lambda_l$ and $\Lambda_r$ and its tangency points:
  \begin{corollary}\label{cor:lifting5b6bQuad} Assume $V(\sextic)$ has no hyperflexes. Fix $\Lambda=\Lambda_l$ or $\Lambda_r$, and let $P''$ be the remaining tangency point between $\Lambda$ and $\Gamma$. If $\sqrt{2}, \sqrt{3}\in \resK$ we have:
    \begin{enumerate}
      \item The liftings $(\ell, p, p', p'')$ of $(\Lambda,P,P',P'')$ are defined over $\K$ if, and only if, $(\ell,p'')$ is.
      \item  If $P''$ is of type (3c), $(\ell,p'')$ is defined over $\K$ if, and only if,  $\sqrt{\overline{a_{1,3}} \,\overline{a_{3,j}} \left(-\overline{a_{2,2}}/\overline{a_{2,3}}\right)^{j-1}}\in \resK$. Here, the point $(3,j)$ forms a triangle with $(2,2)$ and $(2,3)$ in the Newton subdivision of $\sextic$.
\item If $P''$ has type (2a), (4a) or (6a), then  $(\ell, p'')$ is defined over $\K$.
    \end{enumerate}
    \end{corollary}

\section{Type (3) tangencies for trivalent tritangents}\label{sec:type-3-tangencies}

In this section we determine the lifting multiplicities of type (3) tangencies. There are 12 cases to be analyzed, and we discuss them in increasing order of complexity. Several of them can be deduced from the computations done in~\cite{CM20}. In particular, \autoref{pr:Prop5.2Corrected} given in the Appendix to the present paper fixes a small inaccuracy in the proof of~\cite[Proposition 5.2]{CM20} regarding the proof of the lifting multiplicity for a type (3c) tangency for a smooth  tropical quartic curve in~$\TPr^2$.

Throughout this section we assume that $\Lambda$ is trivalent and contains a slope one edge. The combinatorics of the stars of $\Gamma$ at a vertex carrying a type (3h) or (3bb) tangency were determined by 
{Lemmas}~\ref{lm:starsTopVertexMostType3d3h} and~\ref{lm:starsTangency3bb}.


\begin{remark}\label{rk:type3a3choriz}
Type (3) tangencies along a horizontal leg of $\Lambda$ can be treated with techniques developed in~\cite{CM20, LM17}. In particular, \cite[Proposition 5.2]{CM20} ensures that  tangencies of type (3a) along a horizontal leg have local lifting multiplicity two, but a unique value for the initial forms $(\bar{n},\bar{p})$. Horizontal type (3c) tangencies are treated in~\autoref{pr:Prop5.2Corrected}.
\end{remark}

Our first statement discusses types (3cc) and (3ac):

\begin{figure}
  \includegraphics[scale=0.5]{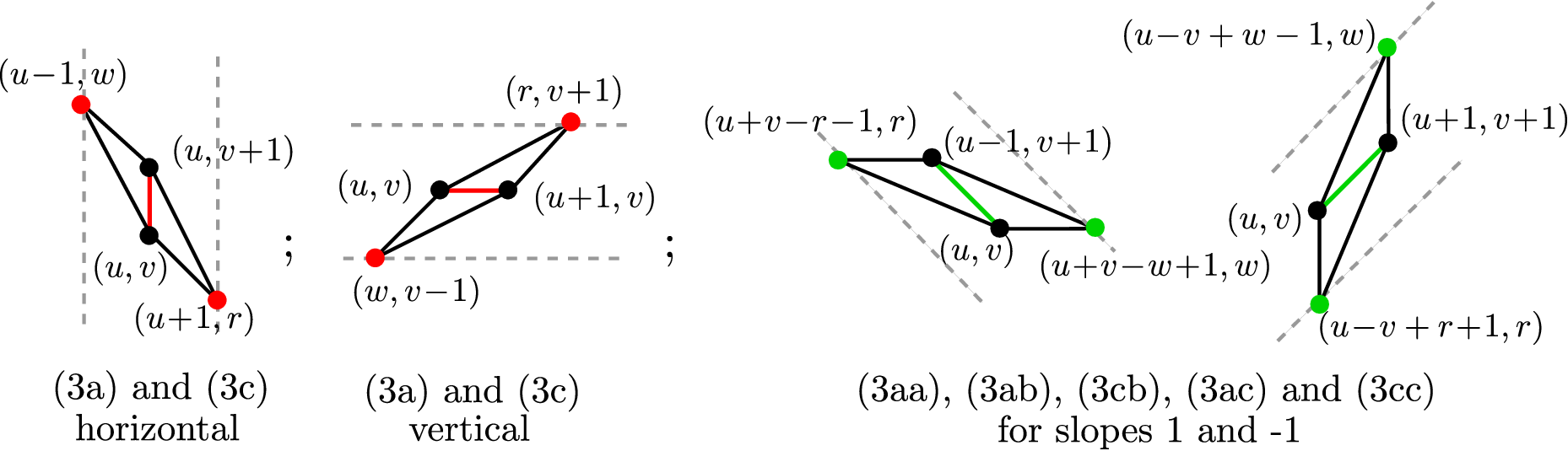}
  \caption{Relevant cells in the Newton subdivision of $\sextic$ for type (3) local tangencies of multiplicity 2.\label{fig:type3Mult2NP}}
\end{figure}

\begin{proposition}\label{pr:Prop5.2diag}
Assume that  $\sextic$ is generic relative to $\Gamma$.   Let $\Lambda$ be tritangent to $\Gamma$ with a slope one edge and a local tangency of type (3cc) or (3ac) at a point $P$ along this edge.  Then, $(\Lambda, P)$ has exactly two lifts $(\ell,p)$, each corresponding to a different choice of  parameter $n$. \end{proposition}

\begin{proof} The choice of representatives for these local tangencies only involves the star of the lower vertex of $\Lambda$. Thus, the result can be deduced from the statement for local tangencies of type (3a) or (3c) between a trivalent tropical line and a smooth quartic curve in $\TPr^2$. The precise formulas for $(\bar{n}, \bar{p})$ are obtained from \autoref{pr:Prop5.2Corrected} and ~\cite[Proposition 5.2]{CM20}, respectively, after permuting the $x$ and $z$ coordinates in $\TPr^2$.

  We start discussing the type (3cc) tangency. The four vertices  in the Newton subdivision of $\sextic$ relevant for the tangency are indicated in~\autoref{fig:type3Mult2NP}. They determine the  indices $u,v,w,r$ used below. The only parameter of $\ell$ involved in the local equations at $P$ is $n$. Furthermore,  any lift $(\ell,p)$ of $(\Lambda, P)$ satisfies $\bar{n} = \overline{a_{u,v}}/\overline{a_{u-1,v+1}}$.

 We set $\tilde{\sextic}(x,y):=\sextic(x, y-(a_{u,v}/a_{u-1,v+1} + n_1)x)$  for some  $n_1\in R$ satisfying  $\val(n_1)>\val(a_{u,v})-\val(a_{u-1,v+1})$.
 The element must be picked to ensure that   the point $(u+v,0)$ is unmarked in the Newton subdivision of $\tilde{\sextic}$. Such choice  forces the   triangle with vertices $(u+v-1,0), (u+v+1, 0)$ and $(u+v-1,1)$, to be a cell in the subdivision. The existence of $n_1$ satisfying these conditions is guaranteed by a result analogous to~\autoref{lm:chooseTail3a}.

 We let $P_1$ be the vertex of $\Gamma'=\Trop \,V(\tilde{f})$ dual to this triangle, and let $p_1$ be the image of $p$ under the map $(x,y)\mapsto (x, y+(a_{u,v}/a_{u-1,v+1} + n_1)x)$. By construction $\Trop \,p_1 =P_1$ and $p_1$ is a tangency point between $V(\tilde{f})$ and $V(\tilde{\ell})$, where  $\tilde{\ell}(x,y):=\ell(x,y-(a_{u,v}/a_{u-1,v+1} + n_1)x)$.

  After the choice of $n_1$ is made, the coefficient $n$ has the form $n = a_{u,v}/a_{u-1,v+1} + n_1 +n_2$ with $\val(n_2)= \val(a_{u+v-w+1,w}) >\val(n_1)$ and $\overline{n_2}$ satisfies
  \begin{equation}\label{eq:type3cdiag_11curve}(-1)^{w}
    \,\overline{a_{u+v-w+1,w}} \left (\frac{\overline{a_{u,v}}}{\overline{a_{u-1,v+1}}}\right)^{w} \!\!\!- \frac{\overline{a_{u-1,v+1}}^2}{4\,\overline{a_{u+v-r-1,r}}} (-1)^{r}
    \left (\frac{\overline{a_{u,v}}}{\overline{a_{u-1,v+1}}}\right )^{2v-r} \overline{n_2}^2 = 0.
  \end{equation}

  The value of $\overline{p_1}$ is uniquely determined from $\overline{n_2}$ and the local equations at $P_1$. We have,
  \begin{equation}\label{eq:3acDiag}\overline{p_1} =
    \left (
    \overline{n_2}\,\frac{(-1)^{v-r}\,\overline{a_{u-1,v+1}}}{2\,\overline{a_{u+v-r-1,r}}}
    {\left (\frac{\overline{a_{u,v}}}{\overline{a_{u-1,v+1}}}\right )}^{v-r},
    \, -\overline{n_2}
  \right ).
\end{equation}

The computations for the type (3ac) are completely analogous, except that we must incorporate the constant coefficient $m$ into the initial modification, setting \[
\tilde{\sextic}(x,y):=\sextic(x, y-m-(a_{u+1,v}/a_{u,v+1} + n_1)x).
\]
The choice of $\overline{m}$ is fixed by one of the remaining tangency points, while $n_1\in  R$ is determined in the same way as was done for the (3cc) tangency. The local equations determined by the tangency $P_1$ match those of the (3cc) tangency because $m$ does not feature in the coefficients associated to $P_1^{\vee}$. Thus, the values for $\bar{n}$ and $\bar{p}$ are determined by expressions~\eqref{eq:type3cdiag_11curve} and~\eqref{eq:3acDiag}.

To finish, we must show that for both tangency types, the values of $(n,p_1)$ are uniquely determined by $(\bar{n}, \bar{p_1})$. To this end, it suffices to check that the determinant of the Jacobian of the tuple $(\tilde{f}_{P_1},\bar{y}-\overline{n_2}\,\bar{x},W(\tilde{f}_{P_1},\bar{y}-\overline{n_2}\,\bar{x}))$ relative to the variables $\bar{x}, \bar{y}, \overline{n_2}$ has non-vanishing expected initial form. Since the local equations at $P_1$ agree for both types, the claim follows from the computation carried out in the proof of~\autoref{pr:Prop5.2Corrected} for a horizontal type (3c) tangency.
\end{proof}

Expression~\eqref{eq:type3cdiag_11curve} yields arithmetic constraints for lifting type (3ac) and (3cc) tangencies:

\begin{corollary}\label{cor:3cc3acDiagonalQuad}
  The two lifts $(\ell,p)$ of a diagonal type (3ac) or (3cc) tangency $(\Lambda,P)$ are defined over a quadratic extension of $\K$. Such lifts lie in $\K$ if, and only if,
  \[\sqrt{ (-1)^{w+r}\, \overline{a_{u+v-w+1,w}}\,\overline{a_{u+v-r-1,r}}\, (\overline{a_{u,v}}/\overline{a_{u-1,v+1}})^{w+r}}\in \resK.\]
  \end{corollary}

Our next statement confirms no tangency type (3bb) is realized by  a classical tritangent $V(\ell)$ to $V(\sextic)$. The analogous result for  bitangent lines to smooth quartics in $\TPr^2$ can be found in~\cite[Proposition A.7]{CM20}.

\begin{proposition}\label{pr:3bb} Assume that $\sextic$ is generic relative to $\Gamma$. Let $\Lambda$ be a tritangent to $\Gamma$ with a type (3bb) tangency. Then, $\Lambda$ does not lift to a classical tritangent to $V(\sextic)$ defined over $\overline{\K}$.
  \end{proposition}

\begin{proof} We must treat two situations, depending on whether the tangency is diagonal or horizontal, as indicated in~\autoref{fig:classificationLocalTangencies}. We let $P$ be the lower vertex of $\Lambda$  and $P'$ be the remaining  tangency point in the type (3bb) tangency.

We first discuss the diagonal case, which we labeled (3bb1) and (3bb2) in the figure.  The star of $\Gamma$ at $P$ is fixed by~\autoref{lm:starsTangency3bb}. Up to translation, the vertices of the cell $P^{\vee}$ are the points $(0,3)$, $(2,0)$ and $(1,2)$. We let $a$, $b$ and $c$ be the corresponding coefficients. 
Up to a Laurent monomial factor, the local equations at $P$ become:
  \begin{equation}\label{eq:3bbdiag}
  \sextic_P =  \bar{c}\,\bar{x}\,\bar{y}^2+ \bar{a}\,\bar{y}^3 + \bar{b}\,\bar{x}^2, \quad \ell_{P}=\bar{y} + \bar{n}\,\bar{x} + \overline{m}\quad \text{ and } \quad W_P = 2\,\bar{c}\,\bar{n}\,\bar{x}\,\bar{y} + 3\,\bar{a}\,\bar{n}\,\bar{y}^2 - \bar{c}\,\bar{y}^2 - 2\,\bar{b}\,\bar{x}.
  \end{equation}
In turn, the local equations at $P'$ fix the value of $\bar{n}$ as $\bar{n}=\bar{c}/\bar{a}$. Eliminating $\bar{y}$ from the system~\eqref{eq:3bbdiag} via the vanishing of $\ell_P$ and substituting $\bar{n}$ by its known value, leads to a linear constraint on $\bar{x}$, namely $\bar{a}\,\overline{m} + \bar{c}\,\bar{x}=0$. This implies $\bar{y}=0$, which is not allowed.

In the horizontal case, the star of $\Gamma$ at $P$ comes in two flavors, both listed in~\autoref{lm:starsTangency3bb} and depicted in~\autoref{fig:classificationLocalTangencies}. The right one produces a non-realizable tritangent curve $\Lambda$ by~\cite[Proposition A.7]{CM20}, so it remains to prove that the left one does not lift either. By construction, the vertices of $P^{\vee}$ become $(1,3)$, $(1,2)$ and $(2,0)$. We label the corresponding coefficients by $a$, $b$ and $c$, respectively. The tangency point $P'$ informs the value of $\overline{m}$:  we have $\overline{m}=\bar{b}/\bar{a}$. 
The local equations at $P$ are
\[
\sextic_P = \bar{c}\,\bar{x}^2+ \bar{a}\,\bar{x}\,\bar{y}^3 + \bar{b}\,\bar{x}\,\bar{y}^2,\quad 
\ell_{P}=\bar{y} + \bar{n}\,\bar{x} + \overline{m}\quad \text{ and } \quad
W_P = 3\,\bar{a}\,\bar{n}\,\bar{x}\,\bar{y}^2 + 2\,\bar{b}\,\bar{n}\,\bar{x}\,\bar{y} - \bar{a}\,\bar{y}^3 - \bar{b}\,\bar{y}^2 - 2\,\bar{c}\,\bar{x}.
\]
Proceeding with the elimination of variables as in the diagonal case, we obtain the restriction $\bar{a}^2\,\bar{c}\,\bar{n}\,\bar{x} = 0$, which is not allow. This concludes our proof.
\end{proof}

Next, we show that  tangency types (3ab) and (3cb) along the diagonal also cannot lift over $\overline{\K}$:

\begin{lemma}\label{lm:3abOr3cbDiagonal} Assume that $\sextic$ is generic relative to $\Gamma$ and let $\Lambda$ be a tritangent to $\Gamma$ with  a local tangency of type (3ab) or (3cb) along its unique edge.  Then, $\Lambda$ does not lift to a classical tritangent to $V(\sextic)$ defined over $\overline{\K}$. 
\end{lemma}

\begin{proof} We pick representatives of these tangency types as in~\autoref{fig:classificationLocalTangencies}. The result is a direct consequence of the genericity conditions listed in~\autoref{def:fgenericRelToGamma}. Indeed, assuming $\Lambda$ is tritangent to $\Gamma$, the initial form of the parameter $\du$ appearing in a potential lift $\ell$ does not feature in any of the local equations at tangency points between $\Gamma$ and $\Lambda$. In this situation, we such equations yield a $9\times 8$ system with no solutions per the genericity assumptions on $\sextic$.
  \end{proof}

Tangencies of type (3f) behave like star tangencies of multiplicity four between tropical trivalent lines and smooth tropical curves in $\TPr^2$. Thus, we can determine both the local lifting multiplicity and the existence of real lifts using~\cite[Proposition 3.12]{LM17} or~\cite[Proposition 6.4]{CM20}.

Before stating the result for this tangency type, we describe the local combinatorics of the relevant cells in the Newton subdivision of $\sextic$. Let $v$ be the lower vertex of $\Lambda$. The bidegree of $\sextic$ restricts the location of the cell dual to $v$ in the Newton subdivision of $\sextic$: it must be the triangle with vertices $(1,1)$, $(1,2)$ and $(2,1)$. The three triangles adjacent to $v^{\vee}$ are determined by three vertices, which we label $(0,i)$, $(j,0)$ and $(k,4-k)$ for $k=1,2$ or $3$. We let $\lambda_1, \lambda_2$ and $\lambda_3$ denote the lengths of the vertical, horizontal and diagonal edge of the connected component of $\Gamma \cap \Lambda$ containing $v$.

A standard computation via chip-firing on this component ensures that the shortest edge among these three is the only one not containing a tangency point. Recall from the genericity constraints imposed on $\Gamma$ that this minimum length is unique. The $\Dn{4}$-symmetry on $\TPr^1\times \TPr^1$ and our choice of representatives of a type (3f) tangencies featured in~\autoref{fig:classificationLocalTangencies} restricts our analysis to three cases:
\begin{equation}\label{eq:relOrderFor3e}
\lambda_1<\lambda_2\leq \lambda_3, \quad \lambda_1<\lambda_3\leq \lambda_2 \quad \text{ or } \quad \lambda_3<\lambda_1\leq \lambda_2.
\end{equation}

\begin{proposition}\label{pr:3f}
  Let $(\Lambda,P,P')$ be one of the three $\Dn{4}$-representative of a local tangency (3f) with $\Gamma$ subject to the length restrictions~\eqref{eq:relOrderFor3e}.
  Then, there are four lifts $(\ell, p, p')$ of $(\Lambda, P, P')$ over $\overline{\K}$.  Such lifts are determined by two independent choices of the initial forms of two parameters, called $m_2$ and $n_2$. Furthermore:
    \begin{enumerate}[(i)]
  \item All four lifts are defined over a  quadratic extension of $\K$.
    \item Either all four lifts can be realized over $\K$ or none of them is. The precise conditions for lifting over $\K$ depend on the marked vertices $(0,i), (j,0)$ and $(k,4-k)$ and the relative order among the edge lengths $\lambda_1, \lambda_2$, and $\lambda_3$.
    \end{enumerate}
\end{proposition}

\begin{proof} The result follows from the analogous statement for bitangent lines to smooth tropical quartics in $\TPr^2$ with a star-like behavior, corresponding to a bitangent class of shape (C) in the notation of~\cite[Figure 6]{CM20}. The bidegree of $\sextic$ and the relative order between the three lengths $\lambda_1, \lambda_2, \lambda_3$ restrict the values of $i,j,k$. In addition, it forces the remaining tangency point between $\Lambda$ and $\Gamma$  to be  on one of the two positive legs of $\Lambda$, and its type depends on the nature of the tangency, i.e., whether it is horizontal, vertical, diagonal or a type (5a).

  The case where $\lambda_1<\lambda_2\leq \lambda_3$ corresponds to the $\Sn{3}$-orbit representative of this local tangency chosen in~\cite{CM20}. In particular, we have $0\leq i \leq 3$ and $1\leq j,k\leq 3$ since $\Gamma$ is a $(3,3)$-curve. The explicit values of $m, n, p, p'$ are obtained in two steps, which we now summarize. First, we  modify $V(\sextic)$ along the curve $y- m_1 - n_1 x$, for suitable choices of parameters $m_1, n_1\in \K$ with $\overline{m_1} = \overline{a_{11}}/\overline{a_{12}}$ and $n_1=\overline{a_{21}}/\overline{a_{12}}$ that yield a desired partial Newton subdivision of $\tilde{\sextic} := \sextic(x, y-m_1-n_1\,x)$, namely, the one depicted in~\cite[Figure 17]{LM17}. The computation then reduces to lifting two type (2a) tangencies, located along the horizontal and diagonal legs of the trivalent tropical line $\Trop\,V(y+m_2+n_2\,x)$. As explained below, the initial forms of $m_2$ and $n_2$ admits two values. Furthermore, the initial forms of the corresponding tangency points lie in $\left (\QQ(\overline{m_2}, \overline{n_2})\right )^2$. A Jacobian computation confirms that these tangencies and the parameters $m_2, n_2$ are completely determined by their initial forms. Setting $m = m_1 + m_2$ and $n= n_1 +n_2$ yields the relevant coefficients of $\ell$.

  The proofs of~\cite[Proposition 6.4, Lemma 6.5]{CM20} give explicit values for the initial forms of $m_2$ and $n_2$, in terms of the parameters $i, j, k$. More precisely, we have:
  \begin{equation}\label{eq:3fmn2Case123}
    \begin{aligned} \overline{m_2} &=
\frac{2}{\overline{a_{11}}}\sqrt{(-1)^{i+j}\overline{a_{0i}}\,\overline{a_{j0}}\,\overline{a_{21}}^{2-j}\,\overline{a_{11}}^{i+j-2}\overline{a_{12}}^{-i}},
  \\
  \overline{n_2} & = 
\pm \frac{2}{\overline{a_{21}}}\sqrt{(-1)^{j+k}\,\overline{a_{j0}}\,\overline{a_{k,4-k}}\,\overline{a_{21}}^{6-j-k}\,\overline{a_{11}}^{j-2}\, \overline{a_{12}}^{k-4}}.
\end{aligned}
  \end{equation}
  The conditions ensuring that $m_2, n_2\in \K$ are obtained by imposing that the  radicands in these formulas are perfect squares in $\resK$.

  The statement for the remaining two cases in~\eqref{eq:relOrderFor3e} follows by symmetry. Indeed, these two cases can be reduced the previous one via the $\pr^2$-symmetries induced by the maps $\nu_0\colon z\leftrightarrow x$ and $\nu_1\colon x\mapsto y\mapsto z \mapsto x$. Notice that the values of the parameters of $\ell$ change as $(m,n)\mapsto (n,m)$ and $(m,n)\mapsto (1/m, m/n)$, respectively. In turn, the tuple of indices $(i,j,k)$ map to $(4-k, 4-j, 4-i)$, and $(k, 4-i, 4-j)$, respectively.  The precise conditions determining whether the liftings occur over $\K$ can be obtained from~\eqref{eq:3fmn2Case123} by applying the  maps $\nu_0$ and $\nu_1^{-1}$, respectively. 
\end{proof}

The next two propositions determine the  local lifting multiplicities for types (3h) and (3d).

\begin{proposition}\label{pr:type3h} 
  Assume that $\sextic$ is generic relative to $\Gamma$ and let $\Lambda$ be  tritangent to $\Gamma$ with a local tangency  of type  (3h). Then, its local lifting multiplicity equals two and both lifts are defined over the same  quadratic extension of $\K$. 
\end{proposition}

\begin{proof} We pick the orbit representative for tangency type listed in~\autoref{fig:classificationLocalTangencies}.  We let $P$ and $P'$ be the upper and lower vertices of $\Lambda$ responsible for the type (3h) tangency, and  $P''$ be the remaining tangency point between $\Lambda$ and $\Gamma$. The stars of $\Gamma$ at both points are determined by~\autoref{lm:starsTopVertexMostType3d3h}. The bidegree of $\sextic$ forces $P''$ to be on a horizontal leg of $\Lambda$, and to be of type (3c) or (2a).

  Upon applying the map $\tau_1^2$ from~\autoref{tab:D4Action} if necessary, we may assume $P''$ lies  in the positive horizontal leg of $\Lambda$. Thus, the triangle $P^{\vee}$ has  vertices  $(0,2)$, $(0,3)$ and $(1,1)$. We let $a,b$ and $c$ be the corresponding coefficients. In turn, the vertices of $(P')^{\vee}$ are $(0,2)$, $(1,1)$ and $(1,0)$. We label the coefficient of $\sextic$ associated to this last point by $b'$.

  We let $(\ell,p,p')$ be any lift of $(\Lambda, P, P')$. In what follows, we show that the lifting multiplicity of this triple is the product of the lifting multiplicities of $P$ and $P'$. We set $\lambda'' = n/\du$.
Since $P''$ has tangency type (2a) or (3c), we know from the local equations at $P''$ that  $\overline{\lambda''}\in \resK$.  In what follows, we express our solutions  $\overline{m}, \bar{n}, \bar{\du}, \bar{p}, \bar{p'}$ for the local systems at $P$ and $P'$ in terms of $\overline{\lambda''}$.

  The values of $\bar{n}$ and $\bar{\du}$ are determined by $\overline{\lambda''}$ and the local equations at $P$, while the value of $\overline{m}$ only features in the equations for $P'$. For this reason, we treat the point $P$ first. We have
  \[
 \sextic_P = (\bar{a}\,\bar{y} + \bar{b}\,\bar{y}^2 + \bar{c}\, \bar{x})\,\bar{y}, \;\quad
             {\ell_P}\!= \bar{y} +  \bar{n}\, \bar{x} + \bar{\du}\,\bar{x}\,\bar{y}
             \;\text{ and } \;
                    {W_P}= 
                    3\,\bar{b}\,\bar{\du}\,\bar{y}^3 + (3\,\bar{b}\,\bar{n}
                    + 2\,\bar{a}\,\bar{\du})\bar{y}^2  + (2\,\bar{a}\,\bar{n}-\bar{c})\,\bar{y} + \bar{c}\,\bar{n}\,\bar{x}.
 \]

 We eliminate  $\bar{x}$ using the vanishing of $\sextic_P$. Manipulating the resulting equations $\ell_P(\bar{y})$ and $W_P(\bar{y})$ in $\resK[\bar{n}^{\pm}, \bar{\du}^{\pm}][\bar{y}^{\pm}]$ yields a  constant polynomial in the corresponding ideal, namely,
 \[
 \bar{b}^2\bar{n}^2-2\,\bar{a}\,\bar{b}\,\bar{n}\,\bar{\du} +\bar{a}^2\bar{\du}^2 + 4\,\bar{b}\,\bar{c}\,\bar{\du}.
   \]
Rewriting the value of $\bar{n}$ in terms of $\overline{\lambda''}$ and $\bar{\du}$ yields a linear equation in $\bar{\du}$, and thus a unique solution in $\bar{\du}, \bar{n}$ and  $\bar{p}$. More precisely, we have
   \begin{equation}\label{eq:3h}
\bar{\du} = -\frac{4\,\bar{b}\,\bar{c}}{(\bar{b}\,\overline{\lambda''}-\bar{a})^2}, \quad \bar{n} = -\frac{4\,\bar{b}\,\bar{c}\,\overline{\lambda''}}{(\bar{b}\,\overline{\lambda''}-\bar{a})^2}     \quad \text{ and } \quad \bar{p} = \left (-\frac{\bar{b}^2\overline{\lambda''}^2-\bar{a}^2}{4\,\bar{b}\,\bar{c}}, \; -\frac{\bar{b}\,\overline{\lambda''}+\bar{a}}{2\,\bar{b}}\right ).
   \end{equation}

   A similar computation can be done to determine  $\overline{m}$ and $\bar{p'}$ from the local equations at $P'$:
   \[\sextic_{P'}=\bar{c}\,\bar{x}\,\bar{y} + \bar{a}\,\bar{y}^2 + \overline{b'}\,\bar{x},\quad \ell_{P'}=\bar{y} + \bar{n}\,\bar{x} + \overline{m}\quad \text{ and } \quad W_{P'}=\bar{c}\,\bar{n}\,\bar{x} +2\,\bar{a}\,\bar{n}\,\bar{y} -\bar{c}\,\bar{y}-\overline{b'}.
   \]
      Eliminating $\bar{x}$ and $\overline{m}$ using $W_{P'}$ and $\ell_{P'}$, and replacing this value in $\sextic_{P'}$ yields the condition \[
   \bar{c}\,(\bar{a}\,\bar{n} - \bar{c})\bar{y}^2 + 2\,\overline{b'}\,(\bar{a}\,\bar{n} - \bar{c})\bar{y} - \overline{b'}^2= 0 .
   \]
    Since $\overline{b}$ plays no role in determining $\overline{\lambda''}$, the genericity of $\sextic$ combined with~\eqref{eq:3h} confirms that the above expression has degree two in $\bar{y}$. Thus, we obtain two solutions in $(\overline{m}, \bar{p'})$, namely,
   \begin{equation}\label{eq:3hBottomVertex}
     \begin{aligned}
       \overline{m} &=
-\frac{\overline{b'}(2\,\bar{a}\,\bar{n} - \bar{c} \pm 2\sqrt{\bar{a}\,\bar{n}(\bar{a}\,\bar{n} - \bar{c})}}{\bar{c}^2}
       \quad \text{ and } \\
       \bar{p'}& =\left (
       \frac{\overline{b'}(2\,\bar{a}\,\bar{n}(\bar{a}\,\bar{n}-\bar{c}) \pm (2\,\bar{a}\,\bar{n}-\bar{c})\sqrt{\bar{a}\,\bar{n}(\bar{a}\,\bar{n} - \bar{c})})}{\bar{c}^2\,\bar{n}(\bar{a}\,\bar{n}-\bar{c})},\;
       -\frac{\overline{b'}(\bar{a}\,\bar{n} - \bar{c} \pm
\sqrt{\bar{a}\,\bar{n}(\bar{a}\,\bar{n} - \bar{c})})}{\bar{c}\,(\bar{a}\,\bar{n}-\bar{c})} 
\right ).
     \end{aligned}
   \end{equation}

   To conclude, we must show that $(\du,p,m, p')$ is uniquely determined by the corresponding tuple of  initial forms. We do so by invoking~\autoref{lm:multivariateHensel}. The determinant of the local equations at $P$ and $P'$ relative to the six variables $(\bar{\du}, \bar{x}, \bar{y}, \overline{m}, \bar{x'}, \bar{y'})$ is block upper-triangular. A~\sage\,computation confirms that the  determinants  of the two diagonal blocks are non-vanishing. Indeed, their values are the  monomials  $-2\,\bar{c}^3\bar{x}^5\bar{\du}^3/\bar{b}$ and
   $\pm\,2\,\overline{b'}\sqrt{\bar{a}\,\bar{n}(\bar{a}\,\bar{n} - \bar{c})}$
    where
   $\bar{x}$ is the first coordinate of $\bar{p}$. 
\end{proof}

\begin{proposition}\label{pr:type3d}
  Assume that $\sextic$ is generic relative to $\Gamma$. Then, the local lifting multiplicity of any type (3d) tangency of $\Gamma$ is two. Furthermore, both  lifts are defined over the same  quadratic extension of $\K$. 
\end{proposition}

\begin{proof} We start by discussing the combinatorics of $\Gamma$ imposed by a  (3d) tangency, using the orbit representative from~\autoref{fig:classificationLocalTangencies}.  We let $P$ be the top vertex of $\Lambda$ and let $P'$ be the tangency point on the negative horizontal leg that are part of the type (3d) tangency.  By construction, the dual cell to the lower vertex $v$ of $\Lambda$ in the Newton subdivision of $\sextic$ is a fixed triangle, with vertices $(1,0)$, $(1,1)$ and $(2,0)$. Similarly,  the cell $P^{\vee}$ is the triangle with has vertices $(1,1)$, $(2,0)$ and $(3,0)$. The remaining tangency point $P''$ between $\Lambda$ and $\Gamma$ lies on the positive vertical leg of $\Lambda$. Its tangency type is either (1a), (2a) or (3c). Since we are interested in lifting $\Lambda$, we can discard the first option.
 In the remainin two cases, $\bar{\du}\in \resK$ and its value  is fixed by the local equations at $P''$.

  We let $(\ell,p, p')$ be a lift of $(\Lambda, P, P')$.   Since $P'$ is a non-transverse horizontal  tangency of multiplicity two, the methods from  the proof of~\autoref{pr:Prop5.2Corrected} allow us to express $m$ in the form $m=a_{1,0}/a_{1,1} + m_1 + m_2$ with $\val(m_2)>\val(m_1)> \val(a_{1,0}/a_{1,1})$ and $m_1\in \K$. The value of $m_1$ is determined by~\autoref{lm:chooseTail3a}. In particular, we know that $\overline{m} = \overline{a_{1,0}}/\overline{a_{1,1}}$  and there are two solutions $(\overline{m_2},\bar{p'})$ imposed by the local equations at the tangency $P'$. Both quantities are defined over the same quadratic extension of $\resK$.

  In what follows we determine the values for $(\bar{n}, \bar{p})$ in terms of $\bar{\du}$, using the local equations at $P$. Denoting  the coefficients of $xy$, $x^2$ and $x^3$ in $\sextic$ by $a, b$ and $c$, respectively, we have
  \[
  \sextic_{P} = \bar{x}(\bar{c}\,\bar{x}^2 + \bar{b}\,\bar{x} + \bar{a}\,\bar{y}),\quad \ell_P = \bar{y} + \bar{n}\,\bar{x} +\bar{\du}\,\bar{x}\,\bar{y} \quad \text{ and } \quad W_P=
3\,\bar{c}\,\bar{\du}\,\bar{x}^3 + (2\,\bar{b}\,\bar{\du} + 3\,\bar{c})\bar{x}^2 - (\bar{a}\,\bar{n}-2\,\bar{b})\bar{x} +\bar{a}\,\bar{y}.
  \]
  The vanishing of $\sextic_P$ allows us to eliminate $\bar{y}$ from $\ell_P$ and $W_P$. Algebraic manipulations in~\sage\, produce linear restrictions for $\bar{n}$ and $\bar{x}$, namely,
  \[
(\bar{b}\,\bar{\du} + \bar{c})\bar{x} - 2\,(\bar{a}\,\bar{n} - \bar{b}) = 0 \quad \text{ and }\quad 4\,\bar{a}\,\bar{c}\,\bar{\du}\,\bar{n} + (\bar{b}\bar{\du}^2 - \bar{c})^2=0.
  \]
Therefore, we have a unique solution in  $(\bar{n}, \bar{p})$ in terms of $\bar{\du}$, namely,
    \begin{equation}\label{eq:3d}
\bar{n} = -(\bar{b}\,\bar{\du} - \bar{c})^2/(4\,\bar{a}\,\bar{c}\,\bar{\du}) \quad \text{ and } \quad
\bar{p} = \left (-(\bar{b}\,\bar{\du} + \bar{c})/(2\,\bar{c}\,\bar{\du}),\;
(\bar{b}^2\bar{\du}^2 -\bar{c}^2)/(4\,\bar{a}\,\bar{c}\,\bar{\du}^2)\right).
  \end{equation}

  We determine the lifting multiplicity of $(\Lambda, P, P')$ by means of~\autoref{lm:multivariateHensel}.   Following the proof of~\autoref{pr:Prop5.2Corrected}, we write $\tilde{\sextic}(x,z) = \sextic(x, z-(\overline{a_{1,0}}/\overline{a_{1,1}}+m_1))$. By construction, the  $6\times 6$ Jacobian of the local equations $(\tilde{\sextic}_{P'}, z-\overline{m_2}, W(\tilde{\sextic}_{P'}, z-\overline{m_2}), \sextic_P, \ell_P, W_P)$ relative to the variables $(\overline{m_2}, \bar{p}, \bar{n},\bar{p'})$ is  block diagonal. A~\sage~computation confirms that the expected initial forms of the determinants of both diagonal blocks are non-zero: the one relative to $P'$ equals $\overline{m}\,\overline{a_{1,1}}^2$, while the one for $P$ is $-2\,\bar{a}\,\bar{c}\,\bar{\du}\,\bar{x}^4$, where $\bar{x}$ denotes the first entry of $\bar{p}$.  Thus, there are exactly two lifts $(m,n,p,p')$, where $n$ and $p'$ are defined over $\K$, whereas $m$ and $p$  are defined over a quadratic extension of $\K$.
\end{proof}

We end this section by analyzing local tangencies of type (3aa) occurring along the  edge of $\Lambda$. 

\begin{proposition}\label{pr:3aaD} Assume that  $\sextic$ is generic relative to $\Gamma$ and let $\Lambda$ be tritangent to $\Gamma$ with a type (3aa) tangency along its unique edge. Then, this local tangency lifts with  multiplicity two. Both local lifts are defined over a quadratic extension of $\K$.
\end{proposition}

\begin{proof}   Without loss of generality, we assume that $\Lambda$ has a slope one edge. We let $P$ be the point carrying the type (3aa) diagonal tangency, and let $(\ell, p)$ be a local lift of $(\Lambda, P)$. The genericity assumptions on $\sextic$ ensure that the remaining two tangencies  between $\Lambda$ and $\Gamma$ cannot belong to  legs of $\Lambda$ adjacent to the same vertex if $\Lambda$ where to lift. By exploiting symmetry, we may place them on  the negative horizontal and positive vertical legs of $\Lambda$. Their types are either (2a) or (3c), so  $\overline{m}$ and $\bar{\du}$ are unique, they lie in $\resK$ and their values are determined by the local equations at these two points. Note that  $\bar{n}$ does not feature in any of them. We determine $n$ with the local system at $P$.

We let $e$ be the slope one edge of $\Gamma$ carrying the input tangency and let $v$ and $v'$ be the bottom and top endpoints of $e$, respectively. 
  The dual cell to $e$ in the Newton subdivision of $\sextic$ has vertices $(2,1)$ and $(3,0)$. In turn, the cells $v^\vee$ and $(v')^{\vee}$ are fixed by a choice of an extra vertex, which we label as $(w,2-w)$ and $(r, 4-r)$, respectively, with $0\leq w\leq 2$ and $1\leq r\leq 3$.

  As usual, we assume that $P=(0,0)$ and $\sextic \in R[x,y]\smallsetminus \mathfrak{M}R[x,y]$, so $\val(a_{21}) = \val(a_{30}) = 0$. A standard argument using chip-firing on $e\cap \Lambda$ confirms that $P$ is the midpoint of $e$. We let $\lambda$ be half the length of $e$. Thus, we have $\val(n) = 0$, $\val(m)=\val(\du)=\lambda$, $\val(a_{ij})>\lambda$ for all $(i,j)$ with $i+j\neq 3$, whereas $\val(a_{03}), \val(a_{12})>0$. Setting $v=(-\mu, -\mu)$ and $v'=(\mu', \mu')$ yields
  \[
  \val(a_{w,2-w})=\mu>\lambda \quad \text{ and } \quad \val(a_{r,4-r})=\mu'>\lambda.
  \]
  
  Our objective is to determine the pair $(\bar{n}, \bar{p})$. Our proposed solution $n$ will be of the form $n=a_{30}/a_{21} + n_1 + n_2$ with $\val(n_2)=\lambda>\val(n_1)>0$. In particular, $\bar{n}=\overline{a_{30}}/\overline{a_{21}}$ for any lifting of $(\Lambda, P)$. The values of $n_1$ and $n_2$ are determined in a two-step procedure, as in the proof of~\autoref{pr:Prop5.2Corrected}. First, a choice of a suitable algebraic lift of the tropical modification of $\RR^2$ along $\Lambda$ produces a re-embedding of $V(\sextic)$ and $V(\ell)$ in
  $(\K^*)^3$ via the ideals
  \[
  I_{\sextic}=\langle \sextic, z-y-m-(a_{30}/a_{21} + n_1)x -\du\, x\, y\rangle \quad \text{ and }\quad I_{\ell} = \langle \ell, z-y-m-(a_{30}/a_{21} + n_1)x -\du \,x \,y\rangle.\]
  Note that these two ideals depend on the parameter $n_1$. The vanishing of the polynomial
  \[
\tilde{\sextic}(x,z)=(1+\du\, x)^3\,f(x,(z-m-(a_{30}/a_{21} + n_1)\,x)/(1+\du\,x)) = \sum_{i,j} \tilde{a}_{ij} \in \K[x,z]
  \]
  determines the projection of $V(I_{\sextic})$ onto the $xz$-plane.  By construction, the Newton polytope of $\tilde{f}$ is the trapezoid with vertices $(0,0)$, $(0,3)$, $(3,3)$ and $(6,0)$.

  \begin{figure}[t]
  \includegraphics[scale=0.11]{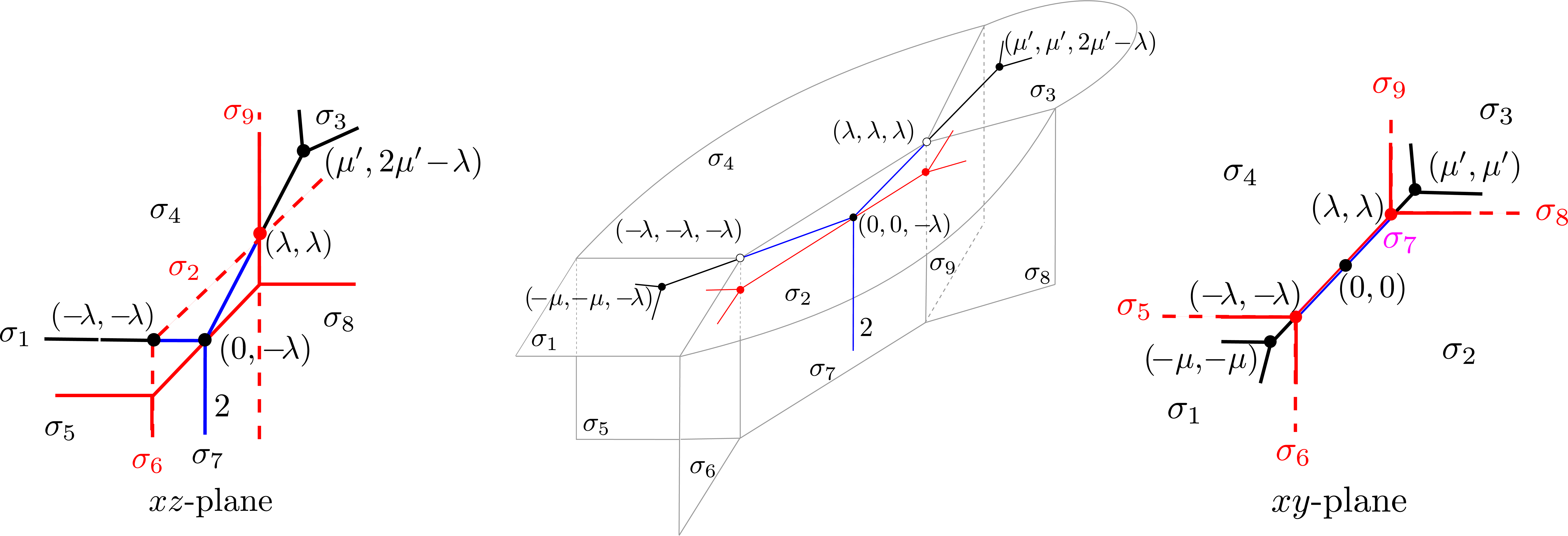}
  \caption{Modification of  $\RR^2$ along $\Lambda$ revealing the local lifting of  a (3aa) tangency. Dashed segments indicate projections of 2-dimensional cells on the modified plane.\label{fig:3aaD}} 
  \end{figure}

  A partial snapshot of the expected  tropical planar curves is depicted on the  left of~\autoref{fig:3aaD} and reveals a type (2a) diagonal tangency at the re-embedded point $P$ with a slope one line, corresponding to the projection of $V(I_{\ell})$.
 The value of $n_1$ producing such tropicalization is fixed by requiring that the  Newton subdivision of $\tilde{\sextic}$ satisfies the following two properties:
  \begin{itemize}
    \item it contains the triangle with vertices $(2,0)$, $(2,1)$ and $(4,0)$,
    \item the vertex $(3,0)$ is unmarked in the subdivision.
  \end{itemize}
  Both combinatorial conditions are a direct consequence of the following claims regarding the  valuations of the relevant coefficients of $\tilde{\sextic}$:
  \begin{equation}\label{eq:valIneq3aaD}
\val(\tilde{a}_{21})=0, \quad  \val(\tilde{a}_{20}) = \val(\tilde{a}_{40}) = \lambda \quad \text{ and } \quad \val(\tilde{a}_{kl})>(1-l)\lambda \quad\text{ for all other }(k,l).
  \end{equation}
  The three equalities can be verified by direct computation. Indeed, it suffices to analyze which of the terms in each expression $a_{ij}x^i(z-m-(a_{30}/a_{21} + n_1)\,x)^j(1+\du x)^{3-j}$ contribute to the coefficient $\tilde{a}_{kl}$ of interest and have  valuation smaller than desired.   Furthermore, the same calculation reveals 
  \begin{equation}\label{eq:relevantInitialsModified3aaD}
    \overline{\tilde{a}_{20}} = -\overline{a_{21}} \,\overline{m}, \quad \overline{\tilde{a}_{21}} =\overline{a_{21}} \quad \text{ and } \quad \overline{\tilde{a}_{40}} = \overline{a_{30}}\,\bar{\du}.
    \end{equation}

  For the inequalities on the right-side of~\eqref{eq:valIneq3aaD}, it suffices to discuss the cases when $l<2$ by our assumptions on $\sextic$ and $\lambda$. Writing the coefficients $\tilde{a}_{kl}$ for $(k,l) \neq (3,0)$ as polynomials in $R(n)[m,\du]$ using~\sage, we certify the desired inequalities by analyzing their linear and constant terms. A simple inspection shows that  all constant terms involve a variable $a_{i,j}$ with $i+j\neq 3$. In turn, for $l=1$, all monomials in  the linear terms in $m,\du$  involve a variable other than $a_{21}, a_{30}, n$.  These observations and the information on the valuations of all $a_{ij}$, $m$, $n$ and $\du$ prove our claim.

  It remains to check that $\val(\tilde{a}_{30})>\lambda$. For this,  we let $\tilde{a}_{3,0}'$  be the sum of all those   terms of   $\tilde{a}_{3,0}$ with valuation at most $\lambda$. The knowledge about the valuations of $m,\du$ and all coefficients $a_{ij}$ obtained earlier, ensures that $\tilde{a}_{3,0}\in R[n_1]$. More precisely, we have
  \begin{equation}\label{eq:valuea30Tang3aaD}
    \tilde{a}_{3,0}'(n_1) :=  a_{30}   - a_{21}(a_{30}/a_{21} + n_1) + a_{12} (a_{30}/a_{21} + n_1)^2 - a_{03} (a_{30}/a_{21} + n_1)^3.
  \end{equation}
  The genericity of $\sextic$ forces $\tilde{a}_{3,0}'(0)\neq 0$. In turn,  
  \autoref{lm:3aaDvaluen1} below ensures the existence of a unique $n_1\in \mathfrak{M}$ with $\tilde{a}_{3,0}'(n_1) = 0$.

  Once the value of $n_1$ is determined, we use the local equations for the tangency point $P_1=(0,-\lambda)$ of $\Trop \,V(\tilde{\sextic})$ to  determine both $\overline{n_2}$ and the initial form of the tangency point $p_1$ lifting $P_1$. The system is given by the vanishing of 
  \begin{equation}\label{eq:3aaDSystemPostModification}
    \tilde{\sextic}_{P_1} = \overline{\tilde{a}_{20}} + \overline{\tilde{a}_{21}}\,\bar{z} + \overline{\tilde{a}_{40}}\, \bar{x}^2,\quad \tilde{\ell}_{P_1} = \bar{z}-\overline{n_2}\,\bar{x}\quad \text{ and } \quad W_{P_1} = 2\,\overline{\tilde{a}_{40}}\,\bar{x} - \overline{\tilde{a}_{21}}\,\overline{n_2}.
  \end{equation}
  A direct computation combined with the identities in~\eqref{eq:relevantInitialsModified3aaD} and the fact that $\overline{m}, \bar{\du}$ are uniquely determined by $\sextic$, confirms that  system has two solutions, namely,
  \[
  \overline{n_2} =
  \pm 2\sqrt{-\overline{a_{21}}\,  \overline{a_{30}}\,\overline{m}\,\bar{\du}\,}/\overline{a_{21}}\quad \text{ and }\quad
  \overline{p_1} =  (\pm \sqrt{-\overline{a_{21}}\,  \overline{a_{30}}\,\overline{m}\,\bar{\du}\,}/(\overline{a_{30}}\,\bar{\du}),\; 2\,\overline{m}).
\]

A direct calculation using~\sage~show that the determinant of the Jacobian of the system~\eqref{eq:3aaDSystemPostModification} has expected initial form $-\overline{n_2}$. Since this value is non-zero, ~\autoref{lm:multivariateHensel} ensures that  each value $(\overline{n_2}, \overline{p_1})$ has a unique lift to a solution $(n_2, p_1)$ for the tangency between $V(\tilde{\sextic})$ and $V(z-n_2x)$. 
\end{proof}

\begin{lemma}\label{lm:3aaDvaluen1}
  The univariate polynomial  $\tilde{a}_{3,0}'(n_1)\in R[n_1]$ from~\eqref{eq:valuea30Tang3aaD} has a unique root $n_1$ in $\K$ of positive valuation if its constant coefficient is non-zero. Furthermore, this solution is unique and has valuation $\val(n_1)= \val(a_{12}a_{21}-a_{03}a_{30})$.
\end{lemma}

\begin{proof} We prove the statement by invoking the Fundamental Theorem of Tropical Algebraic Geometry. The tropicalization of the polynomial $\tilde{a}_{3,0}'(n_1)$ becomes the piecewise linear function 
\[\trop(\tilde{a}_{3,0}')(N_1) = \max\{  - \val(a_{12}a_{21}-a_{03}a_{30}), N_1, -\val(a_{21}a_{12} - 3\,a_{03}\,a_{30}) + 2\,N_1, -\val(a_{03}) + 3\, N_1\}.
\]
The information on the valuations of $a_{30}, a_{21}, a_{03}$ and $a_{21}$ provided earlier ensures that when $N_1< 0$, the max is attained twice for a single value of $N_1$, namely $N_1 = - \val(a_{12}a_{21}-a_{03}a_{30})$. We conclude that $\tilde{a}_{30}'(n_1)$ has a unique root $n_1\in \overline{\K}$ with $\val(n_1)=-N_1$. Furthermore, since 
\[(\tilde{a}_{30}')_{N_1} = \frac{\overline{a_{30}}^2}{\overline{a_{21}}^3} (\overline{a_{12}\,{a_{21}} - a_{30}\,a_{03}}) -\overline{a_{21}}\,\overline{n_1}=0
\quad  \text{ and } \quad
\frac{\partial\, (\tilde{a}_{30}')_{N_1}}{\partial\, \overline{n_1}}|_{\frac{\overline{a_{30}}^2}{\overline{a_{21}}^4}(\overline{a_{12}\,a_{21} - a_{30}\,a_{03}})} = -\overline{a_{21}} \neq 0, 
\]
it follows from \autoref{lm:multivariateHensel} that $n_1\in R$. This concludes our proof.
\end{proof}

\section{Local lifting multiplicities for tropical 4-valent tritangents}\label{sec:4ValentLifts}

In this section, we focus our attention on $4$-valent  tritangent curves. This situation was not encountered in~\cite{CM20, LM17}. The bottom row of~\autoref{fig:classificationLocalTangencies} shows all the possible local tangencies that we can have for such $\Lambda$'s. In what follows, we determine their lifting multiplicities.
The chosen notation  manifests the similarities with the corresponding trivalent cases.  The lifting multiplicities of (1a'), (2a') and (3c') agree with those of types (1a), (2a) and (3c), respectively. Their values are $0$, $1$ and $2$, respectively. This section is devoted to treat the remaining cases.

We start by discussing type (3a'). Our first result confirms that they do not lift over $\overline{\K}$ under  our standard genericity restrictions on $\sextic$.

\begin{lemma}\label{lm:type3a'} 
 Assume that $\Lambda$ has a tropical tangency of type (3a'). If $\sextic$ is generic relative to $\Gamma$, then  $\Lambda$  does not lift  to a tritangent curve to  $V(\sextic)$ defined over $\overline{\K}$.
  \end{lemma}

\begin{proof} Up to symmetry, we may assume that the type (3a') tangency occurs along a negative horizontal leg of $\Gamma$, which we label as $e$. The corresponding stable intersection  consists of two simple points: the vertex $v$ of $\Lambda$ and the rightmost vertex $w$ of $e$. Thus, the corresponding tangency point $P$ is the midpoint between $v$ and $w$, so it lies on the positive horizontal leg of $\Lambda$. The local equations at $P$ yield a unique value for $\bar{n}/\bar{d}$: it is  a rational function of the initial forms of the coefficients of $\sextic$ corresponding to the dual edge $e^{\vee}$ in the Newton subdivision of $\sextic$.

  For bidegree reasons, the positive horizontal leg of $\Lambda$ contains a second tangency point $P'$  in its relative interior, which can be of types (1a'), (2a')  or (3c'). We neglectic the first case if $\Lambda$ lifts. Using \autoref{lm:type2Horiz} and~\autoref{pr:Prop5.2Corrected}, we see that the local equations at $P'$  determines a unique expression for $\bar{n}/\bar{d}$. Furthermore, its value is a rational function of the initial forms of the coefficients of $\sextic$ involved in $(P')^{\vee}$. Since  $e^{\vee}\neq (P')^{\vee}$, our genericity assumptions on $\sextic$ imply the triple $(\Lambda, P, P')$ cannot lift over $\overline{\K}$ to a classical triple $(\ell, p, p')$.
\end{proof}

Next, we discuss the remaining multiplicity two cases, namely, types (4a') and (6a'). We treat them together, with the same methods used in the proof of~\autoref{pr:4a6aWithAdjacentLeg}, since two of the legs adjacent to the vertex of $\Lambda$ will contain the remaining two tangency points $P'$ and $P''$, which are necessarily distinct. In addition, the bidegree of $\Gamma$ forces $P'$ and $P''$ to lie in different legs of $\Lambda$.

To determine the lifting multiplicity of these types, we start by fixing as representatives those depicted in~\autoref{fig:classificationLocalTangencies}.
 We let $e$ be the edge of $\Gamma$ of slope $-1$ containing the tangency point $P$, and $H^{\pm}$ be the hyperplanes determined by the line spanned by $e$. We introduce the following quantity
\begin{equation}\label{eq:mu}
  \mu:=\max\{|H^+\cap \{P',P''\}|,|H^-\cap \{P',P''\}|\}.
\end{equation}
Thus,  $\mu=1$ if $P'$ and $P''$ lie on different halfspaces relative to $e$. Otherwise, we have $\mu=2$.

\begin{proposition}\label{pr:type4ap6ap} Let $\Lambda$ be a 4-valent tritangent (1,1)-curve to $\Gamma$.  Assume that $\Lambda$ is a $4$-valent tritangent to $\Gamma$ with a tangency point of type (4a') or (6a') at its vertex $P$. We  let $P'$ and $P''$ denote the remaining tangency points. 
  If $\sextic$ is generic relative to $\Gamma$, then the lifting multiplicity of $(\Lambda,P)$ equals the quantity $\mu$ from~\eqref{eq:mu}. Furthermore, the lifting multiplicity of the tuple  $(\Lambda, P, P', P'')$ equals the product of the three local lifting multiplicities at all three tangency points.
\end{proposition}

\begin{proof} We use an auxiliary variable $\varepsilon$ to distinguish between the two tangency types, setting  $\varepsilon =0$ if $P$ has type (4a') and $\varepsilon =1$ if it is  (6a'). Up to eliminating a Laurent monomial factor in $\sextic_P$, the local equations at $P$ become
  \[
  \sextic_P = 
  \overline{a_{u,v}} + \overline{a_{u+1,v+1}}\,\bar{x}\,\bar{y} +\varepsilon\, \overline{a_{u,v+1}}\,\bar{y} 
  , \; \ell_P =  \bar{y} + \overline{m} + \bar{n}\,\bar{x} + \bar{\du}\,\bar{x}\,\bar{y} \quad \text{and} \quad  W_P=\det(\Jac(\sextic_P,\ell_P; \bar{x},\bar{y})).
  \]

  We eliminate the $\bar{x}$-variable from the system, setting $\bar{x}= -(\overline{a_{u,v}} + \varepsilon\,\overline{a_{u,v+1}}\,\bar{y})/(\overline{a_{u+1,v+1}}\,\bar{y})$. Algebraic manipulations using~\sage\, produce the following two polynomials in the ideal $\langle \sextic_P(\bar{y}), W_P(\bar{y})\rangle$ of $\resK[\overline{m}^{\pm}, \bar{n}^{\pm}, \bar{\du}^{\pm}][\bar{y}^{\pm}]$:
  \begin{equation}\label{eq:4ap6ap}\begin{aligned}&
(\overline{a_{u+1,v+1}}\,\overline{m}-\varepsilon\,\overline{a_{u,v+1}}\,\bar{n}-\overline{a_{u,v}} \,\bar{\du})^2 + 4\,\overline{a_{u,v}}\,\bar{n}(\overline{a_{u+1,v+1}}-\varepsilon\,\overline{a_{u,v+1}}\,\bar{\du}),
      \\
  &(\overline{a_{u+1,v+1}}\,\overline{m} -\varepsilon\,\overline{a_{u,v+1}}\,\bar{n} -\overline{a_{u,v}}\,\bar{\du})\bar{y} - 2\,\overline{a_{u,v}}\,\bar{n}.
    \end{aligned}
  \end{equation}
  Since $\bar{n}\neq 0$ we see that the coefficient of $\bar{y}$ in the second equation cannot vanish. Thus,  the tuple   $\bar{p}$ is uniquely determined from $\overline{m}, \bar{n}$ and $\bar{\du}$.

  By construction, the remaining tangency points $P'$ and $P''$ can only be of type (1a'), (2a') or (3c'). The first case can be ignored since it will not lift.  The bidegree of $\sextic$ restricts the locations of $P'$ and $P''$ to a horizontal and vertical leg of $\Lambda$, respectively. By the action of $\langle \tau_0, \tau_1^2\rangle \subseteq \Dn{4}$ we may  assume that $P'$ is on the positive horizontal leg. For type (6a'), $P''$  must lie on the negative vertical leg, so $\mu=1$. For type (4a'),  the position of $P''$ is determined by the value of $\mu$.

  By~\autoref{lm:type2Horiz} and~\autoref{pr:Prop5.2Corrected}, the local equations at $P'$  determine a unique value for $\overline{\lambda'}$, where $\lambda':=n/\du$. In turn, a positive vertical tangency at $P''$ fixes a unique value for $\bar{\du}$, whereas one negative vertical leg yields a unique value for $\overline{\lambda''}$, where $\lambda''=m/n$.  

  This information allows us to simplify the top expression in~\eqref{eq:4ap6ap}. The precise formula depends on the value of $\mu$.
  We  show that it has  $\mu$-many solutions $\overline{m}$ in terms of the parameters  $\overline{\lambda'}$ and either   $\overline{\lambda''}$ or  $\bar{\du}$ depending on whether $\mu =1$ or $2$, respectively. Indeed, rewriting the expression in terms of $\overline{m}$ and these parameters using~\sage, and removing Laurent monomial factors yields:
  \begin{equation}\label{eq:4ap6apConstant}\begin{aligned}&
      ( (\overline{a_{u+1,v+1}}\,\overline{\lambda''} -\varepsilon\,\overline{a_{u,v+1}} \,\overline{\lambda'}- \overline{a_{u,v}}      )^2 - 4\, 
      \varepsilon\,\overline{a_{u,v+1}}\,\overline{a_{u,v}}\,\overline{\lambda'})\overline{m} + 4\, \overline{a_{u,v}}\,\overline{a_{u+1,v+1}}\,\overline{\lambda''}\,\overline{\lambda'}^2 & \text{ for } \mu =1,\\
      &
      (\overline{a_{u+1,v+1}}\,\overline{m} - \bar{\du}(      \varepsilon\,\overline{a_{u,v+1}}\,\overline{\lambda'} - \overline{a_{u,v}}))^2 - 4\,\overline{a_{u,v}}\,\bar{\du}\,\overline{\lambda'}
      ( \varepsilon\,\overline{a_{u,v+1}}\,\bar{\du} - \overline{a_{u+1,v+1}})
      & \text{ for } \mu =2.
    \end{aligned}
  \end{equation}
Note that the scalars of $\overline{m}$ in both equations are non-zero, per our genericity assumptions on $\sextic$. Thus, the degree of $\overline{m}$ in both equations matches $\mu$, as we wanted to show.

We should remark that  if $\mu=2$, then the point $P$ is necessarily of type (4a'). Setting $\varepsilon=0$ in the corresponding expression in~\eqref{eq:4ap6apConstant} yields the following two solutions for $(\overline{m},\bar{p})$:
\[
\overline{m}=\frac{\overline{a_{u,v}}\,\bar{\du} \pm 2\,\sqrt{-\overline{a_{u,v}}\,\overline{a_{u+1,v+1}}\,\bar{\du}\,\overline{\lambda'}}}{\overline{a_{u+1,v+1}}}\; \text{ and }\;
\bar{p}=\left ( \mp \frac{\sqrt{-\overline{a_{u,v}}\,\overline{a_{u+1,v+1}}\,\bar{\du}\,\overline{\lambda'}}}{\overline{a_{u+1,v+1}}\,\bar{\du}\,\overline{\lambda'}},\, \pm \frac{\overline{a_{u,v}}\,\bar{\du}\,\overline{\lambda'}}{\sqrt{-\overline{a_{u,v}}\,\overline{a_{u+1,v+1}}\,\bar{\du}\,\overline{\lambda'}}}\right).
\]

\autoref{lm:Jac4ap6ap} below confirms that no matter the value of $\mu$, each of the initial values $\bar{p}$, $\bar{p'}$, $\bar{p''}$ and $(\overline{m},\overline{\lambda'}, \overline{\lambda''})$ (for $\mu=1$), respectively, $(\overline{m},\overline{\lambda'}, \bar{\du})$  (for $\mu=2$) obtained from the local equations for $P, P'$ and $P''$ lift uniquely to a classical tritangent tuple $(\ell, p, p', p'')$. 
\end{proof}

\begin{lemma}\label{lm:Jac4ap6ap} Let $\Lambda, P, P'$, and $P''$  be as in~\autoref{pr:type4ap6ap}. Each solution to the local equations defining these three tangency points has a unique lift to a  tritangent tuple $(\ell, p, p', p'')$. 
\end{lemma}

\begin{proof} The statement is a direct consequence of~\autoref{lm:multivariateHensel}. The choice of the $9\times 9$ Jacobian matrix depends on the nature of the tangencies $P'$ and $P''$. The contribution of $P$ is computed by using the equations $\{\sextic_{P}, \ell_{P}, W_{P}\}$ and the variables $\{x, y, m\}$.   If $P'$ is of type (2a), we use the equations $\sextic_{P'}, \ell_{P'}, W_{P'}$ and the variables $\{x', y', \lambda'\}$. In turn, when $P'$ has type (3c), we use $\{\tilde{\sextic}(x'',z), z-\lambda'_2,W(\tilde{\sextic}, z-\lambda_2')\}$ and $\{x'', z, \lambda_2'\}$, as described in the proof of~\autoref{pr:Prop5.2Corrected} (replacing the variable $m$ by $\lambda'$). Similar choices are made for  $P''$.

  With these conventions, the initial of the Jacobian matrix becomes a block upper-triangular matrix. The block diagonal components correspond to the Jacobians of the local systems at $P$, $P'$ and $P''$. By~\autoref{lm:type2Horiz} and~\autoref{pr:Prop5.2Corrected} the last two block have determinant with non-vanishing expected initial form. In turn, the expected initial form of the determinant of the block corresponding to $P$ depends on $\mu$. Here are the precise values obtained with~\sage\,given $\bar{p}=(\bar{x},\bar{y})$:
  \[\Jac(\sextic_P, \ell_P, W_P; \bar{x}, \bar{y}, \overline{m})(\bar{p}, \overline{m}) =
  \begin{cases}
    -2\,\overline{a_{u,v}}\,\overline{a_{u+1,v+1}}\,\overline{m}\,\bar{y}^3/\overline{\lambda''} & \text{ if }\mu = 1,\\
     2\,\overline{a_{u,v}}^2\,\overline{\du}^3\,\bar{x}\,\overline{\lambda'}^3 & \text{ if } \mu =2.\\
  \end{cases}
  \]
 By construction, these expressions are never vanishing, as we wanted to show.
\end{proof}

  The last two cases  that remain to be analyzed correspond to $4$-valent curves  $\Lambda$ with tropical hyperflexes, namely tangencies of types (4b') or (6b'). Recall that in both situations,  the tropical hyperflex $P$ agrees with the unique vertex of $\Lambda$. In turn, $P$ lies in the relative interior of an edge $e$ of $\Gamma$ or it becomes the endpoint of $e$ adjacent to a leg of $\Gamma$.
  Without loss of generality, we suppose that $P=(0,0)$ and $\sextic(x,y)\in R[x,y]\smallsetminus \mathfrak{M}R[x,y]$.

  Exploiting the $\Dn{4}$-symmetry, and the bidegree of $\sextic$ we can assume the  endpoints of  $e^{\vee}$ are $\{(0,0), (1,3)\}$ and that the  remaining tangency point $P''$ between $\Gamma$ and $\Lambda$ lies on the positive horizontal leg of $\Lambda$.
In particular, by \autoref{lm:type2Horiz} and \autoref{pr:Prop5.2Corrected}, we know that $P''$ determines a unique value for $\overline{\lambda''}$, where $\lambda'':=n/\du$.  As we will see below, the local equations at $P$ determine $(\overline{m},\bar{\du})$ or $(\overline{m},\bar{n})$ in terms of $\overline{\lambda''}$ for both tangency types.

First, we discuss the case when $P$ is of type (4b').  If $(\ell, p,p')$ is a triple lifting $(\Lambda, P)$, the local system determining  the tropical hyperflex is given by the vanishing of the following three equations:
\begin{equation}\label{eq:localMult4_4bCross}
  \sextic_P(x,y) = \bar{a} + \bar{c}\, \bar{x}\,\bar{y}^3, \quad
{\ell_P}\!= \bar{y} + \overline{m} +\bar{n}\, \bar{x} + \bar{\du}\,\bar{x}\,\bar{y}  \;\text{ and }\; {W_P}= (2\,\bar{\du}\,\bar{x}\,\bar{y} + 3\,\bar{n}\,\bar{x} - \bar{y})\,\bar{c}\,\bar{y}^2,
    \end{equation}
where $a, c$ denote the coefficients of $\sextic$ associated to the endpoints of $e^{\vee}$. 
Our next theorem determines the solutions in $\overline{m},\bar{n}$ and $\bar{\du}$ of the system defined by~\eqref{eq:localMult4_4bCross}:

  \begin{theorem}\label{thm:mult4type4_4val}
    Assume that $\sextic$ is generic relative to $\Gamma$, and let $P$ be a type (4b') tangency of $\Gamma$ with local lift $(\ell,p,p')$. Then, the three coefficients $m,n,\du$ of $\ell$ satisfy the following constraints:
    \begin{equation}\label{eq:conditions4bCross}
      (\overline{m} = -8\,\bar{n}/\bar{\du}\;\; \text{ and } \;\;\bar{a}\,\bar{\du}^4 + 64\,\bar{c}\,\bar{n}^3 = 0) \qquad \text{ or } \qquad (\overline{m} = 4\,\bar{n}/\bar{\du} \;\;\text{ and }\;\; \bar{a}\,\bar{\du}^4 - 16\,\bar{c}\,\bar{n}^3 = 0).
    \end{equation}
    Furthermore, if the triple has an initial hyperflex, the second option must occur.
  \end{theorem}
  \begin{proof} The result follows by direct computation. We use  $\sextic_P$ to obtain a Laurent monomial expression for $\bar{x}$, namely $\bar{x} = -\bar{a}/(\bar{c}\,\bar{y}^3)$. Substituting this value in $\ell_P$ and $W_P$ yields an ideal $I$ of $\resK[\overline{m}^\pm,\bar{n}^{\pm}, \bar{\du}^{\pm}][\bar{y}^{\pm}]$ with two generators. Algebraic manipulations between $\ell_P(\bar{y})$ and $W_P(\bar{y})$ using~\sage~produce the following two polynomials in the ideal I:
\begin{equation}\label{eq:gh4bp}
  \begin{aligned}
    g := & \,3\,\bar{a}\,\bar{\du}\,\bar{y}^2 +
    (2\,\bar{a}\,\overline{m}\,\bar{\du} +
    4\,\bar{a}\,\bar{n})\,\bar{y} + 3\,\bar{a}\,\overline{m}\,\bar{n},
    \\ h := & \underbrace{(4\,\bar{c}\,\overline{m}^3\,\bar{\du}^2 +
      7\,\bar{c}\,\overline{m}^2\,\bar{n}\,\bar{\du} +
      16\,\bar{c}\,\overline{m}\,\bar{n}^2 -
      27\,\bar{a}\,\bar{\du}^3)}_{=:A_1}\,\bar{y} +
    6\,\bar{n}\,\underbrace{(\bar{c}\,\overline{m}^3\,\bar{\du} +
      2\,\bar{c}\,\overline{m}^2\,\bar{n} -
      6\,\bar{a}\,\bar{\du}^2)}_{=:A_0}.
  \end{aligned}
\end{equation}
By construction, $I$ has no linear polynomial, so both $A_0$ and $A_1$ must vanish. The ideal $\langle A_0, A_1\rangle$ of $\resK[\bar{\du}^{\pm}, \bar{n}^{\pm}][\overline{m}^{\pm}]$ contains the linear polynomial $(\bar{a}\,\bar{\du}^4 - 96\,\bar{c}\,\bar{n}^3)\overline{m} + 20\,\bar{a}\,\bar{n}\,\bar{\du}^3$. Note that  since $\bar{a}$ does not feature in the local system at $P''$ defining $\bar{n}/\bar{\du}$ and $\sextic$ is generic we know that  $(\bar{a}\,\bar{\du}^4 - 96\,\bar{c}\,\bar{n}^3) \neq 0$. Thus, the parameter $\overline{m}$ is a rational function in  $\bar{n}$ and $\bar{\du}$.

Substituting this expression back in $A_0$ and $A_1$ and clearing the powers of $(\bar{a}\,\bar{\du}^4 - 96\,\bar{c}\,\bar{n}^3)$ featuring in both formulas yields two constraints on $\bar{n}$ and $\bar{\du}$, namely:
\begin{equation*}
  (\bar{a}\,\bar{\du}^4 + 864\,\bar{c}\,\bar{n}^3)(\bar{a}\,\bar{\du}^4 + 64\,\bar{c}\,\bar{n}^3)(\bar{a}\,\bar{\du}^4 - 16\,\bar{c}\,\bar{n}^3) = (3\,\bar{a}\,\bar{\du}^4 + 2272\,\bar{c}\,\bar{n}^3)(\bar{a}\,\bar{\du}^4 + 64\,\bar{c}\,\bar{n}^3)(\bar{a}\,\bar{\du}^4 - 16\,\bar{c}\,\bar{n}^3) = 0.
\end{equation*}
There are three possibilities for solving this system. The first two correspond to the vanishing of the binomials in $\bar{\du}$ and $\bar{n}$ featured in~\eqref{eq:conditions4bCross}. The remaining option is given by the vanishing of the expressions  $(\bar{a}\,\bar{\du}^4 + 864\,\bar{c}\,\bar{n}^3)$ and $(3\,\bar{a}\,\bar{\du}^4 + 2272\,\bar{c}\,\bar{n}^3) = 0$. Their common solution in $(0,0)$ which is not allowed.

To obtain the formulas for $\overline{m}$ listed in the statement we replace $\bar{n}$ by $\bar{\du}\overline{\lambda''}$ in each of our binomial equations featured in~\eqref{eq:conditions4bCross} to obtain a linear expression for $\bar{n}$. Replacing this solution back in  $\overline{m}$ yields the expressions in~\eqref{eq:conditions4bCross}. Finally, substituting these values back in the quadratic polynomial $g$ seen in~\eqref{eq:gh4bp} produces two different solutions $\bar{y} =  2\,\overline{\lambda''}(1 \pm  \sqrt{3})$ when $\overline{m} = -8\,\overline{\lambda''}$, but only one  $\bar{y} = -2\,\overline{\lambda''}$ when $\overline{m}=4\,\overline{\lambda''}$. 
  \end{proof}

  \begin{corollary}\label{cor:valuesMult4Type4} If $(\ell,p,p')$ is a lift of a type (4b') tangency between $\Lambda$ and $\Gamma$, then we have
    \[(\overline{m}, \bar{n}, \bar{\du})  \;=\;
    (4\,\overline{\lambda''}, {16\, \bar{c}\,\overline{\lambda''}^4}/{\bar{a}}, {16\,\bar{c}\,\overline{\lambda''}^3}/{\bar{a}} ) \quad \text{ or }\quad (-8\overline{\lambda''}, {-64\, \bar{c}\,\overline{\lambda''}^4}/{\bar{a}}, {-64\, \bar{c}\,\overline{\lambda''}^3}/{\bar{a}}).
    \]
    Furthermore, in the first case, the triple has initial hyperflex  $\left (\bar{a}/{(8\,\bar{c}\,\overline{\lambda''}^3)}, \, -2\overline{\lambda''}\right )$, whereas for the second value, the initials forms of the tangency points are $\left(-\bar{a}/(8\,\bar{c}\overline{\lambda''}^3\,(1\pm \sqrt{3})^3), 2\,\overline{\lambda''}\,(1\pm \sqrt{3})\right )$.
  \end{corollary}

  As we did with types (6b) and (5b), we must rule out initial hyperflexes for a lifting $(\ell,p,p')$ of a type (4b') tangency. This is the content of the next statement:
  \begin{theorem}\label{thm:lifting4bpInitHyperflex}
    Assume that $\sextic$  is generic relative to $\Gamma$. Fix a triple  $(\ell,p,p')$ lifting a tangency point of type (4b') between $\Lambda$ and $\Gamma$. Then, we have $\overline{p} \neq \overline{p'}$.
  \end{theorem}
  
  \begin{proof} 
    We follow the same steps as in the proof of~\autoref{thm:liftingMult4InitHyperflex}, arguing by contradiction. As usual, we assume the type (4b') tangency occurs at  $P=(0,0)$ and that $\sextic \in R[x,y]\smallsetminus \mathfrak{M} R[x,y]$. The values of $\overline{m}, \bar{n}, \bar{\du}$ and $\bar{p}=\overline{p'}$ are fixed by~\autoref{cor:valuesMult4Type4}. They depend on the initial form of  $\lambda'' := n/\du$. We set $p:=(r,s)$ and $p':=(r',s')$. The vanishing of $\ell$ at $p$ and $p'$ ensures that $x\neq x'$ and $y\neq y'$.

    We pick  two generic elements $\varepsilon_1, \varepsilon_2$ in $\resK$ of valuation $N$, where $N$ is as in~\eqref{eq:N}. We let $I$ and $I_{\ell}$ be the ideals from~\eqref{eq:idealsMult4} and $\tilde{\sextic}(x',y')$, $\tilde{\ell}(x',y')$ be as in~\eqref{eq:modifiedql}. Since $m,n,\du \notin \mathfrak{M}$, the coefficients $m', n'$ from~\eqref{eq:valuesnewmn} satisfy $\val(m')\geq N$, $\val(n)=\val(n')$, and $\overline{n'} = \bar{n}+\bar{\du}\,\bar{s} = -\bar{n} = -16 \bar{c}\,\overline{\lambda''}^4/\bar{a}$. The expected initial term of the $y'$-coefficient of $\tilde{\ell}$ equals $1+\bar{\du}\bar{r} = 3$. Combining these two facts, we see that the expected initial form of $m'$ equals:
        \[
    ( 1+ \bar{\du}\,\overline{r})\overline{\varepsilon_2} + (\bar{n} + \bar{\du}\,\bar{s})\overline{\varepsilon_1} = 3 \,\overline{\varepsilon_2} -\bar{n}\,\overline{\varepsilon_1}. \]
    
        The genericity constraints on $\varepsilon_1, \varepsilon_2$ become
        $3 \overline{\varepsilon_2} -\bar{n}\,\overline{\varepsilon_1}$ and $\tilde{\sextic}(0,0) \neq 0$. They are imposed to ensure  $m'$ has valuation $N$ and that $\tilde{\sextic}$ has a constant term. The latter can be attained since the constant coefficient of $\tilde{\sextic}$ is a non-zero polynomial in $R[\varepsilon_1, \varepsilon_2]$. Furthermore, a~\sage~computation confirms that the expected initial form of the constant coefficient of $\tilde{f}$ vanishes, thus yielding a  positive valuation.

    \begin{figure}[tb]
      \includegraphics[scale=0.55]{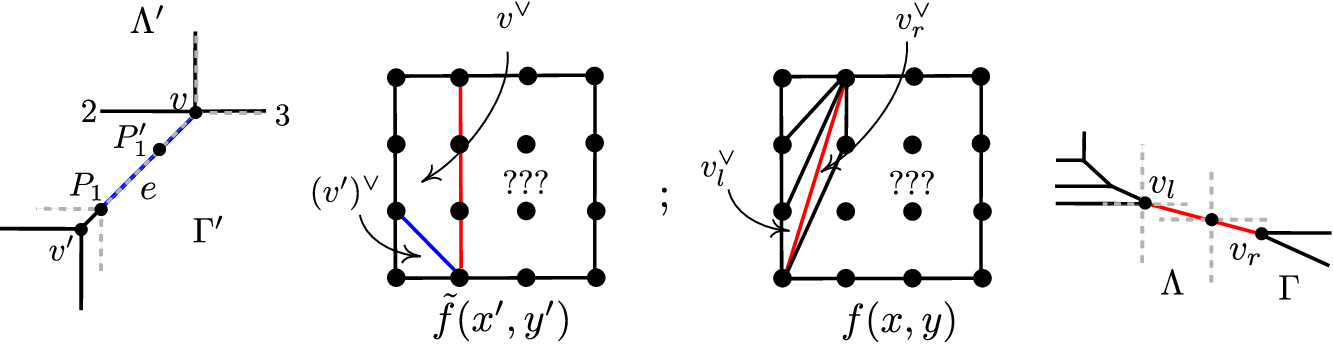}
        \caption{From left to right:  $\Lambda'$ (viewed relative to $\Gamma'$), Newton subdivisions of both $\tilde{\sextic}$ and $\sextic$, and local information about $\Gamma$ and $\Lambda$ in the presence of an initial hyperflex of tangency type (4b') or (6b').\label{fig:4bprime}}
    \end{figure}

    The relevant pieces of the Newton subdivision of $\tilde{\sextic}$ and the relative positions of $\Gamma':=\Trop \,V(\tilde{\sextic})$ and $\Lambda':=\Trop\,V(\tilde{\ell})$ can be seen in~\autoref{fig:4bprime}.    The new tangency points become $p_1=(-\varepsilon_1, -\varepsilon_2)$ and $p_1'=(r'-r-\varepsilon_1, s'-s-\varepsilon_2)$. The point $P_1:=\Trop\, p_1$ becomes the lower vertex of $\Lambda'$, per our genericity assumptions on $\varepsilon_1, \varepsilon_2$. A second modification (verbatim to the one given at the end of the proof of~\autoref{thm:liftingMult4InitHyperflex}) allows us to assume that $P_1$ is not a vertex of  $\Gamma'$. This condition forces  $\tilde{\sextic}_{P_1}$ to have two terms, with monomials $x'$ and $y'$.

    The vanishing at $\overline{p_1}$ of $\tilde{\sextic}_{P_1}, \tilde{\ell}_{P_1}$ and their  Wronskian forces $\overline{\tilde{a}_{10}}\,\overline{\varepsilon_1} + \overline{\tilde{a}_{01}}\,\overline{\varepsilon_2} = 3\,\overline{\tilde{a}_{10}} +\bar{n}\,\overline{\tilde{a}_{01}} = 0$. Since $\overline{\tilde{a}_{10}}, \overline{\tilde{a}_{01}}\neq 0$, we conclude that   $3\,\overline{\varepsilon_2} -\bar{n}\,\overline{\varepsilon_1}=0$. This contradicts the  genericity of $\varepsilon_1$ and $\varepsilon_2$.
  \end{proof}
  
  \begin{corollary}\label{cor:4bCrossLiftsWithNoInitHyperflex} Assume that $\Gamma$ has a tropical hyperflex of type (4b'). If $\sextic$ is generic relative to $\Gamma$, then a pair $(\Lambda,P)$ encoding a type (4b')  tangency lifts to a triple $(\ell,p,p')$ with no initial hyperflexes. Its local lifting multiplicity equals one. Furthermore, the lifting multiplicity of $\Lambda$ agrees with that of $(\Lambda,P'')$, where $P''$ denotes the remaining tangency point between $\Lambda$ and $\Gamma$. 
  \end{corollary}
  \begin{proof} The result follows by combining~\autoref{thm:mult4type4_4val} and the genericity conditions on $\sextic$. Furthermore, the bidegree of $\Gamma$ ensures that $P''$ lies along the positive horizontal leg of $\Lambda$ if the latter lifts. By~\autoref{thm:lifting4bpInitHyperflex}, we know that any lift $(\ell,p,p')$ of $(\Lambda, P)$ has no initial hyperflex. In turn, \autoref{cor:valuesMult4Type4} fixes unique values of $(\overline{m}, \bar{n}, \bar{\du}, \bar{p}, \overline{p'})$ in terms of the quantity $\overline{\lambda''}$, where $\lambda'':=n/\du$. The latter is computed from the local equations at $P''$.

    The claim follows by~\autoref{lm:multivariateHensel}. To build the corresponding $9\times 9$ Jacobian matrix, we treat two cases, depending on the tangency type of $P''$, as we did in the proof of~\autoref{thm:liftingFormulasMult4}.  In both situations, the Jacobian becomes a block upper-triangular matrix. Each diagonal block records the data for the points $P$ and $P''$, respectively. The determinant of the block corresponding to $P''$  has non-vanishing expected initial form, as seen in \autoref{tab:initialFormsType2Horiz} and in the proof of~\autoref{pr:Prop5.2Corrected}. For the block corresponding to $P$, we consider the variables $(r,s,r',s', m, \du)$. A~\sage~computation reveals that its determinant has expected initial form $9437184\,\sqrt{3}\,\bar{a}\,\bar{c}^3\,\overline{\lambda''}^{10}/(\bar{s}\,\overline{s'})$. This quantity is non-vanishing, as we wanted to show.
      \end{proof}

  We end this section by discussing the tangency type (6b'). We keep the notation and assumptions from type (4b'). In particular, the remaining tangency point $P''$ between  $\Lambda$ and $\Gamma$ lies on the positive horizontal leg of $\Lambda$, so it fixes a unique value for $\overline{\lambda''}$, where $\lambda'':=n/\du$. Any lifting $(\ell,p,p')$ of a type (6b') tangency $P$ makes the local equations at $P$ 
  \begin{equation}\label{eq:localMult4_6bCross}
\begin{aligned}      \sextic_P &= \bar{a} + \bar{b}\,\bar{y} + \bar{c}\, \bar{x}\,\bar{y}^3, \;\quad
             {\ell_P}\!= \bar{y} + \overline{m} +\bar{n}\, \bar{x} + \bar{\du}\,\bar{x}\,\bar{y}
             \quad\text{ and}\\
                    {W_P}&= \det(\Jac(\sextic_P,\ell_P; \bar{x},\bar{y})) =
                    2\,\bar{c}\,\bar{\du}\,\bar{x}\,\bar{y}^3 + 3\,\bar{c}\,\bar{n}\,\bar{x}\,\bar{y}^2 - \bar{c}\,\bar{y}^3 + \bar{b}\,\bar{\du}\,\bar{y} + \bar{b}\,\bar{n}
\end{aligned}
  \end{equation}
vanish. Setting $\bar{b}=0$ in the above system recovers~\eqref{eq:localMult4_4bCross}. 

\begin{theorem}\label{thm:lifting6b_Cross}
  Assume that $\Gamma$ has a tropical hyperflex of type (6b'). If $\sextic$ is generic relative to $\Gamma$, then no triple $(\ell,p,p')$ lifting the pair $(\Lambda,P)$ can have an initial hyperflex. In addition, the local lifting multiplicity of $(\Lambda, P)$ equals one. Furthermore, the lifting multiplicity of $\Lambda$  agrees with that of $(\Lambda,P'')$, where $P''$ denotes the remaining tangency point between $\Lambda$ and $\Gamma$.\end{theorem}

\begin{proof} The result follows by direct computation, using the same techniques as those in the proofs of 
  {Theorems}~\ref{thm:mult4type4_4val} and~\ref{thm:lifting4bpInitHyperflex}.
  We start by computing the possible values of $(\overline{m}, \bar{n}, \bar{\du}, \bar{p}, \bar{p'})$. We express $\bar{x}$ as a rational function of $\bar{y}$ using the vanishing of $\sextic_P$. After replacing this expression in $\ell_P$ and $W_P$, algebraic manipulations produce a linear and a quadratic polynomial in $\bar{y}$ in the corresponding elimination ideal. The coefficients of the linear polynomial (called $A_0$ and $A_1$) are cubic in $\overline{m}$. We denote the quadratic polynomial by $g$.

  The ideal $\langle A_0, A_1\rangle$ contains a quadratic and a linear polynomial in $\overline{m}$, called $A_2$ and $A_3$, respectively. The vanishing of $A_3$ characterizes $\overline{m}$ as a rational function in $\bar{n}$ and $\bar{\du}$. Replacing this expression back in $A_0$, $A_1$ and $A_2$ produces three rational constraints between $\bar{n}$ and $\bar{\du}$. Their numerators, called $A_0'$, $A_1'$ and $A_2'$, must vanish. Each of them factors into four, five and four polynomials, respectively.   Three of these factors are common to all three numerators, namely:
  \[
  \begin{aligned}
    h_1&=  \bar{b}^5\bar{n}^2\bar{\du}^3 + 2\,\bar{a}\,\bar{b}^4\bar{n}\,\bar{\du}^4 + \bar{a}^2\bar{b}^3\bar{\du}^5 + 27\,\bar{b}^4\bar{c}\,\bar{n}^4 + 66\,\bar{a}\,\bar{b}^3\bar{c}\,\bar{n}^3\bar{\du} + 94\,\bar{a}^2\bar{b}^2\bar{c}\,\bar{n}^2\bar{\du}^2 + 66\,\bar{a}^3\bar{b}\,\bar{c}\,\bar{n}\,\bar{\du}^3 + 27\,\bar{a}^4\bar{c}\,\bar{\du}^4,\\
h_2&= \bar{b}^5\bar{n}^2\bar{\du}^3 - \bar{a}\,\bar{b}^4\bar{n}\,\bar{\du}^4 + \bar{a}^2\bar{b}^3\bar{\du}^5 + 27\,\bar{b}^4\bar{c}\,\bar{n}^4 + 18\,\bar{a}^2\bar{b}^2\bar{c}\,\bar{n}^2\bar{\du}^2 + 27\,\bar{a}^4\bar{c}\,\bar{\du}^4 - 432\,\bar{a}^3\bar{c}^2\bar{n}^3,\\
h_3& = \bar{b}^4\bar{n}^4 - 4\,\bar{a}\,\bar{b}^3\bar{n}^3\bar{\du} + 6\,\bar{a}^2\bar{b}^2\bar{n}^2\bar{\du}^2 - 4\,\bar{a}^3\bar{b}\,\bar{n}\,\bar{\du}^3 + \bar{a}^4\bar{\du}^4 + 64\,\bar{a}^3\bar{c}\,\bar{n}^3.
  \end{aligned}
  \]

The remaining factors are all distinct. We call them $h_{00}$, $h_{10}$, $h_{11}$ and $h_{20}$. Their total degrees in $\bar{n}$ and $\bar{\du}$ are 14, 14, 1 and 4, respectively. In particular, we have
\[
h_{11}:= \bar{b}\, \bar{n} + \bar{a}\,\bar{u}\; \text{ and } \; h_{20} := \bar{b}^4 \bar{n}^2 \bar{\du}^2 - \bar{a}\,\bar{b}^3 \bar{n}\,\bar{\du}^3 + \bar{a}^2 \bar{b}^2 \bar{\du}^4 + 3\,\bar{a}\,\bar{b}^2 \bar{c}\,\bar{n}^3 - 12\,\bar{a}^2 \bar{b}\,\bar{c}\,\bar{n}^2 \bar{\du} + 3\,\bar{a}^3 \bar{c}\,\bar{n}\,\bar{\du}^2.
\]
Since $P''$ determines $\overline{\lambda''}$ uniquely, where $\lambda''=n/\du$, and $\bar{b}$ does not feature in the local equations for $\lambda''$, we know that $h_{11}$ does not vanish. Thus we can remove it from consideration.

We re-express the  remaining polynomials in terms of $\bar{\du}$ and $\overline{\lambda''}$ and remove any monomial factors. Within the ring $\resK(\overline{\lambda''}^{\pm})[\bar{\du}^{\pm}]$, the new polynomials $h_1, h_3, h_{20}$  become linear in $\bar{\du}$, while $h_{2}$ is quadratic, and the degree of both $h_{00}$ and $h_{10}$ is four.

To prove our statement, we analyze two cases: either $h_1\,h_2\,h_3 = 0$, or  $h_{20}=h_{00}=h_{10}=0$. First, we assume either $h_1$, $h_3$ or $h_{20}$ vanish. Since these polynomials are  linear in $\bar{\du}$, each of them gives a unique (rational) solution $\bar{\du}$ in $\resK(\overline{\lambda''})$. For $h_{20}$ this choice of $\bar{\du}$ is not a root of $h_{00}$, so we can discard this option. Substituting  the values of $\overline{m}$ and $\bar{\du}$ obtained from $h_1$ and, respectively, $h_3$  back into $g$ yields a polynomial in $\bar{y}$. For  $h_1$, the polynomial $g$ becomes a Laurent monomial in $\bar{y}$, which never vanish in $\resK^*$.

In turn, for $h_3$, the polynomial $g$ has two distinct roots in $\bar{y}$. We obtain unique solutions for  $(\overline{m}, \bar{n}, \bar{\du})$ in terms of $\overline{\lambda''}$ and the lift $(\ell,p,p')$ has no initial hyperflex.  The precise values are
\begin{equation}\label{eq:values6b_Cross}
\begin{aligned}
&  (\overline{m}, \bar{n}, \bar{\du}) = \left (
-\frac{8\,\bar{a}\,\overline{\lambda''}(\bar{b}\,\overline{\lambda''} + \bar{a})}{(\bar{b}\,\overline{\lambda''} - \bar{a})^2},
-\frac{64\,\bar{a}^3\bar{c}\,\overline{\lambda''}^4}{(\bar{b}\,\overline{\lambda''} - \bar{a})^4}, -\frac{64\,\bar{a}^3\bar{c}\,\overline{\lambda''}^3}{(\bar{b}\,\overline{\lambda''} - \bar{a})^4}\right ) \,\text{ and }\, \\
&\bar{p},\bar{p'} \in \{\left (-\frac{(3\,\bar{b}^2\overline{\lambda''}^2 \pm 2\,\bar{b}\,\overline{\lambda''} \sqrt{\Delta} + \bar{a}^2)(\bar{b}\,\overline{\lambda''} - \bar{a})^4}{8\,\bar{a}^2\bar{c}^2\overline{\lambda''}^3 (\bar{b}\,\overline{\lambda''} + \bar{a} \pm \sqrt{\Delta})^3},\; 2\,\bar{a}\,\overline{\lambda''} \frac{\bar{b}\,\overline{\lambda''} + \bar{a} \pm \sqrt{\Delta}}{(\bar{b}\,\overline{\lambda''} - \bar{a})^2}\right)\},
\end{aligned}
\end{equation}
  where $\Delta = 3\,\bar{b}^2\overline{\lambda''}^2 - 2\,\bar{a}\,\bar{b}\,\overline{\lambda''} + 3\,\bar{a}^2$.

Finally, if the quadratic polynomial $h_2$ vanishes, we get two solutions in $\bar{\du}$, namely,
\[
\bar{\du} = -3\,\bar{c}\,\frac{9\,\bar{b}^4 \overline{\lambda''}^4 + 6\,\bar{a}^2 \bar{b}^2 \overline{\lambda''}^2 + 9\,\bar{a}^4  \pm \,(3\,\bar{b}^3 \overline{\lambda''}^3 + \bar{a}\,\bar{b}^2 \overline{\lambda''}^2 + \bar{a}^2 \bar{b}\,\overline{\lambda''} + 3\,\bar{a}^3)\sqrt{9\,\bar{b}^2 \overline{\lambda''}^2 - 6\,\bar{a}\,\bar{b}\,\overline{\lambda''} + 9\,\bar{a}^2}}{2\,\,\bar{b}^3(\bar{b}^2 \overline{\lambda''}^2 - \bar{a}\,\bar{b}\, \overline{\lambda''} + \bar{a}^2)}.
\]
Substituting this value and that of $\overline{m}$ back in $g$ gives a polynomial in $\bar{y}$ with a double root. Thus, these values would yield two lifts, each of them with initial hyperflexes.

Modification arguments, analogous to those used in~\autoref{thm:lifting4bpInitHyperflex} confirm that these lifts cannot exist, as we now explain. As usual, we write $p=(r,s)$ and $p'=(r',s')$. We require $\varepsilon_1, \varepsilon_2$ to have valuation $N\gg \max\{\val(r-r'), \val(s-s')\}$ and be generic. The main difference with the (4b') is the lack of explicit expressions for $\bar{p}, \overline{m}$ and $\bar{n}$, which will require some extra algebraic computations to determine expected initial forms of the local system at $P_1$ post modification. 

A \sage\, calculation confirms that $\tilde{a}_{00}$ is a non-constant polynomial in $\varepsilon_1$ and $\varepsilon_2$.
The coefficients $\tilde{a}_{10},\tilde{a}_{01}$ from $\tilde{\sextic}(x',y')$ have expected initial forms
\[\overline{\tilde{a}_{10}}=\frac{\partial \sextic_P}{\partial x}(\bar{p}) \quad \text{ and }\quad \overline{\tilde{a}_{01}}=\frac{\partial \sextic_P}{\partial y}(\bar{p}).
\]
Both expressions are non-zero by ~\autoref{lm:expectedModification6b_Cross} below. Thus, $\tilde{a}_{10}$ and $\tilde{a}_{01}$ have valuation zero.
Combining the above identities with the vanishing of the Wronskian $W_P$ at $\bar{p}$ yields
\[ \overline{\tilde{a}_{10}}(1+\bar{\du}\,\bar{r}) - \overline{\tilde{a}_{01}} (\bar{n}+\bar{\du}\,\bar{s}) = 0.\]

The coefficients $m'$ and $n'$ from $\tilde{\ell}$ satisfy $\val(m')\geq N$, $\val(n')=0$, with  expected initial forms
\[\overline{m'} = \overline{\varepsilon_1}(\bar{n} + \bar{\du}\,\bar{s}) + \overline{\varepsilon_2} (1+\bar{\du}\,\bar{r}) =\frac{(1+\bar{\du}\,\bar{r})}{\overline{\tilde{a}_{01}}}(\overline{\varepsilon_1} \,\overline{\tilde{a}_{10}} + \overline{\varepsilon_2}\, \overline{\tilde{a}_{01}})
\quad  \text{ and } \quad \bar{n'} = \bar{n} + \bar{\du}\,\bar{s} = \frac{
  (1+\bar{\du}\,\bar{r})\overline{\tilde{a}_{10}}}{\overline{\tilde{a}_{01}}}.
\]
Our genericity conditions on $\varepsilon_1,\varepsilon_2$ become $\tilde{a}_{00}\neq 0$ and $\overline{m'} \neq 0$. The latter can be achieved by~\autoref{lm:expectedModification6b_Cross}.
These restrictions ensure that the Newton subdivision of $\tilde{\sextic}$ matches the one seen in~\autoref{fig:4bprime}, and that the tangency point $P_1 = \Trop (-\varepsilon_1, -\varepsilon_2)$ is the lower vertex of the curve $\Lambda':=\Trop \,V(\tilde{\ell})$. As was argued in the proof of \autoref{thm:lifting4bpInitHyperflex}, by performing a second modification if needed, we may assume that $P_1$ lies in the interior of the slope one edge of $\Gamma'=\Trop\, V(\tilde{\sextic})$ adjacent to $v'$. In this situation, we have
\[
\tilde{\sextic}_{P_1}(\overline{p_1}) = -(\overline{\tilde{a}_{10}}\,\overline{\varepsilon_1} + \overline{\tilde{a}_{01}}\,\overline{\varepsilon_2}) = -\overline{m'}\frac{\overline{\tilde{a}_{01}}}{(1+\bar{\du}\,\bar{r})}
\]
which does not vanish by the genericity of $\varepsilon_1, \varepsilon_2$. Thus,  $(\ell,p,p')$ has no initial hyperflex.

Once the lack of initial hyperflexes is established, the usual Jacobian computation of the $6\times 6$ local system imposed by $P$ at $\bar{p}$ and $\bar{p'}$ in the variables $(r,s,r',s', m, u)$ confirms that the local lifting multiplicity of $(\ell,p,p')$ is one. Indeed,  the expected initial of the determinant is
$-\sqrt{\Delta}^5\,\bar{n}^4(\bar{b}\,\overline{\lambda''} - \bar{a})/(16\,\bar{s}\,\bar{s'}\,\bar{a}\,\bar{c}\,\overline{\lambda''}^6)$ and this expression is non-zero by~\eqref{eq:values6b_Cross}. 

By construction, the remaining tangency point $P''$ is necessarily of type (1a), (2a) or (3c) and lies on the positive horizontal leg of $\Lambda$. For types (2a) and (3c), the  corresponding $9\times 9$ Jacobian matrix is block upper-triangular. Thus, the total lifting multiplicity of $(\Lambda, P,P'')$ and $(\Lambda, P'')$ agree, as we wanted to show.
\end{proof}

\begin{lemma}\label{lm:expectedModification6b_Cross}
  Fix a tangency point $P$ of type (6b') between $\Lambda$ and $\Gamma$ and let  $(\ell,p,p')$ be a lift of $(\Lambda, P)$ that  has an initial hyperflex. If $\sextic$ is generic relative to $\Gamma$, and $\sextic_P$ is as in~\eqref{eq:localMult4_6bCross}, then
  \[(1+\bar{\du}\,\bar{x})\frac{\partial \sextic_P}{\partial \bar{x}}(\bar{p}) \,\frac{\partial \sextic_P}{\partial \bar{y}}(\bar{p})\neq 0.\]
\end{lemma}
\begin{proof}  The result follows by a direct computation using~\sage. We write $p = (x,y)$ to match the convention from~\eqref{eq:localMult4_6bCross} but otherwise keep  the notation used in the proof of~\autoref{thm:lifting6b_Cross}. Since $p$ is an initial hyperflex, we know that the polynomial $h_2$ must vanish. We will see that this fact is incompatible with the vanishing of $(1+\bar{\du}\,\bar{x})$ and either of the partials of $\sextic_P$ at $\bar{p}$. 

  The claim for the $\bar{x}$-partial is immediate, since it is a monomial in $\bar{a}\bar{y}$. Next, we discuss the non-vanishing of the $\bar{y}$-partial at $\bar{p}$. We argue by contradiction and assume 
  $\frac{\partial \sextic_P}{\partial \bar{y}}(\bar{p})=
  0$. Adding this identity to the system defined by~\eqref{eq:localMult4_6bCross} determines a unique value for $\bar{\du},\bar{p}$, and linear constraint between $\overline{m}$ and $\bar{n}$, namely:
  \begin{equation*}\label{eq:initials6b_Cross}
    (\bar{\du},\bar{p})= (\frac{27\,\bar{a}^2\,\bar{c}}{4\,\bar{b}^3}, (-\frac{4\,\bar{b}^3}{27\,\bar{c}\,\bar{a}^2}, -\frac{3\,\bar{a}}{2\,\bar{b}}))\quad  \text{ and }\quad 27\,\bar{a}^2\,\bar{c}\,\overline{m} - 4\,\bar{b}^3\,\bar{n} = 0.
  \end{equation*}
  These conditions are incompatible with the vanishing of the polynomial $h_2$ required for having an initial hyperflex if $\sextic$ is generic. Indeed, solving for $\overline{m}$, setting $\bar{n}=\bar{\du}\,\overline{\lambda''}$ and replacing the above values in the polynomial $h_2$ yields the following constraint:
  \[
  (27\,\bar{b}^2\,\overline{\lambda''}^2 + 17\,\bar{a}\,\bar{b}\,\overline{\lambda''} + 15\,\bar{a}^2)\,(2\,\bar{b}\,\overline{\lambda''} - 3\,\bar{a})^2 = 0
\]
Since the coefficient $\bar{b}$ plays no role in determining $\overline{\lambda''}$ from the local equations at the remaining tangency point $P''$ between $\Lambda$ and $\Gamma$, this expression cannot vanish by the genericity of $\sextic$.

To conclude, we must show that $(1+\bar{\du}\,\bar{x})\neq 0$. We argue by contradiction, setting $\bar{x}=-1/\bar{\du}$. Replacing this value in the local system at $P$, we obtain a  quadratic polynomial in the ideal $(f_P(\bar{y}), W_P(\bar{y}))\in \resK[\bar{\du}^{\pm}, \bar{n}^{\pm}]$ with discriminant $(2\,\bar{b}\,\bar{n} - 3\bar{a}\bar{\du})^2$. Its vanishing, required since we have an initial hyperflex, is incompatible with the genericity conditions on $\sextic$ because $\overline{\lambda''}=\bar{n}/\bar{\du}$ is independent of $\bar{b}$.
\end{proof}

\section{Tangencies of multiplicity six}\label{sec:tang-mult-six}

In this section, we discuss  local tangencies of multiplicity six between $\Lambda$ and $\Gamma$. By~\autoref{thm:classificationRealizableLocalTangencies}, such tangencies occur only when $\Lambda$ is trivalent. As in previous sections, we assume $\Lambda$ has a slope one edge. In addition, for each tangency type, we pick those representatives of symmetry classes depicted in~\autoref{fig:classificationLocalTangencies}.

Our first result confirms that type (7) tangencies do not lift to classical tritangents:

\begin{proposition}\label{pr:7a}
Let $\Lambda$ be a  tritangent to $\Gamma$, with a local tangency of type (7). Then, $\Lambda$ does not lift to a classical tritangent to $V(\sextic)$ defined over $\overline{\K}$.  
\end{proposition}

\begin{proof} We let $P$ and $P'$ be the top and bottom vertices of $\Lambda$, respectively. By construction, they are tangency points. The remaining tangency (called $P''$) is the midpoint between  $P$ and $P'$. The cells $P^{\vee}$ and $(P')^{\vee}$ in partial Newton subdivision of $\sextic$ are determined by \autoref{lm:starsTopVertexMostType7a}. The vertices of $P^{\vee}$ are $(0,3)$, $(1,2)$ and $(3,1)$, whereas those of $(P')^{\vee}$ are $(0,3)$, $(1,2)$ and $(2,0)$.
  
By construction, the local equations at $P''$ determine the value of $\bar{n}$, namely $\bar{n}= \overline{a_{12}}/\overline{a_{03}}$. 
  We prove the statement by showing that the local equations at $P$ given by the vanishing of 
  \begin{equation}\label{eq:7a}
    \sextic_P = \bar{y}\,(\overline{a_{03}}\,\bar{y}^2 +\overline{a_{12}}\,\bar{x}\,\bar{y} + \overline{a_{31}}\,\bar{x}^3),\quad \ell_P = \bar{y} + \bar{n}\,\bar{x} + \bar{\du}\,\bar{x}\,\bar{y}\quad \text{ and } \quad W_P =  \det(\Jac(\sextic_P,\ell_P; \bar{x},\bar{y}))
  \end{equation}
  have no solution $(\bar{\du}, \bar{x},\bar{y})\in (\resK^*)^3$. Indeed, a simple manipulation using~\sage~ after substituting the value for $\bar{n}$ obtained earlier confirms that
  \[ W_P  - (3\,\overline{a_{12}}\,\bar{x}-\overline{a_{03}}\,\bar{y}) \sextic_P  +  2\,(\overline{a_{12}}\,\bar{x} + \overline{a_{03}}\,\bar{y})^2\bar{y} \in \langle \ell_P\rangle\resK[\bar{\du}^{\pm}, \bar{x}^{\pm}, \bar{y}^{\pm}].
  \]
  In particular,  any solution to~\eqref{eq:7a} satisfies $\bar{y} = -\overline{a_{12}}\,\bar{x}/\overline{a_{03}}$. However, substituting this value back in $\ell_P$ produces the Laurent monomial  $-(\overline{a_{12}}/\overline{a_{03}})\bar{\du}\,\bar{x}^2$, which does not vanish in $(\resK^*)^3$.
  \end{proof}

In what follows we discuss type (8) tangencies, also referred to as tree-shape ones.  Our objective is to determine its lifting multiplicity and the field of definition for each lift. Here is the precise statement:

\begin{theorem}\label{thm:treeShapeIntersections8} Assume that $\Lambda$ has a tangency of type (8) with $\Gamma$. Let $P$, $P'$ and $P''$ be the corresponding tangency points. Suppose that the edge lengths of $\Gamma\cap\Lambda$ and generic in the sense of~\autoref{rm:genericityOfGamma} (ii). Then, $(\Lambda, P, P', P'')$ has exactly eight lifts to classical tritangents $(\ell, p, p', p'')$ to $V(\sextic)$. Furthermore, all such quadruples are defined over the same quadratic extension of $\K$. In particular, all lifts of $(\Lambda, P, P', P'')$ are defined over $\K$ if, and only if, one of them is.  
  \end{theorem}

\noindent 
The rest of the section is devoted to proving this result. A series of lemmas and propositions discussed in 
{Subsections}~\ref{ssec:build-an-algebr} and~\ref{ssec:build-trop-curve-on-modification} below will simplify the exposition. As in~\autoref{pr:3aaD}, 
we use a tropical modification of $\RR^2$ along the curve $\Lambda$.

We follow the notation of~\autoref{fig:chipFiringTreeShapeIntersections} and let $L_1, \ldots, L_5$ be the lengths of the five edges of the graph $\Gamma \cap \Lambda$. Exploiting symmetries, we may assume $L_1<L_2, L_4, L_5$. We suppose all five edge lengths are generic subject to this condition. In particular, we require the minimum among the  expressions $L_3$, $L_4 -L_1$ and $L_5-L_1$ to  be unique. 
.

\autoref{pr:TreeShapeIntersections} confirms there are three possibilities to distribute the three tangency points on the intersection $\Gamma \cap \Lambda$, depending on the expression realizing $\min\{L_4-L_1, L_3,  L_5-L_1\}$. We label the corresponding cases (I) through (III) as in the bottom of~\autoref{fig:chipFiringTreeShapeIntersections}. We place the lower vertex of $\Lambda$ at $(0,0)$. The location of the three tangency points is fixed for each case. Labeling them by $P, P'$ and $P''$ when reading from left to right, we have $P=(\frac{L_1-L_2}{2}, 0)$ for all three cases, whereas

\begin{minipage}[l]{0.5\textwidth}
\begin{equation}\label{eq:tangenciesTreeShape2D}
 P' = \begin{cases}
  (\frac{L_3+L_4-L_1}{2}, \frac{L_3+L_4-L_1}{2}) & \text{for (I)},\\
(L_3, \frac{L_3+L_4-L_1}{2}) & \text{for (II)},\\
  (\frac{L_3+L_5-L_1}{2}, \frac{L_3+L_5-L_1}{2}) & \text{for (III)},\\
  \end{cases}
\end{equation}
\end{minipage}
\begin{minipage}[l]{0.45\textwidth}
 \[\text{and }\quad P'' =
  \begin{cases}
  (L_3+\frac{L_5-L_4}{2},L_3) & \text{for (I)},\\
  (\frac{L_3+L_5-L_1}{2}, L_3) & \text{for (II)},\\
  (L_3, L_3+\frac{L_4-L_5}{2})  & \text{for (III)}.\\
  \end{cases}\]
\end{minipage}
\vspace{1ex}

This information can be used to determine the initial forms of the parameters of any lifting of $(\Lambda, P, P', P'')$. This is the content of our first lemma:

\begin{lemma}\label{lm:initialFormsTreeShape} Assume that $(\Lambda, P, P', P'')$ has a type (8) tangency with a  classical tritangent lift $(\ell,p,p',p'')$. Then, the initial forms of the coefficients $m, n$, $\du$ and the points $p$, $p'$ and $p''$ are unique and lie in $\resK$. In particular, if the lengths of $\Gamma\cap \Lambda$ satisfy
  $L_1<L_2, L_4, L_5$, we  have
  \[\overline{m} = \overline{a_{11}}/\overline{a_{12}},\quad \bar{n} = \overline{a_{21}}/\overline{a_{12}}\quad \text{ and }\quad \bar{\du} = \overline{a_{22}}/\overline{a_{12}}.\]
  \end{lemma}
\begin{proof} It suffices to prove the uniqueness when  $L_1<L_2,L_4, L_5$. In this situation, there are three possibilities for the combined locations of $P$, $P'$ and $P''$, as indicated in \autoref{pr:TreeShapeIntersections}. Solving the local equations for each of the three cases  determines unique solutions for  the initial forms of $p, p', p''$ and  the triples $(\overline{m},\bar{n}, \bar{n}/\bar{\du})$, $(\overline{m},\bar{\du}, \bar{n}/\bar{\du})$ or $(\overline{m},\bar{n}, \bar{\du})$   for cases (I), (II) and (III), respectively. The values of $(\overline{m}, \bar{n}, \bar{\du})$ are determined from here and they agree with those claimed in the statement.
\end{proof}

\begin{figure}
  \includegraphics[scale=0.12]{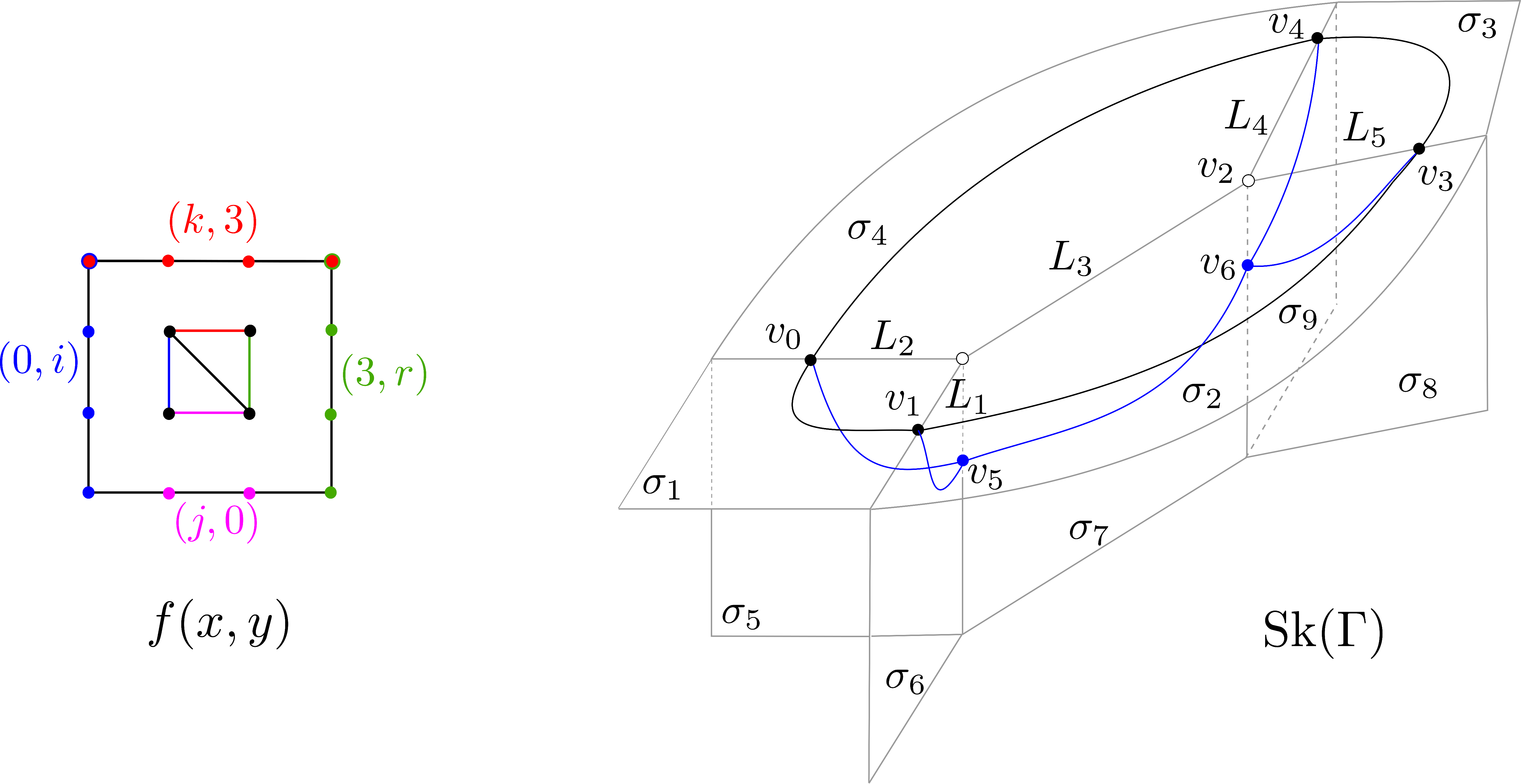}
  \caption{Partial Newton subdivision of $\sextic(x,y)$ and skeleton of the tropical curve $\Trop\,V(I)$. The portions of the curve on the four new cells $\sigma_5, \ldots,\sigma_9$ obtained by modifying $\RR^2$ along $\Lambda$ appear in blue.\label{fig:TreeShapedModifiedSkeleton}}
  \end{figure}

Since the lower vertex of  $\Lambda$ is located at $(0,0)$, we may assume that $\sextic \in R[x,y]\smallsetminus \mathfrak{M}[x,y]$. Furthermore, in view of~\autoref{lm:initialFormsTreeShape} we also have
\begin{equation}\label{eq:FixValuationsanda12Tobe1}
 \val(a_{11})=\val(a_{21})=\val(a_{12})= \val(m)=\val(n) = 0,  \quad \val(\du)=L_3 \quad \text{ and } \quad a_{12}=1.
\end{equation}
Partial information on the Newton subdivision of $\sextic(x,y)$ is recorded on the left of~\autoref{fig:TreeShapedModifiedSkeleton}. In particular, we know that each of the edges of the square with vertices $(1,1)$, $(1,2)$, $(2,2)$ and $(2,1)$ forms a  triangle with a vertex on the  corresponding parallel boundary edge of the Newton polytope of $\sextic$.  We let $(0,i)$, $(j,0)$, $(3,r)$ and $(3,k)$ be these vertices, as indicated in the figure. These triangles are dual cells to the vertices $p_1$, $p_2$, $p_4$ and $p_5$ of the skeleton of $\Gamma$ seen on the top-left of~\autoref{fig:chipFiringTreeShapeIntersections}.

  The edge lengths of $\Gamma \cap \Lambda$ featuring in these two figures impose restrictions on the valuations of the remaining points in the Newton polytope of  $\sextic(x,y)$. Indeed, we have $\val(a_{22})=L_3$, whereas
  \begin{equation}\label{eq:valInequalitiesTreeShape}  \val(a_{0p}) \geq L_2,\; \val(a_{p0}) \geq L_1, \; \val(a_{3p})\geq L_5+pL_3 \text{ and } \val(a_{p3})\geq L_4+pL_3 \;\text{ for }     p=0,\ldots,3.
  \end{equation}
  In addition, equalities for each expression are attained for a single choice of  $p$, namely $p=i,j,r$ and $k$, respectively.   Furthermore, since $L_1<L_2$ and $L_1<L_5$ we know that $j$ is either $1$ or $2$.

\smallskip

\begin{proof}[Proof of~\autoref{thm:treeShapeIntersections8}]  We proceed in two steps, under the assumptions of~\eqref{eq:FixValuationsanda12Tobe1}. First, we perform a modification of $\RR^2$ along $\Lambda$ to reveal the tropical tangencies as proper intersections between the modifications of both $\Gamma$ and $\Lambda$. In order to do so, we must pick a suitable algebraic lift of $\Lambda$. 
Inspired by~\autoref{lm:initialFormsTreeShape} we pick the lift
\begin{equation}\label{eq:idealsModificationTree}
h(x,y):=  y+(a_{11} + m_1) +(a_{21}+ n_1)x+(a_{22} + \du_1)x\,y,
\end{equation}
where $m_1,n_1, \du_1$ satisfy the conditions $\val(m_1),\val(n_1)>0$ and $\val(\du_1)>L_3$. \autoref{ssec:build-an-algebr} discusses how to choose $m_1,n_1, \du_1\in \K$ for each possible distribution of tangency points on $\Lambda$.

Once $h$ is determined, we re-embed the curves $V(\sextic)$ and $V(\ell)$ via the ideals $I:=\langle\sextic, z-h \rangle$  and $I_{\ell}:=\langle \ell, z-h\rangle$ of $\K[x,y,z]$, respectively.
The projections of $V(I)$ to the $xz$- and $yz$-planes are given, respectively, by the vanishing of the following two bivariate polynomials:
\begin{equation}\label{eq:modificationxz_zy}
  \begin{aligned}
\tilde{\sextic}(x,z) &:= f(x, \frac{z-(a_{11} + m_1)-(a_{21}+ n_1)x}{1+(a_{22} + \du_1)x})\,((1+(a_{22} + \du_1)x)^3 = \sum_{i,j} \tilde{a}_{ij} x^i z^j,\\
\hat{\sextic}(z,y) &:= f(\frac{z-(a_{11} + m_1)-y}{(a_{21}+ n_1)+(a_{22} + \du_1)y},y)\,((a_{21}+ n_1)+(a_{22} + \du_1)y)^3 = \sum_{i,j} \hat{a}_{ij} z^i y^j.
  \end{aligned}
\end{equation}
Precise formulas for all relevant coefficients $\tilde{a}_{ij}$ and $\hat{a}_{ij}$ are recorded in~\autoref{sec:CoefficientsModificationSextics}.
The projections of both $\Trop\,V(I)$ and $\Trop\,V(I_{\ell})$ to the $xz$- and $zy$-planes will be denoted as
\begin{equation}\label{eq:tropicalCurvesModifiedTreeShape}
  \tilde{\Gamma}:=\Trop\,V(\tilde{\sextic}(x,z)),\; \hat{\Gamma}:=\Trop\,V(\hat{\sextic}(z,y)), \; \tilde{\Lambda}:=\Trop\,V(\tilde{\ell}(x,z)) \text{ and } \hat{\Lambda}:=\Trop\,V(\hat{\ell}(z,y)).
\end{equation}

As~\autoref{thm:ChoiceOfModificationLift} explains, the choice of $m_1,n_1, \du_1\in \K$ predetermines partial Newton subdivisions of the polynomials in~\eqref{eq:modificationxz_zy} and fixes the combinatorics of the planar tropical curves $\tilde{\Gamma}$ and $\hat{\Gamma}$. Once these parameters are fixed,  we write
\begin{equation}\label{eq:parametersLineTreeShape}
  m:= a_{11} + m_1 + m_2, \quad n:= a_{21} + n_1 + n_2 \quad \text{ and} \quad \du = a_{22} + \du_1 + \du_2, 
\end{equation}
where the parameters $m_2, n_2, \du_2\in \overline{\K}$ 
satisfy $\val(m_2), \val(n_2)>0$ and $\val(\du_2)>L_3$.
Their choice ensures that $P, P'$ and $P''$ are tangency points between $\Trop\, V(I)$ and $\Trop \, V(I_{\ell})$. Their construction is the subject of~\autoref{ssec:build-trop-curve-on-modification}.

\autoref{thm:LocalEquationsAfterModification} confirms that there are precisely eight valid choices of  parameters $m_2, n_2, \du_2$, all defined over the same  quadratic extension of $\K$. The lifts  $p,p',p''\in V(I)$ are uniquely determined by each triple of parameters. The same result confirms that either all lifts are defined over $\K$ or none of them is. 
\end{proof}

\subsection{How to pick the parameters $m_1, n_1$ and $\du_1$}\label{ssec:build-an-algebr}

As we saw in prior sections when discussing other tangency types, the choice of parameters $m_1, n_1, \du_1$ is determined by a need to predict the local equations for the tangency points $P$, $P'$ and $P''$ after the tropical modification is performed. In turn, these systems are determined by the Newton subdivisions of the polynomials $\tilde{\sextic}(x,z)$ and $\hat{\sextic}(z,y)$, or, equivalently, by the skeleton of the tropical curve $\Trop\, V(I)$.

Our next result confirms that we can always pick  parameters $m_1, n_1$ and $\du_1$ to guarantee that the graph depicted on the right of~\autoref{fig:TreeShapedModifiedSkeleton} is the skeleton of $\Trop \,V(I)$:
\begin{theorem}\label{thm:ChoiceOfModificationLift}
  Assume that $L_1<L_2,L_4,L_5$ and $\sextic$ is generic relative to $\Gamma$. Then, there exists parameters $m_1, n_1, \du_1\in \K$ with $\val(m_1), \val(n_1)>0$ and $\val(\du_1)>L_3$ ensuring that the marked Newton subdivisions of $\tilde{\sextic}(x,z)$ and $\hat{\sextic}(z,y)$ are as seen in~\autoref{fig:ModificationsThreeShape}.  
\end{theorem}

\begin{figure}
  \includegraphics[scale=0.2]{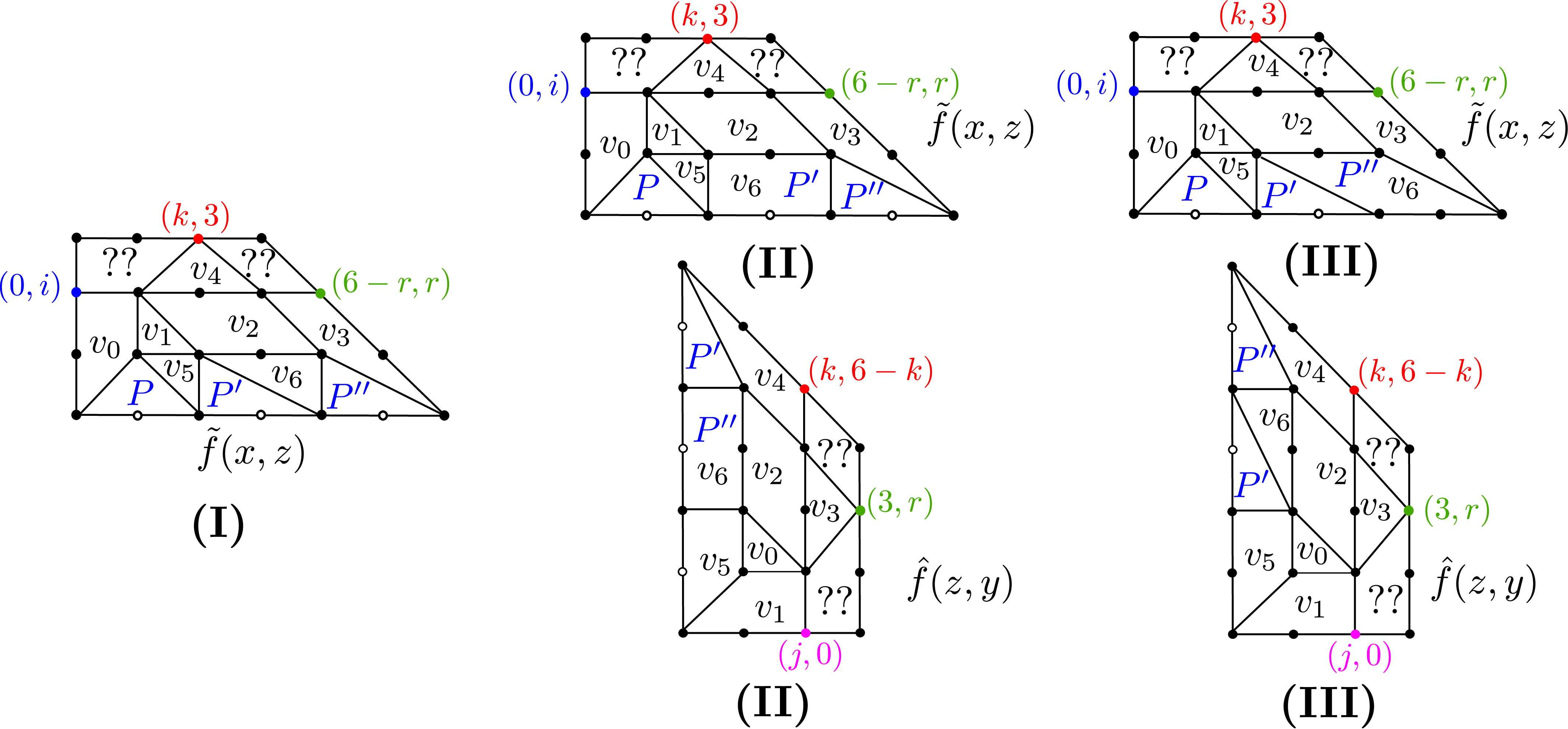}
  \caption{Partial Newton subdivisions of the projections $\tilde{\sextic}(x,z)$ and $\hat{\sextic}(z,y)$ for each of the three possible locations of tritangent points $P$, $P'$ and $P''$ corresponding to cases (I), (II) and (III) from~\autoref{fig:chipFiringTreeShapeIntersections}. Labels on 2-dimensional cells indicate their dual vertices in the tropical curves $\tilde{\Gamma}$ and $\hat{\Gamma}$ seen in~
    {Figures}~\ref{fig:Case1TreeShape} and~\ref{fig:Cases23TreeShape}. Markings ?? indicate possible  subdivisions of cells. Unfilled dots denote unmarked vertices. 
    \label{fig:ModificationsThreeShape}}
\end{figure}

\begin{proof}
  The methods used for finding such parameters will be similar for all three situations, although slight variations will be required for each particular case. Case (I) will be simpler to discuss, since the points $P, P', P''$ will be located on the open charts $\sigma_5$, $\sigma_7$ and $\sigma_8$, and these charts can be recovered from the Newton subdivision of $\tilde{\sextic}(x,z)$. For the other two cases, one of the tangency points will lie on $\sigma_9$, so we will need to use $\hat{\sextic}(x,z)$ as well.

            {Lemmas}~\ref{lm:coefficientsProjectionXZ} and~\ref{lm:coefficientsProjectionZY} below ensure that the cells $v_0^{\vee}$ through $v_4^{\vee}$ in the Newton subdivisions of $\tilde{\sextic}(x,z)$ and $\hat{\sextic}(z,y)$ are as indicated in~\autoref{fig:ModificationsThreeShape} without imposing any further restrictions on $m_1, n_1$ or $\du_1$. The corresponding  vertices $\RR^3$ can be recovered from~\eqref{eq:p1Top6} and the function $\trop \,h$, since the latter does not depend on these parameters. More precisely, we have
  \begin{equation}\label{eq:v0throughv4}
    \begin{aligned}
      v_0& =(-L_2,0,0), \quad v_1=(0,-L_1,0), \quad v_2=(L_3,L_3,L_3),\\
      v_3& =(L_3+L_5, L_3, L_3+L_5) \quad \text{and} \quad v_4=(L_3, L_4+L_5, L_4+L_5).
    \end{aligned}
  \end{equation}

  In what follows, we determine the remaining cells of both Newton subdivisions, on a case-by-case basis. 
  {Figures}~\ref{fig:Case1TreeShape} and~\ref{fig:Cases23TreeShape} show the corresponding dual tropical curves. The points $P, P', P'', v_5, v_6$ are vertices of $\Trop V(I)$. They appear in charts $\sigma_5$ through $\sigma_9$ of the modification of $\RR^2$ along $\Lambda$. Their coordinates are listed in~\autoref{tab:TreeShapeVerticesNew3D} and can be obtained by simple linear algebra computations from these pictures starting from~\eqref{eq:tangenciesTreeShape2D} and~\eqref{eq:v0throughv4}. Notice that the values of $P$ and $v_5$ agree for all three cases, whereas the those of the remaining points vary.

    In turn, the values of these five points imposes necessary and sufficient restrictions on the valuations of the coefficients $\tilde{a}_{0p}$ and $\hat{a}_{0p}$ for $0\leq p\leq 6$ that are needed to obtained the desired marked Newton subdivisions of the two polynomials $\tilde{\sextic}(x,z)$ and $\hat{\sextic}(z,y)$. We listed them on~\autoref{tab:SufficientValuations}, together with sufficient lower bounds for some valuations obtained by further exploiting the edge length restrictions defining each case.

    The coefficients $m_1, n_1, \du_1$ will be chosen precisely  to ensure that these sufficient conditions hold. In order to do so, we look at the formulas from all coefficients listed in~\autoref{sec:CoefficientsModificationSextics}, and collect all terms appearing on each coefficient whose valuations has potentially a lower value than the one required for the given coefficient.  We will denote the resulting polynomials by $\tilde{b}_{ij}$ and $\hat{b}_{ij}$, to match the notation for the original coefficients. The coefficients $m_1, n_1, \du_1$ will be picked as zeroes of a suitable summand of these polynomials. The  precise choice will be done on a case-by-case basis.

      \begin{table}[tb]
  \begin{tabular}{|c||c|c|c|}
    \hline Case & (I) & (II) & (III) \\
    \hline \hline
    $P$ & $(\frac{L_1-L_2}{2}, 0,-\frac{L_1+L_2}{2})$ & $(\frac{L_1-L_2}{2}, 0,-\frac{L_1+L_2}{2})$ & $(\frac{L_1-L_2}{2}, 0,-\frac{L_1+L_2}{2})$ \\
    $P'$ & $(\frac{L_4+L_3-L_1}{2}, \frac{L_4+L_3-L_1}{2}, -L_1)$ & $(L_3, \frac{L_4+L_3-L_1}{2}, -L_1)$ & $(\frac{L_3+L_5-L_1}{2}, \frac{L_3+L_5-L_1}{2},-L_1)$ \\
    $P''$ & $(L_3+\frac{L_5-L_4}{2},L_3, L_3-L_4)$ & $(\frac{L_3+L_5-L_1}{2},L_3, -L_1)$ & $(L_3, L_3+\frac{L_4-L_5}{2}, L_3-L_5)$ \\
\hline    $v_5$ & $(0, 0, -L_1)$ & $(0,0,-L_1)$ & $(0, 0, -L_1)$ \\
    $v_6$ &  $(L_3, L_3, L_3-L_4)$ & $(L_3, L_3, -L_1)$ & $(L_3, L_3, L_3-L_5)$ \\
    \hline 
  \end{tabular}
  \caption{Formulas for  new vertices in $\Trop\,V(I)$. 
    \label{tab:TreeShapeVerticesNew3D}}
\end{table}

      \begin{table}[tb]
  \begin{tabular}{|c||c|c|c||c|c|}
    \hline Case & $\val(\tilde{a}_{10})$ & $\val(\tilde{a}_{30})$ & $\val(\tilde{a}_{50})$ & $\val(\tilde{a}_{20})$ & $\val(\tilde{a}_{40})$  \\
    \hline \hline
    \multirow{2}{*}{(I)} & $>(L_1+L_2)/2$ & {$>(L_3+L_4+L_1)/2$} & {$>(4L_3+L_4+L_5)/2$} & \multirow{2}{*}{$L_1$} & \multirow{2}{*}{$L_3+L_4$} \\
    & ($\geq L_2$) & ($\geq L_1+L_3$) & ($\geq 2L_3+L_5$) & &  \\

    \hline    
    \multirow{2}{*}{(II)} & $>(L_1+L_2)/2$ & \multirow{2}{*}{$\geq L_1+L_3$} & {$>(5L_3+L_5+L_1)/2$} & \multirow{2}{*}{$L_1$} & \multirow{2}{*}{$2L_3+L_1$} \\
    & ($\geq L_2$) & & ($\geq 2L_3+L_5$) & &  \\

    \hline
        \multirow{2}{*}{(III)} & $>(L_1+L_2)/2$ & $>(L_1+L_3+L_5)/2$ & \multirow{2}{*}{$\geq 2L_3+L_5$} & \multirow{2}{*}{$L_1$} & \multirow{2}{*}{$L_3+L_5$} \\
        & ($\geq L_2$) & ($\geq L_1+L_3$) & & &  \\
    \hline 
  \end{tabular}
  \vspace{1ex}
  
  \begin{tabular}{|c||c|c|c||c|c|}
    \hline Case & $\val(\hat{a}_{01})$ & $\val(\hat{a}_{03})$ & $\val(\hat{a}_{05})$ & $\val(\hat{a}_{02})$ & $\val(\hat{a}_{04})$  \\
    \hline \hline
    \multirow{2}{*}{(II)} & \multirow{2}{*}{$\geq L_1$} & \multirow{2}{*}{$\geq L_1+L_3$} & {$>(5L_3+L_4+L_1)/2$} & \multirow{2}{*}{$L_1$} & \multirow{2}{*}{$2L_3+L_1$} \\
    & &   & ($\geq 2L_3+L_4$) & &  \\
    \hline 
    \multirow{2}{*}{(III)} & \multirow{2}{*}{$\geq L_1$} & {$>(L_1+L_3+L_5)/2$} & {$>(4L_3+L_4+L_5)/2$} & \multirow{2}{*}{$L_1$} & \multirow{2}{*}{$L_3+L_5$} \\
    & & ($\geq L_1+L_3$)  & ($\geq 2L_3+L_4$) & &  \\
    \hline 
  \end{tabular}
  \caption{Valuation requirements on the coefficients of $\tilde{\sextic}(x,z)$ and $\hat{\sextic}(z,y)$ ensuring the Newton subdivisions from \autoref{fig:ModificationsThreeShape}. Sufficient conditions appear in between parentheses. These values complement the information provided in~\autoref{lm:coefficientsProjectionXZ}.    \label{tab:SufficientValuations}}
      \end{table}

We start with the coefficient $\tilde{a}_{10}$, since the restriction on its valuation is common to all cases. A simple inspection shows that
\begin{equation}\label{eq:b10xz}\tilde{b}_{10}:= (-a_{11}^3a_{13} + a_{10}) + (-3\,a_{11}^2a_{13} + a_{11}) m_1 + (-3\,a_{11}\,a_{13} + 1)m_1^2 - a_{13}\,m_1^3.
\end{equation}

\begin{figure}
  \includegraphics[scale=0.12]{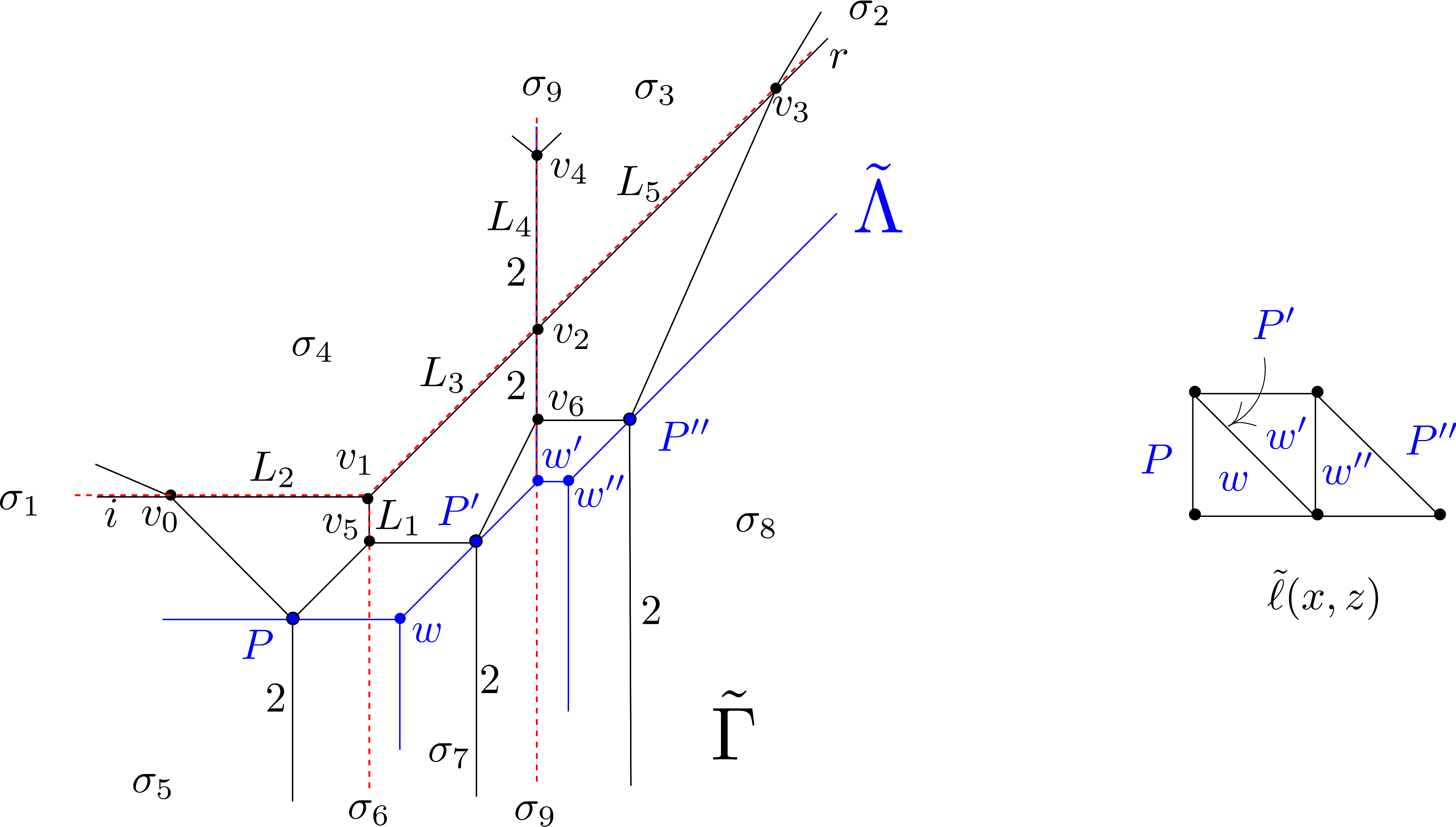}
  \caption{The tropical curves $\tilde{\Gamma}$ and $\tilde{\Lambda}$ from~\eqref{eq:tropicalCurvesModifiedTreeShape}, and Newton subdivision of $\tilde{\ell}(x,z)$ (with labeled dual cells) for case (I).     The vertex $w$ of $\tilde{\Lambda}$ can be located either on $\sigma_5$ on $\sigma_7$, depending on the sign of $L_4+L_3-L_2$. The symbols $\sigma_1, \ldots, \sigma_9$ indicate the charts on the modification of $\RR^2$ under $\Lambda$, following the notation of~\autoref{fig:ModificationsThreeShape}.\label{fig:Case1TreeShape}}
\end{figure}
\begin{figure}
  \includegraphics[scale=0.12]{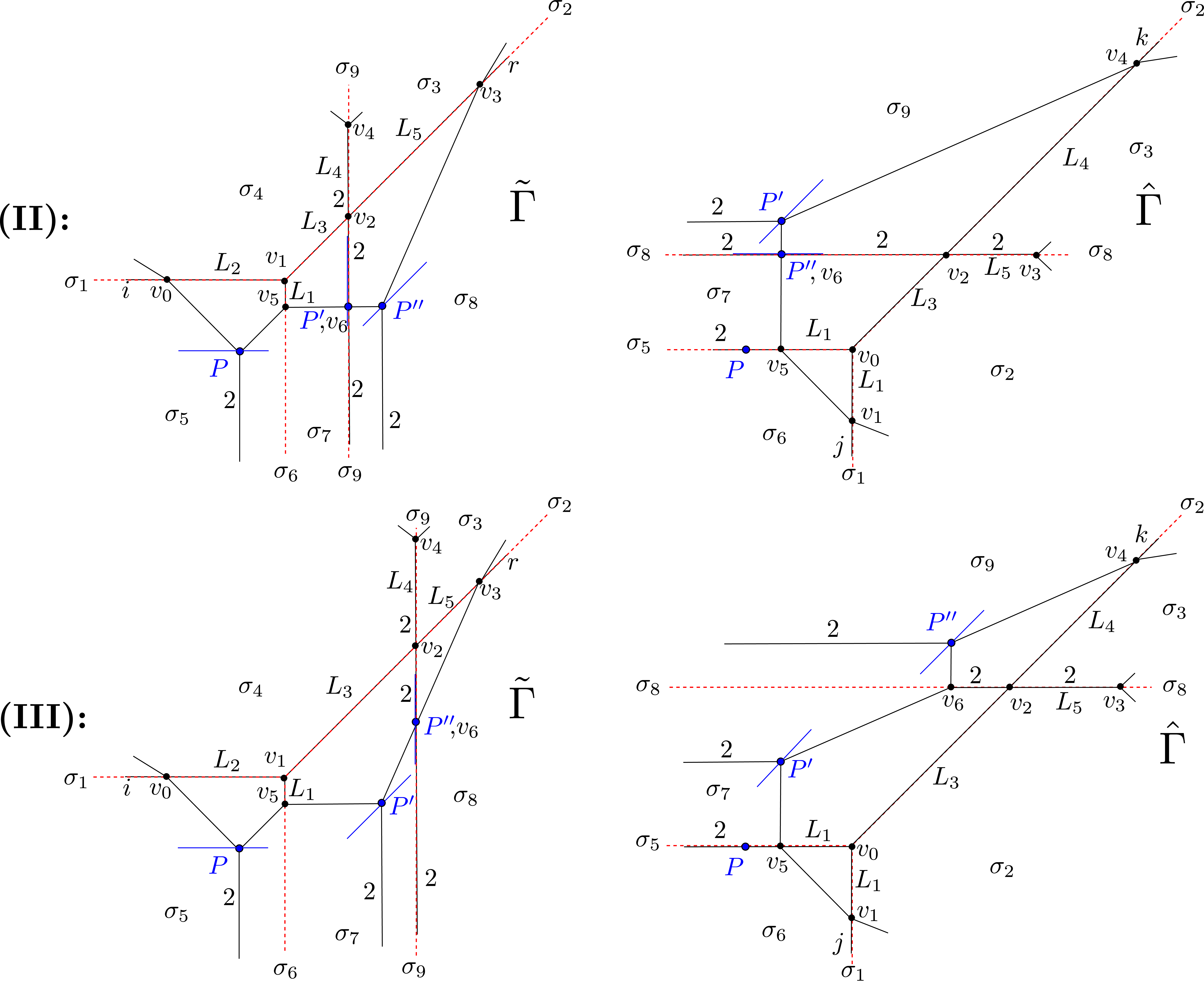}
  \caption{From left to right: the tropical curves $\tilde{\Gamma}, \hat{\Gamma}$ and the relative location of the tangency points on both $\tilde{\Lambda}$ and $\hat{\Lambda}$ for cases  (II) and  (III).
    \label{fig:Cases23TreeShape}}
\end{figure}

\noindent
Note that  $\tilde{b}_{10}(0) \neq 0$ by the genericity of $\sextic$ relative to $\Gamma$. \autoref{lm:choiceForm1} below confirms the existence (and uniqueness) of the parameter $m_1$ subject to the valuation restrictions of the statement.

Once the value of $m_1$ is fixed, a case-by-case analysis confirms that a complementary pair $(n_1, \du_1)$ can be chosen to ensure the valuation requirements  listed in~\autoref{tab:SufficientValuations}. This is the content of 
{Propositions}~\ref{pr:choiceForn1u1CaseI},~\ref{pr:choiceForn1u1CaseIII} and~\ref{pr:choiceForn1u1CaseII} below, corresponding to cases (I), (III) and (II), respectively, presented in increasing order of difficulty 
\end{proof}

As seen in~\autoref{fig:ModificationsThreeShape}, the dual cells corresponding to the projections of the vertices $v_0$, $v_1$, $v_3$, $v_4$ and $v_5$ of $\Trop\, V(I)$ to both planes are common to all three cases. Our next two lemmas describe the valuation and initial forms of the  coefficients of $\tilde{\sextic}(x,z)$ and $\hat{\sextic}(z,y)$ involved in these cells. The indices $i,j,k,r$ in both statements correspond to those in~\autoref{fig:TreeShapedModifiedSkeleton}.

\begin{lemma}\label{lm:coefficientsProjectionXZ} Assume that  $L_1<L_2,L_4, L_5$ and that the parameters $\du_1, m_1, n_1$ satisfy  $\val(\du_1)>L_3$ and $\val(m_1), \val(n_1)>0$. Then, the following identities hold for the coefficients of $\tilde{\sextic}(x,z)$:
  \begin{enumerate}
  \item $\val(\tilde{a}_{12})=\val(\tilde{a}_{11}) = \val(\tilde{a}_{21}) = 0$,  and $\overline{\tilde{a}_{12}} = 1$, with $\overline{\tilde{a}_{11}} = -\overline{a_{11}}$ and $\overline{\tilde{a}_{21}} = -\overline{a_{21}}$;
      \item $\val(\tilde{a}_{32})=\val(\tilde{a}_{41})=2L_3$ and $\val(\tilde{a}_{22}) =\val(\tilde{a}_{31})=L_3$;
      \item $\overline{\tilde{a}_{32}} = \overline{a_{22}}^2$, $\overline{\tilde{a}_{41}} = -\overline{a_{21}}\,(\overline{a_{22}})^2$, $\overline{\tilde{a}_{22}}=2\,\overline{a_{22}}$ and $\overline{\tilde{a}_{31}} = -2\,\overline{a_{21}}\,\overline{a_{22}}$;
      \item $\val(\tilde{a}_{00})=L_2$ and  $\overline{\tilde{a}_{00}} = (-1)^j\,\overline{a_{0i}}\,\overline{a_{11}}^i$;
        \item  $\val(\tilde{a}_{60})=3L_3+L_5$ and $\overline{\tilde{a}_{60}} = (-1)^r \overline{a_{3r}}(\overline{a_{21}})^r(\overline{a_{22}})^{3-r}$;
      \item $\tilde{a}_{p3} = a_{p3}$ for all $p \in \{0,\ldots, 3\}$, and $\val(\tilde{a}_{0p})=L_2$ for all $0<p\leq i$;
        \item $\val(\tilde{a}_{6-p,p}) \geq 3L_3+L_5$ for all $p\in \{1,2,3\}$ and equality holds if $p\leq r$.
 \end{enumerate}
 In particular, the cells $v_0^{\vee}$ through $v_5^{\vee}$  on the Newton subdivision of $\tilde{\sextic}(x,z)$ are independent of the distribution of the tropical tangency points between $\Gamma$ and $\Lambda$.  \end{lemma}

\begin{proof}The claims from items (1) through (5)  follow by inspecting the formulas for each coefficient listed in~\autoref{sec:CoefficientsModificationSextics},  using solely the inequalities $L_1<L_2,L_4, L_5$ and the valuation restrictions on $m_1, n_1$ and $\du_1$. In turn, the statements of items (6) and (7) are determined by a direct computation using the formulas in the appendix.

  The claim about the cells in the Newton subdivision of $\tilde{\sextic}(x,z)$ follows by a simple linear algebra computation combined with the fact that the coordinates of the points $v_0, v_1, v_2, v_4$ and $v_5$ in $\RR^3$ are completely determined by the modification. Their values appear in~\eqref{eq:v0throughv4} and~\autoref{tab:TreeShapeVerticesNew3D}.
\end{proof}

\begin{lemma}\label{lm:coefficientsProjectionZY}
  Assume that  $L_1<L_2,L_4, L_5$ and that the parameters $\du_1, m_1, n_1$ satisfy  $\val(\du_1)>L_3$ and $\val(m_1), \val(n_1)>0$. Then, the following identities hold for the coefficients of $\hat{\sextic}(z,y)$:
  \begin{enumerate}
  \item $\val(\hat{a}_{12})=\val(\hat{a}_{11}) = \val(\hat{a}_{21}) = 0$,  $\overline{\hat{a}_{12}} = -(\overline{a_{21}})^2= - \overline{\hat{a}_{21}}$ and $\overline{\hat{a}_{11}} = -\overline{a_{11}}\,(\overline{a_{21}})^2$;
      \item $\val(\hat{a}_{23})=\val(\hat{a}_{14})=2L_3$ and $\val(\hat{a}_{22}) =\val(\hat{a}_{13})=L_3$;
      \item $\overline{\hat{a}_{23}} = \overline{a_{22}}^2$, $\overline{\hat{a}_{14}} = -(\overline{a_{22}})^2$, $\overline{\hat{a}_{22}}=2\,\overline{a_{21}}\,\overline{a_{22}}= -\overline{\hat{a}_{13}}$;
      \item $\val(\hat{a}_{00})=L_1$ and  $\overline{\hat{a}_{00}} = (-1)^j \overline{a_{j0}}\,(\overline{a_{11}})^j\,(\overline{a_{21}})^{3-j}$;
        \item $\val(\hat{a}_{06})=3L_3+L_4$  and $\overline{\hat{a}_{06}} = (-1)^k \overline{a_{3r}}(\overline{a_{22}})^{3-k}$;
        \item $\hat{a}_{3p}=a_{3p}$ for all $p\in \{0,\ldots,3\}$, and $\val(\hat{a}_{p,0}) = L_1$  if $p\leq j$.
      \item $\hat{a}_{p,6-p} \geq 3L_3+L_4$ for all $p \in \{0,\ldots, 3\}$ whereas equality holds if $p\leq k$;
  \end{enumerate}
 In particular, the cells $v_0^{\vee}$ through $v_4^{\vee}$ and $v_5^{\vee}$ on the Newton subdivision of $\hat{\sextic}(x,z)$ are independent of the distribution of the tropical tangency points between $\Gamma$ and $\Lambda$.\end{lemma}
\begin{proof} The statement follows by using the same methods as in the proof of~\autoref{lm:coefficientsProjectionXZ}.
  \end{proof}

Our next result discusses how to pick the parameter $m_1$:
\begin{lemma}\label{lm:choiceForm1}
  The polynomial  $\tilde{b}_{10}\in \K[m_1]$ from~\eqref{eq:b10xz} has a unique root $m_1\in \overline{\K}$ with $\val(m_1)>0$ if, and only if, $\tilde{b}_{10}(0)\neq 0$. Furthermore, if $\sextic$ is generic relative to $\Gamma$ we have that  $m_1\in \K$, $\val(m_1) = \val(-a_{11}^3a_{13} + a_{10})$,  and   $\overline{m_1}$ satisfies
  \begin{equation}\label{eq:initialEquationForM_1}
  \overline{a_{11}}\,\overline{m_1} + \overline{-a_{11}^3a_{13} + a_{10}}=0.
  \end{equation}
    In particular, $\val(m_1)\geq L_1$ and equality holds if, and only if, $j=1$.
\end{lemma}
\begin{proof} The result follows by the Fundamental Theorem of Tropical Algebraic Geometry. By construction, $\tilde{b}_{10}(0) = -a_{11}^3a_{13} + a_{10}$, so  $\val(\tilde{b}_{10}(0)) \geq \min\{\val(a_{13}), \val(a_{10})\} \geq L_1$ from \eqref{eq:valInequalitiesTreeShape}. Equality holds if, and only if, $j=1$. 
  A direct computation confirms that the map $\tilde{b}_{10}$ is given by the formula
  \[\trop(\tilde{b}_{10})(M_1) = \max\{ -\val(-a_{11}^3a_{13} + a_{10}), M_1, 2\,M_1, -\val(a_{13}) + 3M_1\}.
  \]
  Since $ \val(-a_{11}^3a_{13} + a_{10})>0$ and $\val(a_{13})>0$, we see that this tropical polynomial has a unique root $M_1$ with $M_1<0$, namely, $M_1 = -\val(-a_{11}^3a_{13} + a_{10})$. Therefore, the Fundamental Theorem ensures the existence of a root $m_1$ of $\tilde{b}_{10}$ with $-\val(m_1) = M_1$.

  By construction, $\overline{m_1}$ solves the  local  equation $(\tilde{b}_{10})_{M_1}=0$, which is precisely the one in~\eqref{eq:initialEquationForM_1}. Finally, since this equation has a unique solution in $\resK$, and $(\tilde{b}_{10})'_{M_1}(\overline{m_1})\neq 0$, \autoref{lm:multivariateHensel} implies that the root $m_1$ is unique and lies in $\K$. This concludes our proof.
\end{proof}

Our next statement ensures the existence of a pair $(n_1,\du_1)$ compatible with $m_1$ and  satisfying the valuation requirements for case (I) listed in~\autoref{tab:SufficientValuations}:
\begin{proposition}\label{pr:choiceForn1u1CaseI} Assume that $\sextic$ is generic relative to $\Gamma$ and that the lengths of $\Gamma \cap \Lambda$ satisfy $L_4<L_3+L_1$, $L_4<L_5$ and $L_1<L_2, L_4, L_5$, as in case (I). Let  $m_1 \in \K$ be  the parameter chosen in \autoref{lm:choiceForm1}. Then, there exist a unique pair $(n_1,\du_1)\in (\overline{\K})^2$ with $\val(n_1)>0$ and $\val(\du_1)>L_3$ satisfying $\val(\tilde{a}_{30})\geq L_1+L_3$ and $\val(\tilde{a}_{50})\geq 2L_3+L_5$. Furthermore, such pair lies in $\K^2$, and its coordinatewise valuation has the following properties:
  \begin{enumerate}
  \item $\val(n_1)  \geq L_4$ and equality holds if, and only if, $k=0$;
    \item $\val(\du_1)\geq L_3+L_4$ and equality holds if, and only if, $k=0$ or $2$.
  \end{enumerate}
  Moreover, for such  triples $(m_1,n_1, \du_1)$ we have that $\val(\tilde{a}_{20})=L_1$ and $\val(\tilde{a}_{04})=L_3+L_4$,  with
  \[\overline{\tilde{a}_{40}} = (-1)^k\, \overline{a_{k3}}\,(\overline{a_{22}})^{1-k}\,(\overline{a_{21}})^3
  \quad \text{ and } \quad \overline{\tilde{a}_{20}} = \begin{cases}
        -\overline{a_{10}}\,\overline{a_{21}}/\overline{a_{11}} & \text{ if }j=1,\\
            \overline{a_{20}} & \text{ if }j=2.
  \end{cases}
  \]
\end{proposition}

\begin{proof} Under the conditions of the statement, we get the following expressions for the polynomials $\tilde{b}_{30}$ and $\tilde{b}_{50}$ collecting all terms of $\tilde{a}_{30}$ and $\tilde{a}_{50}$ whose valuations are potentially lower than the required quantities:  
  \begin{equation}\label{eq:b50xzForI}
  \begin{cases}
    \begin{aligned}
    \tilde{b}_{30}:= & (-a_{03}\,a_{21}^3 + a_{30}) + (-3\,a_{03}\,a_{21}^2 + 2\,a_{11}\,a_{22} + a_{21})n_1 + (-3\,a_{03}\,a_{21} +1)n_1^2 - \\ & a_{03}\, n_1^3 +
    (-a_{11}^2a_{22}-2\,a_{11}\,a_{21})\du_1 - a_{11}^2\,\du_1^2.\,\\    
      \tilde{b}_{50} := & (a_{20}a_{22}^3 - a_{21}^3a_{23}) + (a_{21}\,a_{22}^2 - 3\,a_{21}^2a_{23})n_1 + (a_{22}^2 -3\,a_{21}\,a_{23}) n_1^2 - a_{23}\,n_1^3 + \\
    &  (-a_{21}^2a_{22} + 3\,a_{20}\,a_{22}^2)\du_1 + (-a_{21}^2 + 3\,a_{20}\,a_{22})\du_1^2 + a_{20}\,\du_1^3 + a_{22}\,n_1^2\du_1-a_{21}\,n_1\,\du_1^2.
  \end{aligned}
  \end{cases}
\end{equation}
  Notice that both polynomials lie in $\K[n_1,\du_1]$. By~\autoref{lm:BuildSolutionsb30b50CaseI} below, there exist solutions $(n_1,\du_1) \in (\K^*)^2$ to the system  $\tilde{b}_{30}=\tilde{b}_{50}=0$ with the claimed valuation restrictions.

  To conclude, we confirm that for such triples $(m_1, n_1,\du_1)$, the valuations of $\tilde{a}_{02}$ and $\tilde{a}_{40}$ agree with the ones listed in~\autoref{tab:SufficientValuations}. Note that by the length restrictions imposed for case (I), we have $j\in \{1,2\}$ and $k\in \{0,1,2\}$. Using the lower bounds on $\val(n_1)$, $\val(m_1)$ and $\val(\du_1)$ we obtain the following two polynomials 
  \[
\tilde{b}_{20} = a_{21}\,m_1 + a_{20}\quad  \text{ and }\quad  \tilde{b}_{40} = -a_{13}\,a_{21}^3 + 2\,a_{21}\,a_{22}\,n_1 - a_{21}^2\du_1
  \]
  collecting the terms of $\tilde{a}_{20}$ and $\tilde{a}_{40}$ with valuation potentially lower than $L_1$ and $L_3+L_4$, respectively. Their expected valuations are precisely these two quantities. In what follows, we confirm these values are attained.

  By our choice of $m_1$ done in~\autoref{lm:choiceForm1} a single term in $\tilde{b}_{20}$ achieves  valuation $L_1$.   Substituting the corresponding values for the initial form of $m_1$ when $j=1$ confirms the formula for $\overline{\tilde{a}_{20}}$ provided in the statement.
  In turn, per our valuation restrictions on $n_1$ and $\du_1$ it follows that the expected initial form of  $\tilde{b}_{40}$ involves a single summand  when $k=1$ or $2$.  In turn, for $k=0$,  it    equals $\overline{2\,a_{21}\,a_{22}\,n_1 -a_{21}^2\du_1}$.  The claim for $\overline{\tilde{a}_{40}}$  follows by replacing the values for $\overline{n_1}$ and $\overline{\du_1}$ provided in~\autoref{lm:BuildSolutionsb30b50CaseI} in  the three  expressions for $\overline{\tilde{b}_{40}}$ arising for each possible choice of the index $k$.
\end{proof}

\begin{lemma}\label{lm:BuildSolutionsb30b50CaseI} Assume that $\sextic$ is generic relative to $\Gamma$ and that the lengths of $\Gamma \cap \Lambda$ satisfy $L_4<L_3+L_1$, $L_4<L_5$ and $L_1<L_2, L_4, L_5$, as in case (I). Then, the system defined by the vanishing of the polynomials  $\tilde{b}_{50}$ and $\tilde{b}_{30}$ from~\eqref{eq:b50xzForI} has a solution $(n_1, \du_1) \in (\overline{\K}^*)^2$ with $\val(n_1)>0$ and $\val(\du_1)> L_3$. Furthermore, such solution lies in $\K^2$ and both $\val(n_1)$ and $\val(\du_1)$ satisfy the restrictions stated in~\autoref{pr:choiceForn1u1CaseI}.  In addition, $(\overline{n_1},\overline{\du_1})=(\overline{a_{03}}\,(\overline{a_{21}})^2, \overline{a_{03}}\,\overline{a_{22}}\,\overline{a_{21}})$ for $k=0$, whereas
  $\overline{\du_1}=-\overline{a_{23}}\,\overline{a_{21}}/\overline{a_{22}}$ for $k=2$.
  \end{lemma}

  \begin{proof} As in the proof of~\autoref{lm:choiceForm1}, we argue by invoking the Fundamental Theorem of Tropical Geometry. Our goal is to find a solution $(n_1,\du_1)\in (\overline{\K}^*)^2$ to the system $\tilde{b}_{50}=\tilde{b}_{30}=0$ with the restrictions $\val(n_1)>0$ and $\val(\du_1)>L_3$. We do so by picking a point in the intersection of the corresponding tropical curves lying on the cone $\RR_{<0}\times \RR_{<-L_3}$. \autoref{fig:b30b50xz} depicts the Newton subdivisions of both curves and the corresponding plane tropical curves. Recall that  the condition $L_4<L_5$ combined with~\eqref{eq:valInequalitiesTreeShape} forces $k\neq 3$.

The genericity of $\sextic$ relative to $\Gamma$ and our edge length assumptions determine the following facts:
  \begin{itemize}
  \item $\val(\tilde{b}_{30}(0))
    \geq L_4$ and equality holds if, and only if, $k=0$,
  \item  $\val(\tilde{b}_{50}(0))
    \geq 2L_3+L_4$ and equality holds if, and only if, $k=2$.
  \end{itemize}
  In particular, this implies that the vertices $v$ and $v'$ in the figure have coordinates  $v=(-\alpha, -\alpha-L_3)$ and $v'=(-\beta, -\beta)$, where 
  \begin{equation}\label{eq:alphaBetaCaseI}
    \alpha := \val(\tilde{b}_{50}(0)) - 2L_3 
    \geq L_4 \quad \text{ and } \quad \beta := \val(\tilde{b}_{30}(0)) \geq L_4. 
  \end{equation}
  Furthermore, $\alpha = L_4$ if, and only if, $k=2$, whereas $\beta=L_4$ if, and only if, $k=0$.

  In what follows, we determine the intersection $\Trop\,V(\tilde{b}_{30}) \cap \Trop\,V(\tilde{b}_{50})$ within $\RR_{<0}\times \RR_{<-L_3}$  for each  possible value of the index $k$. As~\autoref{fig:b30b50xz} shows, the answer will depend on the relative order between $\alpha$, $\beta$ and $\alpha + L_3$.  In all cases, we will find a unique point $w$ in the intersection where the local system $(\tilde{b}_{30})_w = (\tilde{b}_{50})_w = 0$ is consistent. 

  \begin{figure}[tb]
  \includegraphics[scale=0.17]{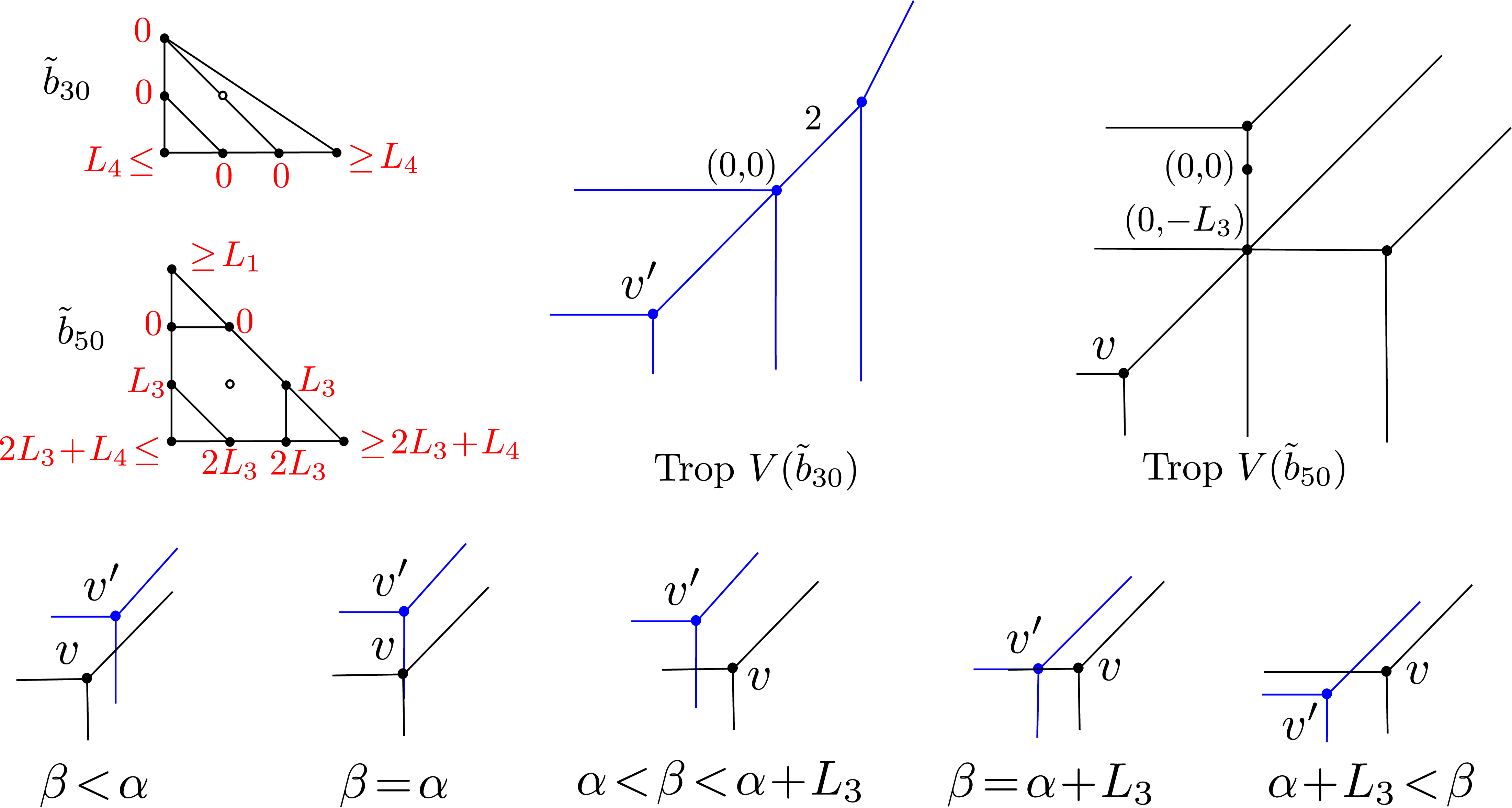}
  \caption{From top to bottom: Newton subdivisions of the polynomials $\tilde{b}_{30}$ and $\tilde{b}_{50}$ from~\eqref{eq:b50xzForI} for case (I), the associated tropical plane curves and their relevant intersections in $\RR_{<0}\times \RR_{<-L_3}$ dependent on $\alpha,\beta$. We record the valuation of the non-zero coefficients of both polynomials and the position of $(0,0)$ on $\Trop\,V(\tilde{b}_{50})$.
     \label{fig:b30b50xz}}
\end{figure}

  First, we assume these three quantities are pairwise distinct. If so, the intersection is transverse and consists of a single point $w$. We  determine its coordinates  from $v$ and $v'$:
  \[
  w = \begin{cases}
    (-\beta, -\beta-L_3) & \text{ if } \beta<\alpha,\\
    (-\beta, -\alpha-L_3) & \text{ if } \alpha<\beta<\alpha+L_3,\\
    (-\alpha-L_3, -\alpha-L_3) & \text{ if } \alpha+L_3<\beta.
  \end{cases}
  \]
  A simple inspection reveals that for all cases, the initial system at $w$ is linear, with unique solution in $(\resK^*)^2$.   By~\autoref{lm:multivariateHensel}, the solution lifts uniquely to a solution of $\tilde{b}_{30}=\tilde{b}_{50}=0$.

  Next, assume $\beta=\alpha$. In this situation, the intersection is non-proper and consists of the ray $v+\RR_{\geq 0}\langle (0,-1)\rangle$. By the genericity of $\sextic$ and~\eqref{eq:b50xzForI}, we see that $v$ is the sole point on this ray  for which the corresponding initial system is consistent. Furthermore, the initial system at $v$ becomes
\[
\overline{-a_{03}\,a_{21}^3+a_{30}} +
\overline{a_{21}}\,\overline{n_1} = \overline{a_{20}\,a_{22}^3 - a_{21}^3a_{23}} + \overline{a_{21}}\,(\overline{a_{22}})^2\overline{n_1} - (\overline{a_{21}})^2\overline{a_{22}} \,\overline{\du_1} =  0.
\]
It has a unique solution in $(\resK^*)^2$, which lifts uniquely to a solution  in $\K^2$ of the original system.

Finally, assume $\beta = \alpha + L_3$. As in the previous case, the intersection is non-proper, but agrees with the horizontal ray $v'+\RR_{\geq 0}\langle (-1,0)\rangle$. As before, the genericity of $\sextic$ ensures that the only local system with a solution is the one corresponding to $v'$. For $w=v'$, the local system becomes
\[
\overline{-a_{03}\,a_{21}^3 + a_{30}} + \overline{a_{21}}\,\overline{n_1} - 2\,\overline{a_{11}}\,\overline{a_{21}}\,\bar{\du_1}= \overline{a_{20}\,a_{22}^3 - a_{21}^3a_{23}}  - (\overline{a_{21}})^2\overline{a_{22}} \,\bar{\du_1} =  0
\]
Its unique solution in $(\resK^*)^2$  lifts uniquely to a solution $(n_1,\du_1)$ in $\K^2$ of the original system.

To conclude, we analyze which of these five possible intersections can occur for each choice of $k$. If $k=0$, we know that $\beta=L_4<\alpha$, so $w=(-L_4,-L_3-L_4)$. This confirms that $n_1$ and $\du_1$ have the prescribed valuations. The local system at $w$ becomes
\[
-\overline{a_{03}}\,(\overline{a_{21}})^3 + \overline{a_{21}}\,\overline{n_1} = -(\overline{a_{21}})^2\,\overline{a_{22}}\,\bar{\du_1} + \overline{a_{21}}\,(\overline{a_{21}})^2\,\overline{n_1} = 0,
\]
so the solution $(\overline{n_1},\bar{\du_1})$ agrees with the one given in the statement.

If $k=2$, we have $\alpha=L_4<\beta$. The valuations of $n_1$ and $\du_1$ are determined by the coordinates of $w$, which vary depending on the  relative order between $\beta$ and $\alpha + L_3$. Analyzing all three options we obtain $\val(n_1)>\alpha_4$ and $\val(\du_1) =L_3+L_4$, as we wanted.  In turn, the vanishing of the polynomial $(\tilde{b}_{50})_w = -(\overline{a_{21}})^3\,\overline{a_{23}}- (\overline{a_{21}})^2\overline{a_{22}}\,\bar{\du_1}$ confirms that $\bar{\du_1}$ has the expected value.

Finally, when $k=1$, we have that $\alpha,\beta>L_4$. All five possible relative orders between $\alpha$, $\beta$ and $\alpha+L_3$ can occur. Inspecting the corresponding five values for $w$ given earlier  confirms that $\val(n_1)>L_4$ and $\val(\du_1)>L_3+L_4$, as we wanted.
\end{proof}

Before presenting the analog of~\autoref{pr:choiceForn1u1CaseI} for cases (II) and (III) we analyze the coefficient $\hat{a}_{05}$. Notice that the corresponding point $(0,5)$ lies in the 
cell in the Newton subdivision of $\hat{f}(z,y)$ dual to the unique tangency point between $\Gamma$ and $\Lambda$ present in the chart $\sigma_9$ of the modification. Using the target lower bound for $\val(\hat{a}_{50})$ indicated in~\autoref{tab:SufficientValuations}, we define the polynomial $\hat{b}_{05}$ collecting terms with potentially lower valuation than desired.
A a simple inspection  confirms that  $\hat{b}_{05}$ is the same for both cases, namely,
\begin{equation}\label{eq:b05zyCasesIIAndIII}  \hat{b}_{05} :=  (a_{02}\,a_{22}^3 - a_{32}) + (3\,a_{02}\,a_{22}^2 - a_{22})\du_1 + (3\,a_{02}\,a_{22}-1) \du_1^2 + a_{02}\,\du_1^3.
\end{equation}
 The vanishing of $\hat{b}_{05}$ determines which properties must be satisfied by the parameter $\du_1$. Here is the precise statement, which depends on the index $r$ seen in~\autoref{fig:TreeShapedModifiedSkeleton}:
\begin{lemma}\label{lm:choiceForu1CasesIIAndIII} Assume that $\val(m_1)\geq L_1$, and that the lengths of $\Gamma \cap \Lambda$ are as in cases (II) or (III). Then, the polynomial  $\hat{b}_{05}\in \K[\du_1]$ from~\eqref{eq:b05zyCasesIIAndIII} has a unique root $\du_1\in \overline{\K}$ with $\val(\du_1)>L_3$ if, and only if, $\hat{b}_{05}(0)\neq 0$. Furthermore, if $\sextic$ is generic relative to $\Gamma$ we have that $\du_1\in \K$, $\val(\du_1) = \val(a_{02}\,a_{22}^3 - a_{32})-L_3$, and   $\overline{\du_1}$ satisfies
  \begin{equation}\label{eq:initialEquationForU1}
  -\overline{a_{22}}\,\overline{\du_1} + \overline{a_{02}\,a_{22}^3 - a_{32}}=0.
  \end{equation}
    In particular, for case (III), we have that $\val(\du_1)\geq L_3+L_5$ and equality holds if, and only if, $r=2$. In turn, for case (II), we have that $\val(\du_1)> 2L_3+L_1$.
\end{lemma}
\begin{proof} The result follows by the Fundamental Theorem of Tropical Algebraic Geometry as in the proof of~\autoref{lm:choiceForm1}. The lower bounds for $\val(\du_1)$ are obtained by using~\eqref{eq:valInequalitiesTreeShape} and the conditions on the edge lengths imposed for cases (II) and (III).
\end{proof}

Next, we discuss how to pick the coefficient $n_1$ for case (III) in a compatible way with our prior choice of $m_1$ and $\du_1$ to ensure that all the sufficient conditions from~\autoref{tab:SufficientValuations} for this case are satisfied. The value of $n_1$ will be determined by the valuation requirements on $\tilde{a}_{30}$ and $\hat{a}_{03}$.

To this end we first build  the polynomials $\tilde{b}_{30}$ and $\hat{b}_{03}$ collecting all terms of $\tilde{a}_{30}$ and $\hat{a}_{03}$ with valuation that can potentially be strictly less than $L_3+L_1$.
The valuation information on $m_1$ and $\du_1$ provided in 
{Lemmas}~\ref{lm:choiceForm1} and~\ref{lm:choiceForu1CasesIIAndIII}, together with the edge length constraints of case (III), yield
\begin{equation}\label{eq:b30CaseIII}
  \tilde{b}_{30}:= (a_{30} - a_{03}\,a_{21}^3) + (a_{21}-3\,a_{03}\,a_{21}^2 + 2\,a_{11}\,a_{22})n_1 + (1-3\,a_{03}\,a_{21})n_1^2 - a_{03}n_1^3 =:  - \hat{b}_{03}.
\end{equation}

Our next result  ensures the existence of a unique root $n_1\in \K$ of $\tilde{b}_{30}$ and $\hat{b}_{30}$ with $\val(n_1)>0$.

\begin{lemma}\label{lm:ChoiceN1ForCaseIII}
  Assume that the lengths of $\Gamma \cap \Lambda$ satisfy $L_5<L_3+L_1$, $L_5<L_4$ and $L_1<L_2, L_4, L_5$, as in case (III). The polynomial $\tilde{b}_{30}$ from~\eqref{eq:b30CaseIII} has a unique solution $n_1\in \overline{\K}^*$ with $\val(n_1)>0$ if, and only if, $\tilde{b}_{30}(0)\neq 0$. In particular,  we have that $\val(n_1) = \val(a_{30}-a_{03}\,a_{21}^3)\geq L_5$, and equality holds if, and only if, $r=0$. Furthermore, $n_1\in \K$ and the initial form of $n_1$ satisfies 
  \[
  \overline{a_{30}-a_{03}\,a_{21}^3} + \overline{a_{21}}\,\overline{n_1}=0.
  \]
\end{lemma}
\begin{proof} The result follows by the same techniques used in the proof of~\autoref{lm:choiceForm1} combined with  the fact that $\val(a_{03})>0$ and $\val(a_{30})>0$.
  \end{proof}

All that remains to establish case (III) is to verify that our choice of parameters $(m_1,n_1,\du_1)$ satisfy the remaining valuation conditions listed in~\autoref{tab:SufficientValuations}. Our next statement confirms this is indeed the case:
\begin{proposition}\label{pr:choiceForn1u1CaseIII} Let $\sextic$ be generic relative to $\Gamma$. 
  Assume that the lengths of $\Gamma \cap \Lambda$ satisfy $L_5<L_3+L_1$, $L_5<L_4$ and $L_1<L_2, L_4, L_5$, as in case (III).   Let $m_1,\du_1,n_1$ be the parameters in $\K$ chosen using 
  {Lemmas}~\ref{lm:choiceForm1},~\ref{lm:choiceForu1CasesIIAndIII} and~\ref{lm:ChoiceN1ForCaseIII}, respectively.
Then,  all the conditions listed in~\autoref{tab:SufficientValuations} are satisfied. Furthermore, the initial forms of the four coefficients $\tilde{a}_{20}, \tilde{a}_{40}, \hat{a}_{02}$ and $\hat{a}_{04}$ verify
  \begin{equation*}\label{eq:initialsCaseIII}
      \overline{\tilde{a}_{20}} = \overline{a_{j0}}(-\overline{a_{10}}/\overline{a_{11}})^{2-j}\! = \overline{\hat{a}_{20}}/ \overline{a_{21}}, \quad
 \overline{\tilde{a}_{40}} = (-\overline{a_{21}})^r\,\overline{a_{3r}}\,(\overline{a_{22}})^{1-r} \text{ and }
      \overline{\hat{a}_{04}} = (-1)^r\,\overline{a_{3r}}(\overline{a_{22}}/\overline{a_{21}})^{1-r}.
  \end{equation*}
\end{proposition}
\begin{proof} Recall that the edge length restriction $L_5<L_4$ for case (III) forces $r$ to lie in $\{0,1,2\}$. Our choice of parameters $m_1,n_1,\du_1$ done in the aforementioned lemmas guarantee that the valuation requirements listed in the table  for $\tilde{a}_{10}$, $\tilde{a}_{30}$, $\hat{a}_{03}$ and $\hat{a}_{05}$ are satisfied.
In turn, the conditions  $\val(n_1)\geq L_5$ and $\val(\du_1)\geq L_3+L_5$ obtained in the lemmas ensure that all terms in $\tilde{a}_{50}$ have valuation at least $2L_3+L_5$. The same reasoning together with the condition $\val(m_1)\geq L_1$ imply that all terms of $\hat{a}_{01}$ have valuation at least $L_1$. Thus, the required conditions for $\val(\tilde{a}_{50})$ and $\val(\hat{a}_{01})$  are also fulfilled.

To finish, we need only verify the valuation expectations for the coefficients $\tilde{a}_{20}$, $\tilde{a}_{40}$, $\hat{a}_{20}$ and $\hat{a}_{40}$ provided in~\autoref{tab:SufficientValuations}, and that the claimed formulas for their initial forms are valid. We argue as in the proof of~\autoref{pr:choiceForn1u1CaseII}. A simple inspection confirms that
\begin{equation}\label{eq:b20b02CaseIII}
  \tilde{b}_{20} := a_{21}\,m_1 + a_{20} \quad \text{ and }\quad  \hat{b}_{02}:= a_{21}\, \tilde{b}_{20}. 
\end{equation}
As in the proof of~\autoref{pr:choiceForn1u1CaseI}, the information on
$m_1$ obtained in~\autoref{lm:choiceForm1} implies both  that
$\val(\tilde{a}_{20})=\val(\hat{a}_{02})=L_1$ and
the formulas for $\overline{\tilde{a}_{20}}$ and $\overline{\hat{a}_{02}}$.

Similarly, the lower bounds on the valuations of both $\du_1$ and $n_1$ discussed earlier yield
\[
\tilde{b}_{40} := (2\,a_{21}\,a_{22}n_1 + 3\,a_{22}\,a_{30}) - a_{21}\,a_{31} - a_{21}^2\,\du_1 \quad \text{ and } \quad \hat{b}_{40} := -a_{22}\,n_1 - a_{31} -a_{21}\du_1.
\]
The expected valuation for the three summands in each expression is achieved for exactly one value of $r$, namely $r=0,1$ and $2$ when reading them from left to right. Thus, both  $\tilde{a}_{40}$ and $\hat{a}_{04}$ have valuation $L_3+L_5$, as desired. The values of their initial forms are obtained from the data for  $\overline{n_1}$ and $\overline{\du_1}$ provided in the aforementioned lemmas for $r=0$ and $r=2$, respectively.
\end{proof}

\begin{remark}\label{rm:markingCaseIII} For case (III) the sum of all terms in $\tilde{a}_{50}$ with the potential of having  valuation exactly $2L_3+L_5$ equals
  \[  (3\,a_{22}^2a_{30}+a_{21}\,a_{22}^2n_1) -2\,a_{21}\,a_{22}\,a_{31}  + (a_{21}^2\,a_{32}  - a_{22}\,a_{21}^2\du_1). \]
  Since $r\neq 3$ for case (III), the properties of  $\du_1$  and $n_1$ obtained in 
  {Lemmas}~\ref{lm:choiceForu1CasesIIAndIII} and~\ref{lm:ChoiceN1ForCaseIII} imply that $\val(\tilde{a}_{50})=2L_3+L_5$ for all three possible values of  $r$. Thus, the point $(5,0)$ is always marked in the Newton subdivision of $\tilde{\sextic}(x,z)$. In turn, a similar argument using the information about $m_1$ gathered in~\autoref{lm:choiceForm1} confirms that the polynomial
  \[ a_{21}^2(a_{11}\,m_1 - a_{10}) + 2\,a_{11}\,a_{20}\,a_{21}
  \]
obtained from $\hat{a}_{10}$ by collecting all terms of potential valuation $L_1$  has valuation precisely $L_1$ for our choice of parameter $m_1$. Thus,  the vertex $(0,1)$ in the Newton subdivision of $\hat{f}(z,y)$ is marked for both $j=1$ and $2$.
\end{remark}

It remains to discusses the last option to have a type (8) tangency, namely, case (II). As with case (III) we let
$m_1, \du_1$  in $\K$ be the  unique roots of the polynomials $\tilde{b}_{10}$ and $\hat{b}_{05}$ from~\eqref{eq:b10xz} and \eqref{eq:b05zyCasesIIAndIII} with $\val(m_1)\geq L_1$ and $\val(\du_1)> 2L_3+L_1$ obtained in 
{Lemmas}~\ref{lm:choiceForm1} and~\ref{lm:choiceForu1CasesIIAndIII}, respectively. In order to determine $n_1$, we look at the valuation conditions of $\tilde{b}_{30}$, $\hat{a}_{03}$ and $\tilde{a}_{50}$ listed in~\autoref{tab:SufficientValuations}.

Our first result show that the conditions on the first two polynomials only impose valuation restrictions on the parameter $n_1$:

\begin{lemma}\label{lm:valRestrictionsn1caseIIIat30Andah03}
  Assume that the lengths of $\Gamma \cap \Lambda$ satisfy $L_3+L_1< L_4 , L_5$ and $L_1<L_2, L_4, L_5$, as in case (II). Let $m_1, \du_1$ be fixed as in 
  {Lemmas}~\ref{lm:choiceForm1} and~\ref{lm:choiceForu1CasesIIAndIII}. Fix a a parameter $n_1$ with $\val(n_1)>0$. Then, $\val(\tilde{a}_{30}), \val(\hat{a}_{03})\geq L_1+L_3$ if, and only if, $\val(n_1)\geq L_1+L_3$. Furthermore, if the latter holds, we have  $\val(\hat{a}_{01})\geq L_1$.

\end{lemma}
\begin{proof}
We start by collecting all terms of $\tilde{a}_{30}$ and $\hat{a}_{03}$ with valuation potentially lower than $L_1+L_3$. We obtain the following formulas for the resulting polynomials $\tilde{b}_{30}$ and $\hat{b}_{03}$:
  \begin{equation}\label{eq:b30b03CaseII}
    \tilde{b}_{30}:=(2\,a_{11}\,a_{22} + a_{21})\,n_1 + n_1^2 \quad \text{ and }\quad \hat{b}_{03}:= -\tilde{b}_{30}.
  \end{equation}

  Since $\val(n_1)>0$, we conclude that  $\val(\hat{b}_{30}) = \val(\tilde{b}_{30}) = \val(n_1)$, with
$\overline{\tilde{a}_{30}} = \overline{\tilde{b}_{30}} = \overline{a_{21}}\,\overline{n_1}$ and $\overline{\hat{a}_{30}} = -\overline{\tilde{a}_{30}}$. 
        Thus, the conditions for the valuations of both $\tilde{a}_{30}$ and $\hat{a}_{03}$ will be satisfied if, and only if, $\val(n_1)\geq L_1+L_3$. A simple inspection of all terms of $\hat{a}_{01}$ ensures that, under this condition, we have $\val(\hat{a}_{01})\geq L_1$. \end{proof}

The valuation restrictions obtained so far for the triple $(m_1,n_1, \du_1)$  determine the formula for the polynomial $\tilde{b}_{05}$. More precisely, we have:
  \begin{equation}\label{eq:b50xzForII}
    \begin{aligned}
      \tilde{b}_{50}:= &(a_{20}\,a_{22}^3 - a_{21}^3\,a_{23}) + (a_{21}\,a_{22}^2 - 3\,a_{21}^2\,a_{23})n_1 + (a_{22}^2-3\,a_{21}\,a_{23})n_1^2 - a_{23}\,n_1^3 + \\
    & (-a_{22}\,a_{21}^2 + 3\,a_{22}^2\,a_{20})\du_1 + (-a_{21}^2 + 3\,a_{20}\,a_{22})\du_1^2 + a_{22}\,n_1^2\du_1 - a_{21}\,n_1\du_1^2 +a_{20}\,\du_1^3.
    \end{aligned}
  \end{equation}

  Our next lemma confirms that this polynomial determines $n_1$ uniquely. Its proof technique is similar to that of~\autoref{lm:BuildSolutionsb30b50CaseI}:

\begin{lemma}\label{lm:pickN1ForII}
  Assume that $\sextic$ is generic relative to $\Gamma$,  and that the lengths of $\Gamma \cap \Lambda$ satisfy $L_3+L_1< L_4 , L_5$ and $L_1<L_2, L_4, L_5$, as in case (II). Let $\du_1$ be fixed as in~\autoref{lm:choiceForu1CasesIIAndIII}.  Then,  the polynomial $\tilde{b}_{50}\in \K[n_1]$ has a unique root $n_1\in \overline{\K}^*$  with $\val(n_1)\geq L_1 + L_3$. Furthermore, such root lies in $\K$ and $\val(n_1)=L_1+L_3$ if, and only if, $j=2$. In addition, when $j=2$, we have
   $\overline{n_1} = -\overline{a_{20}}\,\overline{a_{21}}\,\overline{a_{22}}$.
\end{lemma}

\begin{figure}[tb]
  \includegraphics[scale=0.17]{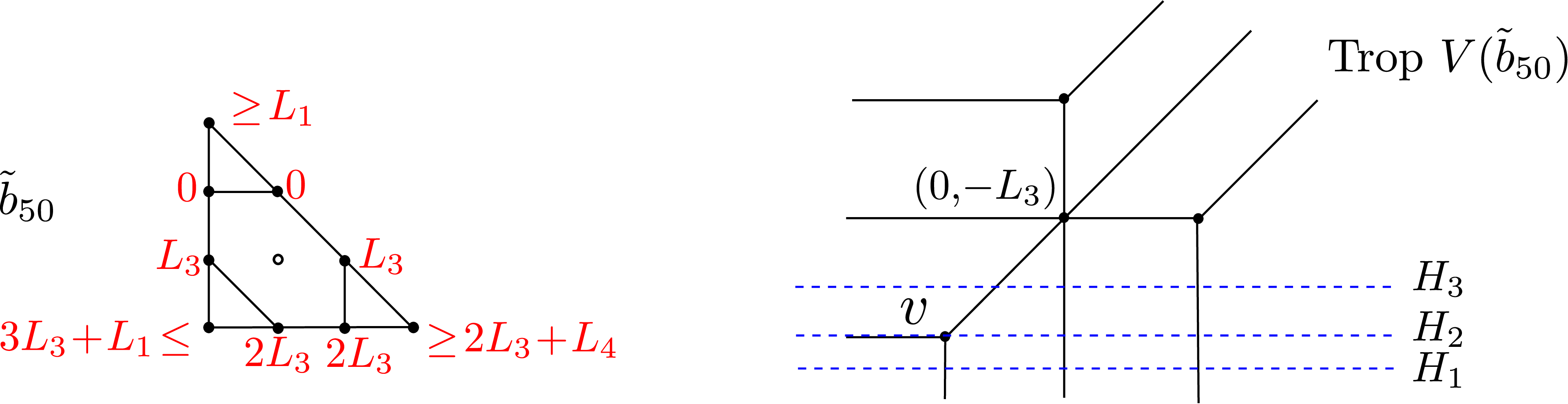}
  \caption{From left to right: Newton subdivision of the polynomial $\tilde{b}_{50}$ from~\eqref{eq:b50xzForII} in the variables $(n_1,\du_1)$ and its dual tropical plane curve for case (II). We include data on the valuation of the relevant coefficients of  $\tilde{b}_{50}$. The three horizontal lines indicate the possible locations  for the horizontal line  recording the negative valuation of the parameter $\du_1$ with respect to the vertex $v$ of $\Trop \,V(\tilde{b}_{50})$.\label{fig:b50CaseIIxz}}
\end{figure}

\begin{proof} 
  Recall that the parameter $\du_1$ has valuation  $\delta:=\val(a_{02}\,a_{22}^3 - a_{32})-L_3$. 
  We prove the statement by invoking the Fundamental Theorem of Tropical Geometry and~\autoref{lm:multivariateHensel}. \autoref{fig:b50CaseIIxz} shows the Newton subdivision of $\tilde{b}_{50}$ and its dual tropical curve. The vertex $v$ has coordinates $v:=(2L_3-\gamma, L_3-\gamma)$ with $\gamma=\val(a_{20}\,a_{22}^3-a_{23}\,a_{21}^3)$.
  Since $L_4>L_1+L_3$, \eqref{eq:valInequalitiesTreeShape} ensures that $\gamma \geq L_1 + 3L_3$ and equality holds if, and only if, $j=2$.  

  Our objective is to find $n_1$ with $\val(n_1)\geq L_1+L_3$ so that $(-\val(n_1), -\val(\du_1))$ lies in the intersection of $\Trop\,V(\tilde{b}_{50})$ and the horizontal line $H$ through $(0,-\delta)$. Note that $\delta>2L_3+L_1$ by~\autoref{lm:choiceForu1CasesIIAndIII}. There are three possible locations for $H$ depending on the relative value between $\gamma$ and $\delta$. We label them as $H_1$, $H_2$ and $H_3$ in the figure. Comparing the value of $v$ and $-\delta$ we see that $H_2$ and $H_3$ can only occur if $L_3-\gamma\leq -\delta$, or equivalently $L_3+\delta <\gamma$. This forces $\gamma>L_1+3L_3$, so we have $j=1$ for $H_2$ and $H_3$.

 In what follows we determine the intersections $\Trop\,V(\tilde{b}_{50})\cap H_s$ for each $s=1,2,$ or $3$, and confirm that the first coordinate of all points in the intersection is bounded above by $-(L_1+L_3)$. Furthermore, we show that a unique point $Q_s$ in each of them yields a solution to the corresponding local equation  $(\tilde{b}_{50})_{Q_s}=0$ that is compatible with the known value of $\bar{\du_1}$ determined by~\eqref{eq:initialEquationForU1}.

 We treat each case separately, starting with $s=1$. For $H_1$ to occur, the quantities $\gamma$ and $\delta$ must satisfy $L_3-\gamma>-\delta$. We obtain a single point in the intersection  between $H_1$ and $\Trop\,V(\tilde{b}_{50})$, namely $Q_1= (2L_3-\gamma, -\delta)$. Our lower bound on $\gamma$ provided below yields $2L_3-\gamma \leq-(L_1+L_3)$, and equality is attained precisely for $j=2$. It follows that  $Q_1 \in \RR_{\leq -(L_1+L_3)}\times \{-\delta\}$, as we wanted. Furthermore, 
the local equation for $\tilde{b}_{50}$ at this point is the following linear expression in $\overline{n_1}$:
\[
\overline{a_{20}\,a_{22}^3 - a_{21}^2\,a_{23}} + \overline{a}_{21}\,(\overline{a}_{22})^2\,\overline{n_1}=0.
\]
It is independent of $\bar{\du_1}$, and its solution lies in $\resK$. ~\autoref{lm:multivariateHensel} and the Fundamental Theorem of Tropical Geometry ensures this solution lifts to a unique root of $\tilde{b}_{50}(n_1)$ with $\val(n_1)=-2L_3+\gamma\geq L_1+L_3$ which must lie in $\K$. Our analysis shows that $\val(n_1) = L_1+L_3$ if, and only if, $j=2$. Moreover, for $j=2$, the value of $\overline{n_1}$ obtained from the local equation at $Q_1$ matches the given one.

Next, we assume $H=H_3$. As with $H_1$ the intersection is proper and we obtain a single point $Q_2 = (-\delta+L_3, -\delta)$. Our lower bound on $\delta$ ensures that $-\delta+L_3<-(L_1+L_3)$. Thus, we obtain the desired condition $Q_2 \in \RR_{< -(L_1+L_3)}\times \{-\delta\}$. The local equation for $\tilde{b}_{50}$ at $Q_2$ becomes:
\[
  \overline{a_{21}}\,(\overline{a_{22}})^2\overline{n_1} - \overline{a_{22}}\,(\overline{a}_{21})^2\,\overline{\du_1} = 0.
\]
Since the value of $\du_1$ is fixed by~\eqref{eq:initialEquationForU1}, the above equation has a unique solution in $\overline{n_1}$. By construction, it lifts uniquely to a solution $n_1\in \K$ with $\val(n_1)>L_1+L_3$. This inequality is compatible with our condition $j=1$.

To conclude, we analyze the case when $H=H_2$. In this situation, we know that $j=1$. In turn, the intersection is non-proper and it agrees with the ray $v+\RR_{\geq}\langle(-1,0)\rangle$. Notice that $-\delta+L_3<-(L_1+L_3)$ son the whole ray lies in $\RR_{< -(L_1+L_3)}\times \{-\delta\}$.
For any point $Q$ in  this ray with $Q\neq v$, the corresponding local equation $(\tilde{b}_{50})_{Q}$ is solely dependent on $\du_1$. The genericity assumptions on $\sextic$ will make it inconsistent with~\eqref{eq:initialEquationForU1}. In turn, the local equation at $v$ becomes
\begin{equation}\label{eq:n1IIIH3}
  \overline{a_{20}\,a_{22}^3 -a_{21}^3\, a_{23}} - \overline{a_{22}}\,(\overline{a_{21}})^2\,\overline{\du_1} + \overline{a_{21}}\,(\overline{a_{22}})^2\,\overline{n_1}=0.
\end{equation}
As with the previous cases, it follows from here that $\tilde{b}_{50}$ has a unique solution $n_1$ in $\K$ with $\val(n_1) > L_3 +L_1$ whose initial form $\overline{n_1}$ satisfies the above local equation when $\bar{\du_1}$ is the unique solution to ~\eqref{eq:initialEquationForU1}.
  \end{proof}

Our final statement in this subsection confirms that out choice of tuple $(m_1,n_1, \du_1) \in \K^3$ satisfies the remaining valuation conditions from~\autoref{tab:SufficientValuations}  for case (II):

\begin{proposition}\label{pr:choiceForn1u1CaseII}
  Assume that the lengths of $\Gamma \cap \Lambda$ satisfy $L_3+L_1< L_4 , L_5$ and $L_1<L_2, L_4, L_5$, as in case (II). 
  Let $m_1,\du_1,n_1$ be the parameters in $\K$ chosen using 
  {Lemmas}~\ref{lm:choiceForm1},~\ref{lm:choiceForu1CasesIIAndIII} and~\ref{lm:pickN1ForII}, respectively. Then,
the triple $(m_1,n_1, \du_1)\in \K^3$  satisfies all sufficient conditions listed in~\autoref{tab:SufficientValuations}. In addition, the initial forms of the coefficients $\tilde{a}_{20}, \tilde{a}_{40}, \hat{a}_{02}$ and $\hat{a}_{04}$ verify   $ \overline{\hat{a}_{20}}= \overline{a_{21}}\,\overline{\tilde{a}_{20}}$,    whereas
\begin{equation*}\label{eq:initialsCaseII}
\overline{\tilde{a}_{20}} = (-1)^j\,\overline{a_{j0}}\,(\frac{\overline{a_{10}}}{\overline{a_{11}}})^{2-j}   ,\; \overline{\tilde{a}_{04}} = (-1)^j\,\overline{a_{j0}}\,(\overline{a_{22}})^2\left (\frac{\overline{a_{21}}}{\overline{a_{11}}}\right )^{2-j}\!\!\!\text{and }\,
 \overline{\hat{a}_{40}} = (-1)^j\,(\overline{a_{22}})^2\,\frac{\overline{a_{j0}}}{\overline{a_{11}}}(\frac{\overline{a_{21}}}{\overline{a_{11}}})^{1-j}.
  \end{equation*}
\end{proposition}

\begin{proof}
  Our choice of parameters ensures that the  valuation conditions are satisfied for $\tilde{a}_{10}$, $\tilde{a}_{30}$, $\tilde{a}_{50}$, $\hat{a}_{01}$, $\hat{a}_{03}$ and $\hat{a}_{05}$. 
  All that remains is to confirm the claims for the valuations and initial forms of the four coefficients $\tilde{a}_{20}$, $\hat{a}_{02}$, $\tilde{a}_{40}$, and $\hat{a}_{04}$.

  The formulas for $\tilde{b}_{20}$ and $\hat{b}_{02}$ agree with those in~\eqref{eq:b20b02CaseIII}. Thus, the proof of~\autoref{pr:choiceForn1u1CaseIII} confirms the claims for the first two coefficients. Finally, the inequalities on $\val(m_1)$, $\val(n_1)$ and $\val(\du_1)$ ensure that the valuations of all terms featuring in $\tilde{a}_{40}$ or $\hat{a}_{04}$ are bounded below by $L_1+2 L_3$. Furthermore, collecting all terms of potential valuation $L_1+2L_3$ from both polynomials we get
  \begin{equation}\label{eq:b40b04CaseIII}
    \tilde{b}_{40}:= a_{21}\,a_{22}^2\,m_1 + (3\,a_{20}\,a_{22}^2 + 2\,a_{21}\,a_{22}\,n_1)
 \quad \text{ and }\quad
  \hat{b}_{04}:= a_{22}^2\,m_1 -a_{22}\, n_1.
  \end{equation}

  Looking at both polynomials in~\eqref{eq:b40b04CaseIII}, we see from the valuation information on $m_1$ and $n_1$ that the valuation of each summand of the corresponding polynomial has value $L_1+2L_3$ for exactly one choice of $j$.
  Since $\overline{m_1}=-\overline{a_{10}}/\overline{a_{11}}$ when $j=1$, and  $\overline{n_1}= -\overline{a_{20}}\overline{a_{22}}^2/\overline{a_{21}}$ when $j=2$, we can replace these values on the relevant summand of $\tilde{a}_{40}$ and $\hat{a}_{04}$ to certify the proposed formulas for the initials of   $\tilde{a}_{40}$ and $\hat{a}_{04}$. 
\end{proof}

\begin{remark}\label{rm:markingCaseII}
  For case (II), the value of the index $j$ determines whether the points $(3,0)$ and $(0,3)$ are unmarked or not in the Newton subdivisions of $\tilde{f}(x,z)$ and $\hat{f}(z,y)$. Indeed, using~\eqref{eq:b30b03CaseII} we see that these two points will be marked  if, and only if, $\val(n_1)=L_1+L_3$, or equivalently, when $j=2$. In turn,  the point $(0,1)$ will always be marked in the subdivision of $\hat{f}(z,y)$. Indeed, the sum of all the terms in $\hat{a}_{01}$ of valuation $L_1$ equals $a_{21}^2(a_{11}m_1 - a_{10}) + 2\,a_{11}\,a_{20}\,a_{21}$, and this expression has valuation $L_1$ for both $j=1$ or $2$. 
\end{remark}

\subsection{Finding the tuple $(\ell,p,p',p'')$}\label{ssec:build-trop-curve-on-modification}
In this subsection, we turn our attention to  determining the values of the parameters $m_2, n_2$ and $\du_2$ featured in~\eqref{eq:parametersLineTreeShape} as well as the values of  the three tangency points $p$, $p'$ and $p''$.  Recall from~\autoref{ssec:build-an-algebr} that $\val(m_1), \val(n_1)>0$ and $\val(\du_1)>L_3$. The choice of $m_1, n_1,\du_1$ was determined by the desired to predetermine the  Newton subdivisions seen in  
{Figures}~\ref{fig:Case1TreeShape} and~\ref{fig:Cases23TreeShape} for the polynomials  $\tilde{f}(x,z)$ and $\hat{f}(z,y)$ from~\eqref{eq:tropicalCurvesModifiedTreeShape}. The choice depends on the location of the tropical tangency points on the input plane curve $\Lambda$. In what follows, we compute the values of $m_2, n_2$ and $\du_2$. Their valuations must satisfy $\val(m_2), \val(n_2)>0$ and $\val(\du_2)>L_3$.

The polynomial $h$ defined in~\eqref{eq:idealsModificationTree} induces a re-embedding of the $(1,1)$-curve $V(\ell)$ into $\K^3$ via the ideal $I_{\ell} =\langle \ell, z-h\rangle$.  The projections  of $V(I_{\ell})$ to the $xz$- and $zy$-planes are given, respectively, by the vanishing of the polynomials
\begin{equation}\label{eq:modificationLinexz_zy}
  \begin{aligned}
\tilde{\ell}(x,z) &:= \ell(x, \frac{z-(a_{11} + m_1)-(a_{21}+ n_1)x}{1+(a_{22} + \du_1)x})\,((1+(a_{22} + \du_1)x)^3 =  \sum_{i,j} \tilde{\ell}_{ij} x^i z^j,\\ 
\hat{\ell}(x,z) &:= \ell(\frac{z-(a_{11} + m_1)-y}{(a_{21}+ n_1)+(a_{22} + \du_1)y},y)\,((a_{21}+ n_1)+(a_{22} + \du_1)y)^3 = \sum_{i,j} \hat{\ell}_{ij} z^i y^j.
  \end{aligned}
\end{equation}
The Newton polytopes of both $\tilde{\ell}$ and $\hat{\ell}$ are trapezoids with five lattice points. The corresponding five coefficients have the following explicit formulas:
 \begin{equation}\label{eq:coefficientsLines}
  \begin{aligned}
    \tilde{\ell}_{00} &= m_2, \quad
  \tilde{\ell}_{10} = (a_{22} + \du_1) m_2 - (a_{11} + m_1) \du_2 + n_2,\quad
  \tilde{\ell}_{20} = (a_{22} + \du_1) n_2 - (a_{21} + n_1) \du_2, \\
  \tilde{\ell}_{01} &= 1, \quad
  \tilde{\ell}_{11} = \hat{\ell}_{11}= (a_{22} + \du_1) + \du_2\;; \quad
  \hat{\ell}_{00} = (a_{21} + n_1)m_2 - (a_{11} + m_1)n_2, \\
  \hat{\ell}_{01}& =  \tilde{\ell}_{10} - 2 n_2, \quad
\hat{\ell}_{02} = -\du_2, \quad   
 \text{and} \quad \hat{\ell}_{10}  =  (a_{21} + n_1) + n_2.\\
  \end{aligned}
\end{equation}

 \begin{remark}\label{rm:ValsAndInitialFormsProjectionsLine} By construction, we have that $\tilde{\ell}_{11} = \hat{\ell}_{11}=\du$ and $\hat{\ell}_{10} =n$. It follows from~\eqref{eq:FixValuationsanda12Tobe1} that $\val(\tilde{\ell}_{11}) = \val(\hat{\ell}_{11}) = L_3$ and $\val(\hat{\ell}_{10}) = 0$. In addition, we know that $\overline{\tilde{\ell}_{11}} = \overline{\hat{\ell}_{11}}= \overline{a_{22}}$ and $\overline{\hat{\ell}_{10}} = \overline{a_{21}}$.
\end{remark}

 \begin{remark}\label{rm:NewtonSubdivisionsCases23} The Newton subdivision of $\tilde{\ell}(x,z)$ and the marking of cells dual to the tangency points $P$, $P'$ and $P''$ is fixed for case (I), and it is shown in~\autoref{fig:Case1TreeShape}. In turn, there are various options for the Newton subdivisions of $\tilde{\ell}(x,z)$ and $\hat{\ell}(z,y)$ and the dual cells to $P, P'$ and $P''$, for cases (II) and (III).  The possibilities for $\hat{\ell}(z,y)$ are determined from the Newton subdivision of  $\tilde{\ell}(x,y)$ by observing the behavior of $\Trop\,V(\tilde{\ell})$ within the charts $\sigma_5$, $\sigma_7$ and $\sigma_9$. \autoref{fig:Cases23LinesNP} depicts representatives for all possibilities under our standard length assumptions on $\Gamma \cap \Lambda$. 
\end{remark}

\begin{figure}
\includegraphics[scale=0.12]{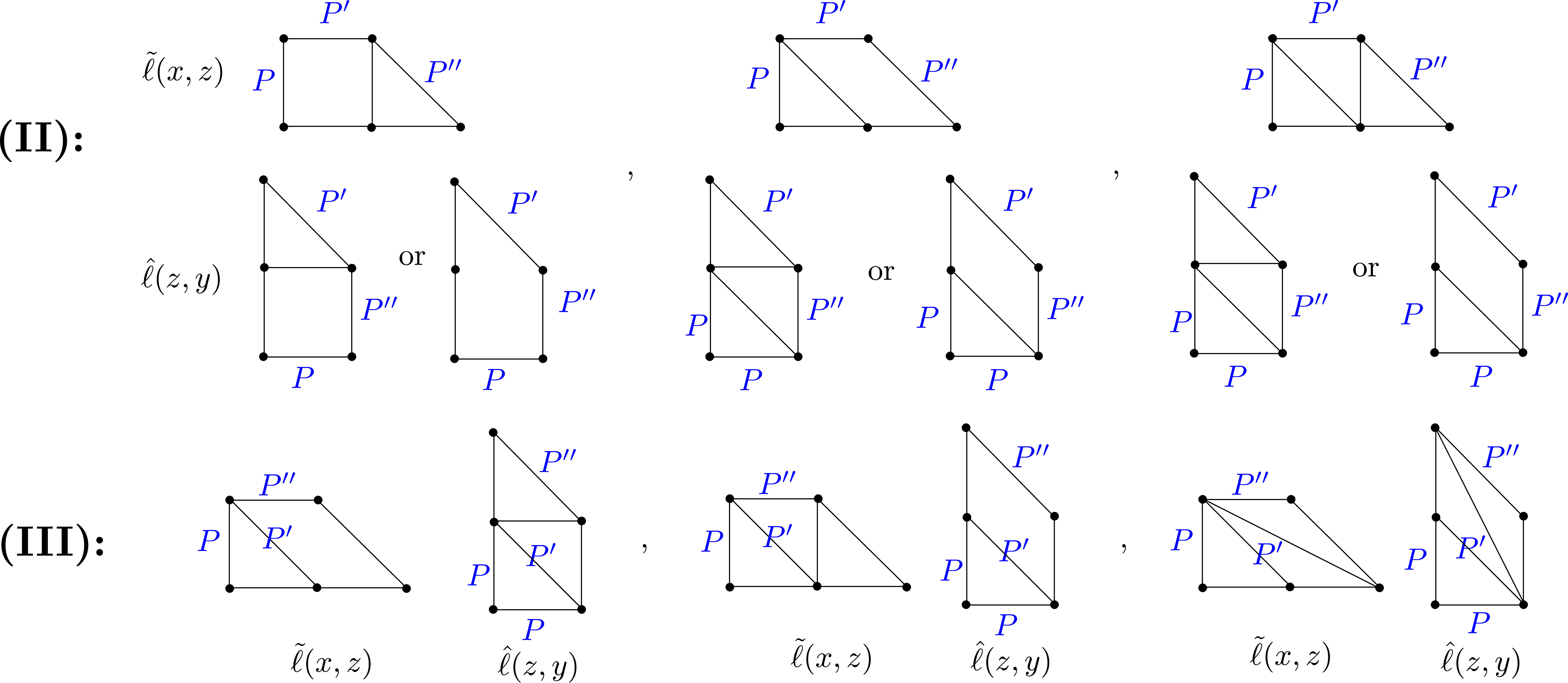}
\caption{Relevant cells in the Newton subdivisions of both $\tilde{\ell}(x,z)$ and $\hat{\ell}(z,y)$ for cases (II) and (III), assuming the lengths of $\Gamma \cap \Lambda$ are generic subject to $L_1<L_2, L_4, L_5$. For some of the Newton subdivisions of  $\hat{\ell}(z,y)$, there are two possible assignments for the cell dual to $P$.\label{fig:Cases23LinesNP}}
\end{figure}

As 
{Figures}~\ref{fig:Case1TreeShape} and~\ref{fig:Cases23TreeShape} show, each tropical tangency point lies in a different chart of the modification. Their location determines partial information about the coefficients listed in~\eqref{eq:coefficientsLines} and  the initial form of the planar projection of the corresponding classical tangency point. Here is the precise statement:

\begin{lemma}\label{lm:initialParamFromTangencies}
  A tropical tangency point between $\Gamma$ and $\Lambda$ in $\sigma_5$  fixes $\val(m_2)$ and $\overline{m_2}$. In turn, one in $\sigma_7$ fixes $\val(\tilde{\ell}_{10})$ and $\overline{\tilde{\ell}_{10}}$, one in $\sigma_8$ determines $\val(\tilde{\ell}_{20})$ and $\overline{\tilde{\ell}_{20}}$, whereas one in $\sigma_9$ fixes $\val(\du_2)$ and $\overline{\du_2}$. Furthermore, the initial forms of each parameter and the coordinates of  the corresponding projected classical tangency are defined over the same quadratic extension of $\resK$.
\end{lemma}

\begin{proof} The result follows by looking at the local equations at each of the corresponding tangency points. The label of the chart containing the point of interest determines the corresponding dual cell in the Newton subdivision of either $\tilde{\ell}$ or $\hat{\ell}$, as seen in 
  {Figures}~\ref{fig:Case1TreeShape} and~\ref{fig:Cases23TreeShape}. 
  
  We write the explicit systems for each chart, denoting the tangency point by $Q$ to avoid conflicts with the notation in the figures. 
  Thanks to~\autoref{rm:ValsAndInitialFormsProjectionsLine},  we can ensure that the systems involve exactly three unknowns, namely the initial forms of one coefficient of either $\tilde{\ell}(x,z)$ or $\hat{\ell}(z,y)$ and two coordinates of the classical tangency point. 
  
  First, assume that $Q$ lies in $\sigma_5$. Then, the dual cell $Q^{\vee}$ in the Newton subdivision of  $\tilde{\ell}(x,z)$ is the edge with vertices $(0,0)$ and $(0,1)$, whereas the one for $\tilde{\sextic}(x,z)$ is the triangle with vertices $(0,0)$, $(1,1)$ and $(2,0)$ featured in~\autoref{fig:ModificationsThreeShape}. Moreover, the point $(1,0)$ is unmarked in the Newton subdivision of $\tilde{f}$.   The local system at $Q$ is given by the vanishing of the following three equations:
  \begin{equation*}\label{eq:TangencySigma5}
\tilde{\sextic}_Q=\overline{\tilde{a}_{00}} + \overline{\tilde{a}_{11}}\,\bar{x}\,\bar{z} + \overline{\tilde{a}_{20}}\,\bar{x}^2, \quad \tilde{\ell}_{Q} =\overline{m_2} + \bar{z} \;\text{ and }\;  W_Q =  \det(\Jac(\sextic_Q,\ell_Q; \bar{x},\bar{y})) = 2\overline{\tilde{a}_{20}}\,\bar{x} + \overline{\tilde{a}_{11}}\,\bar{z}.
  \end{equation*}
  A direct computation reveals that this system has exactly two solutions $(\overline{m_2}, \bar{x}, \bar{z})$, namely
  \begin{equation}\label{eq:m2}
    (\overline{m_2}, \bar{x}, \bar{z}) = \left (\pm \frac{2}{\overline{\tilde{a}_{11}}} \sqrt{\overline{\tilde{a}_{20}}\,\overline{\tilde{a}_{00}}}, \pm \frac{1}{\overline{\tilde{a}_{20}}} \sqrt{\overline{\tilde{a}_{20}}\,\overline{\tilde{a}_{00}}}, \mp \frac{2}{\overline{\tilde{a}_{11}}} \sqrt{\overline{\tilde{a}_{20}}\,\overline{\tilde{a}_{00}}}\right).
    \end{equation}

  Second, suppose that $Q$ lies in $\sigma_7$. Then, we are in cases (I) or (III). The local equations arise from the polynomials $\tilde{\sextic}$ and $\tilde{\ell}$. The corresponding dual cells to $Q$ have vertices $\{(2,0), (2,1), (4,0)\}$ and $\{(1,0), (0,1)\}$, respectively. In addition, $(3,0)$ is unmarked in the Newton subdivision of $\tilde{\sextic}$. Thus, the local equations at $Q$ become
\[
\bar{x}^2(\overline{\tilde{a}_{20}} + \overline{\tilde{a}_{21}}\,\bar{z} + \overline{\tilde{a}_{40}} \,\bar{x}^2) = \bar{z} + \overline{\tilde{\ell}_{10}} \,\bar{x} = 2\,\overline{\tilde{a}_{40}}\,\bar{x} - \overline{\tilde{a}_{21}}\,\overline{\tilde{\ell}_{10}}
= 0.
  \]
  The system has two solutions, namely,
  \begin{equation}\label{eq:ltilde10}
(\overline{\tilde{\ell}_{10}}, \bar{x}, \bar{z})= \left ( \pm \frac{2}{\overline{\tilde{a}_{21}}} \sqrt{\overline{\tilde{a}_{20}}\,\overline{\tilde{a}_{40}}}, \pm \frac{1}{\overline{\tilde{a}_{40}}} \sqrt{\overline{\tilde{a}_{20}}\,\overline{\tilde{a}_{40}}},  - 2\,\frac{\overline{\tilde{a}_{20}}}{\overline{\tilde{a}_{21}}}\right ).
  \end{equation}

  Next, assume that $Q$ lies in $\sigma_8$. In this situation, we are in cases (I) or (II). The dual cells to $Q$ corresponding to $\tilde{f}$ and $\tilde{\ell}$ have vertices $\{(4,0), (4,1), (6,0)\}$ and $\{(1,1), (2,0)\}$, respectively.  The point $(5,0)$ is unmarked in the Newton subdivision of $\tilde{f}$. Thus, the local equations at $Q$ become
  \[
  \bar{x}^4(\overline{\tilde{a}_{40}} + \overline{\tilde{a}_{41}} \,\bar{z} + \overline{\tilde{a}_{60}}\,\bar{x}^2) = \bar{x}\,(\overline{a_{22}}\,\bar{z} +  \overline{\tilde{\ell}_{20}}\,\bar{x}) = 2\,\overline{\tilde{a}_{60}}\,\overline{a_{22}} \,\bar{x} - \overline{\tilde{a}_{41}} \,\overline{\tilde{\ell}_{20}} = 0.
  \]
  A direct computation gives the two solutions to this system:
  \begin{equation}\label{eq:ltilde20Over11}
    (\overline{\tilde{\ell}_{20}}, \bar{x}, \bar{z}) = \left (\pm 2\,\frac{\overline{a_{22}}}{\overline{\tilde{a}_{41}}} \,\sqrt{\overline{\tilde{a}_{40}}\,\overline{\tilde{a}_{60}}}, \pm \frac{1}{\overline{\tilde{a}_{60}}}\,\sqrt{\overline{\tilde{a}_{40}}\,\overline{\tilde{a}_{60}}}, - 2\,\frac{\overline{\tilde{a}_{40}}}{\overline{\tilde{a}_{41}}}\right ).
  \end{equation}

  Finally, we discuss the case when $Q\in \sigma_9$. In this situation, we are in cases (II) or (III). The combinatorial data can only be recovered from the $zy$-projections, i.e., from  $\hat{f}(x,y)$ and $\hat{\ell}(z,y)$. The vertices of $Q^{\vee}$ in the Newton subdivisions of each polynomial are $\{(0,4), (1,4), (0,6)\}$ and $\{(0,2), (1,1)\}$, respectively. The point $(0,5)$ is unmarked in the subdivision of $\hat{f}(x,y)$. Since the values of $\hat{\ell}_{02}$ and $\hat{\ell}_{11}$ are fixed by~\autoref{rm:ValsAndInitialFormsProjectionsLine}, the local equations at $Q$ become
  \[
\bar{y}^4(\overline{\hat{a}_{04}} + \overline{\hat{a}_{14}} \, \bar{z} + \overline{\hat{a}_{06}}\, \bar{y}^2) = \bar{y}(-\overline{\du_2}\,\bar{y} + \overline{a_{22}} \,\bar{z}) = \overline{\du_2}\,\overline{\hat{a}_{14}} + 2\,\overline{\hat{a}_{06}} \,\overline{a_{22}}\,\bar{y}=0.
  \]
We obtain the following two tuples of solutions:
\begin{equation}\label{eq:u2}
  (\overline{\du_2}, \bar{y}, \bar{z}) = \left ( \pm 2\,\frac{\overline{a_{22}}}{\overline{\hat{a}_{14}}} \,\sqrt{\overline{\hat{a}_{04}}\,\overline{\hat{a}_{06}}}, \mp \frac{1}{\overline{\hat{a}_{06}}}\,\sqrt{\overline{\hat{a}_{04}}\,\overline{\hat{a}_{06}}}, -  2\,\frac{\overline{\hat{a}_{04}}}{\overline{\hat{a}_{14}}}\right ).
\end{equation}

The above description of the dual cell to $Q$ in the Newton subdivisions of either $\tilde{\ell}$ or $\hat{\ell}$ allows us to  determine the valuation of the parameters $m_2$, $\tilde{\ell}_{10}$, $\tilde{\ell}_{20}$ or $\du_2$, according to which chart of the modification of $\RR^2$ along $\Lambda$ contains $Q$. In all cases, $Q^{\vee}$ is a primitive edge, with one vertex placed at a known height ($0$ for $Q\in \sigma_5\cup\sigma_7)$ and $-L_3$ for the other two charts), and the other one at height given by the negative valuation of the chart-dependent choice of parameter. Since the coordinates of $Q$ can be recovered from~\autoref{tab:TreeShapeVerticesNew3D} given the chart information, the unknown height 
can be determined by a simple linear algebra calculation. We obtain the following precise formulas:
\begin{equation}\label{eq:valsParam2n2du2TreeShape}
  \begin{aligned}  \val(m_2) & = \frac{L_1+L_2}{2}, 
\qquad\qquad\qquad\qquad    \qquad \val(\tilde{\ell}_{10}) = \frac{L_1+L_3 + \min\{L_4,L_5\}}{2}, 
      \\
\val(\tilde{\ell}_{20})  &= L_3+ \frac{L_5 + \min\{L_4,L_1+L_3\}}{2}, \;\quad    \val(\du_2)  = L_3 +\frac{L_4+ \min\{L_5,L_1+L_3\}}{2} 
     .
  \end{aligned}
\end{equation}

In order to prove that the four three tuples computed in expressions~\eqref{eq:m2} through \eqref{eq:u2} are defined over a quadratic extension of $\resK$ it is enough to notice that the initial forms of $\tilde{a}_{20}$, $\tilde{a}_{40}$, $\hat{a}_{04}$, $\tilde{a}_{00}$, $\tilde{a}_{60}$ and $\hat{a}_{06}$ lie in $\resK$. The claim for the first three parameters is case dependent, and follows from 
{Propositions}~\ref{pr:choiceForn1u1CaseI},~\ref{pr:choiceForn1u1CaseIII} and~\ref{pr:choiceForn1u1CaseII}. Explicit formulas for the initial forms of the last three coefficients can be found in~\autoref{lm:coefficientsProjectionXZ} (4), (5), and~\autoref{lm:coefficientsProjectionZY} (5), respectively. \end{proof}

\begin{remark}\label{rm:initialsJacobians}
  The proof of~\autoref{lm:initialParamFromTangencies}, combined with a simple computation in~\sage, confirms that the  expected initial forms of the determinant of each of the $3\times 3$-Jacobians associated to the local systems at each tangency point $Q$ is non-vanishing. The relevant three variables needed for each Jacobian depends on the chart containing $Q$. They are listed in expressions~\eqref{eq:m2} through~\eqref{eq:u2}.
\end{remark}

\begin{remark}\label{rm:missingEntriesForModifications} Even though each new chart of the modification involves only two out of the three coordinates of a classical tangency, the remaining entry of either $\bar{p}, \bar{p'}, \bar{p''}$ can be recovered from the relation $z = h(x,y)$ and~\autoref{lm:initialFormsTreeShape}. The value of the missing  coordinate depends on which chart contains the given tropical tangency. More precisely, we have $\bar{y}=-\overline{m}= -\overline{a_{11}}$ for a tangency in $\sigma_7$, whereas $\bar{y}= -\bar{n}\,\bar{x} = -\overline{a_{21}}\,\bar{x}$ for one $\sigma_7$, $\bar{y} = -\bar{n}/\bar{\du} = -\overline{a_{21}}/\overline{a_{22}}$ for one in $\sigma_8$, and $\bar{x}=-1/\bar{\du} = -1/\overline{a_{22}}$ for one in $\sigma_9$.
\end{remark}

Next, we present the  main statement in this subsection. It confirms that the local equations at the three tangency points have exactly eight tuples of solutions $(m_2, n_2, \du_2,p, p', p'')$, each defined over a quadratic extension of $\K$:

\begin{theorem}  \label{thm:LocalEquationsAfterModification}
The local equations for $\tilde{f}(x,z)$ and $\hat{f}(z,y)$ determining $P$, $P'$ and $P''$  as tropical tangencies impose independent conditions on the initial forms of three parameters $m_2$, $n_2$ and $\du_2$ and the tangent points $p$, $p'$ and $p''$. There are precisely eight solutions, all defined over the same quadratic extension of $\K$. Furthermore, either all of them lie in $\K$, or none of them does.
\end{theorem}

\begin{proof} As usual, we write $p=(x,y,z)$, $p'=(x', y', z')$ and $p''=(x'', y'', z'')$ for the values of the three tangency points in $\overline{\K}^3$. 
  The locations of the corresponding three tropical tangencies in $\Lambda$ are fixed for each of the three cases seen in~\autoref{fig:chipFiringTreeShapeIntersections}. 
  By~\autoref{rm:missingEntriesForModifications}, they can be used to recover $\bar{p}$, $\overline{p'}$ and $\overline{p''}$ from the formulas obtained in the proof of~\autoref{lm:initialParamFromTangencies}.
  
  Recall that one of these points (labeled $P$) is always in  the chart $\sigma_5$. 
  {Figures}~\ref{fig:Case1TreeShape} and~\ref{fig:Cases23TreeShape} confirms that the locations of the remaining two points, labeled $(P',P'')$, are $(\sigma_7, \sigma_8)$ for case (I),  $(\sigma_9,\sigma_8)$ for case (II), and $(\sigma_7, \sigma_9)$ for case (III). 
  By~\autoref{lm:initialParamFromTangencies},  the local equations at $P$ determine two solutions $(\overline{m_2}, \bar{x}, \bar{z})$ in all three cases. In turn, the local equations at $P'$ and $P''$ fix four different 6-tuples of initial forms, namely,
$(\overline{\tilde{\ell}_{10}}, \overline{x'}, \overline{z'};\overline{\tilde{\ell}_{20}}, \overline{x''}, \overline{z''})$ for case (I), 
  $(\overline{\du_2}, \overline{y'}, \overline{z'}; \overline{\tilde{\ell}_{20}}, \overline{x''}, \overline{z''})$ for  case (II) and $(\overline{\tilde{\ell}_{10}}, \overline{x'}, \overline{z'}; \overline{\du_2}, \overline{y''}, \overline{z''})$  for case (III). The same lemma confirms that in all cases, the eight 9-tuples of initial forms solving these three systems  are defined over a quadratic extension of $\resK$.

We wish to show that each of these solutions has a unique lift to a quadratic extension of $\K$. To this end, we look at the Jacobian of the three local systems with respect to nine suitable variables, namely, $(m_2, x, y; \tilde{\ell}_{10}, x', z'; \tilde{\ell}_{20}, x'', z'')$ for case (I), $(m_2, x, y; \du_2, y', z'; \tilde{\ell}_{20}, x'', z'')$  for case (II) and $(m_2, x, y; \tilde{\ell}_{10}, x', z'; \du_2, y'', z'')$ for case (III). A simple inspection reveals the these three Jacobians are block upper-triangular, with diagonal blocks recording the $3\times 3$-Jacobian of the local systems for $P$, $P'$ and $P''$, respectively.

  By~\autoref{rm:initialsJacobians}, the expected initial form of the determinants of each Jacobian matrix is non-vanishing. Thus, \autoref{lm:multivariateHensel} ensures that the  each of the 9-tuples of initial forms listed above have unique lifts over a quadratic extension of $\K$, namely the composite of $\K(\sqrt{\tilde{a}_{20}\,\tilde{a}_{00}})$ and either the extensions $\K(\sqrt{\tilde{a}_{20}\,\tilde{a}_{40}}, \sqrt{\tilde{a}_{40}\,\tilde{a}_{60}})$,
$\K(\sqrt{\hat{a}_{04}\,\hat{a}_{06}}, \sqrt{\tilde{a}_{40}\,\tilde{a}_{60}})$
 or $\K(\sqrt{\tilde{a}_{20}\,\tilde{a}_{40}}, \sqrt{\hat{a}_{04}\,\hat{a}_{06}})$ for cases (I), (II) and (III), respectively. In all cases, the eight lifts are defined over $\K$ if, and only if, one of them is, since the information solely depends on the whether the initial forms of these three square roots belong to $\resK$.

 To finish, we must show that these unique lifts determine unique tuples $(m_2, n_2, \du_2, p, p', p'')$ over the same quadratic fields. By construction, it suffices to prove the statement for $(m_2, n_2, \du_2)$. We claim that this triple can be obtained from the triple of parameters featured in each case (which are defined over over the same quadratic extension of $\K$), via a linear change of coordinates with coefficients in $\K$.  Indeed, using~\eqref{eq:coefficientsLines} we see that
 $(m_2, \tilde{\ell}_{10},     \tilde{\ell}_{20})^t = M_{1} (m_2,n_2, \du_2)^t$,  $(m_2, \tilde{\ell}_{20}, \du_2)^t = M_{2} (m_2,n_2, \du_2)^t$ and $(m_2, \tilde{\ell}_{10}, \du_2)^t = M_{3} (m_2,n_2, \du_2)^t$ for cases (I) through (III), respectively, where
 \begin{equation}\label{eq:matricesForOriginalm2n2d2Parameters}
M_1:=
 \begin{pmatrix}
   1 & 0 & 0\\
   a_{22}+\du_1 & 1 & -a_{11}-m_1\\
   0 & a_{22} + \du_1 & -a_{21}-n_1   
 \end{pmatrix},
 \quad
 M_2:=
 \begin{pmatrix}
   1 & 0 & 0\\
0 &    a_{22}+\du_1  & -a_{21}-n_1\\
   0 & 0 & 1
 \end{pmatrix} \quad \text{ and }
 \end{equation}
 \begin{equation*}
 M_3:=
 \begin{pmatrix}
   1 & 0 & 0\\
   a_{22}+\du_1 & 1 & -a_{11}-m_1\\
   0 & 0 & 1
 \end{pmatrix}.
 \end{equation*}
 
By construction, the determinants of these three matrices have non-vanishing expected initial forms, namely $-\overline{a_{21}}$, $\overline{a_{22}}$ and $1$, respectively. Thus, they are all invertible over $\K$. Since the eight triples of parameters obtained  earlier for each case are defined over $\K$ if, and only if, one of them is, the same is true for the eight tuples $(m_2, n_2, \du_2)$. This concludes our proof. 
\end{proof}

Our last statement in this section is obtined by combining the previous result with the  formulas in~\eqref{eq:valsParam2n2du2TreeShape}. It confirms that the valuation requirements on the three parameters $m_2, n_2$ and $\du_2$ hold. 

\begin{proposition}\label{pr:valm2n2du2} The parameters $m_2, n_2$ and $\du_2$ satisfy $\val(m_2), \val(n_2)>0$ and $\val(\du_2)>L_3$.
\end{proposition}

\begin{proof} We proceed by a case-by-case analysis, using the formulas in~\eqref{eq:coefficientsLines}  to recover $(n_2,\du_2)$ from $m_2$ and the pair of parameters determined in~\autoref{lm:initialParamFromTangencies} using the matrices $M_1$, $M_2$ and $M_3$ from~\eqref{eq:matricesForOriginalm2n2d2Parameters}. 
   The valuations of the corresponding parameters appear in~\eqref{eq:valsParam2n2du2TreeShape}. In particular, we know that $\val(m_2)=(L_1+L_2)/2>L_1>0$.

   For case (I), we recover $n_2$ and $\du_2$ from the matrix $M_1$ and the quantities $\tilde{\ell}_{10}, \tilde{\ell}_{20}$ and $m_2$ as
   \begin{equation}\label{eq:systemIVals}
   \begin{pmatrix}
     n_2\\
     \du_2
   \end{pmatrix}
   = \frac{1}{\det(M_1)}
\begin{pmatrix}
   -(a_{21} + n_1) & a_{11}+m_1\\
   -(a_{22} + \du_1) & 1
   \end{pmatrix}
   \begin{pmatrix}
   \tilde{\ell}_{10}-(a_{22}+\du_1)\,m_2\\
\tilde{\ell}_{20}   
   \end{pmatrix}.
   \end{equation}
   The conditions $\val(m_1), \val(n_1)>0$ and $\val(\du_1)>L_3$, in addition to the quantities listed in \eqref{eq:FixValuationsanda12Tobe1} fix the valuation of all entries in the  above matrix, namely,  $L_3$ for the $(2,1)$-entry and $0$ for the remaining ones. In particular, we have $\val(\det(M_1)) = 0$. In turn, the edge length restrictions on $\Gamma \cap \Lambda$ for case (I) and the formulas from~\eqref{eq:valsParam2n2du2TreeShape} yield
   \[\val(\tilde{\ell}_{10}) = \frac{L_3+L_1+L_4}{2} < L_3 + L_1 <L_3 + \frac{L_1+L_2}{2} = \val((a_{22}+d_1)m_2).
   \]
   Thus, we get $\val(\tilde{\ell}_{10}) = \val(\tilde{\ell}_{10} - (a_{22}+\du_1)m_2)$. This identity and~\eqref{eq:systemIVals} provide the desired lower bounds for $\val(n_2)$ and $\val(\du_2)$, namely:
\begin{equation*}
   \begin{aligned}
   \val(n_2) &\geq \min\{\val(\tilde{\ell}_{10}), \val(\tilde{\ell}_{20})\} = \frac{1}{2}\min\{L_1+L_3+L_4, 2L_3+L_4+L_5\}>\frac{L_3}{2}+L_1>0,\; \text{ and}\\
   \val(\du_2) &\geq \min\{L_3 + \val(\tilde{\ell}_{10}), \val(\tilde{\ell}_{20})\} = L_3 + \frac{1}{2}\min\{ L_3 + L_1+L_4, L_4+L_5\} >L_3 + 2L_4>L_3.
   \end{aligned}
\end{equation*}

The argument for cases (II) and (III) is simpler, since $\du_2$ features as a known parameter. The formula seen in~\eqref{eq:valsParam2n2du2TreeShape} confirms that $\val(\du_2)>L_3$. In turn, the valuation of $n_2$ can be determined using one of the following two  identities:
\[ \textbf{(II):}\;\; (a_{22}+\du_1)n_2 = \tilde{\ell}_{20} + (a_{21}+n_1)\du_2   \quad \text{ and }\quad \textbf{(III):}\;\; n_2 = \tilde{\ell}_{20} - (a_{22}+\du_1) m_2 +  (a_{11}+m_1)\du_2.
\]
Combining the conditions $\val(a_{22})=L_3<\du_1$, $\val(a_{21})=0<\val(n_1)$ and $\val(a_{11})=0<\val(m_1)$ with the formulas for $\val(\tilde{\ell}_{20})$ and $\val(\tilde{\ell}_{10})$ seen in~(\ref{eq:valsParam2n2du2TreeShape}), we obtain the desire lower bound for $\val(n_2)$ in both cases. More precisely, we have:
\begin{equation*}
  \begin{aligned}\textbf{(II):}& \val(n_2)  \geq-L_3 + \min\{\val(\tilde{\ell}_{20}), \val(d_2)\}\!= \min\{\frac{L_3\!+\!L_5\!+\!L_1}{2}, \frac{L_4\!+\!L_3\!+\!L_1}{2}\}>L_3 + L_1 >0,\\
    \textbf{(III):}& \val(n_2)  \geq \min\{\val(\tilde{\ell}_{10}), L_3\!+\!\val(m_2), \val(\du_2)\} \!= \min\{\frac{L_1\!+\!L_3\!+\!L_5}{2}, L_3\!+\!\frac{L_1\!+\!L_2}{2}, L_3 \!+\! \frac{L_4 \!+\! L_5}{2}\} \\
     &\quad \quad \quad\;
     >\min\{L_5, L_3+L_1, L_3 + L_5\} = L_5 >0.
  \end{aligned} \qedhere
\end{equation*}
   \end{proof}

\section{Proof of Theorems~\ref{thm:main2} and~\ref{thm:main3}}\label{sec:proofsMain2-3}

In the last five sections, explicit computations were carried out to confirm the  local lifting multiplicities formulas  recorded in~\autoref{tab:LiftingMultiplicities}. Our goal in this section is prove 
{Theorems}~\ref{thm:main2} and~\ref{thm:main3} using this data. Throughout, we assume that $\Gamma$ is generic in the sense of~\autoref{rm:genericityOfGamma}. Recall that this condition is only relevant in the presence of a type (3f) or (8) tangency.

Our first result explains the numerology behind the lifting multiplicities claimed in~\autoref{thm:main2}:

\begin{theorem}\label{thm:LiftingEachTritangentCurve} Assume that $\sextic$ is generic relative to $\Gamma$. Then, any given tropical  curve  $\Lambda$ tritangent to $\Gamma$ lifts to 0, 1, 2, 4 or 8 tritangent tuples $(\ell,p,p',p'')$ to $V(\sextic)$ defined over $\overline{\K}$. Each non-zero value arises as the product of the local lifting multiplicities of each tropical tangency, after appropriate interpretation.
\end{theorem}

\begin{proof} As we saw in~\autoref{sec:preliminaries} the initial data of each  tritangent tuple solves the local system of equations determined by each tropical tangency. Thus, in the presence of a tangency type with local lifting multiplicity zero (seen on the first row of~\autoref{tab:LiftingMultiplicities}), $\Lambda$ will not lift to a classical tritangent tuple.
For this reason,  we  assume from now on that each tropical tangency in $\Lambda$ has positive local lifting multiplicity.

  Our goal is to prove the validity of the product formula  whenever $\Lambda$ lifts.
  We have three different possibilities, depending on the number of connected components of $\Gamma \cap \Lambda$, which we label by $s$. We analyze the  $4$-valent and trivalent cases separately.

  If $\Lambda$ is 4-valent, then we know that $s=2$ or $3$. In the first situation, there are two options for the multiplicity four tangency, namely (4b') or (6b'). Since both local lifting multiplicities are one, the product formula holds as a consequence of~\autoref{cor:4bCrossLiftsWithNoInitHyperflex} and~\autoref{thm:lifting6b_Cross}, respectively. Similarly, if $s=3$, then the bidegree of $\Gamma$ and the absence of type (3a') tangencies combined  ensure one of the tangencies of $\Lambda$ is at its vertex. Its type is (4a') o (6a'), so the result  follows from~\autoref{pr:type4ap6ap}.

It remains to prove the statement in the trivalent case for all three values of $s$.  Following our conventions from~\autoref{fig:classificationLocalTangencies} we assume that the unique edge of $\Lambda$  has slope one.
  If $s=1$, then $\Lambda$ must  have a type (8) tangency. The result follows by~\autoref{thm:treeShapeIntersections8}. If $s=2$ or $3$, we must first determine what initial data of a tritangent tuple $(\ell, p,p',p'')$ is provided by the  local equations at each tropical tangency.

  \begin{figure}[t]
  \includegraphics[scale=0.45]{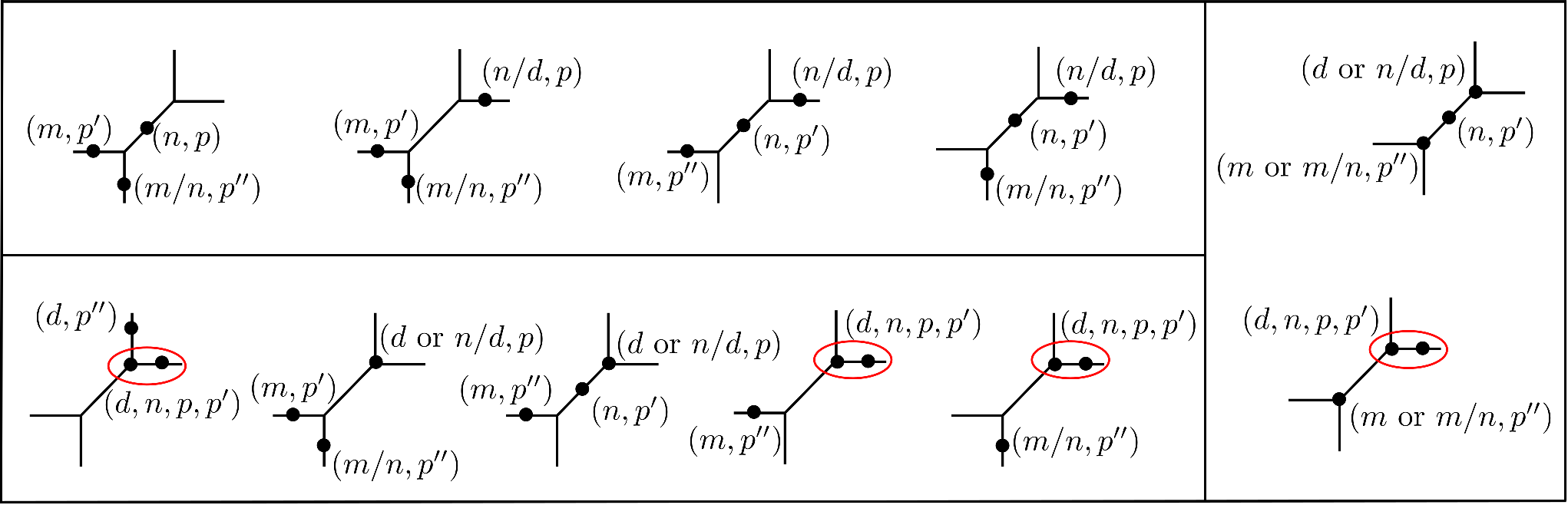}
  \caption{Data determined by the tropical tangency points when $\Lambda$ is trivalent and $\Gamma \cap \Lambda$ has three components, organized by the number of tangencies at vertices.\label{fig:parameters2x3Tangency}}
\end{figure}

  We start with the case $s=3$. Since $\Lambda$ lifts, we know no two tangencies can lie in the relative interior of the same edge or leg, by the genericity assumption on $\sextic$.
  \autoref{fig:parameters2x3Tangency} shows $\Dn{4}$-representatives of all remaining possible distribution of the three tangency points on  $\Lambda$. 
  The tuples listed in the picture indicate the information obtained from the local equations for each tangency type. They can be deduced from the results in 
  {Sections}~\ref{sec:trivalentLifts} and~\ref{sec:type-3-tangencies}, 
   as we now explain.

  By~\autoref{lm:type2Horiz}, a type (2a)  tangency at a point $P$ on the positive horizontal leg yields $(n/d,p)$, whereas one on the slope one edge of $\Lambda$ gives $(n,p)$. In turn, a type (3a) or (3c) tangency at a point $P$ on the negative horizontal leg of $\Lambda$ determines $(m,p)$ by~\autoref{rk:type3a3choriz}. The formulas for tangencies at other legs  of $\Lambda$ are obtained by invoking the action of $\Dn{4}$ described in~\autoref{tab:D4Action}. In turn, a type  (3ac), (3cc) or (3aa) tangency on $\Lambda$ determines $(n,p)$ by 
  {Propositions}~\ref{pr:Prop5.2diag} and~\ref{pr:3aaD}, respectively.

  If a tangency $P$ lies on a vertex of $\Lambda$, its type can be (4a), (6a) or (5a). The assignment of the corresponding tuple of parameters depends on whether or not any leg adjacent to it contains a tangency in its relative interior and whether the tangency at $P$  is horizontal, vertical or diagonal. If no adjacent leg adjacent to $P$ carries a tangency, the tuple for $P$ is obtained from the representative cases treated in 
  {Propositions}~\ref{pr:4a6aNoAdjacentLeg} and~\ref{pr:5aWithOrWithoutAdjacentLeg}, using the action of $\Dn{4}$ when necessary. Finally, if a leg adjacent to $P$ also carries a tangency, we group both tangencies (and their local equations) and determine the corresponding joint tuple from~\autoref{pr:4a6aWithAdjacentLeg} and~\autoref{rm:(5a)withAdjacentLegQuadratic}.
  
  Next, assume $s=2$. In this case, ~\autoref{tab:LiftingMultiplicities} confirms that the multiplicity four component of $\Gamma \cap \Lambda$ has type (3f),  (3h), (3d), (5b) or (6b). The bidegree of $\Gamma$, our genericity assumption on $\sextic$ and the liftability of $\Lambda$ combined restrict the location and type of the complementing tangency on $\Lambda$, which we label as $P''$. More precisely, for type (3f), $P''$ can have type (2a), (3a), (3c), (4a), (6a) or  (5a). In turn, for types (3d) and (3h), $P''$ can only be of type (2a) or (3c) and must belong to  a fixed leg of $\Lambda$. Finally, for types (5b) and (6b), $P''$ must be a horizontal  tangency and  of type (2a), (3c), (4a) or (6a).
  As with the $s=3$ case, we record the tuple of parameters obtained for each combination (up to $\Dn{4}$-symmetry) in~\autoref{fig:parameters4-2Tangency}. The information is obtained from the proofs of 
  {Propositions}~\ref{pr:3f},~\ref{pr:type3h},~\ref{pr:type3d} and~\autoref{thm:liftingFormulasMult4}, respectively.

  Equipped with the partition of parameters recorded in 
  {Figures}~\ref{fig:parameters2x3Tangency} and~\ref{fig:parameters4-2Tangency} we can now establish the product formula for the lifting multiplicity of $\Lambda$. Indeed, whenever the three parameters of $\ell$ cannot be obtained from the tuples, the genericity of $\sextic$ confirms that such tritangents will not lift since the local equations will yield an overdetermined system of 9 equations in at most 8 unknowns. In all other cases, the Jacobian of the corresponding systems for the local equations becomes block upper-triangular. The proofs of all statements referenced earlier show that the expected initial form of the determinant of each diagonal block is non-vanishing. Thus, using~\autoref{lm:multivariateHensel} we conclude that all the local solutions (after grouping if necessary as indicated in~\autoref{fig:parameters2x3Tangency}) can be combined to produce a unique tritangent tuple $(\ell,p,p',p'')$ from this initial data. This confirms that the product formula holds in the trivalent case as well.
\end{proof}

  \begin{figure}
  \includegraphics[scale=0.45]{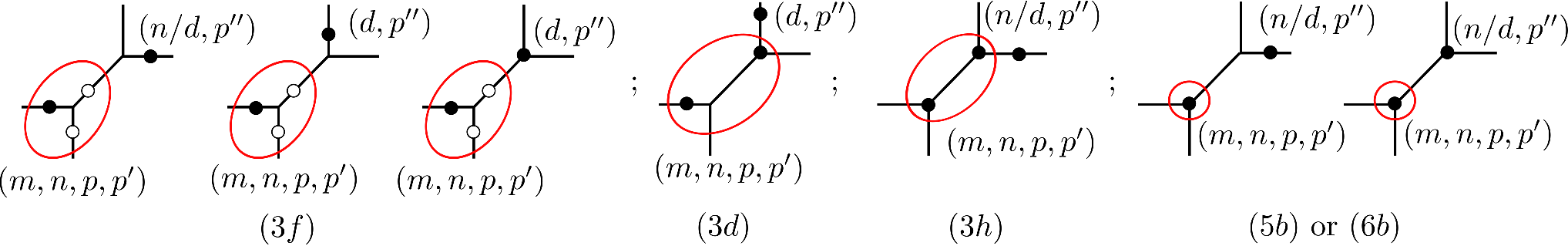}
  \caption{Data determined by the tropical tangency points when $\Lambda$ is trivalent and $\Gamma \cap \Lambda$ has two components. The higher multiplicity ones are highlighted in red.\label{fig:parameters4-2Tangency}}
    \end{figure}

  \begin{remark} The previous statement establishes that the lifting multiplicity of each tropical tritangent to $\Gamma$ is either 0 or a power of two. On the other hand, \autoref{thm:15x8=120} confirms that the  lifting multiplicity of any given tritangent class to $\Gamma$ over $\overline{\K}$ equals eight. Thus, it is natural to ask, which combinations of non-zero lifting multiplicities  can occur within a tritangent class to a fixed   $\Gamma$. An answer to this question is provided in the companion article by some of the present authors~\cite{CLMR25Partition}.
  \end{remark}

  Equipped with the multiplicity formulas for lifting each tritangent over $\overline{\K}$ we turn our attention to arithmetic considerations. More precisely, we wish to determine the field of definition of each tritangent tuple to a fixed $\Gamma$. We start our discussion with a definition:

  \begin{definition}\label{def:totallyKRational}
  Let $V(\ell)$ be a tritangent $(1,1)$-curve to a smooth $(3,3)$-curve $V(\sextic)$ defined over $\K$. Let $\LL$ be an algebraic field extension of $\K$. We say $V(\ell)$  is totally $\LL$-rational if it is $\LL$-rational and all tangency points between $V(\ell)$ and $V(\sextic)$ lie in $\pr^1_{\LL}\times \pr^1_{\LL}$.
\end{definition}

The following statement yields~\autoref{thm:main3} as a corollary.   Recall that a quadratic extension of $\K$ is any field extension obtained as a (possibly empty) tower of degree two field extensions of $\K$. 

\begin{theorem}\label{thm:rationalTotallyRatQuadraticExtension}
  Let $\Lambda$ be  tropical tritangent to $\Gamma$ with positive lifting multiplicity, and fix a tritangent tuple $(\ell, p,p',p'')$ lifting of $\Lambda$. Let $\LL$ be the field extension of $\K$ obtained by adjoining  the parameters $m,n,\du$ of $\ell$. Then:
  \begin{enumerate}[(i)]
  \item \label{item:quadraticExtension}  $\LL$ is a quadratic extension of $\K$;
  \item \label{item:sameFieldForAllLifts} any other lift of $\Lambda$ is defined over $\LL$;
  \item \label{item:LTotallyRatl} if $\Lambda$ contains  no tangency of type (5b), (6b), (4b') or (6b'), the lift $V(\ell)$ is totally $\LL$-rational;
  \item \label{item:exceptions} in all other cases,  the tangencies $p,p,p''$ are defined over a degree two extension of $\LL$.
  \end{enumerate}
\end{theorem}

\begin{proof} 
  We establish each item of the statement separately, starting with the first two. In particular, we use the explicit formulas for the solutions to all local equations at tropical tangencies between $\Lambda$ and $\Gamma$ provided in 
  {Sections}~\ref{sec:trivalentLifts} through~\ref{sec:tang-mult-six}.

  The proof of the product formula provided in~\autoref{thm:LiftingEachTritangentCurve} confirms that the parameters of $\ell$ are determined uniquely by combining the solutions to the local equations corresponding to each tangency (or group of tangencies). If the local lifting multiplicity of a given tangency type equals one, the initial forms of the relevant parameters solving  the corresponding local equations are elements of $\resK$. In turn, for multiplicities two, four or eight, the initial forms for the relevant parameters (or ratios thereof) are obtained as solutions to one, two or three quadratic equations over $\resK$, depending on the number of  tangency points in the corresponding component of $\Gamma \cap \Lambda$. These two observations confirm items~\eqref{item:quadraticExtension} and~\eqref{item:sameFieldForAllLifts}.

  To establish~\eqref{item:LTotallyRatl} it suffices to show that ignoring the four exceptional tangency types, the formulas for the initial forms of the classical tangencies lifting the tropical ones lie in the extension of $\resK$ generated by the initials of the relevant parameters of $\ell$. Thus, applying~\autoref{lm:multivariateHensel} to obtain the tritangent tuple from the initial values confirms that $p,p'$ and $p''$ are defined over $\K(m,n,\du)$. This confirms the totally $\LL$-rationality of any $\LL$-rational tritangent outside the exceptional cases.

Finally, item~\eqref{item:exceptions} follow from the proofs of the local lifting multiplicity values for the exceptional tangency types. Indeed, for types (5b), (6b) and (4b'), the formulas for $\bar{p}$ and $\bar{p'}$ provided in~\autoref{tab:liftingMult45b6b} and ~\autoref{cor:valuesMult4Type4}  confirm that the initial values of the tangency points are defined over $\resK(\sqrt{3})$ or $\resK(\sqrt{2})$, respectively. In turn, for type (6b'), the values listed in~\eqref{eq:values6b_Cross} show that $\bar{p}$ and $\bar{p'}$ are defined over $\resK(\sqrt{\Delta})$ for a fixed element $\Delta\in \resK$. The result follows from these facts and~\autoref{lm:multivariateHensel}.
\end{proof}

\section{Avoidance loci and real lifts of tropical tritangents}\label{sec:avoidance-loci-real}

In this section we discuss liftings of tropical tritangents over real closed non-Archimedean valued fields. To highlight this property, we write such fields as  $\KR$, and let $\K$ be the corresponding algebraic closure. By construction, we have  $\K= \overline{\KR}=\KR(\sqrt{-1})$. In addition to $\RR$, a useful example to keep in mind is the field of real Puiseux series $\PSR$, where  $x>y$ if the leading coefficient of $x-y$ is positive.

\begin{remark}\label{rm:valRCExtension} Recall that, by definition, the valuation on $\KR$ is compatible with the order, that is $0< a < b$ implies $\val(a)\geq \val(b)$. This condition implies the existence of a unique valuation on $\K$ extending that of $\KR$. More precisely, we have $\val(a+\sqrt{-1}\,b) = \min \{\val(a), \val(b)\}$ for any $a,b\in \KR$.
\end{remark}

When $\KR=\RR$, results of Gross and Harris~\cite{gro.har:81} confirm that the number of real tritangent planes to smooth real space sextics depends both on topological data on the real locus, namely, the number of connected components of the real locus and whether or not it disconnects the corresponding complex Riemann surface. The number of such planes can be 0, 8, 16, 24, 32, 64 or 120.


Given a smooth space sextic $C$ define over $\KR$, we consider the tropical tritangent planes to $\Trop\,C$ in $\TPr^3$. Tropical tritangent classes of $\Trop\,C$ record continuous deformations of these planes that preserve the tritangency condition. We identify each tropical plane with the location of its unique vertex. Thus, each such  class becomes a closed connected set in $\TPr^3$. Furthermore,~\cite[Corollary 3.5]{CLMR25Partition} confirms they can be endowed with  a  rational polyhedral complex structure.

The tropical method  allows us to prove that, near the tropical limit, the total count of $\KR$-rational tritangent planes to such curves  is always a multiple of 8. Here is our main result:

\begin{theorem}\label{thm:realLifts} Let $C$ be a smooth space sextic curve defined over $\KR$. Assume that  $\Trop\, C\subseteq \TPr^3$ is smooth. Then, each tropical tritangent class to $\Trop \,C$  has either zero or exactly eight lifts defined over $\KR$. 
\end{theorem}

The proof of this statement involves both avoidance loci~\cite{kum:19} and  convex closures of sets in $\check{\pr}_{\KR}^3$. Our approach follows analogous results by the third author, Payne and Shaw~\cite{MPS24} regarding the count of real bitangent lines to smooth plane quartics. Following Tarski's Principle, we refer to any geometric object defined over $\KR$ or with coordinates in this field as real. 

\begin{definition} The \emph{avoidance locus}  of a smooth curve $C\subseteq \pr_{\K}^3$ defined over $\KR$ is the set of planes defined over $\KR$ that do not intersect the real locus $C(\KR)$ of $C$.  We denote it by $\mathcal{A}_C$.
\end{definition}
The avoidance locus is locally closed in the Euclidean topology on the real dual projective space $\check{\pr}_{\KR}^3$. Note that with the exception of $\RR$, real closed fields are totally disconnected in the order topology. Thus, as explained in~\cite{MPS24}, rather than topological connectedness, we must use a notion of definable connectedness to characterize the components of $\mathcal{A}_{C}$ (see~\cite[Chapter 6]{vdDries:98}). A definable set is \emph{definably connected} if it cannot be written as a disjoint union of two non-empty definable open sets. For simplicity, we refer to the definable connected components of $\mathcal{A}_{C}$ as its components.

For degree reasons, any real tritangent plane to $C$ contains a real tangency point. In turn, non-real tangencies correspond to a pair of complex conjugate points (i.e., points that are Galois conjugate in the algebraic closure  $\K$ of $\KR$.). Thus, when $C(\KR)$ is empty, none of the tritangent planes to $C$ are real. 

By~\cite[\S 4]{kum:19}, when $C(\KR)$ is non-empty, the avoidance locus of $C$ has 1,2, 4, 8 or 15 components. It follows from~\cite[Corollary 2.4]{kum:19} and  the correspondence between tritangents and real odd theta characteristics, together with the non-vanishing real definite differential in $H^0(C, \Omega_{X|\KR})$ associated to them that all real tritangent planes to $C$ lie in the closure of $\mathcal{A}_C$. Furthermore, each component of $\mathcal{A}_{C}$ has precisely eight real tritangent planes in its closure. Furthermore, each of them lies in the closure of a unique component of $\mathcal{A}_{C}$. Thus, the real tritangents to $C$ are canonically partitioned into sets of eight.

\begin{remark}\label{rm:convexComponentsAC} Tarski's Principle combined with the natural generalization of \cite[Proposition 5.1]{KKPSS17} to smooth curves in $\pr^n_{\RR}$ 
  confirm that the cone in $\KR^4$ over the Euclidean closure of each  component of $\mathcal{A}_C$ is a convex set. 
\end{remark}

The following two lemmas relate classical avoidance loci and tropical tritangent classes:

\begin{lemma}\label{lm:tropicalizeAC}
  Let $C$ be a smooth sextic curve in $\pr_{\K}^3$ defined over $\KR$. If $Y$ is an element of $\mathcal{A}_C$, then $\Trop\,Y$ is a tropical tritangent to $\Trop \,C \subseteq \TPr^3$.
  \end{lemma}
\begin{proof} Since $Y$ does not meet $C(\KR)$, it follows that $Y$ intersects $C$ in three pairs of complex conjugate points. Since complex conjugate points in $\pr^3_{\K}$ have equal tropicalization by~\autoref{rm:valRCExtension}, it follows that  each connected component of $\Trop\,Y\cap \Trop\,C$ has stable intersection multiplicity 2, 4 or 6. The precise number depends on how many pairs lie in the given component after tropicalization (see~\cite[Theorem 6.4]{ossrab:13}. Thus,  $\Trop\, Y$ is tritangent to $\Trop \,C$.
\end{proof}

\begin{lemma}\label{lm:tropConnectedComponentsAC} Let $Y_1$ and $Y_2$ be two planes  in the same component of $\mathcal{A}_{C}$. Then, their tropicalizations in $\TPr^3$ belong to the same tropical tritangent class of $\Trop \,C$.
\end{lemma}

\begin{proof} Recall from~\autoref{rm:convexComponentsAC} that the cone in $\KR^4$ over each component of $\mathcal{A}_C$ is convex. Thus, the line segment $Z$ in $\check{\pr}_{\KR}^3$ joining $Y_1$ and $Y_2$ lies in the closure of the component of $\mathcal{A}_C$ containing both $Y_1$ and $Y_2$.

  Let $a, a'$ be the points in $\check{\pr}_{\KR}^3$ corresponding to $Y_1$ and $Y_2$, respectively.  Write  $v=\Trop\, a$ and $v'=\Trop\, a'$. A direct computation confirms that the tropical line segment
  \[
  \operatorname{tconv}(v,v'):=\{b\odot v \oplus b' \odot v': b\oplus b' = 0\}\subseteq \check{\TPr}^3
  \]
  is contained in the Euclidean closure of $\trop(Z)$.   By~\cite[Proposition 3]{dev.stu:04}, $  \operatorname{tconv}(v,v')$ is a concatenation of at most three Euclidean line segments, whose direction is a zero-one vector. In particular,  it produces a continuous path  deforming the tropical tritangent $\Trop\, Y_1$ into $\Trop\, Y_2$.
Since the tropicalization of each point in $Z$ is a tropical tritangent to $\Trop \,C$ by~\autoref{lm:tropicalizeAC}, we conclude that this deformation preserves the tritangency condition. Thus, both tropical tritangent planes belong to the same tropical tritangents class of $\Trop \,C$.
\end{proof}

\begin{corollary}\label{cor:closureOfClassesMapToTritangentClasses} Let $S$ be a component of $\mathcal{A}_C$ and  $\overline{S}$ be its closure in ${\check{\pr}_{\KR}}^3$. Then, the tropicalization of all members of  $\overline{S}$ lies in the same tropical tritangent class of $\Trop \,C$.
\end{corollary}

\begin{proof} The result is a direct consequence of~\autoref{lm:tropConnectedComponentsAC} and the equality $\Trop(S) = \Trop(\overline{S})$. The latter holds since infinitesimal deformations of a point in $S$ associated to a plane $H$,  with respect to the non-Archimedean norm on $\KR$, produce planes with the same tropicalization as $H$.
\end{proof}

Our next result generalizes~\cite[Theorem 2]{MPS23} and establishes~\autoref{thm:realLifts} as a direct corollary:

\begin{theorem}\label{thm:tropicalizeAC} Let $C$ be a smooth sextic curve in $\pr_{\KR}^3$ with smooth tropicalization $\Trop\, C\subseteq \TPr^3$, and $S$ be a fix connected component of $\mathcal{A}_C$. Then,
  the tropicalization of the closure of $S$ in $\check{\pr}_{\KR}^3$
  lies in a unique tropical tritangent class to $\Trop \,C$.
  Furthermore, such a  class 
  has eight lifts  over $\KR$. 
\end{theorem}

\begin{proof} We fix a component $S$ of $\mathcal{A}_C$ and let $\overline{S}$ be its closure in ${\check{\pr}_{\KR}}^3$.
By~\autoref{cor:closureOfClassesMapToTritangentClasses}, there exists  a tritangent class $\Omega$ to $\Trop\,C$ containing all tropical planes $\Trop H$ with $H \in \overline{S}$.  In particular, the eight real tritangents planes to $C$ in $\overline{S}$ tropicalize to members of  $\Omega$. 

By~\cite[Theorem 4.5]{JL18}, each tritangent class contains the tropicalization of exactly eight classical tritangents defined over $\K$. Thus, the tropicalization of the closure of two distinct components of $\mathcal{A}_C$ cannot lie in the same tropical tritangent class, so $S$ has precisely eight $\KR$-lifts. 
\end{proof}

\begin{proof}[Proof of~\autoref{thm:realLifts}] Let $S$ be a tritangent class to $\Trop\,C$ with at least one member lifting to a $\KR$-rational tritangent plane $H$ to $C$. Since $H$ lies in the closure of a single component of $\mathcal{A}_C$, it follows from~\autoref{thm:tropicalizeAC} that $S$ has precisely eight lifts over $\KR$.  
\end{proof}

The following result is a direct consequence of our proof methods and it is central to our constructions in the next section.

\begin{corollary}\label{cor:totallyReal}
  Let $C$ be a smooth sextic curve $\pr_{\KR}^3$. Assume $C$ is contained in the standard Segre surface, and $\Trop \,C$ is smooth when viewed in $\TPr^1\times \TPr^1$. If $C$ is generic relative to $\Trop\,C$ and $V(\ell)$ is a tritangent plane to $C$ defined over $\KR$, then all tangency points between $C$ and  $V(\ell)$ are also defined over $\KR$.
\end{corollary}

\begin{proof} We view $C=V(\sextic)\subseteq \pr^1\times \pr^1$ where $\sextic \in \KR[x,y]$ is the corresponding $(3,3)$-polynomial.  Since $\sextic$ is generic relative to $\Trop\,C\subseteq \TPr^1\times \TPr^1$, we know that $C$ has no hyperflexes.
  
  Fix any tropical tritangent tuple $(\Lambda, P, P',P'')$ to $\Trop\,C$, and assume $(\ell,p,p',p'')$ is a lift of $\Lambda$ defined over $\KR$.   Since the tuple $(\ell,p,p',p'')$ is uniquely determined by the corresponding initial data through~\autoref{lm:multivariateHensel}, the $\KR$-rationality of the lift $V(\ell)$ authomatically implies its total $\KR$-rationality  whenever  $\bar{p}$, $\overline{p'}$, $\overline{p''} \in \resKR^2$ where $\resKR$ dentoes the residue field of $\KR$. The latter can be determined from our lifting formulas, as we now explain.

 \autoref{thm:rationalTotallyRatQuadraticExtension} (\ref{item:LTotallyRatl}) confirms that in the absence of a tropical hyperflex, the lift $V(\ell)$ is totally $\KR$-rational. Thus, it suffices to analyze the behavior in the presence of a tangency of type (5b), (6b), (4b') ot (6b').    
 In the first three cases, ~\autoref{tab:liftingMult45b6b} and~\autoref{cor:valuesMult4Type4} ensures  the total $\KR$-rationality since  $\sqrt{2}, \sqrt{3} \in \resKR$. For type (6b'), the behavior depends on whether or not the quantity $\sqrt{\Delta}$ from~\eqref{eq:values6b_Cross} lies in $\resKR$. Since $\Delta = (\bar{b}\overline{\lambda''}-\bar{a})^2 + 2(\bar{b}^2\overline{\lambda''}^2 + \bar{a}^2)\geq 0$, total $\KR$-rationality also holds in the presence of this tangency type.
\end{proof} 



\section{Constructing totally real tritangents}\label{sec:examples}

As we discussed in~\autoref{sec:avoidance-loci-real}, the number of real tritangents to a smooth real sextic curve $C$ in $\pr_{\RR}^3$ depends both on the topology of  the real curve and its relative position within the corresponding complex Riemann surface. In turn, work of Kummer~\cite{kum:19} provided lower and upper bounds for the number of totally real tritangent planes to $C$ given this topological data. While all but one of these bounds is know to be sharp by~\cite{HKSS,KRSMS}, \autoref{cor:totallyReal} confirms that the tropical techniques developed in the present paper can provide a new source of examples with maximal number of totally real tritangents for curves near their tropical limit.

For simplicity, we assume that $C$ is contained in the Segre quadric surface in $\pr^3_{\PSR}$.
In this setting, Viro's patchworking method~\cite{IMS09,Viro} determines the topological type of the real locus of $C$ solely in terms of $\Gamma$ and a distribution of signs for all 16 coefficients of the defining  bidegree $(3,3)$-polynomial $\sextic$. When $\K = \PSR$, the sign of a series matches that of its initial form in $\RR$.

In what follows, we provide two concrete examples, with 64 and 120 totally real tritangents, respectively, for a fixed sign assignment to all coefficients of $\sextic$. The corresponding tropical curves  appear in 
{Figures}~\ref{fig:example64} and~\ref{fig:honeycomb}, respectively. Members for each of the 15 tropical tritangent class are recorded by a pair of vertices, labeling the corresponding class. A vertex with a multilabel is shared by two or more tritangents $\Lambda$.

\autoref{tab:signConditionsExamples} shows the sign conditions for deciding realness of lifts of some local tangencies types relevant to the examples below. They can be obtained from the corresponding rules in~\cite[Table 10]{CM20} for lifting tropical bitangent classes to smooth plane quartic curves over $\PSR$, using the $\Dn{4}$-action determined by the maps $\tau_0$ and $\tau_1$ from~\autoref{tab:D4Action}.

            \begin{table}[tb]
              \begin{tabular}{|c|c||c|c|}
                \hline
Type &                  $(\Lambda,P)$  & Coefficient &  Conditions for real solutions\\
                \hline\hline
 & \multirow{2}{*}{\includegraphics[scale=0.3]{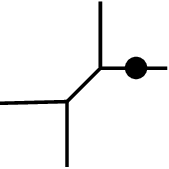} \quad \includegraphics[scale=0.3]{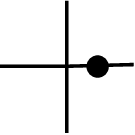} \quad\includegraphics[scale=0.3]{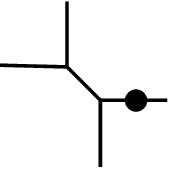}} & \multirow{2}{*}{$\bar{n}/\bar{\du}$} &  \multirow{4}{*}{$ (-s_{u,v}\,s_{u,v+1})^{r+w}s_{u+1,r}\,s_{u-1,w}>0$}  
     \\   & &  & \\ \cline{2-3}
\multirow{2}{*}{(3c)}&  \multirow{2}{*}{\includegraphics[scale=0.3]{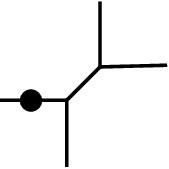}  \quad \includegraphics[scale=0.3]{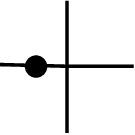} \quad\includegraphics[scale=0.3]{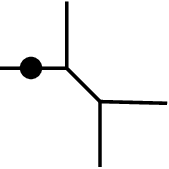}} &  
  \multirow{2}{*}{$\overline{m}$} & \\ & & & \\
\cline{2-4}
\multirow{2}{*}{(3c')} & \multirow{2}{*}{\includegraphics[scale=0.3]{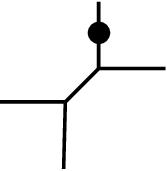}   \quad \includegraphics[scale=0.3]{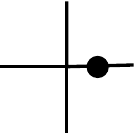} \quad\includegraphics[scale=0.3]{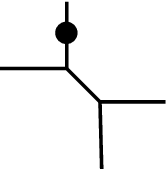}} & \multirow{2}{*}{$\bar{\du}$} &  \multirow{4}{*}{
$ (-s_{u,v}\,s_{u+1,v})^{r+w}s_{r,v+1}\,s_{w,v-1}>0$}
       \\&&&  \\ \cline{2-3}
 &  \multirow{2}{*}{\includegraphics[scale=0.3]{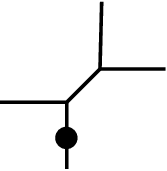}  \quad \includegraphics[scale=0.3]{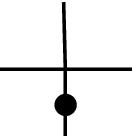} \quad\includegraphics[scale=0.3]{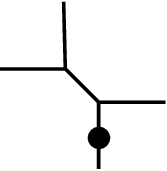}} &  
           \multirow{2}{*}{$\overline{m}/\bar{n}$} & \\
           & & & \\
 \hline \hline
  \multirow{4}{*}{ (3cc)} &  \multirow{2}{*}{ \includegraphics[scale=0.3]{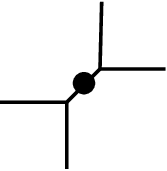}}  & \multirow{2}{*}{$\bar{n}$} & \multirow{2}{*}{$(-s_{u,v}\,s_{u-1,v+1})^{r+w}s_{u+v-r-1,r}\,s_{u+1-w+1,w}>0$} \\
 & & &\\
\cline{2-4}
& \multirow{2}{*}{ \includegraphics[scale=0.3]{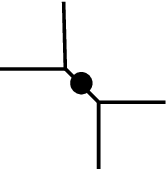}}  & \multirow{2}{*}{$\overline{m}/\bar{\du}$} & \multirow{2}{*}{$(-s_{u,v}\,s_{u+1,v+1})^{r+w} s_{u-v+1+r,r}\,s_{u-v-1+w,w}>0$}\\
 & & &\\
        \hline\hline
\multirow{16}{*}{(3a)} & \multirow{2}{*}{\includegraphics[scale=0.3]{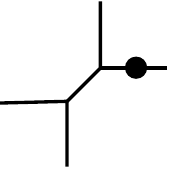}}  & \multirow{4}{*}{$\bar{n}/\bar{\du}$} &  \multirow{2}{*}{$ -(-s_{u,v}\,s_{u,v+1})^{r+v+1}s_{u+1,r}\,s_{u,v}\sign(\bar{n})>0$}
\\ &&&\\\cline{2-2} \cline{4-4} 
& \multirow{2}{*}{\includegraphics[scale=0.3]{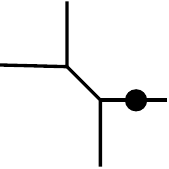}}  &  &  \multirow{2}{*}{$ -(-s_{u,v}\,s_{u,v+1})^{r+v}s_{u+1,r}\,s_{u,v+1}\sign(\overline{m}/\bar{\du})>0$}
\\&&& \\ \cline{2-4} 
& \multirow{2}{*}{\includegraphics[scale=0.3]{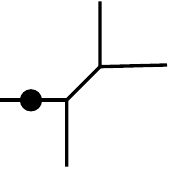}}  & \multirow{4}{*}{$\overline{m}$} &  \multirow{2}{*}{$ -(-s_{u,v}\,s_{u,v+1})^{w+v}s_{u-1,w}\,s_{u,v+1}\sign(\bar{n})>0$}
\\ &&&\\\cline{2-2} \cline{4-4} 
 & \multirow{2}{*}{\includegraphics[scale=0.3]{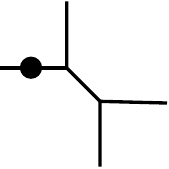}}  &  &  \multirow{2}{*}{$-(-s_{u,v}\,s_{u,v+1})^{w+v+1}s_{u-1,w}\,s_{u,v}\sign(\overline{m}/\bar{\du})>0$} 
\\&&& \\ \cline{2-4} 
& \multirow{2}{*}{\includegraphics[scale=0.3]{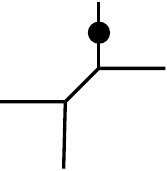}}  & \multirow{4}{*}{$\bar{\du}$} &  \multirow{2}{*}{$ -(-s_{u,v}\,s_{u+1,v})^{r+u+1}s_{r,v+1}\,s_{u,v}\sign(\bar{n})>0$}
\\ &&&\\\cline{2-2} \cline{4-4} 
& \multirow{2}{*}{\includegraphics[scale=0.3]{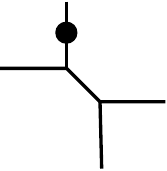}}  &  &  \multirow{2}{*}{$ -(-s_{u,v}\,s_{u+1,v})^{r+u}s_{r,v+1}\,s_{u+1,v}\sign(\overline{m}/\bar{\du})>0$}
\\&&& \\ \cline{2-4} 
& \multirow{2}{*}{\includegraphics[scale=0.3]{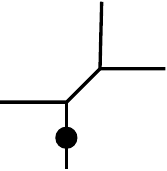}}  & \multirow{4}{*}{$\overline{m}/\bar{n}$} &  \multirow{2}{*}{$-(-s_{u,v}\,s_{u+1,v})^{w+u}s_{w,v-1}\,s_{u+1,v}\sign(\bar{n})>0$}
\\ &&&\\\cline{2-2} \cline{4-4} 
 & \multirow{2}{*}{\includegraphics[scale=0.3]{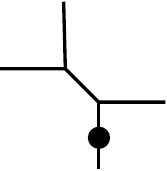}}  &  &  \multirow{2}{*}{$ -(-s_{u,v}\,s_{u+1,v})^{w+u+1}s_{w,v-1}\,s_{u,v}\sign(\overline{m}/\bar{\du})>0$} 
\\&&& \\ \hline              
  \end{tabular}
  \caption{Sign conditions determining the existence of real  local liftings for $\overline{m}$, $\bar{n}$, $\bar{\du}$, $\bar{n}/\bar{\du}$, $\overline{m}/\bar{\du}$ or $\overline{m}/\bar{n}$ imposed by some local tangency types of lifting multiplicity two located along a segment of $\Gamma \cap \Lambda$  whose midpoint $P$ is marked on each figure in the second column. The indices $u,v,r$ and $w$ correspond to those indicated in~\autoref{fig:type3Mult2NP}.\label{tab:signConditionsExamples}}
            \end{table}

  \begin{figure}[t]
    \includegraphics[scale=0.75]{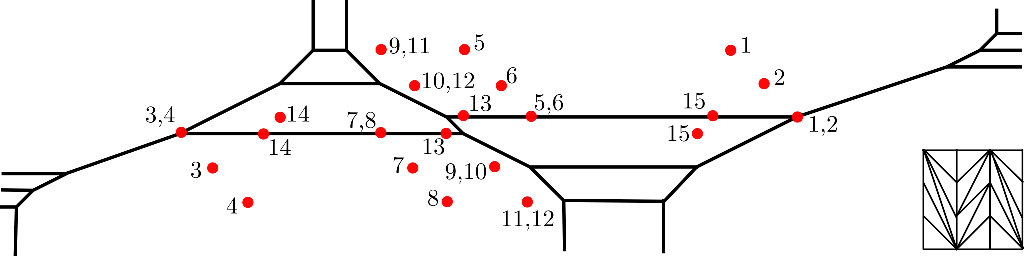}
    \caption{Location of the vertices of 15 non-equivalent tritangent lines to  $\Trop \,V(\sextic)$ and the Newton subdivision of $\sextic$ for the curve $V(\sextic)$ from~\autoref{ex:64}.\label{fig:example64}}
    \end{figure}

\begin{example}\label{ex:64} 
  Consider a  bidegree $(3,3)$-polynomial $\sextic$ with Newton subdivision as in~\autoref{fig:example64} and generic relative to the tropical curve $\Gamma$ seen in the same picture. We assume all 16 coefficients $(a_{i,j})_{i,j}$ of $\sextic$ are positive.

  Combinatorial patchwork confirms that the real locus of  $V(\sextic)$ consists of 4 ovals, two of which are nested. Thus, by~\cite[Table 2]{kum:19}, there are 64 real tritangents to $V(\sextic)$. We can reprove this fact using our methods. In what follows, we show that exactly eight of the 15 tritangent classes of $\Gamma$, namely the first eight, lift over $\PSR$.

  Notice that both the Newton subdivision of $\sextic$ and the sign distribution are symmetric under the action of the map $\tau_1^2$. The tritangent classes of $\Gamma$ are grouped into nine orbits under this action, namely, $\{1,4\}$, $\{2,3\}$, $\{5,8\}$, $\{6,7\}$, $\{9,12\}$, $\{10\}$, $\{11\}$, $\{13\}$ and $\{14,15\}$.  Thus, to prove our real lifting claim for the first eight classes, it suffices to certify it for four of them. We chose the ones labeled 1, 2, 5 and 6.  Similarly, to show the last three classes do not lift over $\PSR$, we need only treat the classes 13 and 14.

  \begin{table}
    \includegraphics[scale=0.45]{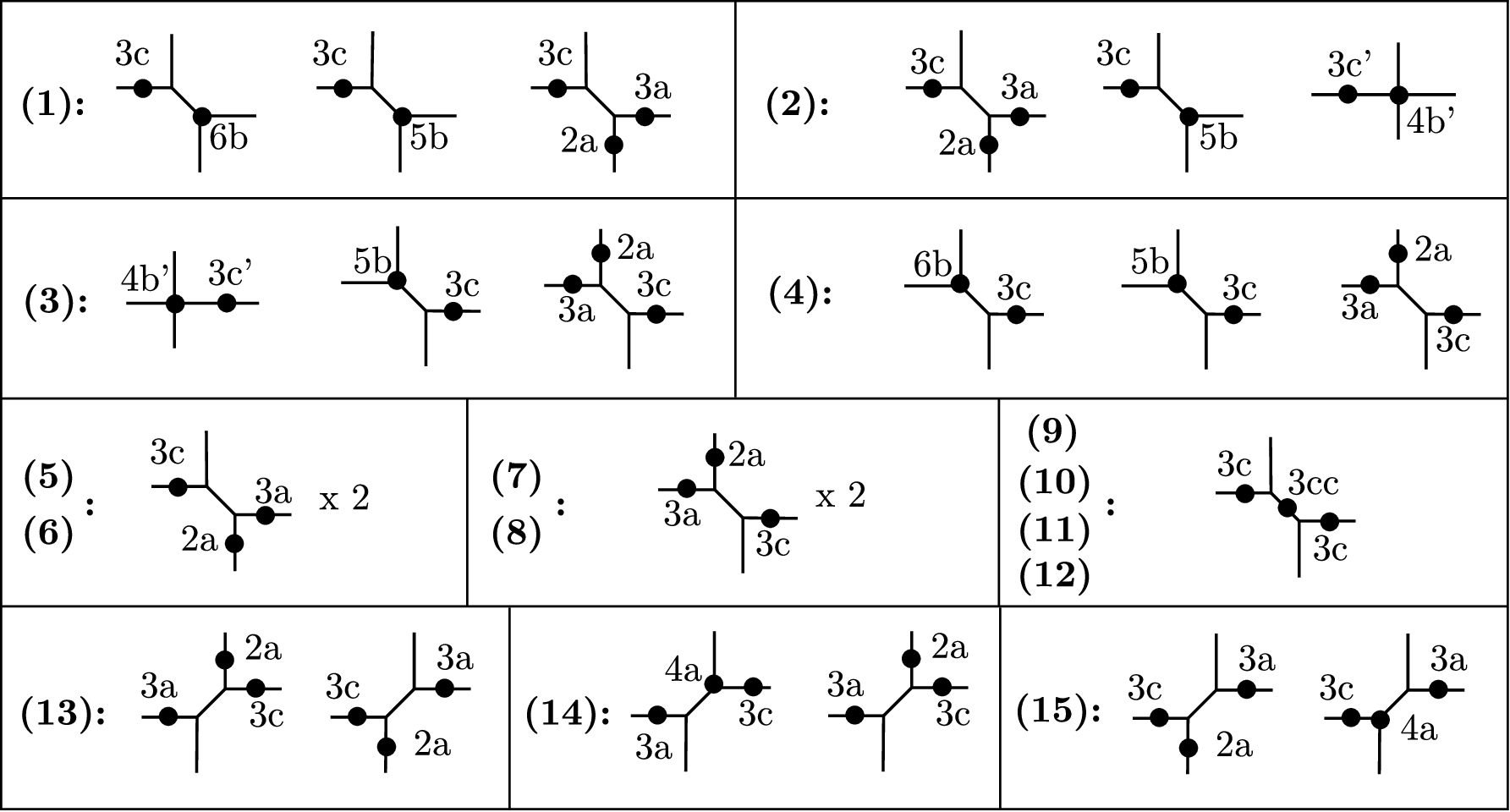}
    \caption{Realizable tritangents for all 15 classes of~\autoref{ex:64}, and the local tangency types for each one. Members from each class are listed from left to right.\label{tab:localTangenciesEx64}}
  \end{table}
  
  \autoref{tab:localTangenciesEx64} shows the distribution of local tangencies for the liftable members of each tritangent class. This information, combined with~\autoref{tab:LiftingMultiplicities} and~\autoref{pr:4a6aWithAdjacentLeg}, confirms that the answer to the lifting question of each member over $\PSR$ is determined by what happens for the local tangencies of type (3a), (3c), (3c') and (3cc). Precise sign rules for each type are listed in~\autoref{tab:signConditionsExamples}.

  We start by discussing classes 9 through 12, since they all share a tangency point of type (3cc). The relevant coefficients are $u=v=1$, $w=3$ and $r=0$. The formula from~\autoref{tab:signConditionsExamples} confirms this local tangency does not lift over $\PSR$, as we wanted to show.

  Next, we show that the classes 13 and 14 do not lift over $\PSR$. Note that both classes have two members with a common type (3a) tangency along the negative horizontal leg of $\Lambda$ and the bottom-left horizontal edge of $\Gamma$. The values of the relevant parameters in the formula from the table are $u=1$, $v=0$, $w=3$ and the sign of $\bar{n}$. To prove this local tangency does not lift over $\PSR$, we must check that $\bar{n}<0$. We do so by looking at the left member for class 13 and the right one for class 14 seen in~\autoref{tab:localTangenciesEx64}. The type (3c) tangency on the positive horizontal leg of $\Lambda$ confirms that $\bar{n}/\bar{\du}>0$, whereas the type (2a) tangency  along the positive vertical leg of $\Lambda$ yields $\bar{\du}<0$ in both cases. Indeed, it suffices to combine the formulas from~\autoref{tab:initialFormsType2Horiz} for (2D), respectively (2E), and the map $\tau_0$. Thus, $\bar{n}$ is negative, as we wanted.

  To conclude, we discuss classes 1, 2, 5 and 6.  The tangencies (3c) and (3c')  for all members of classes 1 and 5 seen in~\autoref{tab:localTangenciesEx64} occur along  the top-left horizontal edge of $\Gamma$. We have $u=1$, $v=2$ and $w=r=3$. Similarly, the tangencies (3c) and (3c') for classes 2 and 6 lie on the middle-left horizontal edge of $\Gamma$, with $u=1$, $v=1$ and $w=r=3$. Since $r=w$ in both cases, the corresponding sign rule from~\autoref{tab:signConditionsExamples} confirms that these local tangencies lift over $\PSR$. 

  In turn, a type (3a) tangency for members of these four classes occurs on the top-right horizontal edge of $\Gamma$. We have $u=v=2$ and $r=0$.  To prove that this local tangency lifts over $\PSR$, we must check that $\overline{m}/\bar{\du}<0$. The local equations at this tangency ensure that  $\bar{n}/\bar{\du} = \overline{a_{22}}/\overline{a_{23}}$, so it is positive. In turn, applying the map $\tau_0$ to our formulas from~\autoref{tab:initialFormsType2Horiz} for (2D) and (2E) we see that the type (2a) tangency along the negative vertical leg of these four tritangents yields $\bar{n}/\overline{m}<0$. Thus, the quantity $\overline{m}/\bar{\du}$ is negative, as we wanted to show.
\end{example}

  \begin{figure}
    \includegraphics[scale=0.75]{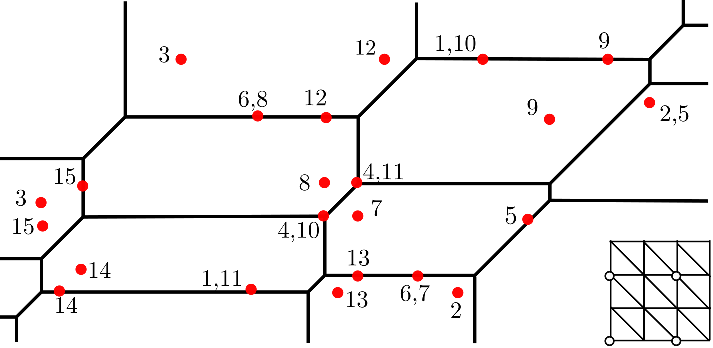}
    \caption{Location of the vertices of 15 non-equivalent tritangent lines to a tropical smooth $(3,3)$-curve in honeycomb form. The four marked vertices in the  subdivision correspond to the positive coefficients of $\sextic$. The remaining ones are negative.\label{fig:honeycomb}}
  \end{figure}

\begin{example}\label{ex:120}
  Let $V(\sextic)$ be a generic lift  to the tropical curve $\Gamma$ in honeycomb form seen in~\autoref{fig:honeycomb}. We suppose that $\sextic$ is defined over $\PSR$ and assign the following signs to its coefficients $a_{ij}$: we set $a_{00}, a_{20}, a_{02}$ and $a_{22}$ to be positive, whereas all remaining ones are negative. The Newton subdivision and the sign distribution are invariant under the action of the map $\tau_0$.

  Combinatorial patchworking confirms that the real curve consists of 5 ovals, two of which are nested. Thus, this curve has 120 real tritangents.
  Our methods yield an alternative proof of this fact. Furthermore, they imply that  all tritangents to $V(\sextic)$ are totally real. Notice that having more that 64 real tritangents forces the total number to be 120. Thus, it suffices to show that at least nine of the 15 tritangents to $\Gamma$ lift over $\PSR$.

  We claim this is so for the tritangents labeled by 2, 3, 7, 8, 9 and 12 through 15 in the picture.
  We chose these nine because these classes are singletons, and each intersection $\Gamma \cap \Lambda$ has three components whose local tangency types are (3a), (3c) or (3cc). Thus, all three have  local lifting multiplicity two over $\PSR$. Furthermore, many of them are shared by  different classes.

  A direct computation using the formulas from~\autoref{tab:signConditionsExamples} confirms that all type (3c) and (3cc) tangencies lift over $\PSR$. Thus, classes 2 and 3 lift over $\PSR$. To conclude, we must check the sign rules for type (3a) tangencies, which occur for tritangent classes 7, 8, 9 and 12 through 15. Again, the formulas in the table require us to verify that  $\overline{m}/\bar{\du}<0$ for the first two and $\bar{n}>0$ for the remaining five. We analyze each case separately.

  For class 7, the type (3c) tangency on the negative horizontal leg determines the sign of $\overline{m}$ to be positive, since $\overline{m} = \overline{a_{11}}/\overline{a_{12}}$. In turn, the type (3c) tangency on the positive vertical leg implies $\bar{\du} = \overline{a_{22}}/\overline{a_{12}}<0$. Thus, the sign of $\overline{m}/\bar{\du}$ is negative, as we wanted.
Similarly, for class 8, we obtain $\overline{m}/\bar{n} = \overline{a_{11}}/\overline{a_{12}}>0$ and $\bar{n}/\bar{\du} = \overline{a_{21}}/\overline{a_{22}}<0$ from the vertical and horizontal type (3c) tangencies, respectively. Once again, $\overline{m}/\bar{\du}<0$, as we needed. We conclude that the tritangent classes 7 and 8 lift over $\PSR$.

For class 9, the two type (3c) tangencies give  $\overline{m}/\bar{n}=\overline{a_{21}}/\overline{a_{31}}>0$ and $\overline{m} = \overline{a_{12}}/\overline{a_{13}}>0$, so $\bar{n}>0$. Similarly, for class 14, we obtain $\bar{\du}= \overline{a_{12}}/\overline{a_{02}}<0$ and $\bar{n}/\bar{\du} = \overline{a_{20}}/\overline{a_{21}}<0$, which ensures that $\bar{n}>0$. Therefore, classes 9 and 14 also lift over $\PSR$.

For the remaining three classes, computing the sign of $\bar{n}$ requires the local information provided by the tangency (3a) itself. For class 12, we get $\overline{m} = \overline{a_{12}}/\overline{a_{13}}$ from this tangency, whereas the vertical (3c) tangency gives $\overline{m}/\bar{n} = \overline{a_{11}}/\overline{a_{21}}>0$, so $\bar{n}>0$ holds. Similarly, for class 13, the (3a) tangency yields $\bar{n}/\bar{\du} = \overline{a_{20}}/\overline{a_{21}}<0$, while the vertical (3c) tangency gives $\bar{\du} = \overline{a_{22}}/\overline{a_{12}}<0$. Therefore, we have $\bar{n}>0$. Finally, for class 15, we obtain $\bar{\du} = \overline{a_{12}}/\overline{a_{02}}<0$ and   $\bar{n}/\bar{\du} = \overline{a_{21}}/\overline{a_{22}}<0$, which ensures $\bar{n}>0$. Thus, classes 12, 13 and 15 also lift over $\PSR$.
\end{example}

\section{Arithmetic considerations}
\label{sec:arithm-constr-lift}

As was proved in 
{Sections}~\ref{sec:trivalentLifts} through~\ref{sec:tang-mult-six} the lifting multiplicities of each of the 38 local tangency types between $\Lambda$ and $\Gamma$ over $\overline{\K}$ are either 0 or a power of 2. \autoref{tab:LiftingMultiplicities} summarizes our findings. It is natural to ask what happens over other field extensions of $\K$.

In view of the formulas obtained for tangency types (5b), (6b) and (4b') seen in~\autoref{tab:liftingMult45b6b} and~\autoref{cor:valuesMult4Type4}, we assume $\sqrt{2}, \sqrt{3} \in \resK$. Similarly, in the presence of a type (6b') tangency, we assume $\sqrt{\Delta}\in \resK$, when $\Delta$ is the quantity from~\eqref{eq:values6b_Cross}. Finally, we suppose throughout that  $\sextic$ is generic relative to $\Gamma$.




Following~\autoref{def:totallyKRational}, we can extend the notion of rationality and total-rationality over a field extension of $\K$ to the local setting, that is to liftings $(\ell,p)$ of tropical pairs $(\Lambda, P)$.
The second entry of the local tropical tangency $(\Lambda,P)$,    refers to all tangencies present in the connected component of $\Gamma \cap \Lambda$ containing the tangency $P$. In particular, we can   consider  arithmetic-dependent lifting multiplicities, which  incorporate the fields of definition into our earlier notion. Their explicit computation will be  simplified thanks to the work done in previous sections.

\begin{definition}
  Given a field extension $\LL$ of $\K$, and a tangency point $P$,  we let $\mult(\Lambda, P, \LL)$ be  the number of local liftings of $(\Lambda,P)$ that are $\LL$-rational.
\end{definition}

Our first result determines the possible values of $\mult(\Lambda, P, \LL)$ and their behavior under further field extensions:
\begin{theorem}\label{thm:liftingMultiplicitiesForOtherFields}
  Let $\LL$ be an algebraic field extension of $\K$.  Assume $P$ is a tangency point between $\Lambda$ and $\Gamma$. Then, $\mult(\Lambda, P, \LL)$ is either $0$ or it equals $\mult(\Lambda, P, \overline{\K})$. In the latter case, the local lifts are totally $\LL$-rational.
\end{theorem}

\begin{proof} We prove the statement by analyzing the possible values of $\mult(\Lambda, P, \overline{\K})$. By construction, any local lifting  is defined over a finite field extension of $\K$, so we may assume $\LL$ is finite. In particular, if $\mult(\Lambda, P, \overline{\K}) = 0$, the same is true for the field $\LL$.

  Next, assume $\mult(\Lambda, P, \overline{\K}) = 1$. In this situation, the local lifts are defined over $\K$ since we assumed $\sqrt{\Delta}, \sqrt{2}, \sqrt{3}\in \resK$. Thus, we get $\mult(\Lambda, P, \LL) = 1$ as well. Notice that the both the relevant parameters on $\ell$ and the tangency points lifting $P$ will be defined over $\LL$.

  Finally, assume $\mult(\Lambda, P, \overline{\K})\geq 2$, then the initial forms of parameters in the local equations determined by $P$ are determined by solutions to a univariate quadratic equation over $\resK$. Thus, either both solutions lie in $\resL$, or neither of them do. This applies to $\overline{p}$ as well as the parameters on $\ell_P$. The same arguments carried out in the proof of~\autoref{thm:rationalTotallyRatQuadraticExtension} confirms that these initial forms determine the parameters uniquely without the need of a further field extension. The statement follows from this fact.
\end{proof}

As a consequence, we obtain the lifting multiplicity of any tropical tritangent $(1,1)$-curve to $\Gamma$ over any field extension $\LL$ of $\K$ and confirm total $\LL$-rationality of each such lift:

\begin{corollary}\label{cor:LiftingEachTritangent} The lifting multiplicity $\mult(\Lambda, \LL)$ of any given tropical tritangent curve $\Lambda$ to $\Gamma$ over a field extension  $\LL$ of $\K$ is either 0 or it equals  $\mult(\Lambda, \overline{\K})$. Furthermore, its possible values are 0, 1, 2, 4 or 8. In addition, each $\LL$-rational lift of $\Lambda$ is totally $\LL$-rational.
\end{corollary}

\begin{proof}  
  By construction, in most cases, local tangencies between $\Lambda$ and $\Gamma$ in different connected components of $\Gamma \cap \Lambda$ impose independent conditions on the three parameters of the equation $\ell$ defining  any  $(1,1)$-curve lifting $\Lambda$ and tritangent to $V(\sextic)$. The exceptional cases arise in the presence of tangencies of type (4a), (4a'), (6a) or (6a'). 
  {Propositions}~\ref{pr:4a6aNoAdjacentLeg},~\ref{pr:4a6aWithAdjacentLeg} and~\ref{pr:type4ap6ap} imply that in these special situations, the joint lifting multiplicities over $\LL$ are also either $0$ or they agree with the ones over $\overline{\K}$. 
  Thus, the properties satisfied by the local lifting multiplicities listed in~\autoref{thm:liftingMultiplicitiesForOtherFields}   are also valid for $\mult(\Lambda, \LL)$.

  The quantity $\mult(\Lambda, \overline{\K})$ is obtained as the product of the local lifting multiplicities at each tropical tangency point, whenever it is non-zero. If $\Lambda$ is trivalent, the exact formula depends on whether $\Lambda$ carries a tangency of type (4a) or (6a) along its unique edge, as was indicated in 
  {Corollaries}~\ref{cor:4a6aField} and~\ref{pr:4a6aWithAdjacentLeg}. The rationality statement follows from~\autoref{thm:liftingMultiplicitiesForOtherFields}.
\end{proof}

In the remainder of this section, we focus on the question about lifting tritangents over an arbitrary Henselian valued field $\K$ whose value group $\val(\K)$ is $2$-divisible, non-trivial and admits a splitting. We start by providing examples of fields with these properties.

\begin{example}\label{ex:hensel2-divisible}
  Such fields exists with any given residue field $\resK$, for instance, the Puiseux series  $\resK\{\!\{t\}\!\}$, the generalized power series field $\resK(\!(t^{\RR})\!)$, and their completions are all valid examples. The required conditions are also satisfied by mixed characteristic fields. For examples, we can consider $\K=\QQ_p(p^{1/2},p^{1/4}, p^{1/8},\ldots)$, where $p^{1/2^{n+1}}$ is a square root of $p^{1/2^n}$ for each $n\in \ZZ_{\geq_0}$. A splitting of the $p$-adic valuation is given by $v\mapsto p^{v}$.
\end{example}

Our objective is two-fold: discuss what determines the $\K$-rationality of a given lift, and whether or not $\K$-rationality of all liftable tritangents of a given class can be determined by a single member.
Our first result stems from a simple observation: the lifting of a given $\Lambda$ to a classical tritangent in $\pr^1\times \pr^1$ over $\overline{\K}$ is determined by the local lifting equations imposed by a tropical tangency point.

The initial forms of the solutions are elements of the algebraic closure of the residue field $\resK$. Their values are determined by the initial forms of the coefficients of the input $(3,3)$-polynomial $\sextic$, or by square roots thereof. In turn, Hensel's lemma determines the field of definition of the tritangent tuple solely in terms of the initial data. An immediate consequence of this result is the following statement, whose analog for the quartic case is given in~\cite[Theorem 3.3]{MPS23}:

\begin{theorem}\label{thm:solutionsUpToSquares} Let $V(\sextic)$ be a smooth $(3,3)$-curve in $\pr^1\times \pr^1$ defined over $\K$ that is generic relative to its tropicalization $\Gamma$. Assume $\Gamma$ is generic. Let $\Lambda$ be a tritangent $(1,1)$-curve to $\Gamma$ that lifts over $\overline{\K}$. Whether or not the lift is $\K$-rational is determined solely by $\Gamma$ and the equivalence classes in  $\resK^*/(\resK^*)^2$ of $-1$ and the initials of all coefficients of $\sextic$.
\end{theorem}

In analogy with the result for smooth plane quartics provided in~\cite[Theorem 3.14]{MPS23}, we conjecture that the total lifting behavior is the same as the one for $\RR$:

\begin{conjecture}\label{conj:LiftingOverOtherFields} Assume the residue characteristic of $\K$ is not 2 or 3.  Let $C$ be a smooth $(3,3)$-curve in $\pr^1\times \pr^1$ defined over $\K$. If $\Trop \,C$ is generic and smooth, then each tritangent class to $\Trop \,C$ has either zero or exactly eight lifts over $\K$.
\end{conjecture}

A natural strategy to prove this conjecture mimics the approach followed in~\cite{CM20, MPS23} for plane quartic curves. Namely, since the local lifting equations at each tangency point are field independent, it suffices to compare the $\K$-rationality of the solutions of the systems determined by the tangency points of each  member of a fixed tritangent class that lift over $\overline{\K}$.  As the combinatorics of local tangencies of $(3,3)$-curve is much more involved than the one for plane quartics, making a classification of tritangent classes a rather challenging endeavor, we refrain from pursuing this method here and leave it for future work.

As we saw in~\autoref{thm:realLifts}, these combinatorial difficulties  can be avoided for $\RR$ (or any real closed field), by viewing $C$ as a smooth space sextic curve in $\pr^3$ and exploiting information provided by its avoidance loci. Even though these techniques do not extend to arbitrary fields,  we can still verify~\autoref{conj:LiftingOverOtherFields} for those fields whose residue fields are  comparable to $\RR$. In what follows, we discuss this approach, building on results from~\cite[Section 3.3]{MPS23}.

In view of~\autoref{thm:solutionsUpToSquares}, the following notion provides a natural framework for comparing fields:

\begin{definition}
  Let $\K_1$ and $\K_2$ be Henselian valued fields with residue fields $\resK_1$ and $\resK_2$, respectively. We say $\K_1$ and $\K_2$ are \emph{comparable} if  there exists an isomorphism of groups
  \begin{equation}\label{eq:comparingFields}
    \phi\colon \resK_1^*/(\resK_1^*)^2\to \resK_2^*/(\resK_2^*)^2.
  \end{equation}
\end{definition}

Our first result generalizes~\cite[Theorem 3.16]{MPS23} and is a direct consequence of~\autoref{thm:solutionsUpToSquares}:

\begin{theorem}\label{thm:ComparingFieldsWithNoSquareRootFor-1}
  Let $\Gamma$ be a smooth generic tropical $(3,3)$-curve in $\TPr^1\times \TPr^1$ and fix a tropical curve $\Lambda$ tritangent to $\Gamma$.
  Let $\K_1$ and $\K_2$ be two  fields of characteristic not 2 or 3 and consider two generic lifts $f_1:= \sum_{i,j} a_{ij} x^i b^j$ and $f_2:=\sum_{i,j} b_{ij} x^i b^j$ of $\Gamma$ defined over $\K_1$ and $\K_2$ respectively.

  Assume that $\K_1$ and $\K_2$ are comparable via an isomorphism $\phi$ as in~\eqref{eq:comparingFields} with $\phi([-1]) = [-1]$, and  $\phi([\overline{a_{ij}}]) = [\overline{b_{ij}}]$ for all $(i,j)\in \{0,1,2,3\}^2$. Then, $\Lambda$ has a $\K_1$-rational lift to a tritangent to $V(f_1)$  if, and only if, it has a $\K_2$-rational lift to a tritangent to $V(f_2)$.
\end{theorem}

Legendre (or Jacobi) symbols over residue fields can provide a way to establish when two fields are comparable. These symbols determines whether a non-zero element of a field is a square or not. Here is the precise definition, which can be found in~\cite[$\S$2.2]{Cohen}:
\begin{definition}\label{def:LegendreSymbol} Let $k$ be an arbitrary field. The \emph{Legendre symbol} on $k$ is the isomorphism 
  \begin{equation*}
\left(\frac{-}{k}\right)\colon k^*\to \ZZ/2\ZZ \qquad   a\mapsto  \left(\frac{a}{k}\right) := \begin{cases} \;\;\;1 & \text{ if } x^2 = a \text{ has a solution in }k,\\
      -1 & \text{ otherwise.}
    \end{cases}
  \end{equation*}
\end{definition}
\noindent For example, when $k=\RR$ the Legendre symbol remembers the sign of the corresponding non-zero real number. 

\begin{remark}\label{rm:LegendreSymbol-1} The  value of the Legendre symbol of $-1$ on any finite field   $k=\mathbb{F}_{p^n}$  of odd prime characteristic $p$ can be easily determined thanks to the multiplicativity of the Legendre symbol, combined with~\cite[Proposition 2.2.4]{Cohen}. More precisely, we have
 \begin{equation}\left(\frac{-1}{k}\right) = (-1)^{\frac{p^n-1}{2}}.
  \end{equation}
 In particular, when $n$ is odd and $p \equiv 3$ mod 4, it follows that $\left(\frac{-1}{\mathbb{F}_{p^n}}\right) = \left(\frac{-1}{\RR}\right) = -1$.
\end{remark}

\autoref{thm:ComparingFieldsWithNoSquareRootFor-1} becomes a useful tool to determine $\K$-rationality of a tritangent lift whenever $\K$ is comparable to $\PSR$. Examples of such fields are those where $\left(\frac{-1}{\resK}\right) = -1$. As a consequence of~\autoref{rm:LegendreSymbol-1} this will hold whenever $\resK$ is a finite field of odd prime characteristic $p$ with $p\equiv 3$ mod 4,  and whose order is an odd power of $p$.

Our next result is analogous to~\cite[Corollary 3.17]{MPS23}, and it follows directly from~\autoref{thm:ComparingFieldsWithNoSquareRootFor-1} when $\K_2 = \PSR$:

\begin{corollary}\label{cor:LiftingFiniteResidueFields} Let $p$ be a prime with $p\equiv 3$ mod 4 and $n$ be an odd positive integer. Consider any  Henselian valued field  $\K$ with residue field $\mathbb{F}_{p^n}$. 
  Let $\Gamma$ be a generic smooth tropical $(3,3)$-curve and pick two generic lifts   $\sextic := \sum_{i,j} a_{ij}x^iy^j$  and $\sextic':=\sum_{i,j} b_{ij} x^iy^j$ of $\Gamma$ defined over $\K$ and $\PSR$, respectively. Assume that
 \[\left(\frac{\overline{a_{ij}}}{\resK}\right) = \left(\frac{\overline{b_{ij}}}{\RR}\right) \quad \text{ for all } i,j \in \{0,1,2,3\}.
 \]
 Then, a tropical tritangent $\Lambda$ to $\Gamma$ has a $\K$-rational lift to a tritangent to $V(\sextic)$ if, and only if, it lifts to a $\PSR$-rational tritangent to $V(\sextic')$.
\end{corollary}

Combining this last result with~\autoref{thm:realLifts} we obtain new evidence in favor of~\autoref{conj:LiftingOverOtherFields}:

\begin{corollary}\label{cor:examplesForConjecture}
Let $p$ be a prime with $p\equiv 3$ mod 4 and $n$ be an odd positive integer, and $\K$ be a  Henselian valued field with residue field $\mathbb{F}_{p^n}$. Fix  a generic smooth tropical $(3,3)$-curve $\Gamma$ with generic lift $V(\sextic)$ defined over $\K$. Then, a tritangent class of $\Gamma$ has either 0 or 8 lifts over $\K$.
\end{corollary}

\appendix

  \section{Type (3) tangencies for tropical lines in $\TPr^2$}\label{sec:correctionCM20}

  In this section, we clarify a mistake made in the proof of~\cite[Proposition 5.2]{CM20} when the tangency between a tropical line and a smooth tropical curve in $\TPr^2$ has type (3c). The error stems from a typo in an estimated lower bound for the valuation of certain coefficients. While the result remains true, we decided to include the correct proof as an appendix to this paper because this result is regularly cited in the present article. 

  By exploiting the action of the symmetric group $\Sn{3}$ we may assume the tangency occurs along the horizontal leg of the tropical line in $\TPr^2$. 

  \begin{proposition}\label{pr:Prop5.2Corrected} Let $\Lambda$ be a tropical line  that is tangent to a smooth plane tropical quartic curve $\Gamma$ along a type (3c) point $P$. Assume $P$ lies on the negative horizontal leg of $\Lambda$.  Fix a generic lift $V(\sextic)$ of $\Gamma$. Then, there are either zero or two real solutions in $m$ for the local equations in $(m,n,p)$ corresponding to the  line $V(\ell)$ tangent to $V(\sextic)$ at $p$ with $\ell=y+m+n\,x$ and   $\Trop\, p=P$.
    \end{proposition}
  \begin{proof} As usual, we assume $P=(0,0)$ and $\sextic\in R[x,y]\smallsetminus \mathfrak{M}R[x,y]$, with $\Gamma=\Trop\, V(\sextic)$. We assume that $P$ is  the midpoint of a horizontal edge $e$ of $\Gamma$, whose dual edge $e^{\vee}$ has vertices $(u,v)$ and $(u,v+1)$. The edge $e^{\vee}$ forms triangles with two vertices in the Newton subdivision of $\sextic$, namely $(u-1,w)$ and $(u+1,r)$, for given $w,r$. By construction, these triangles are dual to the left and right endpoints of $e$, which we write as $(-\lambda,0)$ and $(\lambda, 0)$ for some $\lambda >0$.    With this setup, we have
    \[\val(a_{u,v})=\val(a_{u,v+1}) = 0, \quad \val(a_{u-1,w}) = \val(a_{u+1,r})=\lambda \quad \text{ and } \val(a_{u,j})>0 \; \text{ for all }j\neq v, v+1.
    \]

    Our proposed solution $m$ will have the form $m= a_{u,v}/a_{u,v+1} + m_1 + m_2$ with $\val(m_2)=\lambda >\val(m_1)>0$. In particular, we have $\overline{m} = \overline{a_{u,v}}/\overline{a_{u,v+1}}$, so the sign of $\overline{m}$ is predetermined by the input polynomial $\sextic$: its value is $s_{u,v}s_{u,v+1}$. In what follows, we determine the values of $m_1$ and $m_2$. We do this in two steps.

    First, we re-embed $V(\sextic)$ in $(\K^*)^3$ via the ideal $I=\langle \sextic, z-y-(a_{u,v}/a_{u,v+1} + m_1)\rangle$. The polynomial
    \[\tilde{\sextic}(x,z) = \sextic(x, z-(a_{u,v}/a_{u,v+1} + m_1)) = \sum_{i,j} \tilde{a}_{i,j}\,x^iz^j\in \K[x,z]
    \]
    defines the projection of $V(I)$ onto the $xz$-plane. The parameter $m_1$ must be chosen to ensure that the Newton subdivision of $\tilde{\sextic}$ satisfies two combinatorial properties:
    \begin{itemize}
    \item it contains the triangle with vertices $(u-1,0)$, $(u,1)$ and $(u+1,0)$;
    \item the lattice point $(u,0)$ is not a vertex of the  subdivision.
    \end{itemize}
    Both conditions can  be achieved if the valuations of the relevant coefficients of $\tilde{\sextic}$ are constraint by:
    \[
    \val(\tilde{a}_{u-1,0})=\val(\tilde{a}_{u+1,0}) = \lambda, \quad \val(\tilde{a}_{u,1})=0 \quad \text{ and } \quad \val(\tilde{a}_{u,0})>\lambda.
    \]
 The first three hold and can be verified by direct computation. Moreover, the same calculation shows that  $\val(\tilde{a}_{u,1} - (-a_{u,v}/a_{u,v+1})^v a_{u,v+1})>0$, whereas
\begin{equation}\label{eq:coeffs_case3a}
     \val(\tilde{a}_{u-1,0} - a_{u-1,w} (-a_{u,v}/a_{u,v+1})^{w}) > \lambda \quad \text{ and } \quad
     \val(\tilde{a}_{u+1,0} - a_{u+1,r} (-a_{u,v}/a_{u,v+1})^{r})  > \lambda.
\end{equation}

The coefficient $\tilde{a}_{u,0}$ becomes a polynomial in $R[m_1]$. Thus, to ensure the inequality $\val(\tilde{a}_{u,0})>\lambda$ holds for an appropriate $m_1$, it is enough to find a root $m_1$ of
\begin{equation} \label{eq:univpoly3a}
  \tilde{a}_{u,0}(m_1)=\sum_{k} b_k \, m_1^k = \sum_{j} a_{u,j} (-1)^j (a_{u,v}/a_{u,v+1} + m_1)^j
\end{equation}
with positive valuation. Our reasoning is similar to the one used in the proof of~\cite[Proposition 5.2]{CM20} for the type (3a) tangency. The result is a consequence of~\autoref{lm:chooseTail3a} below combined with the genericity of $\sextic$.

Once this step is completed, we determine the coefficient $m_2$ using the method outlined in~\cite[Lemma~3.9]{LM17}. Indeed, the tropical tangency point $P$ viewed in $\Trop \,V(I)$ becomes $(0,-\lambda)$, the valuation of $m_2$ equals $\lambda$ and there are two solutions for $m_2$ each determined by its initial form. Such form $\overline{m_2}$ solves the quadratic equation:
\[(-1)^w\,\overline{a_{u-1,w}} \left(\frac{\overline{a_{u,v}}}{\overline{a_{u,v+1}}}\right)^w - \frac{(-1)^r\,\overline{a_{u,v+1}}^2}{4\,\overline{a_{u+1,r}}}\, \left(\frac{\overline{a_{u,v}}}{\overline{a_{u,v+1}}}\right)^{2v-r} \overline{m_2}^2 = 0.
\]
Its two solutions agree with the ones seen in~\cite[(5.2)]{CM20}.  Furthermore, the local equations describing $p$ as a tangency point have unique solutions in  the initial form of the tangency point $\bar{p}$ once the value of $\overline{m_2}$ is fixed.

The uniqueness of $m_2$ given the value $\overline{m_2}$ follows by \autoref{lm:multivariateHensel}: the Jacobian of the local equations of $(\tilde{\sextic}, \ell:=z-m_2, W)$ in the variables $(m_2, x, z)$ has expected valuation zero and expected initial form $\overline{m_2}\,\overline{a_{u,v}}^{2v}/\overline{a_{u,v+1}}^{2(v-1)}$. This monomial does not vanish, so the lifting of each fixed tuple $(\overline{m_2}, \overline{p})$ to $(m_2,p)$ is unique.
  \end{proof}

\begin{lemma}\label{lm:chooseTail3a} The univariate polynomial $\tilde{a}_{u,0}(m_1)\in R[m_1]$ from~\eqref{eq:univpoly3a} has a root in $\mathfrak{M}$ if its constant coefficient is non-zero. Furthermore, such root is unique, it lies in $\K$ and has  valuation $\val(\tilde{a}_{u,0}(0))$.
\end{lemma}

\begin{proof} By the Fundamental Theorem of Tropical Algebraic Geometry, it is enough to show that the tropical hypersurface $\Trop\, V(\tilde{a}_{u,0})$ has a point in $\RR_{<0}$ if  $b_0\neq 0$. 

  By construction, we have that $\val(b_k)\geq 0$ for all $k$ since $\sextic\in R[x,y]$ and also that
  \begin{equation}\label{eq:bk}
    b_k = \sum_{j\geq k} a_{u,j} \binom{j}{k} (-1)^j (a_{u,v}/a_{u,v+1})^{j-k}.
  \end{equation}
  Exploiting our knowledge of the valuations of all coefficients $a_{u,j}$ of $\sextic$, we  determine the valuations of all relevant coefficients $b_k$. In particular, we have $\val(b_k)>0$ for $k>v+1$, whereas $\val(b_{v+1})=0$. In turn, we have $\val(b_1)=0$ since the sum of the terms in $R\smallsetminus \mathfrak{M}$ contributing to $b_1$ equals 
  \[
  \begin{cases}
    (-1)^{v+1} a_{u,v+1} (a_{u,v}/a_{u,v+1})^v & \text{ if }v\geq 1, \\
    - a_{u,1} & \text{ if }v=0.
  \end{cases}
  \]
Finally, we know that $\val(b_0)>0$ since the sum of the terms of valuation zero in $b_0$ featuring in~\eqref{eq:bk} has positive valuation by construction.

The tropicalization of $\tilde{a}_{u,0} \in R[m_1]$ in the variable $M_1$ is non-linear, and its given by the  formula
\[
\trop(\tilde{a}_{u,0})(M_1) = \max_{k\geq 0} \{-\val(b_k) + k \,M_1\}.
\]
The above characterization of the valuations of the coefficients $b_k$ ensures that if $\trop(\tilde{a}_{u,0})(M_1)$ were to achieve this maximum value  for $M_1<0$ at least twice, the corresponding indices would be $k=0,1$. From here, we deduce that $M_1=-\val(b_0)<0$.

A direct computation combined with the lower bounds on the valuations of all $b_k$ obtained earlier reveals that this point lies in $\Trop \,V(\tilde{a}_{u,0})$. Thus, by the Fundamental Theorem of Tropical Geometry, $\tilde{a}_{u,0}$ would admit a solution $m_1\in \overline{\K}$ with $-\val(m_1)=M_1$. Since $M_1$ has multiplicity one in $\Trop\,V(\tilde{a}_{u,0})$, we conclude that  $m_1$ is a simple root of $\tilde{a}_{u,0}$ as well. Furthermore, its initial form equals $\overline{m_1}= -\overline{b_0}/\overline{b_1}$. Since ${(\tilde{a}_{u,0}')}_{M_1}(\overline{m_1}) = \overline{b_1}\neq 0$, it follows from~\autoref{lm:multivariateHensel} that $m_1 \in \K$.
  \end{proof}

\section{Tree-Shape intersections: coefficients of $\tilde{\sextic}(x,z)$ and $\hat{\sextic}(z,y)$}\label{sec:CoefficientsModificationSextics}

In this section, we record the 40 relevant coefficients of the bivalent polynomials $\tilde{\sextic}(x,z)$ and $\hat{\sextic}(z,y)$ from~\eqref{eq:modificationxz_zy} obtained from planar projections of the re-embedding of $V(\sextic)$ induced by~\eqref{eq:idealsModificationTree}, under the assumption that $a_{12}=1$.
We start by writing the  coefficients of $\tilde{\sextic}(x,z)$ corresponding to interior points of its Newton polytope and those on the $x$-axis,  moving from left to right and from bottom to top. The remaining coefficients appear at the end of the list.

\small{\begin{flalign*}
 \tilde{a}_{00} &=\!- a_{03}a_{11}^3 \!-\! 3a_{03}a_{11}^2m_1 -\! 3a_{03}a_{11}m_1^2 \!-\! a_{03}m_1^3 \!+\! a_{02}a_{11}^2 \!+\! 2a_{02}a_{11}m_1 +\! a_{02}m_1^2 \!-\! a_{01}a_{11} \!-\! a_{01}m_1 +\! a_{00},&&
  \end{flalign*}
  }
\small{\begin{flalign*}
  \tilde{a}_{10} &=  \!- a_{11}^3a_{13} \!-\! 3a_{03}a_{11}^2a_{21} \!+\! a_{02}a_{11}^2a_{22} \!-\! 3a_{11}^2a_{13}m_1 -\! 6a_{03}a_{11}a_{21}m_1 +\! 2a_{02}a_{11}a_{22}m_1 -\! 3a_{11}a_{13}m_1^2 \!-\! && \\
  &3a_{03}a_{21}m_1^2 \!+\! a_{02}a_{22}m_1^2 \!-\! a_{13}m_1^3 \!-\! 3a_{03}a_{11}^2n_1 -\! 6a_{03}a_{11}m_1n_1 -\! 3a_{03}m_1^2n_1 +\! a_{02}a_{11}^2\du_1 +\! 2a_{02}a_{11}m_1\du_1 +\! && \\
  &a_{02}m_1^2\du_1 +\! 2a_{02}a_{11}a_{21} \!-\! 2a_{01}a_{11}a_{22} \!+\! 2a_{02}a_{21}m_1 -\! 2a_{01}a_{22}m_1 +\! 2a_{02}a_{11}n_1 +\! 2a_{02}m_1n_1 -\!
  2a_{01}a_{11}\du_1 -\! &&\\
  &2a_{01}m_1\du_1 -\! a_{01}a_{21} \!+\! 3a_{00}a_{22} \!+\! a_{11}m_1 +\! m_1^2 \!-\! a_{01}n_1 +\! 3a_{00}\du_1 +\! a_{10},&&
  \end{flalign*}
  }
\small{\begin{flalign*}
  \tilde{a}_{20}&=   \!- 3a_{11}^2a_{13}a_{21} \!-\! 3a_{03}a_{11}a_{21}^2 \!+\! 2a_{02}a_{11}a_{21}a_{22} \!-\! a_{01}a_{11}a_{22}^2 \!-\! a_{11}^3a_{23} \!-\! 6a_{11}a_{13}a_{21}m_1 -\! 3a_{03}a_{21}^2m_1 +\! && \\
  &2a_{02}a_{21}a_{22}m_1 -\! a_{01}a_{22}^2m_1 -\! 3a_{11}^2a_{23}m_1 -\! 3a_{13}a_{21}m_1^2 \!-\! 3a_{11}a_{23}m_1^2 \!-\! a_{23}m_1^3 \!-\! 3a_{11}^2a_{13}n_1 -\! 6a_{03}a_{11}a_{21}n_1 +\! && \\
  & 2a_{02}a_{11}a_{22}n_1 -\! 6a_{11}a_{13}m_1n_1 -\! 6a_{03}a_{21}m_1n_1 +\! 2a_{02}a_{22}m_1n_1 -\! 3a_{13}m_1^2n_1 -\! 3a_{03}a_{11}n_1^2 \!-\! 3a_{03}m_1n_1^2 \!+\! && \\
  &2a_{02}a_{11}a_{21}\du_1 -\! 2a_{01}a_{11}a_{22}\du_1 +\! 2a_{02}a_{21}m_1\du_1 -\! 2a_{01}a_{22}m_1\du_1 +\! 2a_{02}a_{11}n_1\du_1 +\! 2a_{02}m_1n_1\du_1 -\! a_{01}a_{11}\du_1^2 \!-\! && \\
  &a_{01}m_1\du_1^2 \!+\! a_{02}a_{21}^2 \!-\! 2a_{01}a_{21}a_{22} \!+\! 3a_{00}a_{22}^2 \!+\! 2a_{11}a_{22}m_1 +\! 2a_{22}m_1^2 \!+\! 2a_{02}a_{21}n_1 -\! 2a_{01}a_{22}n_1 +\! a_{02}n_1^2 \!-\! a_{11}^2\du_1 -\! && \\
  & 2a_{01}a_{21}\du_1 +\! 6a_{00}a_{22}\du_1 +\! m_1^2\du_1 -\! 2a_{01}n_1\du_1 +\! 3a_{00}\du_1^2 \!+\! 3a_{10}a_{22} \!+\! a_{21}m_1 +\! a_{11}n_1 +\! 2m_1n_1 +\! 3a_{10}\du_1 +\!  a_{20}, &&
  \end{flalign*}
  }
\small{\begin{flalign*}
  \tilde{a}_{30} &= \!- 3a_{11}a_{13}a_{21}^2 \!-\! a_{03}a_{21}^3 \!+\! a_{02}a_{21}^2a_{22} \!-\! a_{01}a_{21}a_{22}^2 \!+\! a_{00}a_{22}^3 \!-\! 3a_{11}^2a_{21}a_{23} \!-\! a_{11}^3a_{33} \!-\! 3a_{13}a_{21}^2m_1 +\! a_{11}a_{22}^2m_1 -\! &&\\
  &6a_{11}a_{21}a_{23}m_1 -\! 3a_{11}^2a_{33}m_1 +\! a_{22}^2m_1^2 \!-\! 3a_{21}a_{23}m_1^2 \!-\! 3a_{11}a_{33}m_1^2 \!-\! a_{33}m_1^3 \!-\! 6a_{11}a_{13}a_{21}n_1 -\! 3a_{03}a_{21}^2n_1 +\! &&\\
  &2a_{02}a_{21}a_{22}n_1 -\! a_{01}a_{22}^2n_1 -\! 3a_{11}^2a_{23}n_1 -\! 6a_{13}a_{21}m_1n_1 -\! 6a_{11}a_{23}m_1n_1 -\! 3a_{23}m_1^2n_1 -\! 3a_{11}a_{13}n_1^2 \!-\! 3a_{03}a_{21}n_1^2 \!+\! && \\
  &a_{02}a_{22}n_1^2 \!-\! 3a_{13}m_1n_1^2 \!-\! a_{03}n_1^3 \!+\!a_{02}a_{21}^2\du_1 -\! a_{11}^2a_{22}\du_1 -\! 2a_{01}a_{21}a_{22}\du_1 +\! 3a_{00}a_{22}^2\du_1 +\! a_{22}m_1^2\du_1 +\! 2a_{02}a_{21}n_1\du_1 -\!  && \\
  &2a_{01}a_{22}n_1\du_1 +\! a_{02}n_1^2\du_1 -\! a_{11}^2\du_1^2 \!-\! a_{01}a_{21}\du_1^2 \!+\! 3a_{00}a_{22}\du_1^2 \!-\! a_{11}m_1\du_1^2 \!-\! a_{01}n_1\du_1^2 \!+\! a_{00}\du_1^3 \!+\! 3a_{10}a_{22}^2 \!+\! a_{11}^2a_{32} \!+\! &&\\
  &2a_{21}a_{22}m_1 +\! 2a_{11}a_{32}m_1 +\! a_{32}m_1^2 \!+\! 2a_{11}a_{22}n_1 +\! 4a_{22}m_1n_1 -\! 2a_{11}a_{21}\du_1 +\! 6a_{10}a_{22}\du_1 +\! 2m_1n_1\du_1 +\! 3a_{10}\du_1^2 \!+\! && \\
  &3a_{20}a_{22} \!-\! a_{11}a_{31} \!-\! a_{31}m_1 +\! a_{21}n_1 +\! n_1^2 \!+\! 3a_{20}\du_1 +\! a_{30},&&
  \end{flalign*}
  }
\small{\begin{flalign*}
  \tilde{a}_{40} &=  \!- a_{21}a_{31} \!-\!a_{13}a_{21}^3 \!+\! a_{10}a_{22}^3 \!-\! 3a_{11}a_{21}^2a_{23} \!+\! a_{11}^2a_{22}a_{32} \!-\! 3a_{11}^2a_{21}a_{33} \!+\! a_{21}a_{22}^2m_1 -\! 3a_{21}^2a_{23}m_1 +\! 2a_{11}a_{22}a_{32}m_1 -\! &&\\
  &6a_{11}a_{21}a_{33}m_1 +\! a_{22}a_{32}m_1^2 \!-\! 3a_{21}a_{33}m_1^2 \!-\! 3a_{13}a_{21}^2n_1 +\! a_{11}a_{22}^2n_1 -\! 6a_{11}a_{21}a_{23}n_1 -\! 3a_{11}^2a_{33}n_1 +\! 2a_{22}^2m_1n_1 -\! &&\\
  &6a_{21}a_{23}m_1n_1  -\! 6a_{11}a_{33}m_1n_1 -\! 3a_{33}m_1^2n_1 -\! 3a_{13}a_{21}n_1^2 \!-\! 3a_{11}a_{23}n_1^2 \!-\! 3a_{23}m_1n_1^2 \!-\! a_{13}n_1^3 \!-\! 2a_{11}a_{21}a_{22}\du_1 +\! && \\
  &3a_{10}a_{22}^2\du_1 +\! a_{11}^2a_{32}\du_1 +\! 2a_{11}a_{32}m_1\du_1 +\! a_{32}m_1^2\du_1 +\! 2a_{22}m_1n_1\du_1 -\! 2a_{11}a_{21}\du_1^2 \!+\! 3a_{10}a_{22}\du_1^2 \!-\! a_{21}m_1\du_1^2 \!-\! && \\
  &a_{11}n_1\du_1^2 \!+\! a_{10}\du_1^3 \!+\! 3a_{20}a_{22}^2 \!-\! 2a_{11}a_{22}a_{31} \!+\! 2a_{11}a_{21}a_{32} \!-\! 2a_{22}a_{31}m_1 +\! 2a_{21}a_{32}m_1 +\! 2a_{21}a_{22}n_1 +\! 2a_{11}a_{32}n_1 +\!&&\\
  &2a_{32}m_1n_1 +\! 2a_{22}n_1^2 \!-\! a_{21}^2\du_1 +\! 6a_{20}a_{22}\du_1 -\! 2a_{11}a_{31}\du_1 -\! 2a_{31}m_1\du_1 +\! n_1^2\du_1 +\! 3a_{20}\du_1^2 \!+\! 3a_{22}a_{30} \!-\! a_{31}n_1 +\! 3a_{30}\du_1,&&
  \end{flalign*}
  }
\small{\begin{flalign*}
  \tilde{a}_{50} &= a_{21}^2a_{32}\!+\! a_{20}a_{22}^3 \!-\! a_{21}^3a_{23} \!-\! a_{11}a_{22}^2a_{31} \!+\! 2a_{11}a_{21}a_{22}a_{32} \!-\! 3a_{11}a_{21}^2a_{33} \!-\! a_{22}^2a_{31}m_1 -\! 3a_{21}a_{23}n_1^2  \!-\! 3a_{21}^2a_{33}m_1 +\! &&\\
  &a_{21}a_{22}^2n_1 -\! 3a_{21}^2a_{23}n_1 +\! 2a_{11}a_{22}a_{32}n_1 -\! 6a_{11}a_{21}a_{33}n_1 +\! 2a_{22}a_{32}m_1n_1 -\! 6a_{21}a_{33}m_1n_1 +\! a_{22}^2n_1^2 \!+\! 2a_{21}a_{22}a_{32}m_1 -\! &&\\
  &3a_{11}a_{33}n_1^2 \!-\! 3a_{33}m_1n_1^2 \!-\! a_{23}n_1^3 \!-\! a_{21}^2a_{22}\du_1 +\! 3a_{20}a_{22}^2\du_1 -\! 2a_{11}a_{22}a_{31}\du_1 +\! 2a_{11}a_{21}a_{32}\du_1 -\! 2a_{22}a_{31}m_1\du_1 +\! &&\\
  &2a_{21}a_{32}m_1\du_1 +\! 2a_{11}a_{32}n_1\du_1 +\! 2a_{32}m_1n_1\du_1 +\! a_{22}n_1^2\du_1 -\! a_{21}^2\du_1^2 \!+\! 3a_{20}a_{22}\du_1^2 \!-\! a_{11}a_{31}\du_1^2 \!-\! a_{31}m_1\du_1^2 \!-\! a_{21}n_1\du_1^2 \!+\! && \\
  &a_{20}\du_1^3 \!+\! 3a_{22}^2a_{30} \!-\! 2a_{21}a_{22}a_{31}   \!-\! 2a_{22}a_{31}n_1 +\! 2a_{21}a_{32}n_1 +\! a_{32}n_1^2 \!+\! 6a_{22}a_{30}\du_1 -\! 2a_{21}a_{31}\du_1 -\! 2a_{31}n_1\du_1 +\! 3a_{30}\du_1^2,&&
  \end{flalign*}
  }
\small{\begin{flalign*}
  \tilde{a}_{60} &= a_{22}^3a_{30} \!-\! a_{21}a_{22}^2a_{31} \!+\! a_{21}^2a_{22}a_{32} \!-\! a_{21}^3a_{33} \!-\! a_{22}^2a_{31}n_1 +\! 2a_{21}a_{22}a_{32}n_1 -\! 3a_{21}^2a_{33}n_1 +\! a_{22}a_{32}n_1^2 \!-\! 3a_{21}a_{33}n_1^2 \!-\! && \\
  &a_{33}n_1^3 \!+\! 3a_{22}^2a_{30}\du_1 -\! 2a_{21}a_{22}a_{31}\du_1 +\! a_{21}^2a_{32}\du_1 -\! 2a_{22}a_{31}n_1\du_1 +\! 2a_{21}a_{32}n_1\du_1 +\! a_{32}n_1^2\du_1 +\! 3a_{22}a_{30}\du_1^2 \!-\!&& \\
  &a_{21}a_{31}\du_1^2 \!-\! a_{31}n_1\du_1^2 \!+\! a_{30}\du_1^3,&& 
  \end{flalign*}
  }
\small{\begin{flalign*}
    \tilde{a}_{11} &=  3a_{11}^2a_{13} \!+\! 6a_{03}a_{11}a_{21} \!-\! 2a_{02}a_{11}a_{22} \!+\! 6a_{11}a_{13}m_1 +\! 6a_{03}a_{21}m_1 -\! 2a_{02}a_{22}m_1 +\! 3a_{13}m_1^2 \!+\! 6a_{03}a_{11}n_1 +\! && \\
    &6a_{03}m_1n_1 -\! 2a_{02}a_{11}\du_1 -\! 2a_{02}m_1\du_1 -\! 2a_{02}a_{21} \!+\! 2a_{01}a_{22} \!-\! 2a_{02}n_1 +\! 2a_{01}\du_1 -\! a_{11} \!-\! 2m_1, &&
  \end{flalign*}
  }
\small{\begin{flalign*}
      \tilde{a}_{21} &=  6a_{11}a_{13}a_{21} \!+\! 3a_{03}a_{21}^2 \!-\! 2a_{02}a_{21}a_{22} \!+\! a_{01}a_{22}^2 \!+\! 3a_{11}^2a_{23} \!+\! 6a_{13}a_{21}m_1 +\! 6a_{11}a_{23}m_1 +\! 3a_{23}m_1^2 \!+\! 6a_{11}a_{13}n_1 +\! &&\\
  &6a_{03}a_{21}n_1 -\! 2a_{02}a_{22}n_1 +\! 6a_{13}m_1n_1 +\! 3a_{03}n_1^2 \!-\! 2a_{02}a_{21}\du_1 +\! 2a_{01}a_{22}\du_1 -\! 2a_{02}n_1\du_1 +\! a_{01}\du_1^2 \!-\! 2a_{11}a_{22} &&\\
  &\!-\! 4a_{22}m_1 -\! 2m_1\du_1 -\! a_{21} \!-\! 2n_1,&&
  \end{flalign*}
  }
\small{\begin{flalign*}
  \tilde{a}_{31} &=  3a_{13}a_{21}^2 \!-\! a_{11}a_{22}^2 \!+\! 6a_{11}a_{21}a_{23} \!+\! 3a_{11}^2a_{33} \!-\! 2a_{22}^2m_1 +\! 6a_{21}a_{23}m_1 +\! 6a_{11}a_{33}m_1 +\! 3a_{33}m_1^2 \!+\! 6a_{13}a_{21}n_1 +\! &&\\
  &6a_{11}a_{23}n_1 +\! 6a_{23}m_1n_1 +\! 3a_{13}n_1^2 \!-\! 2a_{22}m_1\du_1 +\! a_{11}\du_1^2 \!-\! 2a_{21}a_{22} \!-\! 2a_{11}a_{32} \!-\! 2a_{32}m_1 -\! 4a_{22}n_1 -\! 2n_1\du_1 +\! a_{31},&&
  \end{flalign*}
  }
\small{\begin{flalign*}
  \tilde{a}_{41} &=  \!- a_{21}a_{22}^2 \!+\! 3a_{21}^2a_{23} \!-\! 2a_{11}a_{22}a_{32} \!+\! 6a_{11}a_{21}a_{33} \!-\! 2a_{22}a_{32}m_1 +\! 6a_{21}a_{33}m_1 -\! 2a_{22}^2n_1 +\! 6a_{21}a_{23}n_1 +\! 6a_{11}a_{33}n_1 +\! &&\\
 &6a_{33}m_1n_1 +\! 3a_{23}n_1^2 \!-\! 2a_{11}a_{32}\du_1 -\! 2a_{32}m_1\du_1 -\! 2a_{22}n_1\du_1 +\! a_{21}\du_1^2 \!+\! 2a_{22}a_{31} \!-\! 2a_{21}a_{32} \!-\! 2a_{32}n_1 +\! 2a_{31}\du_1,&&
  \end{flalign*}
  }
\small{\begin{flalign*}
  \tilde{a}_{12} &=  \!- 3a_{11}a_{13} \!-\! 3a_{03}a_{21} \!+\! a_{02}a_{22} \!-\! 3a_{13}m_1 -\! 3a_{03}n_1 +\! a_{02}\du_1 +\! 1,&&
  \end{flalign*}
  }
\small{\begin{flalign*}
\tilde{a}_{22} &=  \!- 3a_{13}a_{21} \!-\! 3a_{11}a_{23} \!-\! 3a_{23}m_1 -\! 3a_{13}n_1 +\! 2a_{22} \!+\! \du_1,&&
  \end{flalign*}
  }

\small{\begin{flalign*}
\tilde{a}_{32} &=  a_{22}^2 \!-\! 3a_{21}a_{23} \!-\! 3a_{11}a_{33} \!-\! 3a_{33}m_1 \!-\! 3a_{23}n_1 +\! a_{22}\du_1 +\! a_{32},&&
  \end{flalign*}
  }
\small{\begin{flalign*}
\tilde{a}_{01} &= 3a_{03}a_{11}^2 \!+\! 6a_{03}a_{11}m_1 +\! 3a_{03}m_1^2 \!-\! 2a_{02}a_{11} \!-\! 2a_{02}m_1 +\! a_{01}, \qquad \tilde{a}_{02} = \!-\!3a_{03}a_{11} \!-\! 3a_{03}m_1 +\! a_{02},&&
  \end{flalign*}
  }
\small{\begin{flalign*}
\tilde{a}_{51} & = a_{22}^2a_{31} \!-\! 2a_{21}a_{22}a_{32} \!+\! 3a_{21}^2a_{33} \!-\! 2a_{22}a_{32}n_1 +\! 6a_{21}a_{33}n_1 +\! 3a_{33}n_1^2 \!+\! 2a_{22}a_{31}\du_1 -\! 2a_{21}a_{32}\du_1 -\! 2a_{32}n_1\du_1 +\! a_{31}\du_1^2,&&
  \end{flalign*}
  }
\small{\begin{flalign*}
\tilde{a}_{42} & =  a_{22}a_{32} \!-\! 3a_{21}a_{33} \!-\! 3a_{33}n_1 +\! a_{32}\du_1,\qquad   \tilde{a}_{p3}=a_{p3} \quad\text{ for } p =0,1,2,3.&&
  \end{flalign*}
  }

\normalsize

Next, we record the coefficients of $\hat{\sextic}(z,y)$. First, we write those corresponding to points in the interior of the Newton polytope and those on the $y$-axis,   from top to bottom and from left to right. The remaining ones are can be found at the end of the list.
\small{\begin{flalign*}
  \hat{a}_{00} &=  a_{11}^2a_{20}a_{21} \!-\! a_{10}a_{11}a_{21}^2 \!+\! a_{00}a_{21}^3 \!-\! a_{11}^3a_{30} \!+\! 2a_{11}a_{20}a_{21}m_1 -\! a_{10}a_{21}^2m_1 -\! 3a_{11}^2a_{30}m_1 +\! a_{20}a_{21}m_1^2 \!-\! && \\
  &3a_{11}a_{30}m_1^2 \!-\! a_{30}m_1^3 \!+\! a_{11}^2a_{20}n_1 -\! 2a_{10}a_{11}a_{21}n_1 +\! 3a_{00}a_{21}^2n_1 +\! 2a_{11}a_{20}m_1n_1 -\! 2a_{10}a_{21}m_1n_1 +\! a_{20}m_1^2n_1 -\! && \\
  & a_{10}a_{11}n_1^2 \!+\! 3a_{00}a_{21}n_1^2 \!-\! a_{10}m_1n_1^2 \!+\! a_{00}n_1^3, &&
  \end{flalign*}
  }
\small{\begin{flalign*}
  \hat{a}_{01} &= a_{01}a_{21}^3 \!+\! a_{11}^2a_{20}a_{22} \!-\! 2a_{10}a_{11}a_{21}a_{22} \!+\! 3a_{00}a_{21}^2a_{22} \!-\! a_{11}^3a_{31} \!+\! a_{11}a_{21}^2m_1 +\! 2a_{11}a_{20}a_{22}m_1 -\! 2a_{10}a_{21}a_{22}m_1 -\! && \\
  &  3a_{11}^2a_{31}m_1 +\! a_{21}^2m_1^2 \!+\! a_{20}a_{22}m_1^2 \!-\! 3a_{11}a_{31}m_1^2 \!-\! a_{31}m_1^3 \!-\! a_{11}^2a_{21}n_1 +\! 3a_{01}a_{21}^2n_1 -\! 2a_{10}a_{11}a_{22}n_1 -\! 2a_{10}a_{21}n_1+\! && \\
  &  6a_{00}a_{21}a_{22}n_1 -\!  2a_{10}a_{22}m_1n_1 +\! a_{21}m_1^2n_1 -\! a_{11}^2n_1^2 \!+\! 3a_{01}a_{21}n_1^2 \!+\! 3a_{00}a_{22}n_1^2 \!-\! a_{11}m_1n_1^2 \!+\! a_{01}n_1^3 \!+\! a_{11}^2a_{20}\du_1 -\! && \\
  & 2a_{10}a_{11}a_{21}\du_1 +\! 3a_{00}a_{21}^2\du_1 +\! 2a_{11}a_{20}m_1\du_1 -\! 2a_{10}a_{21}m_1\du_1 +\! a_{20}m_1^2\du_1 -\! 2a_{10}a_{11}n_1\du_1 +\! 6a_{00}a_{21}n_1\du_1 -\! && \\
  & 2a_{10}m_1n_1\du_1 +\! 3a_{00}n_1^2\du_1 +\! 2a_{11}a_{20}a_{21} \!-\! a_{10}a_{21}^2 \!-\! 3a_{11}^2a_{30} \!+\! 2a_{20}a_{21}m_1 -\! 6a_{11}a_{30}m_1 -\! 3a_{30}m_1^2 \!+\! 2a_{11}a_{20}n_1 +\! && \\
  & 2a_{20}m_1n_1 -\! a_{10}n_1^2,&&\
  \end{flalign*}
  }
\small{\begin{flalign*}
  \hat{a}_{02} &=  a_{02}a_{21}^3 \!+\! 3a_{01}a_{21}^2a_{22} \!-\! a_{10}a_{11}a_{22}^2 \!+\! 3a_{00}a_{21}a_{22}^2 \!-\! a_{11}^3a_{32} \!+\! 2a_{11}a_{21}a_{22}m_1 -\! a_{10}a_{22}^2m_1 -\! 3a_{11}^2a_{32}m_1 +\! &&\\
  &2a_{21}a_{22}m_1^2 \!-\! 3a_{11}a_{32}m_1^2 \!-\! a_{32}m_1^3 \!+\! 3a_{02}a_{21}^2n_1 -\! a_{11}^2a_{22}n_1 +\! 6a_{01}a_{21}a_{22}n_1 +\! 3a_{00}a_{22}^2n_1 +\! a_{22}m_1^2n_1 +\! && \\
  &3a_{02}a_{21}n_1^2 \!+\! 3a_{01}a_{22}n_1^2 \!+\! a_{02}n_1^3 \!-\! a_{11}^2a_{21}\du_1 +\! 3a_{01}a_{21}^2\du_1 -\! 2a_{10}a_{11}a_{22}\du_1 +\! 6a_{00}a_{21}a_{22}\du_1 -\! 2a_{11}^2n_1\du_1 -\! && \\
  &2a_{10}a_{22}m_1\du_1 +\! a_{21}m_1^2\du_1 +\! 6a_{01}a_{21}n_1\du_1 +\! 6a_{00}a_{22}n_1\du_1 -\! 2a_{11}m_1n_1\du_1 +\! 3a_{01}n_1^2\du_1 -\! a_{10}a_{11}\du_1^2 \!+\! 3a_{00}a_{21}\du_1^2 \!-\! && \\
  & a_{10}m_1\du_1^2 \!+\! 3a_{00}n_1\du_1^2 \!+\! 2a_{11}a_{20}a_{22} \!-\! 2a_{10}a_{21}a_{22} \!-\! 3a_{11}^2a_{31} \!+\! a_{21}^2m_1 +\! 2a_{20}a_{22}m_1 -\! 6a_{11}a_{31}m_1 -\! 3a_{31}m_1^2 \!-\!&& \\
  & 2a_{11}a_{21}n_1 -\! 2a_{10}a_{22}n_1 -\! 2a_{11}n_1^2 \!-\! m_1n_1^2 \!+\! 2a_{11}a_{20}\du_1 -\! 2a_{10}a_{21}\du_1 +\! 2a_{20}m_1\du_1 -\! 2a_{10}n_1\du_1 +\! a_{20}a_{21} \!-\! && \\
  & 3a_{11}a_{30} \!-\! 3a_{30}m_1 +\! a_{20}n_1,&&
  \end{flalign*}
  }
\small{\begin{flalign*}
  \hat{a}_{03} &= \!-\! a_{30} \!-\!a_{11}a_{13}a_{21}^2 \!+\! a_{03}a_{21}^3 \!+\! 3a_{02}a_{21}^2a_{22} \!+\! 3a_{01}a_{21}a_{22}^2 \!+\! a_{00}a_{22}^3 \!+\! a_{11}^2a_{21}a_{23} \!-\! a_{11}^3a_{33} \!-\! a_{13}a_{21}^2m_1 +\! &&\\
  &a_{11}a_{22}^2m_1 +\! 2a_{11}a_{21}a_{23}m_1 -\! 3a_{11}^2a_{33}m_1 +\! a_{22}^2m_1^2 \!+\! a_{21}a_{23}m_1^2 \!-\! 3a_{11}a_{33}m_1^2 \!-\! a_{33}m_1^3 \!-\! 2a_{11}a_{13}a_{21}n_1 +\! &&\\
  &a_{20}\du_1 +\! 3a_{03}a_{21}^2n_1 +\! 6a_{02}a_{21}a_{22}n_1 +\! 3a_{01}a_{22}^2n_1 +\! a_{11}^2a_{23}n_1 -\! 2a_{13}a_{21}m_1n_1 +\! 2a_{11}a_{23}m_1n_1 +\! a_{23}m_1^2n_1 -\! && \\
  &a_{11}a_{13}n_1^2 \!+\! 3a_{03}a_{21}n_1^2 \!+\! 3a_{02}a_{22}n_1^2 \!-\! a_{13}m_1n_1^2 \!+\! a_{03}n_1^3 \!+\! 3a_{02}a_{21}^2\du_1 -\! a_{11}^2a_{22}\du_1 +\! 6a_{01}a_{21}a_{22}\du_1 +\! && \\
  &a_{22}m_1^2\du_1 +\! 3a_{00}a_{22}^2\du_1 +\! 6a_{02}a_{21}n_1\du_1 +\! 6a_{01}a_{22}n_1\du_1 +\! 3a_{02}n_1^2\du_1 -\! a_{11}^2\du_1^2 \!+\! 3a_{01}a_{21}\du_1^2 \!+\! 3a_{00}a_{22}\du_1^2 \!-\! &&  \\
  & a_{11}m_1\du_1^2 \!+\! 3a_{01}n_1\du_1^2 \!+\! a_{00}\du_1^3 \!-\! a_{10}a_{22}^2 \!-\! 3a_{11}^2a_{32} \!+\! 2a_{21}a_{22}m_1 -\! 6a_{11}a_{32}m_1 -\! 3a_{32}m_1^2 \!-\! 2a_{11}a_{22}n_1 -\! && \\
  & 2a_{11}a_{21}\du_1 -\! 2a_{10}a_{22}\du_1 -\! 4a_{11}n_1\du_1 -\! 2m_1n_1\du_1 -\! a_{10}\du_1^2 \!+\! a_{20}a_{22} \!-\! 3a_{11}a_{31} \!-\! 3a_{31}m_1 -\! a_{21}n_1 -\! n_1^2 ,&&
  \end{flalign*}
  }

\small{\begin{flalign*}
  \hat{a}_{04} &= \!-\!2a_{11}a_{13}a_{21}a_{22} \!+\! 3a_{03}a_{21}^2a_{22} \!+\! 3a_{02}a_{21}a_{22}^2 \!+\! a_{01}a_{22}^3 \!+\! a_{11}^2a_{22}a_{23} \!-\! 2a_{13}a_{21}a_{22}m_1 +\! 2a_{11}a_{22}a_{23}m_1 +\! &&\\
  &a_{22}a_{23}m_1^2 \!-\! 2a_{11}a_{13}a_{22}n_1 +\! 6a_{03}a_{21}a_{22}n_1 +\! 3a_{02}a_{22}^2n_1 -\! 2a_{13}a_{22}m_1n_1 +\! 3a_{03}a_{22}n_1^2 \!-\! 2a_{11}a_{13}a_{21}\du_1 +\! && \\
  &3a_{03}a_{21}^2\du_1 +\! 6a_{02}a_{21}a_{22}\du_1 +\! 3a_{01}a_{22}^2\du_1 +\! a_{11}^2a_{23}\du_1 -\! 2a_{13}a_{21}m_1\du_1 +\! 2a_{11}a_{23}m_1\du_1 +\! a_{23}m_1^2\du_1 -\! &&\\
  & 2a_{11}a_{13}n_1\du_1 +\! 6a_{03}a_{21}n_1\du_1 +\! 6a_{02}a_{22}n_1\du_1 -\! 2a_{13}m_1n_1\du_1 +\! 3a_{03}n_1^2\du_1 +\!3a_{02}a_{21}\du_1^2 \!+\! 3a_{01}a_{22}\du_1^2 \!+\! &&\\
  &3a_{02}n_1\du_1^2 \!+\! a_{01}\du_1^3 \!-\! a_{13}a_{21}^2 \!+\! 2a_{11}a_{21}a_{23} \!-\! 3a_{11}^2a_{33} \!+\! a_{22}^2m_1 +\! 2a_{21}a_{23}m_1 -\! 6a_{11}a_{33}m_1 -\! 3a_{33}m_1^2 \!-\! && \\
  &2a_{13}a_{21}n_1 +\! 2a_{11}a_{23}n_1 +\! 2a_{23}m_1n_1 -\! a_{13}n_1^2 \!-\! 2a_{11}a_{22}\du_1 -\! 2a_{11}\du_1^2 \!-\! m_1\du_1^2 \!-\! 3a_{11}a_{32} \!-\! 3a_{32}m_1 -\! &&\\
  &a_{22}n_1 -\! a_{21}\du_1 -\! 2n_1\du_1 -\! a_{31},&&
  \end{flalign*}
  }
\small{\begin{flalign*}
  \hat{a}_{05} &= \!-\!a_{11}a_{13}a_{22}^2 \!+\! 3a_{03}a_{21}a_{22}^2 \!+\! a_{02}a_{22}^3 \!-\! a_{13}a_{22}^2m_1 +\! 3a_{03}a_{22}^2n_1 -\! 2a_{11}a_{13}a_{22}\du_1 +\! 6a_{03}a_{21}a_{22}\du_1 +\! && \\
  & 3a_{02}a_{22}^2\du_1 -\! 2a_{13}a_{22}m_1\du_1 +\! 6a_{03}a_{22}n_1\du_1 -\! a_{11}a_{13}\du_1^2 \!+\! 3a_{03}a_{21}\du_1^2 \!+\! 3a_{02}a_{22}\du_1^2 \!-\! a_{13}m_1\du_1^2 \!+\! 3a_{03}n_1\du_1^2 \!+\! &&\\
  &a_{02}\du_1^3 \!-\! 2a_{13}a_{21}a_{22} \!+\! 2a_{11}a_{22}a_{23} \!+\! 2a_{22}a_{23}m_1 -\! 2a_{13}a_{22}n_1 -\! 2a_{13}a_{21}\du_1 +\! 2a_{11}a_{23}\du_1 +\! 2a_{23}m_1\du_1 -\! && \\
  & 2a_{13}n_1\du_1 +\! a_{21}a_{23} \!-\! 3a_{11}a_{33} \!-\! 3a_{33}m_1 +\! a_{23}n_1 -\! a_{22}\du_1 -\! \du_1^2 \!-\! a_{32},&&
  \end{flalign*}
  }
\small{\begin{flalign*}
  \hat{a}_{06} &= a_{03}a_{22}^3 \!+\! 3a_{03}a_{22}^2\du_1 +\! 3a_{03}a_{22}\du_1^2 \!+\! a_{03}\du_1^3 \!-\! a_{13}a_{22}^2 \!-\! 2a_{13}a_{22}\du_1 -\! a_{13}\du_1^2 \!+\! a_{22}a_{23} \!+\! a_{23}\du_1 -\! a_{33},&&
  \end{flalign*}
  }
\small{\begin{flalign*}
    \hat{a}_{11} &= \!-\!a_{11}a_{21}^2 \!-\! 2a_{11}a_{20}a_{22} \!+\! 2a_{10}a_{21}a_{22} \!+\! 3a_{11}^2a_{31} \!-\! 2a_{21}^2m_1 -\! 2a_{20}a_{22}m_1 +\! 6a_{11}a_{31}m_1 +\! 3a_{31}m_1^2 \!+\! && \\
    & 2a_{10}a_{22}n_1 -\! 2a_{21}m_1n_1 +\! a_{11}n_1^2 \!-\! 2a_{11}a_{20}\du_1 +\! 2a_{10}a_{21}\du_1 -\! 2a_{20}m_1\du_1 +\! 2a_{10}n_1\du_1 -\! 2a_{20}a_{21} \!+\! &&\\
    & 6a_{11}a_{30} \!+\! 6a_{30}m_1 -\! 2a_{20}n_1,&&
  \end{flalign*}
  }
\small{\begin{flalign*}
    \hat{a}_{12} &= \!-\!2a_{11}a_{21}a_{22} \!+\! a_{10}a_{22}^2 \!+\! 3a_{11}^2a_{32} \!-\! 4a_{21}a_{22}m_1 +\! 6a_{11}a_{32}m_1 +\! 3a_{32}m_1^2 \!-\! 2a_{22}m_1n_1 +\! 2a_{10}a_{22}\du_1 -\! &&\\
    &2a_{21}m_1\du_1 +\! 2a_{11}n_1\du_1 +\! a_{10}\du_1^2 \!-\! a_{21}^2 \!-\! 2a_{20}a_{22} \!+\! 6a_{11}a_{31} \!+\! 6a_{31}m_1 +\! n_1^2 \!-\! 2a_{20}\du_1 +\! 3a_{30},&&
  \end{flalign*}
  }
\small{\begin{flalign*}
    \hat{a}_{13} &=  a_{13}a_{21}^2 \!-\! a_{11}a_{22}^2 \!-\! 2a_{11}a_{21}a_{23} \!+\! 3a_{11}^2a_{33} \!-\! 2a_{22}^2m_1 -\! 2a_{21}a_{23}m_1 +\! 6a_{11}a_{33}m_1 +\! 3a_{33}m_1^2 \!+\! 2a_{13}a_{21}n_1 -\! &&\\
    &2a_{11}a_{23}n_1 -\! 2a_{23}m_1n_1 +\! a_{13}n_1^2 \!-\! 2a_{22}m_1\du_1 +\! a_{11}\du_1^2 \!-\! 2a_{21}a_{22} \!+\! 6a_{11}a_{32} \!+\! 6a_{32}m_1 +\! 2n_1\du_1 +\! 3a_{31},&&
  \end{flalign*}
  }
\small{\begin{flalign*}
    \hat{a}_{14} &=  2a_{13}a_{21}a_{22} \!-\! 2a_{11}a_{22}a_{23} \!-\! 2a_{22}a_{23}m_1 +\! 2a_{13}a_{22}n_1 +\! 2a_{13}a_{21}\du_1 -\! 2a_{11}a_{23}\du_1 -\! 2a_{23}m_1\du_1 +\!  &&\\
    & 2a_{13}n_1\du_1 -\! a_{22}^2\!-\! 2a_{21}a_{23} \!+\! 6a_{11}a_{33} \!+\! 6a_{33}m_1 -\! 2a_{23}n_1 +\! \du_1^2 \!+\! 3a_{32},&&
  \end{flalign*}
  }
\small{\begin{flalign*}
    \hat{a}_{21} &= a_{21}^2 \!+\! a_{20}a_{22} \!-\! 3a_{11}a_{31} \!-\! 3a_{31}m_1 +\! a_{21}n_1 +\! a_{20}\du_1 -\! 3a_{30},&&
  \end{flalign*}
  }
\small{\begin{flalign*}
        \hat{a}_{22} &=  2a_{21}a_{22} \!-\! 3a_{11}a_{32} \!-\! 3a_{32}m_1 +\! a_{22}n_1 +\! a_{21}\du_1 -\! 3a_{31},&&
  \end{flalign*}
  }
\small{\begin{flalign*}
        \hat{a}_{23} &= a_{22}^2 \!+\! a_{21}a_{23} \!-\! 3a_{11}a_{33} \!-\! 3a_{33}m_1 +\! a_{23}n_1 +\! a_{22}\du_1 -\! 3a_{32},&&
  \end{flalign*}
  }
\small{\begin{flalign*}
    \hat{a}_{10} &=
    \!-\!2a_{11}a_{20}a_{21} \!+\! a_{10}a_{21}^2 \!+\! 3a_{11}^2a_{30} \!-\! 2a_{20}a_{21}m_1 +\! 6a_{11}a_{30}m_1 +\! 3a_{30}m_1^2 \!-\! 2a_{11}a_{20}n_1 +\! 2a_{10}a_{21}n_1 -\! &&
    \\ & 2a_{20}m_1n_1 +\! a_{10}n_1^2,&&
  \end{flalign*}
  }
\small{\begin{flalign*}
    \hat{a}_{15}& =
a_{13}a_{22}^2 \!+\! 2a_{13}a_{22}\du_1 +\! a_{13}\du_1^2 \!-\! 2a_{22}a_{23} \!-\! 2a_{23}\du_1 +\! 3a_{33}, \qquad
    \hat{a}_{24} =
a_{22}a_{23} \!+\! a_{23}\du_1 -\! 3a_{33}, &&
  \end{flalign*}
  }
\small{\begin{flalign*}
    \hat{a}_{20} &=
a_{20}a_{21} \!-\! 3a_{11}a_{30} \!-\! 3a_{30}m_1  +\! a_{20}n_1, \qquad
\hat{a}_{3p} = a_{3p} \quad\text{ for } p =0,1,2,3.&&
  \end{flalign*}
  }

\normalsize

 \bibliographystyle{plain}

  \vspace{2ex}
  

  \noindent
\textbf{\small{Authors' addresses:}}
\smallskip
\

\noindent
\small{M.A.\ Cueto,  Mathematics Department, The Ohio State University, 231 W 18th Ave, Columbus, OH 43210, USA.
\\
\noindent \emph{Email address:} \texttt{cueto.5@osu.edu}}
\vspace{2ex}

\noindent
\small{Y.~Len, Mathematical Institute, University of St Andrews, North Haugh
St Andrews, KY16 9SS, UK.\\
\noindent \emph{Email address:} \texttt{yoav.len@st-andrews.ac.uk}}
\vspace{2ex}

\noindent
\small{H.~Markwig, Eberhard Karls Universit\"at T\"ubingen, Fachbereich Mathematik, Auf der Morgenstelle 10, 72108 T\"ubingen, Germany.
  \\
  \noindent \emph{Email address:} \texttt{hannah@math.uni-tuebingen.de}}
\vspace{2ex}

\noindent
\small{Y.~Ren, Mathematical Sciences \& Computer Science Building, Upper Mountjoy Campus, Stockton Road, Durham University, DH1 3LE, UK.
 \\
  \noindent \emph{Email address:} \texttt{yue.ren2@durham.ac.uk}}

\end{document}